%% file: main.tex
\pdfoutput=1 

\documentclass[12pt]{amsart}
\usepackage[margin=1in, marginparwidth=0.8in, marginparsep=0.1in]{geometry}
\usepackage{enumitem, verbatim}
\usepackage{microtype}
\allowdisplaybreaks

\usepackage[export]{adjustbox}
\usepackage{amsmath,amssymb,amsthm,mathtools}
\usepackage{graphicx}
\usepackage[hyphens]{url}
\usepackage[pdfusetitle]{hyperref}
\usepackage[nameinlink]{cleveref}
\usepackage{color}
\hypersetup{colorlinks,linkcolor=blue,citecolor=cyan}
\usepackage{physics}
\usepackage{asymptote}
\usepackage{standalone}
\usepackage{subcaption}
\usepackage{wasysym}
\usepackage{stmaryrd}
\usepackage{float}
\usepackage{subcaption}
\usepackage{bbm}
\usepackage{comment}
\usepackage{adjustbox}

\usepackage{tikz-cd, tikz-3dplot}
\usetikzlibrary{patterns, knots, calc, arrows.meta, decorations.markings, shapes.misc, math, hobby, intersections, spath3}
\usetikzlibrary{shadings}
\usepgflibrary {shadings}
\usetikzlibrary{snakes}

\tikzset{
    anchorbase/.style={baseline={([yshift=-0.5ex]current bounding box.center)}}
}
\tikzstyle directed=[postaction={decorate,decoration={markings,
    mark=at position #1 with {\arrow{>}}}}]
\tikzset{
    cross/.style={cross out, draw=black, minimum size=2*(#1-\pgflinewidth), inner sep=0pt, outer sep=0pt},
    cross/.default={1pt}
}
\tikzset{
    partial ellipse/.style args={#1:#2:#3}{
        insert path={+ (#1:#3) arc (#1:#2:#3)}
    }
}

\makeatletter
\renewcommand*\env@matrix[1][\arraystretch]{%
  \edef\arraystretch{#1}%
  \hskip -\arraycolsep
  \let\@ifnextchar\new@ifnextchar
  \array{*\c@MaxMatrixCols c}}
\makeatother

\makeatletter
\newcommand{\labeleditem}[2]{%
  \item[(#1)]%
  \edef\@currentlabel{#1}%
  \label{#2}%
}
\makeatother

\newtheorem{thm}{Theorem}[section]
\newtheorem{mainthm}{Theorem}
\renewcommand{\themainthm}{\Alph{mainthm}}
\newtheorem{cor}[thm]{Corollary}
\newtheorem{lem}[thm]{Lemma}
\newtheorem{prop}[thm]{Proposition}
\newtheorem{conj}[thm]{Conjecture}

\theoremstyle{definition}
\newtheorem{defn}[thm]{Definition}
\newtheorem{eg}[thm]{Example}

\newtheorem{rmk}[thm]{Remark}

\DeclareMathOperator{\Sk}{\mathrm{Sk}} 
\DeclareMathOperator{\SkAlg}{\mathrm{SkAlg}} 
\DeclareMathOperator{\dual}{d}
\DeclareMathOperator{\SQGM}{\mathrm{SQGM}} 
\DeclareMathOperator{\QT}{Q} 
\DeclareMathOperator{\SQTS}{\mathrm{SQTS}} 
\DeclareMathOperator{\evmap}{\mathrm{ev}} 
\DeclareMathOperator{\st}{\mathrm{st}} 
\DeclareMathOperator{\CT}{\mathsf{c}_{\mathbb{T}}} 
\DeclareMathOperator{\CB}{\mathsf{c}_{\mathbb{B}}} 
\DeclareMathOperator{\tri}{\triangle} 

\title{Compatibility of quantum trace and UV-IR maps}
\author[S. Panitch]{Samuel Panitch}
\address{Department of Mathematics, Yale University, New Haven, CT 06511, USA}
\email{\href{mailto:sam.panitch@yale.edu}{sam.panitch@yale.edu}}

\author[S. Park]{Sunghyuk Park}
\address{Department of Mathematics, Harvard University, Cambridge, MA 02138, USA \and Center of Mathematical Sciences and Applications, Harvard University, Cambridge, MA 02138, USA}
\email{\href{mailto:sunghyukpark@math.harvard.edu}{sunghyukpark@math.harvard.edu}}

\begin{document}

\maketitle

\begin{abstract}
This paper studies the connection between the quantum trace map \cite{BW, PP, GY}
-- which maps the $\mathfrak{sl}_2$-skein module to the quantum Teichm\"uller space for surfaces and to the quantum gluing module for 3-manifolds --
and the quantum UV-IR map \cite{NY}
-- which maps the $\mathfrak{gl}_2$-skein module to the $\mathfrak{gl}_1$-skein module of the branched double cover. 
We show that the two maps are compatible in a precise sense, and that the compatibility map is natural under changes of triangulation; 
for surfaces, this resolves a conjecture of Neitzke and Yan \cite{NY}. 
As a corollary, under a mild hypothesis on the 3-manifold, the quantum trace map can be recovered from the quantum UV-IR map, hence providing yet another construction of the 3d quantum trace map recently introduced in \cite{PP, GY}. 
\end{abstract}

\tableofcontents

\section{Introduction}
At the crossroads of low-dimensional topology and quantum algebra, skein modules have emerged as powerful invariants of 3-manifolds. 
Since their inception in the 1990s \cite{Przytycki-skein, Tur}, they have driven advances in quantum topology
and created bridges to areas spanning 3- and 4-dimensional TQFTs \cite{Walker, GJS, Jordan}, factorization homology \cite{MW, BZBJ, Cooke}, character varieties \cite{Bullock, PS}, and -- crucial to this work -- cluster algebras \cite{BW, Muller, JLSS}. 
Central to this cluster-theoretic picture are ``abelianization'' maps that, given an ideal triangulation of a surface, send its skein algebra to a quantum torus, and that, in the 3-manifold setting, produce certain quotients of quantum tori, thereby providing cluster-like coordinates on skein modules that are natural under changes of triangulation. 
There are two such maps known in the literature: the quantum trace map and the quantum UV-IR map. 

Let $\Sigma$ be a surface equipped with an ideal triangulation $\tau$. 
The \emph{quantum trace map} for $\Sigma$, introduced by Bonahon and Wong \cite{BW}, is an algebra homomorphism 
\[
\Tr_{\tau} : \SkAlg^{\mathfrak{sl}_2}_A(\Sigma) \rightarrow \SQTS_{\tau}(\Sigma)
\]
from the $\mathfrak{sl}_2$-skein algebra $\SkAlg^{\mathfrak{sl}_2}_A(\Sigma)$ of $\Sigma$ into a quantum torus called the \emph{square-root quantum Teichm\"uller space} $\SQTS_{\tau}(\Sigma)$, commonly known as the \emph{square-root Chekhov-Fock algebra} of $(\Sigma, \tau)$. 

The \emph{quantum UV-IR map} for $\Sigma$, introduced by Neitzke and Yan \cite{NY} drawing inspirations from \cite{GLM, Gabella}, is an algebra homomorphism
\[
F_{\tau} : \SkAlg^{\mathfrak{gl}_2}_q(\Sigma) \rightarrow \SkAlg^{\mathfrak{gl}_1}_q(\widetilde{\Sigma}_{\tau})
\]
from the $\mathfrak{gl}_2$-skein algebra $\SkAlg^{\mathfrak{gl}_2}_q(\Sigma)$ of $\Sigma$ to the $\mathfrak{gl}_1$-skein algebra $\SkAlg^{\mathfrak{gl}_1}_q(\widetilde{\Sigma}_{\tau})$ of the branched double cover $\widetilde{\Sigma}_{\tau}$ of $\Sigma$ associated to the ideal triangulation $\tau$. 

Both the quantum trace map and the quantum UV-IR map behave naturally under changes of triangulation: for different triangulations, the corresponding maps are related via quantum cluster transformations. 
Moreover, both of these maps admit generalizations to 3-manifolds. 
The first is the 3d quantum trace map for an ideal triangulation $\mathcal{T}$ of a 3-manifold $Y$,
\[
\Tr_{\mathcal{T}} : \Sk^{\mathfrak{sl}_2}_{A}(Y) \rightarrow \SQGM_{\mathcal{T}}(Y),
\]
which was recently introduced in \cite{PP, GY}.
Here, $\SQGM_{\mathcal{T}}(Y)$ denotes the 3d analog of the square-root quantum Teichm\"uller space, called the \emph{square-root quantum gluing module}.
The second is the 3d quantum UV-IR map
\[
F_{\mathcal{T}} : \Sk^{\mathfrak{gl}_2}_{q}(Y) \rightarrow \Sk^{\mathfrak{gl}_1}_{q}(\widetilde{Y}_{\mathcal{T}})
\]
which is described in Section \ref{subsec:3dUVIR} of this paper. 

While these two maps are both homomorphisms of skein algebras and modules into quantum tori and their quotients, 
it is worth noting that the construction of the quantum trace map is largely algebraic, whereas the quantum UV-IR map has a more geometric and topological flavor. 
The quantum trace map is defined explicitly in terms of generators and relations in the skein module, reflecting combinatorial data of the triangulation. 
In contrast, the quantum UV-IR map relies on a 1-dimensional foliation of the 3-manifold induced by the triangulation, which has its origin in symplectic geometry. 

\bigskip
The main purpose of this paper is to bridge the gap between the two maps -- the quantum trace map and the quantum UV-IR map -- and show that they are compatible in the following sense. 
Given a 3-manifold $Y$ with an ideal triangulation $\mathcal{T}$ and the associated branched double cover $\widetilde{Y}_\mathcal{T}$ (see Section \ref{subsec:3dSpecNet}), 
we construct a commutative square
\begin{equation}\label{eq:main-thm-commutative-square}
\begin{tikzcd}
\Sk^{\mathfrak{gl}_2}_{q}(Y) \arrow[d, "\pi"] \arrow[r, "F_{\mathcal{T}}"] & \Sk^{\mathfrak{gl}_1}_{q}(\widetilde{Y}_{\mathcal{T}}) \arrow[d, "\evmap"]\\
\Sk^{\mathfrak{sl}_2}_{A}(Y) \otimes \Sk^{\mathfrak{gl}_1}_{-A}(Y) \arrow[r, "\Tr_{\mathcal{T}} \otimes \mathrm{id}"] & \SQGM_{\mathcal{T}}(Y) \otimes \Sk^{\mathfrak{gl}_1}_{-A}(Y)
\end{tikzcd},
\end{equation}
where the four arrows are as follows. 
\begin{itemize}
\item The bottom horizontal arrow is the 3d quantum trace map \cite{PP} 
\[
\Tr_{\mathcal{T}} : \Sk^{\mathfrak{sl}_2}_{A}(Y) \rightarrow \SQGM_{\mathcal{T}}(Y)
\]
for the ideal triangulation $\mathcal{T}$ of $Y$, which we review in Section \ref{subsec:3d-quantum-trace}, tensored with the identity map on $\Sk_{-A}^{\mathfrak{gl}_1}(Y)$. 
\item The top horizontal arrow
\[
F_{\mathcal{T}} : \Sk^{\mathfrak{gl}_2}_{q}(Y) \rightarrow \Sk^{\mathfrak{gl}_1}_{q}(\widetilde{Y}_{\mathcal{T}})
\]
is the 3d generalization of the quantum UV-IR map of \cite{NY}, which we describe in detail in Section \ref{subsec:3dUVIR}. 
\item The left vertical arrow $\pi$ is the ``$\mathfrak{gl}_2$-$\mathfrak{sl}_2$ map'' 
\begin{align*}
\pi : \Sk^{\mathfrak{gl}_2}_{q}(Y) &\rightarrow \Sk^{\mathfrak{sl}_2}_{A}(Y) \otimes \Sk^{\mathfrak{gl}_1}_{-A}(Y)\\
[\vec{L}]^{\mathfrak{gl}_2} &\mapsto (-1)^{b(\vec{L})}[L]^{\mathfrak{sl}_2} \otimes [\vec{L}]^{\mathfrak{gl}_1}
\end{align*}
that factors the $\mathfrak{gl}_2$-skeins into tensor products of $\mathfrak{sl}_2$- and $\mathfrak{gl}_1$-skeins, 
defined in Section \ref{subsec:gl2tosl2}. 
\item The right vertical arrow is the ``evaluation map''
\[
\evmap : \Sk^{\mathfrak{gl}_1}_{q}(\widetilde{Y}_{\mathcal{T}}) \rightarrow \SQGM_{\mathcal{T}}(Y) \otimes \Sk^{\mathfrak{gl}_1}_{-A}(Y)
\]
which we construct in Sections \ref{sec:surface-compatibility} and \ref{sec: 3-manifold compatibility}. 
\end{itemize}

We do this by constructing such commutative squares locally (i.e., for elementary pieces) and then showing that they can be glued consistently. 

\begin{mainthm}[Compatibility theorem] \label{thm:main-thm-compatibility}
The quantum trace map is compatible with the quantum UV-IR map in the sense that there are commutative squares \eqref{eq:main-thm-commutative-square}. 
Moreover, these commutative squares are compatible with Pachner moves. 
\end{mainthm}

Along the way, we prove the conjecture of \cite[Sec. 9.2]{NY} on the relation between the 2d quantum trace map and the 2d quantum UV-IR map. 
\begin{mainthm}[Proof of Neitzke-Yan conjecture; detailed version in Theorem \ref{thm:NeitzkeYanConjIsTrue}] \label{thm:conjectureOfAndyAndFei}
The conjecture of Neitzke-Yan \cite{NY} on the relation between the 2d quantum trace map and the 2d quantum UV-IR map is true. 
\end{mainthm}

There is complementary work in the literature touching on related themes:
\cite{KLS} established a connection between the 2d quantum trace map and the map of Gabella \cite{Gabella}, 
while \cite{KorinmanQuesney} studied its relation to the non-abelianization map of \cite{GMNsn, HollandsNeitzke}, which may be viewed as the classical UV-IR map for character varieties. 
Our Theorem \ref{thm:conjectureOfAndyAndFei}, however, furnishes a direct and concrete relationship between the 2d quantum trace map and the quantum UV-IR map. 

Our compatibility theorem (Theorem \ref{thm:main-thm-compatibility}) shows, in particular, that the quantum trace map $\Tr_\mathcal{T}$, tensored with the identity map on $\Sk_{-A}^{\mathfrak{gl}_1}(Y)$, can be recovered from the quantum UV-IR map $F_\mathcal{T}$. 
Under a mild assumption on $Y$ to ensure that $\Sk_{-A}^{\mathfrak{gl}_1}(Y)$ is torsion-free, we can fully recover the quantum trace map from the quantum UV-IR map: 
\begin{mainthm}\label{thm:quantum-trace-from-quantum-UV-IR}
Suppose $Y$ is a $3$-manifold such that the intersection pairing between $H_1(Y;\mathbb{Z})$ and $H_2(Y;\mathbb{Z})$ is zero (e.g., $Y$ can be any knot complement). 
Then, for any ideal triangulation $\mathcal{T}$ of $Y$, the quantum trace map
\[
\Tr_{\mathcal{T}} : \Sk_{A}^{\mathfrak{sl}_2}(Y) \rightarrow \SQGM_{\mathcal{T}}(Y)
\]
can be recovered from the quantum UV-IR map
\[
F_{\mathcal{T}} : \Sk_{q}^{\mathfrak{gl}_2}(Y) \rightarrow \Sk_{q}^{\mathfrak{gl}_1}(\widetilde{Y}_{\mathcal{T}}). 
\]
Explicitly, for any framed, unoriented link $L \subset Y$, if $\vec{L}$ denotes the same link equipped with an arbitrary choice of orientation, then
\[
\Tr_{\mathcal{T}}([L]^{\mathfrak{sl}_2}) = p_{\vec{L}} \circ \evmap \circ F_{\mathcal{T}}([\vec{L}]^{\mathfrak{gl}_2}),
\]
where 
\begin{align*}
p_{\vec{L}} : \SQGM_{\mathcal{T}}(Y) \otimes \Sk^{\mathfrak{gl}_1}_{-A}(Y) &\rightarrow \SQGM_{\mathcal{T}}(Y) \\
x \otimes [\vec{K}] &\mapsto 
\begin{cases}
(-A)^{w(\vec{K},\vec{L})} x &\text{ if } [\vec{K}] = [\vec{L}] \text{ in }H_1(Y;\mathbb{Z}) \\
0 &\text{otherwise}
\end{cases},
\end{align*}
and $w(\vec{K},\vec{L})$ denote the relative writhe between the two framed oriented links $\vec{K}$ and $\vec{L}$, i.e., it is the unique integer such that $[\vec{K}]^{\mathfrak{gl}_1} = (-A)^{w(\vec{K},\vec{L})} [\vec{L}]^{\mathfrak{gl}_1}$ in $\Sk^{\mathfrak{gl}_1}_{-A}(Y)$. 
\end{mainthm}

We conclude by highlighting some promising avenues for future work. 
It should be possible to extend our analysis to higher rank. 
The $\mathfrak{sl}_n$ 2d quantum trace is constructed in \cite{LY}, while the $\mathfrak{gl}_3$ version of the 2d quantum UV-IR map is studied in \cite{NY2}.
It is natural to expect that these maps both admit 3d generalizations, and that both the 2d and 3d versions for higher rank fit into a similar commutative square to \eqref{eq:main-thm-commutative-square}. 

Even more interesting is to extend our compatibility theorem to closed surfaces by, e.g., comparing the quantum trace map of \cite{DetcherrySantharoubane} to the quantum UV-IR map associated to the Fenchel-Nielsen spectral networks \cite{HollandsNeitzke}.

\subsection*{Organization of the paper}
This paper is organized as follows. 

In Section \ref{sec:quantumtrace}, we recall the construction of both the 2d and 3d quantum trace maps, following \cite{BW,Le,PP}. 
This includes a review of stated $\mathfrak{sl}_2$-skein modules and their splitting homomorphisms. 
Moreover, we remind the reader of the naturality of the 2d quantum trace map with respect to flips (changes of the ideal triangulation), and prove a corresponding result for the 3d quantum trace map with respect to 2-3 Pachner moves. 

Section \ref{sec:quantumUVIR} is dedicated to the quantum UV-IR map. 
We begin by reviewing the 2d construction, following \cite{NY}, and demonstrating that it is natural with respect to flips.
Next, using the theory of 3d spectral networks developed in \cite{FN}, we generalize the map to ideally triangulated 3-manifolds. 
It is then shown that this 3d quantum UV-IR map is natural with respect to Pachner moves. 
We conclude the section by introducing a stated version of the quantum UV-IR map (Theorem \ref{thm:stated-quantum-UV-IR}), which will allow for computing the quantum UV-IR map locally. 
This stated version is paramount, as it brings the topological construction of the quantum UV-IR map closer to the algebraic formulation of the quantum trace map. 

With both a quantum trace map and a quantum UV-IR map in hand, Section \ref{sec:surface-compatibility} handles the compatibility of the maps for surfaces. 
We construct the projection map from stated $\mathfrak{gl}_2$-skeins to $\mathfrak{sl}_2$-skeins (Proposition \ref{prop:gl2tosl2}) and twist the product on the $\mathfrak{gl}_2$-skein modules so that this map is an algebra homomorphism in the surface case and a bimodule homomorphism in the 3-manifold case. 

The main idea in proving the compatibility of both the 2d and 3d maps is to first prove that the maps fit into a commutative square locally, and then ``glue'' these local squares together. 
To this end, the evaluation map for triangles is constructed in such a way that the maps are manifestly compatible on a triangle. 
It is a simple matter of pasting these local squares together to obtain the main theorem of the section, Theorem \ref{thm:2d-compatibility}, from which we get Theorem \ref{thm:conjectureOfAndyAndFei} as a corollary. 
Section \ref{sec:surface-compatibility} concludes with a proof that the compatibility between the 2d quantum UV-IR map and the 2d quantum trace is also natural with respect to changes in the ideal triangulation of the surface. 

Section \ref{sec: 3-manifold compatibility} extends the analysis of the previous section to the 3-manifold case. 
The commutative square for a face suspension is established and shown to respect the gluing relations imposed by the splitting maps. 
This culminates in Theorem \ref{thm:main-thm-compatibility}.
We go on to demonstrate that the compatibility of the 3d quantum trace map and 3d quantum UV-IR map is also natural with respect to 2-3 Pachner moves.
The remainder of the section shows that the established compatibility between the maps allows for the 3d quantum trace map to be recovered from the UV-IR map (Theorem \ref{thm:quantum-trace-from-quantum-UV-IR}). 

Section \ref{sec:examples} contains a concrete example of the compatibility established in Theorem \ref{thm:main-thm-compatibility} for a skein in the figure-8 knot complement $Y = S^3 \setminus \mathbf{4_1}$. 

Finally, the definitions and properties of stated $\mathfrak{gl}_2$- and $\mathfrak{gl}_1$-skein modules are detailed in the appendices. 

\subsection*{Acknowledgements}
We are deeply grateful to Andy Neitzke for uncountably many helpful discussions. 
We also thank Thang L\^e for stimulating conversations. 
Part of this work was carried out while S.Park was visiting Yale University, and he thanks both the department and especially Andy Neitzke for their hospitality. 

S.Park was supported in part by Simons Foundation through Simons Collaboration on Global Categorical Symmetries.

\section{Quantum trace map}\label{sec:quantumtrace}
In this section, we review the quantum trace map for surfaces \cite{BW} and for 3-manifolds \cite{PP, GY}.

\subsection{2d quantum trace map}
The quantum trace map of Bonahon and Wong \cite{BW} is an algebra homomorphism from the $\mathfrak{sl}_2$-skein algebra of a surface to its quantum Teichm\"uller space. 
It is a quantization of the classical trace map that expresses the loop coordinates on the $\mathfrak{sl}_2$-character variety in terms of the trace of matrices of Thurston's sheer coordinates on Teichm\"uller space. 
The construction of the quantum trace map can be best described in terms of stated skein modules, which we describe next.

\subsubsection{Stated $\mathfrak{sl}_2$-skeins}
Here we recall the definition of stated $\mathfrak{sl}_2$-skein modules \cite{Le, CL, CL2, PP} for boundary-marked 3-manifolds, following \cite[Sec. 3]{PP}. 

\begin{defn}[{\cite[Def. 3.1]{PP}}]
Let $Y$ be an oriented 3-manifold with boundary. 

A \emph{boundary marking} for $Y$ is a smoothly embedded oriented graph $\Gamma\subset \partial Y$ where every vertex is either a source or a sink.\footnote{In 2d (i.e. for stated skein algebras), the boundary markings will always be disjoint unions of intervals. In 3d, however, the vertices of the boundary markings play a crucial role.} 
For example, a disjoint union of oriented intervals on $\partial Y$ is a boundary marking with univalent vertices. 

A \emph{boundary-marked 3-manifold} is a pair $(Y,\Gamma)$ consisting of a 3-manifold $Y$ and a boundary marking $\Gamma$ on $\partial Y$. 

A \emph{ribbon tangle in $(Y,\Gamma)$} is a ribbon (i.e., framed)\footnote{We will allow half-twists of the ribbon, so by ``framing'', we mean a choice of a section of the $\mathbb{RP}^1$-bundle over the tangle, given by the unit normal bundle modulo $\mathbb{Z}/2$.} 
tangle in $Y$ whose boundary points lie on the boundary marking $\Gamma$ away from the vertices $V(\Gamma)$ and whose framing at the boundary points are parallel (or anti-parallel) to the tangent vector of the boundary marking. 

A \emph{stated ribbon tangle} is a ribbon tangle equipped with a labeling of each boundary point $p$ by a sign $\mu_p \in \{\pm\}$ (i.e., a ``state''). 
See Figure \ref{fig:stated-ribbon-tangle-example} for an example. 
\end{defn}
\begin{figure}[htbp]
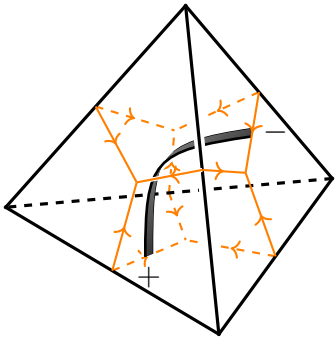

\centering
$
\tdplotsetmaincoords{60}{80}

$
\caption{An example of a stated ribbon tangle; this one is in a tetrahedron ($\approx B^3$) with boundary marking.}
\label{fig:stated-ribbon-tangle-example}
\end{figure}

\begin{defn}
Let $(Y, \Gamma)$ be a boundary marked 3-manifold, and let $R:=\mathbb{Z}[A^{\pm\frac{1}{2}},(-A^2)^{\pm \frac{1}{2}}]$ be the base ring. 
The \emph{stated $\mathfrak{sl}_2$-skein module} 
$\Sk^{\mathfrak{sl}_2}_A(Y, \Gamma)$ 
is the $R$-module generated by the isotopy classes of unoriented, stated ribbon tangles in $Y$, 
modulo the following skein relations:\footnote{
All the framed links, thought of as ribbons, are drawn flat on the page, except of course on the LHS of the relation (\ref{item:halfTwistRel}) which has a positive half-twist.
}
\footnote{
In the relations \eqref{item: sl2 stated skein rel 1} and \eqref{item: sl2 stated skein rel 2}, the orange line denotes a boundary marking, which should be thought of as sitting at a fixed height in a thickened version of the pictures drawn.
}
\footnote{
Here we use the manifestly symmetric version of the stated skein relations, as in \cite[Def. 3.3]{PP}, using the central half-twist relation \eqref{item:halfTwistRel}. 
In \cite[Rmk. 3.5]{PP}, it is explained why these relations are equivalent to the usual non-symmetric version of stated skein relations \cite{Le}. 
}
\begin{gather}
\vcenter{\hbox{

}}
\;, \label{item: sl2 stated skein rel 1}
\\
\vcenter{\hbox{
%
}}
\;, \label{item: sl2 stated skein rel 2}
\\
\vcenter{\hbox{
%
}}
\;. \label{item:halfTwistRel}
\end{gather}
\end{defn}

For an oriented surface $\Sigma$ with a choice of a set of marked points $P\subset \partial \Sigma$ on the boundary, the corresponding cylinder $(\Sigma\times I,P\times I)$ is a boundary-marked 3-manifold, where $P \times I$ is oriented according to the orientation of the interval $I$. 
The associated stated skein module $\Sk^{\mathfrak{sl}_2}_A(\Sigma\times I,P\times I)$ has a natural (unital, associative) algebra structure given by stacking along the $I$-direction. 
The \emph{stated skein algebra} \cite{Le, CL} of $(\Sigma, P)$ is defined to be this algebra: 
\[
\SkAlg^{\mathfrak{sl}_2}_A(\Sigma,P):=\Sk^{\mathfrak{sl}_2}_A(\Sigma\times I,P\times I).
\]
When $\Sigma$ is a punctured bordered surface (in the sense of \cite{Le, CL}) so that every connected component of $\partial \Sigma$ is an interval, we can choose one marked point for each boundary interval, and this gives a canonical choice of boundary marking for $\Sigma \times I$. 
In this case, for simplicity of notation, we will often denote the corresponding stated skein algebra simply by $\SkAlg^{\mathfrak{sl}_2}_A(\Sigma)$. 

\begin{eg}
An example of a punctured bordered surface is the $n$-gon $D_n$, which is the 2-disk $D^2$ with $n$ punctures on the boundary. 
The associated canonical boundary marking (in the example of $n=6$) is illustrated in Figure \ref{fig:D_n-example}. 
\begin{figure}[htbp]
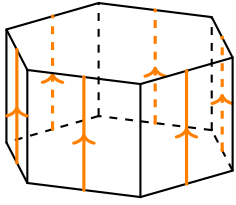

\centering
$
\vcenter{\hbox{
\tdplotsetmaincoords{70}{80}

}}
$
\caption{$D_6 \times I$ with the canonical boundary marking}
\label{fig:D_n-example}
\end{figure}
The associated stated skein algebras $\SkAlg_A^{\mathfrak{sl}_2}(D_n)$, especially the ones for biangles ($n=2$) and triangles ($n=3$), will play a crucial role in the construction of quantum trace maps. 
\end{eg}

The stated skein algebras are designed in such a way that there are natural maps associated to the operation of splitting a punctured bordered surface along an ideal arc. 

\begin{thm}[2d splitting map {\cite[Thm. 1]{Le}}]\label{thm:2d-splitting-map}
Let $\Sigma$ be a punctured bordered surface and let $\Sigma'$ be the surface obtained by splitting $\Sigma$ along an ideal arc $c$ embedded in $\Sigma$. 
Then, there is an algebra homomorphism (in fact an embedding)
\[
\sigma_c : \SkAlg^{\mathfrak{sl}_2}_A(\Sigma) \rightarrow \SkAlg^{\mathfrak{sl}_2}_A(\Sigma'),
\]
defined on stated tangles by
\[
[L] \mapsto  \sum_{\vec{\epsilon} \in \{\pm\}^{(c \times I) \cap L}}  [L_{\Sigma'}^{\vec{\epsilon}}],
\]
where $L_{\Sigma'}^{\vec{\epsilon}} \subset \Sigma' \times I$ denotes the stated tangle obtained by splitting $L \subset \Sigma \times I$ along $c \times I$ and assigning the state $\epsilon_p \in \{\pm\}$ to the two newly created boundary points corresponding to the intersection point $p \in (c \times I) \cap L$.\footnote{To be more precise, in order to get a ribbon tangle after splitting, we need to split after some isotopy of $L$ so that all the intersection points in $(c \times I) \cap L$ are of different height and the framing of $L$ at each intersection point is parallel to the $I$-direction. 
Part of the claim of this theorem is that the map is independent of the choice of such isotopy.} 
\end{thm}
The 2d splitting map is illustrated in Figure \ref{fig:2d-splitting-map}. 
\begin{figure}[htbp]
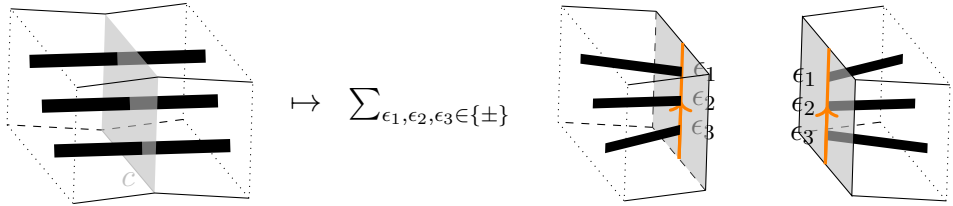

\centering
$
\vcenter{\hbox{
\tdplotsetmaincoords{20}{40}

}}
$
\caption{2d splitting map. The gray surface is $c \times I$, the surface we are cutting $\Sigma \times I$ along. }
\label{fig:2d-splitting-map}
\end{figure}

\begin{defn}\label{defn:bad-arcs}
A \emph{bad arc} is any stated tangle in $\Sigma \times I$ with 1 component that connects two boundary components abutting the same boundary puncture in a trivial way, with boundary states $-$ and $+$ in the counter-clockwise order when viewed from the boundary puncture; see Figure \ref{fig:2d-bad-arc}. 
\end{defn}
\begin{figure}[htbp]
\centering
$
\vcenter{\hbox{
\tdplotsetmaincoords{70}{-20}
\begin{tikzpicture}[tdplot_main_coords, scale=0.8]
\filldraw[black] (1, 0, {1-0.05}) .. controls (1/2, {sqrt(3)/2}, {1-0.05}) .. (-1/2, {sqrt(3)/2}, {1-0.05}) -- (-1/2, {sqrt(3)/2}, {1+0.05}) .. controls (1/2, {sqrt(3)/2}, {1+0.05}) .. (1, 0, {1+0.05}) -- cycle;
\draw[white, line width=2] (0, 0, 0) -- (0, 0, 2);
\draw[gray] (0, 0, 0) -- (0, 0, 2);
\filldraw[gray, opacity=0.3] (0, 0, 0) -- (-1, {sqrt(3)}, 0) -- (-1, {sqrt(3)}, 2) -- (0, 0, 2) -- cycle;
\filldraw[gray, opacity=0.3] (0, 0, 0) -- (2, 0, 0) -- (2, 0, 2) -- (0, 0, 2) -- cycle;
\draw[very thick, orange] (-1/2, {sqrt(3)/2}, 0) -- (-1/2, {sqrt(3)/2}, 2);
\draw[very thick, orange, ->] (-1/2, {sqrt(3)/2}, 0) -- (-1/2, {sqrt(3)/2}, 1);
\draw[very thick, orange] (1, 0, 0) -- (1, 0, 2);
\draw[very thick, orange, ->] (1, 0, 0) -- (1, 0, 1);
\node[anchor=west] at (1, 0, 1){$-$};
\node[anchor=east] at (-1/2, {sqrt(3)/2}, 1){$+$};
\end{tikzpicture}
}}.
$
\caption{A bad arc}
\label{fig:2d-bad-arc}
\end{figure}
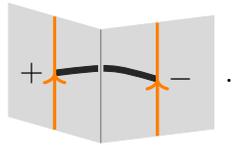

\begin{defn}[{\cite{BW}, \cite[Sec. 7]{CL}}]\label{defn:reduced-stated-skein-algebra}
The \emph{reduced stated skein algebra} $\overline{\SkAlg}^{\mathfrak{sl}_2}_A(\Sigma)$ is the quotient
\[
\overline{\SkAlg}^{\mathfrak{sl}_2}_A(\Sigma) := \SkAlg^{\mathfrak{sl}_2}_A(\Sigma)/I^{\mathrm{bad}},
\]
where $I^{\mathrm{bad}}$ denotes the two-sided ideal of $\SkAlg_A^{\mathrm{sl}_2}(\Sigma)$ generated by the bad arcs. 
\end{defn}
It is easy to see that the image of a bad arc in $\Sigma$ under the splitting map (Theorem \ref{thm:2d-splitting-map}) is in the two-sided ideal generated by bad arcs in $\Sigma'$, and it follows that the splitting map descends to a splitting map of reduced stated skein algebras:
\[
\overline{\sigma}_c : \overline{\SkAlg}^{\mathfrak{sl}_2}_A(\Sigma) \rightarrow \overline{\SkAlg}^{\mathfrak{sl}_2}_A(\Sigma').
\]
This map also turns out to be an embedding. 
As a special case, consider a punctured bordered surface $\Sigma$ equipped with an ideal triangulation $\tau$; 
let $\tau^{(d)}$ denote the set of all $d$-simplices in the triangulation. 
Then we have:
\begin{cor}[\cite{Le}]\label{cor:2d-splitting-triangles}
There is an algebra homomorphism (in fact an embedding)
\begin{align*}
\overline{\sigma}^{\mathfrak{sl}_2}_\tau : \overline{\SkAlg}^{\mathfrak{sl}_2}_A(\Sigma) &\rightarrow \bigotimes_{\tri\in \tau^{(2)}}\overline{\SkAlg}^{\mathfrak{sl}_2}_A(\tri),\\
[L] &\mapsto \sum_{\vec{\epsilon} \in \{\pm\}^{(E\times I) \cap L}}\otimes_{\tri\in \tau^{(2)}} [L_{\tri}^{\vec{\epsilon}}],
\end{align*}
where $E \subset \Sigma$ denotes the union of all the edges of the triangulation, and
$L_{\tri}^{\vec{\epsilon}}$ denotes the part of $L$ in $\tri \times I$ after splitting, with states assigned to newly created boundary points according to $\vec{\epsilon}$. 
\end{cor}

The reduced stated skein algebras of biangles and triangles, which for simplicity we will call the \emph{biangle algebra} and the \emph{triangle algebra} respectively, have particularly simple structures (e.g. they are quantum tori) and will play a crucial role in the construction of both 2d and 3d quantum trace maps. 
\begin{thm}[{\cite{BW}, \cite[Sec. 7.6-7.7]{CL}}]\label{thm:biangle-triangle-algebras}
\begin{enumerate}
\item The biangle algebra is
\[
\mathbb{B} := \overline{\SkAlg}_A^{\mathfrak{sl}_2}(D_2) \cong R[x,x^{-1}], 
\]
with the isomorphism given by
\[
\vcenter{\hbox{

}}
\leftrightarrow
x^{-1}
.
\]
\item The triangle algebra is
\[
\mathbb{T} := \overline{\SkAlg}_A^{\mathfrak{sl}_2}(D_3) \cong \frac{R\langle \alpha^{\pm 1}, \beta^{\pm 1}, \gamma^{\pm 1}\rangle}{\langle \beta\alpha = A\alpha\beta, \gamma\beta = A\beta\gamma, \alpha\gamma = A\gamma\alpha \rangle},
\]
with the isomorphism given by
\[
\vcenter{\hbox{
%
}}
\leftrightarrow
\alpha^{-1}
,
\]
and likewise for $\beta$ and $\gamma$. 
\end{enumerate}
\end{thm}

\subsubsection{2d quantum trace map}
Here we briefly review the construction of the 2d quantum trace map \cite{BW, Le, CL};
the details can be found in the cited references. 

\begin{defn}\label{defn:quantum-torus}
Let $\Gamma$ be a lattice equipped with an antisymmetric bilinear form 
$\langle \, , \rangle : \Gamma \times \Gamma \rightarrow \mathbb{Z}$. 
The \emph{quantum torus} $\QT_{\Gamma}$ associated to $\Gamma$ is the unital associative $R$-algebra with basis $\{x_\gamma\}_{\gamma \in \Gamma}$ and the product law
\[
x_{\gamma} x_{\gamma'} = A^{\frac{\langle \gamma, \gamma'\rangle}{2}}x_{\gamma + \gamma'}.
\]
Note, if we are given a basis $e_1, \cdots, e_n$ of the lattice, 
the quantum torus is generated by $x_1, \cdots, x_n$ (where $x_i := x_{e_i}$) and their inverses, 
subject to commutation relations
\[
x_i x_j = A^{\langle e_i, e_j\rangle} x_j x_i. 
\]

The \emph{Weyl-ordered product} $[x_{i_1} \cdots x_{i_k}]$ of $x_{i_1}, \cdots, x_{i_k}$ is defined to be
\[
[x_{i_1} \cdots  x_{i_k}] := x_{e_{i_1} + \cdots + e_{i_k}} = A^{-\frac{\sum_{1\leq j < l \leq k} \langle e_{i_{j}}, e_{i_{l}}\rangle}{2}} x_{i_1} \cdots  x_{i_k}. 
\]
The Weyl-ordered product is designed to be independent of the ordering of the variables. 
\end{defn}

\begin{defn}\label{def:extended-triangle-algebra}
Fix an invertible scalar $\CT \in R^\times$. 
The \emph{extended triangle algebra} $\widetilde{\mathbb{T}} = \widetilde{\mathbb{T}}_{\CT}$ is
an extension of the triangle algebra $\mathbb{T}$ (as in Theorem \ref{thm:biangle-triangle-algebras})
defined as
\[
\widetilde{\mathbb{T}} := \frac{R\langle a^{\pm 1}, b^{\pm 1}, c^{\pm 1}\rangle}{\langle ba = Aab, cb = Abc, ac = Aca \rangle},
\]
equipped with an embedding
\begin{align*}
\Tr_{\CT} : \mathbb{T} &\hookrightarrow \widetilde{\mathbb{T}} \\
\alpha &\mapsto \CT[bc], \\
\beta &\mapsto \CT[ca], \\
\gamma &\mapsto \CT[ab].
\end{align*}
We view the generators $a, b, c$ as being associated to the three edges of the triangle in counter-clockwise order, as in Figure \ref{fig:extended-triangle-algebra}. 
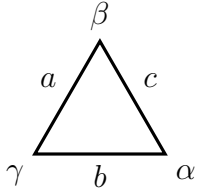
\begin{figure}[htbp]
\centering
$
\vcenter{\hbox{
\begin{tikzpicture}[scale=1.0]
\coordinate (alpha) at ({sqrt(3)/2}, -1/2);
\coordinate (beta) at (0, 1);
\coordinate (gamma) at ({-sqrt(3)/2}, -1/2);
\coordinate (a) at ({-sqrt(3)/4}, 1/4);
\coordinate (b) at (0, -1/2);
\coordinate (c) at ({sqrt(3)/4}, 1/4);
\draw[very thick] (alpha) -- (beta) -- (gamma) -- cycle;
\node[anchor = north west] at (alpha){$\alpha$}; 
\node[anchor = south] at (beta){$\beta$};
\node[anchor = north east] at (gamma){$\gamma$}; 
\node[anchor = south east] at (a){$a$};
\node[anchor = north] at (b){$b$};
\node[anchor = south west] at (c){$c$};
\end{tikzpicture}
}}
$
\caption{Edges $a, b, c$ (and their inverses) generate the extended triangle algebra $\widetilde{\mathbb{T}}$. }
\label{fig:extended-triangle-algebra}
\end{figure}
\end{defn}

Suppose $\Sigma$ is a punctured surface (without boundary) equipped with an ideal triangulation $\tau$. 
For each edge $e \in \tau^{(1)}$ of the ideal triangulation $\tau$, consider the following element $\hat{x}_e$ of the big quantum torus $\bigotimes_{\tri \in \tau^{(2)}} \widetilde{\mathbb{T}}$. 
\begin{enumerate}
\item If there are two ideal triangles $\tri_1, \tri_2 \in \tau^{(2)}$ abutting $e$, define
\[
\hat{x}_e := a_1 \otimes a_2,
\]
where $a_1, a_2$ are the edges of $\tri_1, \tri_2$ which are glued to give $e$, respectively. 
\item If there is only one ideal triangle $\tri$ abutting $e$, define
\[
\hat{x}_e := [a_1 a_2],
\]
where $a_1, a_2$ are the two edges of $\tri$ which are glued to give $e$. 
\end{enumerate}
By convention, we are suppressing the entries of the tensor product which are $1$. 
The element $\hat{x}_e \in \bigotimes_{\tri \in \tau^{(2)}} \widetilde{\mathbb{T}}$ is called the \emph{square-root quantized shear parameter} associated to the edge $e$ of the ideal triangulation $\tau$ of $\Sigma$. 
The idea behind the definition of $\hat{x}_e$ is that each edge $e \in \tau^{(1)}$ is obtained by gluing two bare edges (i.e. edges of the ideal triangles before gluing), as in Figure \ref{fig:square-root-quantized-shear-parameter}. 
\begin{figure}[htbp]
\centering
$
\vcenter{\hbox{
\begin{tikzpicture}[scale=1.0]
\coordinate (alpha) at (-1/2, {sqrt(3)/2});
\coordinate (beta) at (1, 0);
\coordinate (beta2) at (-2, 0);
\coordinate (gamma) at (-1/2, {-sqrt(3)/2});
\coordinate (a) at (1/4, {-sqrt(3)/4});
\coordinate (b) at (-1/2, 0);
\coordinate (c) at (1/4, {sqrt(3)/4});
\draw[very thick] (alpha) -- (beta) -- (gamma) -- cycle;
\draw[very thick] (alpha) -- (beta2) -- (gamma) -- cycle;
\node[left] at (-1/2, 0){$e$};
\end{tikzpicture}
}}
\;\;\rightsquigarrow\;\;
\vcenter{\hbox{
\begin{tikzpicture}[scale=1.0]
\begin{scope}
\coordinate (alpha) at (-1/2, {sqrt(3)/2});
\coordinate (beta2) at (-2, 0);
\coordinate (gamma) at (-1/2, {-sqrt(3)/2});
\coordinate (a) at (1/4, {-sqrt(3)/4});
\coordinate (b) at (-1/2, 0);
\coordinate (c) at (1/4, {sqrt(3)/4});
\draw[very thick] (alpha) -- (beta2) -- (gamma) -- cycle;
\node[left] at (-1/2, 0){$a_1$};
\end{scope}
\begin{scope}[shift={(1/2, 0)}]
\coordinate (alpha) at (-1/2, {sqrt(3)/2});
\coordinate (beta) at (1, 0);
\coordinate (gamma) at (-1/2, {-sqrt(3)/2});
\coordinate (a) at (1/4, {-sqrt(3)/4});
\coordinate (b) at (-1/2, 0);
\coordinate (c) at (1/4, {sqrt(3)/4});
\draw[very thick] (alpha) -- (beta) -- (gamma) -- cycle;
\node[right] at (-1/2, 0){$a_2$};
\end{scope}
\end{tikzpicture}
}}
$
\caption{An edge $e$ of triangulation $\tau$ is split into two bare edges $a_1$ and $a_2$.}
\label{fig:square-root-quantized-shear-parameter}
\end{figure}

\begin{defn}
The \emph{square-root quantum Teichm\"uller space} (a.k.a. \emph{square-root Chekhov-Fock algebra}), denoted $\SQTS_\tau(\Sigma)$, is the sub-quantum torus of $\bigotimes_{\tri \in \tau^{(2)}} \widetilde{\mathbb{T}}$ generated by $\{\hat{x}_e\}_{e\in \tau^{(1)}}$. 

In other words, $\SQTS_\tau(\Sigma)$ is the quantum torus $\QT_{\Gamma_\tau}$ for the lattice $\Gamma_\tau$ generated by the edges of the triangulation $\tau$, equipped with the antisymmetric bilinear form given by
\[
\langle e, e'\rangle := a_{ee'} - a_{e'e} \in \{0, \pm 1, \pm 2\},
\]
where $a_{ee'}$ denotes the number of angular sectors delimited by $e$ and $e'$ in the faces of $\tau$, with $e$ coming first in counter-clockwise order. 
\end{defn}

Combining the splitting map (Corollary \ref{cor:2d-splitting-triangles}) and the embedding of triangle algebras into the extended triangle algebras (Definition \ref{def:extended-triangle-algebra}), we obtain the following sequence of maps:
\[
\SkAlg_A^{\mathfrak{sl}_2}(\Sigma) = \overline{\SkAlg}_A^{\mathfrak{sl}_2}(\Sigma) 
\overset{\overline{\sigma}}{\hookrightarrow} 
\bigotimes_{\tri \in \tau^{(2)}} \overline{\SkAlg}_A^{\mathfrak{sl}_2}(\tri) 
\xhookrightarrow{\otimes_{\tri\in \tau^{(2)}}\Tr_{\CT}} 
\bigotimes_{\tri \in \tau^{(2)}} \widetilde{\mathbb{T}}. 
\]
It is easy to see that the image of the composition of these maps is contained in $\SQTS_\tau(\Sigma)$. 

\begin{thm}[\cite{BW, Le}]\label{thm:2d-quantum-trace-map}
There is an algebra embedding $\Tr_\tau:\SkAlg_A^{\mathfrak{sl}_2}(\Sigma) \rightarrow \SQTS_\tau(\Sigma)$
defined as the composition: 
\[
\begin{tikzcd}
\SkAlg_A^{\mathfrak{sl}_2}(\Sigma) \arrow[hookrightarrow, rrrd, swap, "\Tr_\tau"] \arrow[r, equal] & \overline{\SkAlg}_A^{\mathfrak{sl}_2}(\Sigma) \arrow[hookrightarrow, r, "\overline{\sigma}"] & \underset{\tri \in \tau^{(2)}}{\bigotimes} \overline{\SkAlg}_A^{\mathfrak{sl}_2}(\tri) \arrow[hookrightarrow, r, "\otimes_{\tri\in \tau^{(2)}}\Tr_{\CT}"] & \underset{\tri \in \tau^{(2)}}{\bigotimes} \widetilde{\mathbb{T}} \\
& & & \SQTS_\tau(\Sigma) \arrow[hookrightarrow, u, ""]
\end{tikzcd}.
\]
\end{thm}

\begin{rmk}
While the 2d quantum trace map $\Tr_\tau = \Tr_{\tau, \CT}$, as we presented in Theorem \ref{thm:2d-quantum-trace-map}, depends on the choice of an invertible scalar $\CT \in R^\times$, this dependence can be absorbed into rescaling of the generators of $\SQTS_{\tau}(\Sigma)$. 
That is, for any $\CT, \CT' \in R^\times$, we have a commutative diagram
\[
\begin{tikzcd}
 & \SQTS_{\tau}(\Sigma) \arrow[dd, "\lambda_{\CT \rightarrow \CT'}", "\rotatebox{90}{\(\sim\)}"']\\
\SkAlg_A^{\mathfrak{sl}_2}(\Sigma) \arrow[hookrightarrow, ur, "\Tr_{\tau, \CT}"] \arrow[hookrightarrow, dr, "\Tr_{\tau, \CT'}"] & \\
 & \SQTS_{\tau}(\Sigma)
\end{tikzcd},
\]
where 
\begin{align*}
\lambda_{\CT \rightarrow \CT'} : \SQTS_{\tau}(\Sigma) &\overset{\sim}{\rightarrow} \SQTS_{\tau}(\Sigma) \\
\hat{x}_e &\mapsto \frac{\CT'}{\CT}\hat{x}_e
\end{align*}
is an automorphism that simply rescales the generators of the quantum torus. 
\end{rmk}

\subsubsection{Naturality of the 2d quantum trace maps with respect to flips}\label{subsubsec:2d-quantum-trace-naturality}
Suppose $\tau$ and $\tau'$ are two ideal triangulations of $\Sigma$ that differ by a flip, i.e. they look identical outside of an ideal quadrilateral in which we change the triangulation as in Figure \ref{fig:flip}. 
\begin{figure}[htbp]
\centering
$
\vcenter{\hbox{
\begin{tikzpicture}[scale=1.0]
\draw[very thick] (0, 0) -- (2, 0) -- (2, 2) -- (0, 2) -- cycle;
\draw[very thick] (0, 0) -- (2, 2);
\node[above] at (1, 1){$\hat{x}$};
\node[below] at (1, 0){$\hat{y}$};
\node[right] at (2, 1){$\hat{z}$};
\node[above] at (1, 2){$\hat{v}$};
\node[left] at (0, 1){$\hat{w}$};
\end{tikzpicture}
}}
\;\;\rightsquigarrow\;\;
\vcenter{\hbox{
\begin{tikzpicture}[scale=1.0]
\draw[very thick] (0, 0) -- (2, 0) -- (2, 2) -- (0, 2) -- cycle;
\draw[very thick] (0, 2) -- (2, 0);
\node[above] at (1, 1){$\hat{x}'$};
\node[below] at (1, 0){$\hat{y}$};
\node[right] at (2, 1){$\hat{z}$};
\node[above] at (1, 2){$\hat{v}$};
\node[left] at (0, 1){$\hat{w}$};
\end{tikzpicture}
}}
$
\caption{A flip on edge $x$}
\label{fig:flip}
\end{figure}
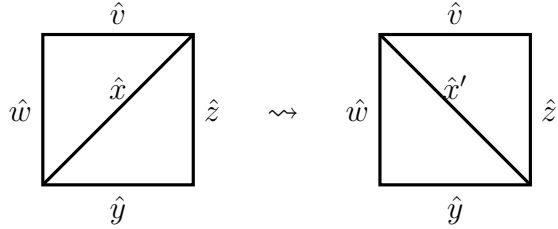
Then, there is an algebra isomorphism between (some appropriate completions of) the square-root quantum Teichm\"uller spaces given by\footnote{
Here, we describe this for $\CT = 1$; in order to get a transition map compatible with the rescaled quantum trace map $\Tr_{\CT}$, replace every $\hat{x}'$ by $\CT \hat{x}'$. 
}
\begin{align*}
\theta_{\tau \rightarrow \tau'} : \widehat{\SQTS}_{\tau}(\Sigma) &\overset{\sim}{\rightarrow} \widehat{\SQTS}_{\tau'}(\Sigma)\\
\hat{x} &\mapsto \hat{x}'^{-1},\\
\hat{y} &\mapsto \hat{y} f(A^{\frac{1}{2}}\hat{x}') = f(A^{-\frac{1}{2}}\hat{x}') \hat{y},\\
\hat{z} &\mapsto \hat{z} g(A^{-\frac{1}{2}}\hat{x}') = g(A^{\frac{1}{2}}\hat{x}') \hat{z},\\
\hat{v} &\mapsto \hat{v} f(A^{\frac{1}{2}}\hat{x}') = f(A^{-\frac{1}{2}}\hat{x}') \hat{v},\\
\hat{w} &\mapsto \hat{w} g(A^{-\frac{1}{2}}\hat{x}') = g(A^{\frac{1}{2}}\hat{x}') \hat{w},\\
\end{align*}
where $f$ and $g$ are some versions of the quantum dilogarithm given by
\begin{align*}
f(x) &:= x \exp\qty(\sum_{d\geq 1} \frac{(-1)^d}{d(A^d + A^{-d})}x^{2d}) = \exp\qty(\sum_{d\geq 1} \frac{(-1)^d}{d(A^d + A^{-d})}x^{-2d}),\\
g(x) &:= \exp\qty(-\sum_{d\geq 1} \frac{(-1)^d}{d(A^d + A^{-d})}x^{2d}) = x \exp\qty(-\sum_{d\geq 1} \frac{(-1)^d}{d(A^d + A^{-d})}x^{-2d}),
\end{align*}
and $\hat{e} \mapsto \hat{e}$ for every edge $e \in \tau^{(1)}$ other than the 5 edges involved in the flip. 
Note, $f$ and $g$ satisfy
\begin{align*}
f(x)g(x) &= x, \\
f(x)g(x^{-1}) &= 1, \\
f(A^{-\frac{1}{2}}x)f(A^{\frac{1}{2}}x) &= \frac{1}{1+x^{-2}}, \\
g(A^{-\frac{1}{2}}x)g(A^{\frac{1}{2}}x) &= 1+x^2.
\end{align*}
From these identities, it follows that, if we restrict to the even part of the algebra (i.e., the subalgebra generated by all monomials whose total degree of variables associated to each triangle is even), the transition map $\theta_{\tau \rightarrow \tau'}$ is given by (see \cite{Hiatt}, \cite[Sec. 7]{BW})
\begin{align*}
[\hat{y}\hat{z}] &\mapsto [\hat{y}\hat{x}'\hat{z}], \\
[\hat{y}^{-1}\hat{z}] &\mapsto \hat{y}^{-1}(\hat{x}' + \hat{x}'^{-1})\hat{z}, \\
[\hat{w}\hat{x}\hat{y}] &\mapsto [\hat{w}\hat{y}], \\
[\hat{w}\hat{x}^{-1}\hat{y}] &\mapsto [\hat{w}\hat{x}'^2 \hat{y}], \\
[\hat{w}^{-1}\hat{x}\hat{y}] &\mapsto A^{-\frac{1}{2}} \hat{w}^{-1} \frac{\hat{x}'^{-1}}{\hat{x}' + \hat{x}'^{-1}} \hat{y}, \\
[\hat{w}^{-1}\hat{x}^{-1}\hat{y}] &\mapsto A^{-\frac{1}{2}} \hat{w}^{-1} \frac{\hat{x}'}{\hat{x}' + \hat{x}'^{-1}} \hat{y}, \\
[\hat{y}\hat{x}\hat{v}] &\mapsto \hat{y} \frac{1}{\hat{x}' + \hat{x}'^{-1}} \hat{v}, \\
[\hat{y}\hat{x}^{-1}\hat{v}] &\mapsto \hat{y} \frac{\hat{x}'^2}{\hat{x}' + \hat{x}'^{-1}} \hat{v}, \\
[\hat{y}^{-1}\hat{x}\hat{v}] &\mapsto [\hat{y}^{-1}\hat{x}'^{-1}\hat{v}], \\
[\hat{y}^{-1}\hat{x}^{-1}\hat{v}] &\mapsto [\hat{y}^{-1}\hat{x}'\hat{v}], \\
[\hat{w}\hat{x}\hat{z}] &\mapsto \hat{w} (\hat{x}' + \hat{x}'^{-1}) \hat{z}, \\
[\hat{w}\hat{x}^{-1}\hat{z}] &\mapsto \hat{w} \hat{x}'^2(\hat{x}' + \hat{x}'^{-1}) \hat{z}, \\
[\hat{w}^{-1}\hat{x}\hat{z}] &\mapsto [\hat{w}^{-1}\hat{x}'^{-1}\hat{z}], \\
[\hat{w}^{-1}\hat{x}^{-1}\hat{z}] &\mapsto [\hat{w}^{-1}\hat{x}'\hat{z}], 
\end{align*}
and likewise for monomials obtained by a 180 degree rotation (i.e. after simultaneously replacing $\hat{y} \leftrightarrow \hat{v}$ and $\hat{z} \leftrightarrow \hat{w}$). 
In particular, when restricted to the even part, $\theta_{\tau \rightarrow \tau'}$ restricts to an algebra isomorphism between fractional division algebras of the quantum tori. 

These algebra isomorphisms satisfy the pentagon relation (see Figure \ref{fig:pentagon-relation}), 
and as a result, if $\tau$ and $\tau'$ are any two ideal triangulations of $\Sigma$, then we can choose any sequence of flips from $\tau$ to $\tau'$ and define $\theta_{\tau \rightarrow \tau'} : \widehat{\SQTS}_{\tau}(\Sigma) \overset{\sim}{\rightarrow} \widehat{\SQTS}_{\tau'}(\Sigma)$ to be the composition of the transition maps associated to those flips; 
by the pentagon relation, the resulting map $\theta_{\tau \rightarrow \tau'}$ is independent of the choice of the sequence of flips from $\tau$ to $\tau'$. 
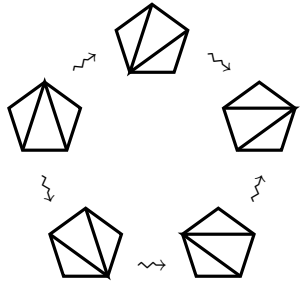
\begin{figure}[htbp]
\centering
$
\vcenter{\hbox{
\begin{tikzpicture}[scale=0.5]
\newcommand*{\defcoords}{
\coordinate (a) at ({cos(1/20*360)}, {sin(1/20*360)});
\coordinate (b) at ({cos(5/20*360)}, {sin(5/20*360)});
\coordinate (c) at ({cos(9/20*360)}, {sin(9/20*360)});
\coordinate (d) at ({cos(13/20*360)}, {sin(13/20*360)});
\coordinate (e) at ({cos(17/20*360)}, {sin(17/20*360)});
}
\begin{scope}[shift={($3*({cos(1/20*360)}, {sin(1/20*360)})$)}]
\defcoords
\draw[very thick] (a) -- (b) -- (c) -- (d) -- (e) -- cycle;
\draw[very thick] (c) -- (a) -- (d);
\end{scope}
\begin{scope}[shift={($3*({cos(5/20*360)}, {sin(5/20*360)})$)}]
\defcoords
\draw[very thick] (a) -- (b) -- (c) -- (d) -- (e) -- cycle;
\draw[very thick] (a) -- (d) -- (b);
\end{scope}
\begin{scope}[shift={($3*({cos(9/20*360)}, {sin(9/20*360)})$)}]
\defcoords
\draw[very thick] (a) -- (b) -- (c) -- (d) -- (e) -- cycle;
\draw[very thick] (e) -- (b) -- (d);
\end{scope}
\begin{scope}[shift={($3*({cos(13/20*360)}, {sin(13/20*360)})$)}]
\defcoords
\draw[very thick] (a) -- (b) -- (c) -- (d) -- (e) -- cycle;
\draw[very thick] (b) -- (e) -- (c);
\end{scope}
\begin{scope}[shift={($3*({cos(17/20*360)}, {sin(17/20*360)})$)}]
\defcoords
\draw[very thick] (a) -- (b) -- (c) -- (d) -- (e) -- cycle;
\draw[very thick] (a) -- (c) -- (e);
\end{scope}
\node[rotate=-36] at ($3*({cos(3/20*360)}, {sin(3/20*360)})$){$\rightsquigarrow$};
\node[rotate=36] at ($3*({cos(7/20*360)}, {sin(7/20*360)})$){$\rightsquigarrow$};
\node[rotate=-72] at ($3*({cos(11/20*360)}, {sin(11/20*360)})$){$\rightsquigarrow$};
\node[rotate=0] at ($3*({cos(15/20*360)}, {sin(15/20*360)})$){$\rightsquigarrow$};
\node[rotate=72] at ($3*({cos(19/20*360)}, {sin(19/20*360)})$){$\rightsquigarrow$};
\end{tikzpicture}
}}
$
\caption{Pentagon relation ensures that different sequences of flips from $\tau$ to $\tau'$ induce the same transition map.}
\label{fig:pentagon-relation}
\end{figure}
The 2d quantum trace map is compatible with these transition maps; 
the following diagram commutes:\footnote{
Here, we are setting $\CT = 1$. 
For rescaled quantum trace maps, we need to rescale the transition maps accordingly; see the previous footnote. 
}
\[
\begin{tikzcd}
 & \widehat{\SQTS}_{\tau}(\Sigma) \arrow[dd, "\theta_{\tau \rightarrow \tau'}", "\rotatebox{90}{\(\sim\)}"']\\
\SkAlg_A^{\mathfrak{sl}_2}(\Sigma) \arrow[hookrightarrow, ur, "\Tr_{\tau}"] \arrow[hookrightarrow, dr, "\Tr_{\tau'}"] & \\
 & \widehat{\SQTS}_{\tau'}(\Sigma)
\end{tikzcd}.
\]

\subsection{3d quantum trace map}\label{subsec:3d-quantum-trace}

In this subsection, we review the construction of the 3d quantum trace map of \cite{PP}. 

\subsubsection{Bimodule structure and the splitting map}
Let $(Y,\Gamma)$ be a boundary-marked 3-manifold. 
A crucial observation \cite[Sec. 3.1]{PP} is that $\Sk^{\mathfrak{sl}_2}_A(Y,\Gamma)$ has a natural module structure over $\SkAlg^{\mathfrak{sl}_2}_A(D_n)$ for each vertex of the boundary marking, where $n$ is the degree of the vertex. 
This can be seen as follows. 
For any vertex $v$ of $\Gamma$ of degree $\deg v = n$, 
we can take the complement of a small neighborhood of $v$ to get a local picture homeomorphic to a cylinder with $n$ boundary markings; see Figure \ref{fig: module structure}. 
\begin{figure}[htbp]
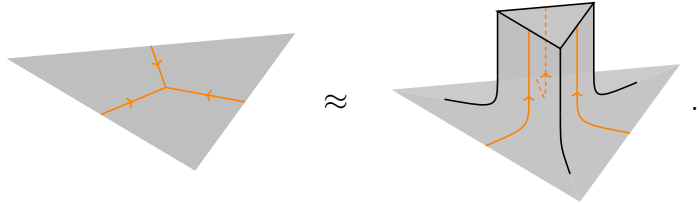

\centering
\begin{math}
\vcenter{\hbox{
\includestandalone[scale=0.5]{figures/module_structureLHS}
}}
\;\approx\;
\vcenter{\hbox{
\includestandalone[scale=0.5]{figures/module_structureRHS}
}}.
\end{math}
\caption{$\SkAlg(D_{\deg v})$-module structure at $v \in V(\Gamma)$}
\label{fig: module structure}
\end{figure}
The action of a stated skein in $\SkAlg(D_n)$ is then obtained by stacking it on top of this cylinder; this action is either a left or right action depending on whether the vertex is a sink or a source. 
For simplicity of notation, let's write
\begin{equation}\label{eq: vertex skein algebra}
\SkAlg_A^{\mathfrak{sl}_2}(V(\Gamma)^{\pm}) := \bigotimes_{v\in V(\Gamma)^{\pm}} \SkAlg_A^{\mathfrak{sl}_2}(D_{\deg v})
\end{equation}
where $V(\Gamma)^{+}$ and $V(\Gamma)^{-}$ denote the set of sink and source vertices of $\Gamma$, respectively. 
The following proposition summarizes this module structure: 
\begin{prop}[{\cite[Prop. 3.13]{PP}}]\label{prop:bimodule-structure}
The stated skein module $\Sk^{\mathfrak{sl}_2}_A(Y,\Gamma)$ has a natural 
$\SkAlg_A^{\mathfrak{sl}_2}(V(\Gamma)^{+})$--$\SkAlg_A^{\mathfrak{sl}_2}(V(\Gamma)^{-})$-bimodule structure. 
\end{prop}
This bimodule structure plays an essential role in understanding the behavior of stated skein modules under cutting a 3-manifold into simpler pieces.
In the natural 3d analog of the splitting homomorphism (Theorem \ref{thm:2d-splitting-map}), the codomain must be the relative tensor product  with respect to these bimodule structures; 
see \cite[Sec. 3.2]{PP}. 
For our purposes, it is enough to consider the special case of the splitting homomorphism corresponding to cutting an ideally triangulated 3-manifold into elementary pieces called \emph{face suspensions}, which we review below in Theorem \ref{thm:3d-splitting-map-for-Sf}. 
Readers interested in the general form of 3d splitting homomorphisms can refer to \cite[Thms. 3.21, 3.24]{PP}. 

As in the 2d case, the 3d quantum traces we are going to define factor through a quotient of the stated skein module obtained by setting all of the bad arcs to $0$. 
Recall from Definition \ref{defn:bad-arcs} that \emph{bad arcs} are stated tangles of a specific form near the boundary;
for each vertex $v \in V(\Gamma)$ of the boundary marking $\Gamma$, the bad arcs of $\SkAlg^{\mathfrak{sl}_2}_A(D_{\deg v})$ are shown in Figure \ref{fig:bad-arcs-3d}.\footnote{
The pictures shown are projections of the stated tangles in $(Y,\Gamma)$ on the boundary $\partial Y$ in a neighborhood of $v$, viewed from \emph{outside} of $Y$ (so, mirror the pictures when viewed from \emph{inside} of $Y$). 
The dashed orange lines indicate that there may be more edges of $\Gamma$ incident to $v$ in those regions of $\partial Y$. 
}
\begin{figure}[htbp]
\centering
$
\vcenter{\hbox{
\begin{tikzpicture}[scale=0.8]
\coordinate (o) at (0, 0);
\coordinate (a) at ({sqrt(3)}, -1);
\coordinate (a1) at ({sqrt(3)/3}, -1/3);
\coordinate (b) at (-{sqrt(3)}, -1);
\coordinate (b1) at (-{sqrt(3)/3}, -1/3);
\draw[line width=3] (a1) arc (-30:-150:2/3);
\begin{scope}[very thick,decoration={
    markings,
    mark=at position 0.5 with {\arrow{>}}}
    ] 
    \draw[postaction={decorate}, orange, ultra thick] (a) -- (o);
    \draw[postaction={decorate}, orange, ultra thick] (b) -- (o);
\end{scope}
\filldraw[orange] (o) circle (0.1em);
\draw[thick, orange, dashed] (0, 0) [partial ellipse=-30 : 210 : 1.0];
\node[anchor = south] at (a1) {$-$};
\node[anchor = south] at (b1) {$+$};
\end{tikzpicture}
}}
,\;
\vcenter{\hbox{
\begin{tikzpicture}[scale=0.8]
\coordinate (o) at (0, 0);
\coordinate (a) at ({sqrt(3)}, 1);
\coordinate (a1) at ({sqrt(3)/3}, 1/3);
\coordinate (b) at (-{sqrt(3)}, 1);
\coordinate (b1) at (-{sqrt(3)/3}, 1/3);
\draw[line width=3] (a1) arc (30:150:2/3);
\begin{scope}[very thick, decoration={
    markings,
    mark=at position 0.5 with {\arrow{>}}}
    ] 
    \draw[postaction={decorate}, orange, ultra thick] (o) -- (a);
    \draw[postaction={decorate}, orange, ultra thick] (o) -- (b);
\end{scope}
\filldraw[orange] (o) circle (0.1em);
\draw[thick, orange, dashed] (0, 0) [partial ellipse=150 : 390 : 1.0];
\node[anchor = north] at (a1) {$-$};
\node[anchor = north] at (b1) {$+$};
\end{tikzpicture}
}}
$
\caption{Bad arcs in $\SkAlg^{\mathfrak{sl}_2}_A(D_{\deg v})$}
\label{fig:bad-arcs-3d}
\end{figure}
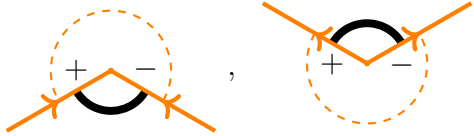
Analogously to Definition \ref{defn:reduced-stated-skein-algebra}, we define:
\begin{defn}[{\cite[Def. 3.34]{PP}}]
The \emph{reduced stated skein module} $\overline{\Sk}^{\mathfrak{sl}_2}_A(Y, \Gamma)$ is the quotient
\[
\overline{\Sk}^{\mathfrak{sl}_2}_A(Y, \Gamma) := 
I^{\mathrm{bad}, +} \backslash \Sk^{\mathfrak{sl}_2}_A(Y, \Gamma)/I^{\mathrm{bad}, -}
 = \frac{\Sk^{\mathfrak{sl}_2}_A(Y, \Gamma)}{I^{\mathrm{bad}, +}\Sk^{\mathfrak{sl}_2}_A(Y,\Gamma) + \Sk^{\mathfrak{sl}_2}_A(Y,\Gamma)I^{\mathrm{bad}, -}},
\]
where $I^{\mathrm{bad}, +}$ denotes the right ideal of $\SkAlg_A^{\mathfrak{sl}_2}(V(\Gamma)^+)$ generated by the bad arcs near the sinks, and $I^{\mathrm{bad}, -}$ denotes the left ideal of $\SkAlg_A^{\mathfrak{sl}_2}(V(\Gamma)^-)$ generated by the bad arcs near the sources. 
\end{defn}
Similarly to \eqref{eq: vertex skein algebra}, let's write, for simplicity of notation, 
\[
\overline{\SkAlg}_A^{\mathfrak{sl}_2}(V(\Gamma)^{\pm}) := \bigotimes_{v\in V(\Gamma)^{\pm}} \overline{\SkAlg}_A^{\mathfrak{sl}_2}(D_{\deg v}).
\]
Then, it is clear from the definition that the reduced stated skein module $\overline{\Sk}^{\mathfrak{sl}_2}_A(Y,\Gamma)$ has a $\overline{\SkAlg}_A^{\mathfrak{sl}_2}(V(\Gamma)^{+})$-$\overline{\SkAlg}_A^{\mathfrak{sl}_2}(V(\Gamma)^{-})$-bimodule structure. 

Suppose $Y$ is a cusped 3-manifold (without boundary) equipped with an ideal triangulation $\mathcal{T}$; 
let $\mathcal{T}^{(d)}$ denote the set of all $d$-simplices in the triangulation. 
We call the vertices, edges, and faces of the tetrahedra in $\mathcal{T}^{(3)}$ before gluing the \emph{bare vertices}, \emph{bare edges}, and \emph{bare faces}, respectively. 
For each ideal tetrahedron $T \in \mathcal{T}^{(3)}$, we fix a base point in the interior and call it the \emph{barycenter} of $T$. 
For each bare vertex $v$ (resp. bare edge $e$, resp. bare face $f$) of $T$, the \emph{vertex cone} $Cv$ (resp. \emph{edge cone} $Ce$, resp. \emph{face cone} $Cf$) is the join of $v$ (resp. $e$, resp. $f$) with the barycenter of $T$. 
Each face $f \in \mathcal{T}^{(2)}$ splits into two bare faces, and the \emph{face suspension} $Sf$ is the union of the two face cones. 
See Figure \ref{fig:vertex-cone-edge-cone-face-suspension}. 
\begin{figure}[htbp]
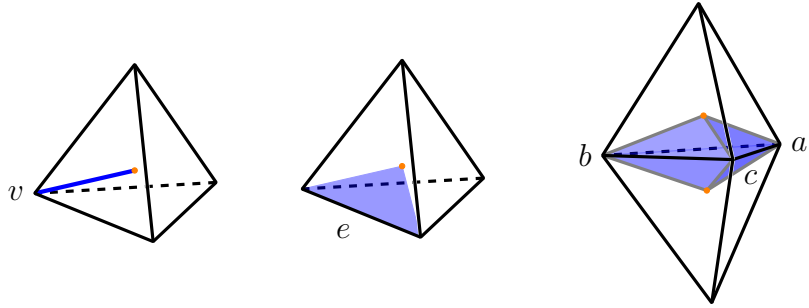

\centering
$
\vcenter{\hbox{
\tdplotsetmaincoords{70}{80}

}}
$
\caption{A vertex cone $Cv$, an edge cone $Ce$, and a face suspension $Sf$ for the face with vertices $a, b, c$}
\label{fig:vertex-cone-edge-cone-face-suspension}
\end{figure}

Consider the decomposition $Y= \bigcup_{f\in \mathcal{T}^{(2)}} Sf$ of $Y$ into face suspensions. 
Note, under this decomposition, each edge cone splits into two bare edge cones. 
We equip each face suspension $Sf$ with a boundary marking by drawing an edge from each of the 3 side edges to the 2 cone points; 
see Figure \ref{fig:face_suspension}. 
Each edge of the boundary marking corresponds to a bare edge cone (i.e. a face of $Sf$). 
We may label the top and the bottom bare edge cones abutting the edge $a$ by $a_1$ and $a_2$, respectively, and likewise for the edges $b$ and $c$. 
\begin{figure}[htbp]
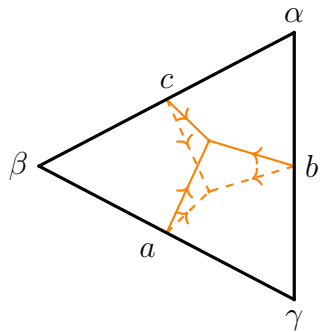

    \centering
    \includestandalone[scale=1]{figures/triangular_pillow}
    \caption{A face suspension with the standard boundary marking}
    \label{fig:face_suspension}
\end{figure}
This boundary marking has 2 sink vertices of degree 3 and 3 source vertices of degree 2, so the reduced stated skein module $\overline{\Sk}^{\mathfrak{sl}_2}_A(Sf)$ has a natural $\mathbb{T}^{\otimes 2}$--$\mathbb{B}^{\otimes 3}$-bimodule structure. 
In \cite[Sec. 4]{PP}, the structure of such bimodules is completely determined. 
To state this structure theorem for face suspensions, let $x_a$ denote the elementary tangle connecting the bare edge cones $a_1$ and $a_2$ with states $++$ (and likewise for $x_b$ and $x_c$), and let $\alpha_i$ ($i = 1, 2$) denote the elementary tangle connecting the bare edge cones $b_i$ and $c_i$ with states $++$ (and likewise for $\beta_i$ and $\gamma_i$). 
\begin{prop}[{\cite[Cor. 4.9]{PP}}]
The reduced stated skein module of a face suspension is given by 
\[
\overline{\Sk}^{\mathfrak{sl}_2}_A(Sf) \cong \frac{\mathbb{T}^{\otimes 2} \otimes \mathbb{B}^{\otimes 3}}{
\langle
(-A^2)\alpha_1\alpha_2 = x_b x_c, \;
(-A^2)\beta_1\beta_2 = x_c x_a, \;
(-A^2)\gamma_1\gamma_2 = x_a x_b
\rangle}
\]
as a left $\mathbb{T}^{\otimes 2} \otimes (\mathbb{B}^{\mathrm{op}})^{\otimes 3} = \mathbb{T}^{\otimes 2} \otimes \mathbb{B}^{\otimes 3}$-module, 
where
\begin{align*}
\mathbb{T}^{\otimes 2} &=
\frac{R\langle \alpha_1^{\pm 1}, \beta_1^{\pm 1}, \gamma_1^{\pm 1}\rangle}{\langle \beta_1\alpha_1 = A\alpha_1\beta_1,\; \gamma_1\beta_1 = A\beta_1\gamma_1,\; \alpha_1\gamma_1 = A\gamma_1\alpha_1 \rangle}\\
&\quad \otimes
\frac{R\langle \alpha_2^{\pm 1}, \beta_2^{\pm 1}, \gamma_2^{\pm 1}\rangle}{\langle \alpha_2\beta_2 = A\beta_2\alpha_2,\; \beta_2\gamma_2 = A\gamma_2\beta_2,\; \gamma_2\alpha_2 = A\alpha_2\gamma_2 \rangle}
\end{align*}
and
\[
\mathbb{B}^{\otimes 3} = R[x_a^{\pm 1}, x_b^{\pm 1}, x_c^{\pm 1}]. 
\]
\end{prop}

In order to state the splitting map associated to the decomposition of $Y$ into face suspensions, we need to first define the relative tensor product of these bimodules. 
\begin{defn}\label{defn:relative-tensor-product-Sf}
The \emph{relative tensor product} $\overline{\bigotimes}_{f\in \mathcal{T}^{(2)}}\overline{\Sk}^{\mathfrak{sl}_2}_A(Sf)$ of the bimodules $\overline{\Sk}^{\mathfrak{sl}_2}_A(Sf)$, $f\in \mathcal{T}^{(2)}$, is defined to be the quotient of the ordinary tensor product (as $R$-modules) $\bigotimes_{f\in \mathcal{T}^{(2)}}\overline{\Sk}^{\mathfrak{sl}_2}_A(Sf)$ by the following relations: 
\begin{itemize}
\item \label{item:redSkMod2} For each vertex cone $Cv$, we have the following relations among left actions of $\bigotimes_{\substack{f \in \mathcal{T}^{(2)} \\ Sf \text{ abutting }Cv}}\mathbb{T}$ on $\bigotimes_{f\in \mathcal{T}^{(2)}} \overline{\Sk}^{\mathfrak{sl}_2}_A(Sf)$:
\begin{gather} 
\vcenter{\hbox{

}}
\;\;=\;\; 
0
\;,
\end{gather}
where each sector in the above diagrams represents one of the three face suspensions surrounding $Cv$ (viewed from the vertex $v$), and the markings shown are on the bare edge cones abutting $Cv$. 
\item \label{item:redSkMod1} For each internal edge $e \in \mathcal{T}^{(1)}$ 
we have the following relations among right actions of $\bigotimes_{\substack{f \in \mathcal{T}^{(2)} \\f \text{ abutting }e}}\mathbb{B}$ on $\bigotimes_{f\in \mathcal{T}^{(2)}}\overline{\Sk}^{\mathfrak{sl}_2}_A(Sf)$: 
\begin{equation} 
\vcenter{\hbox{
%
}}
\;,
\quad \epsilon \in \{\pm\},
\end{equation}
where each sector in the above diagrams represents one of the face suspensions surrounding $e$ (as many as the number of tetrahedra abutting $e$), and the markings shown are on the bare edge cones abutting $e$. 
\end{itemize} 
\end{defn}

\begin{thm}[{\cite[Cor. 3.38]{PP}}]\label{thm:3d-splitting-map-for-Sf}
There is a well-defined splitting map 
\begin{align*}
\overline{\sigma}^{\mathfrak{sl}_2}:\overline{\Sk}^{\mathfrak{sl}_2}_A(Y) &\rightarrow \underset{f\in \mathcal{T}^{(2)}}{\overline{\bigotimes}} \overline{\Sk}^{\mathfrak{sl}_2}_A(Sf),\\
[L] &\mapsto  \qty[\sum_{\vec{\epsilon} \in \{\pm\}^{\cup_{f}Sf \cap L}}\otimes_{f\in \mathcal{T}^{(2)}} [L_{f}^{\vec{\epsilon}}]], 
\end{align*}
where $L_f^{\vec{\epsilon}}$ denotes the part of $L$ in $Sf$ after splitting, with boundary states determined by $\vec{\epsilon}$. 
\end{thm}

\subsubsection{3d quantum trace map}\label{subsubec:3d-quantum-trace}
The codomain of the 3d quantum trace map is the 3d analog of the square-root quantum Teichm\"uller space called the \emph{square-root quantum gluing module}. 
Just like the square-root quantum Teichm\"uller space was built out of extended triangle algebras, the square-root quantum gluing module can be built out of the following basic building blocks:
\begin{defn}
The \emph{face suspension module} $\mathbb{S}f$ of $Sf$ is the quantum torus
\begin{align*}
\mathbb{S}f := \widetilde{\mathbb{T}}^{\otimes 2} &= 
\frac{R\langle a_1^{\pm 1}, b_1^{\pm 1}, c_1^{\pm 1}\rangle}{\langle b_1a_1 = Aa_1b_1,\; c_1b_1 = Ab_1c_1,\; a_1c_1 = Ac_1a_1\rangle} \\
&\quad \otimes \frac{R\langle a_2^{\pm 1}, b_2^{\pm 1}, c_2^{\pm 1}\rangle}{\langle a_2b_2 = Ab_2a_2,\; b_2c_2 = Ac_2b_2,\; c_2a_2 = Aa_2c_2\rangle}
\end{align*}
generated by the 6 bare edge cones of $Sf$, viewed
as a regular $\widetilde{\mathbb{T}}^{\otimes 2}$--$\widetilde{\mathbb{T}}^{\otimes 2}$-bimodule. 
\end{defn}

\begin{thm}
Let $\CT, \CB \in R^{\times}$ be invertible scalars such that 
\begin{equation}\label{eq:scaling-parameters}
\CB^2 = q \CT^2.
\end{equation}
Then there is a well-defined $\mathbb{T}^{\otimes 2}$-$\mathbb{B}^{\otimes 3}$-bimodule homomorphism (in fact, an embedding)
\[
\Tr_{Sf} : \overline{\Sk}(Sf) \rightarrow \mathbb{S}f
\]
mapping the empty skein $[\emptyset]$ to $1$, induced by the embeddings of algebras
\begin{align*}
\mathbb{T}^{\otimes 2} &\hookrightarrow \widetilde{\mathbb{T}}^{\otimes 2}\\
\alpha_1 &\mapsto \CT [b_1c_1],\\
\beta_1 &\mapsto \CT [c_1a_1],\\
\gamma_1 &\mapsto \CT [a_1b_1],\\
\alpha_2 &\mapsto \CT [b_2c_2],\\
\beta_2 &\mapsto \CT [c_2a_2],\\
\gamma_2 &\mapsto \CT [a_2b_2],
\end{align*}
and
\begin{align*}
\mathbb{B}^{\otimes 3} &\hookrightarrow \widetilde{\mathbb{T}}^{\otimes 2}\\
x_a &\mapsto \CB a_1 \otimes a_2,\\
x_b &\mapsto \CB b_1 \otimes b_2,\\
x_c &\mapsto \CB c_1 \otimes c_2.
\end{align*}
\end{thm}

\begin{rmk}
Note, scaling both $\CT$ and $\CB$ by a common factor just corresponds to rescaling the generators of $\mathbb{S}f$. 
Up to scaling, there are two possible choices of the scalars $\CT$ and $\CB$:\footnote{While these two choices can be related by scaling the generators of one copy of $\widetilde{\mathbb{T}}$ by $(-1)^{\frac{1}{2}}$ and the generators of the other copy by $(-1)^{-\frac{1}{2}}$, there is no canonical such choice since the two copies are identical.}
\[
(\CT, \CB) = (q^{-\frac{1}{2}}, 1)\quad \text{or}\quad (\CT, \CB) = (-q^{-\frac{1}{2}}, 1).
\]
In \cite{PP}, we used $(\CT, \CB) = (A^{-1}, (-1)^{-\frac{1}{2}})$ which, up to scaling, corresponds to the second choice. 
\end{rmk}

For each edge cone $Ce$ corresponding to a bare edge $e$, consider the element $\hat{z}_e = a \otimes a'$ of the quantum torus
$\bigotimes_{f\in \mathcal{T}^{(2)}}\mathbb{S}f$, 
where $a$ and $a'$ are the two bare edge cones corresponding to $Ce$ after splitting into face suspensions; see Figure \ref{fig:edge-cone-split}.  
The element $\hat{z}_e \in \bigotimes_{f\in \mathcal{T}^{(2)}}\mathbb{S}f$ is called the \emph{square-root quantized shape parameter} associated to the bare edge $e$ of the ideal triangulation $\mathcal{T}$ of $Y$. 
\begin{figure}[htbp]
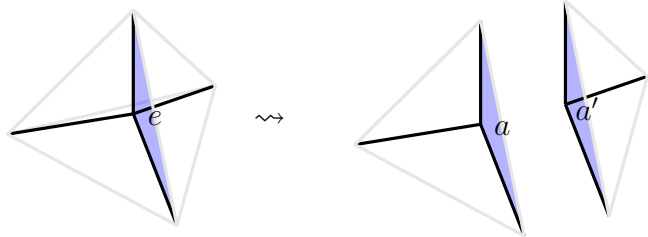

\centering
$
\vcenter{\hbox{
\tdplotsetmaincoords{50}{70}

}}
$
\caption{An edge cone $Ce$ splits into two bare edge cones $a$ and $a'$. Note, this is a cone over Figure \ref{fig:square-root-quantized-shear-parameter}.}
\label{fig:edge-cone-split}
\end{figure}

The square-root quantized shape parameters generate a sub-quantum torus of $\bigotimes_{f\in \mathcal{T}^{(2)}}\mathbb{S}f$. 
The \emph{square-root quantum gluing module} $\SQGM_{\mathcal{T}}(Y)$ that we define below is a two-sided quotient of this sub-quantum torus.\footnote{
The definition of $\SQGM_{\mathcal{T}}(Y)$ we use here is slightly different from the one we gave in \cite[Sec. 5]{PP}; 
the one we use here is the quotient of the one in \cite{PP} by setting $\hat{z}_e = \hat{z}_{e'}$ for each pair of opposite bare edges $e, e'$. 
The effect of this quotient is minor, as the vertex relations already imply that $\hat{z}_e^2 = \hat{z}_{e'}^2$. 
It may seem like there is another natural quotient given by setting $\hat{z}_e = -\hat{z}_{e'}$, but such a quotient turns out to be inconsistent with the 2-3 Pachner move. 
}
\begin{defn}\label{defn:SQGM}
With the scaling parameters $\CB$ and $\CT$, 
the \emph{square-root quantum gluing module} $\SQGM_{\mathcal{T}}(Y)$ is the two-sided quotient of the quantum torus 
\[
\bigotimes_{T \in \mathcal{T}^{(3)}} \frac{R \langle \hat{z}^{\pm 1}, \hat{z}^{'\pm 1}, \hat{z}^{''\pm 1}\rangle}{\langle \hat{z}'\hat{z} = A \hat{z} \hat{z}',\; \hat{z}''\hat{z}' = A \hat{z}' \hat{z}'',\; \hat{z}\hat{z}'' = A \hat{z}'' \hat{z}\rangle},
\]
where $\hat{z}$, $\hat{z}'$, $\hat{z}''$ are (generators associated to) the 3 pairs of opposite edge cones of $T$, in the anti-clockwise order when viewed from a vertex
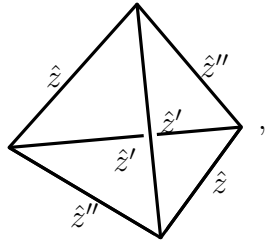
\[
\vcenter{\hbox{
\tdplotsetmaincoords{60}{80}
\begin{tikzpicture}[tdplot_main_coords, scale=0.8]
\begin{scope}[scale = 0.8, tdplot_main_coords]
    \newcommand*{\defcoords}{
        \coordinate (a) at (0, 0, 3);
        \coordinate (b) at ({2*sqrt(2)}, 0, -1);
        \coordinate (c) at ({-sqrt(2)}, {sqrt(6)}, -1);
        \coordinate (d) at ({-sqrt(2)}, {-sqrt(6)}, -1);
        \coordinate (e) at (0, 0, -5);
        \coordinate (ab) at ($1/2*(a) + 1/2*(b)$);
        \coordinate (ac) at ($1/2*(a) + 1/2*(c)$);
        \coordinate (ad) at ($1/2*(a) + 1/2*(d)$);
        \coordinate (eb) at ($1/2*(e) + 1/2*(b)$);
        \coordinate (ec) at ($1/2*(e) + 1/2*(c)$);
        \coordinate (ed) at ($1/2*(e) + 1/2*(d)$);
        \coordinate (bc) at ($1/2*(b) + 1/2*(c)$);
        \coordinate (bd) at ($1/2*(b) + 1/2*(d)$);
        \coordinate (cd) at ($1/2*(c) + 1/2*(d)$);
    }
    
    \defcoords

    \draw[white, line width=5] (b) -- (c);
    \draw[white, line width=5] (b) -- (d);
    \draw[very thick] (b) -- (c);
    \draw[very thick] (b) -- (d);
    
    \draw[very thick] (c) -- (d);
    
    \draw[white, line width=5] (a) -- (b);
    \draw[very thick] (a) -- (b);
    
    \draw[very thick] (b) -- (c);
    \draw[very thick] (b) -- (d);
    
    \draw[very thick] (a) -- (c);
    \draw[very thick] (a) -- (d);
    
    \filldraw (b) circle (0.05em);
    \filldraw (c) circle (0.05em);
    \filldraw (d) circle (0.05em);
    \filldraw (a) circle (0.05em);
    
    \node[left] at (ad){$\hat{z}$};
    \node[right] at (bc){$\hat{z}$};
    \node[right] at (ab){$\hat{z}'$};
    \node[below] at (cd){$\hat{z}'$};
    \node[right] at (ac){$\hat{z}''$};
    \node[below] at (bd){$\hat{z}''$};

\end{scope}
\end{tikzpicture}
}}
,
\]
by the following relations: 
\begin{enumerate}
\labeleditem{V}{item:vertex} Vertex relations (central):
For each tetrahedron, 
\[
[\hat{z} \hat{z}' \hat{z}''] = q^{-1}\CT^{-3},
\]
\labeleditem{L}{item:Lagrangian} Lagrangian relations (as left actions):
For each tetrahedron, 
\[
1 = \CB^{-2} \hat{z}^{-2} + \CB^2 \hat{z}''^2,
\]
\labeleditem{G}{item:gluing} Gluing relations (as right actions):
For each edge, 
\[
\hat{e} = q \CB^{-k},
\]
where
\[
\hat{e} := \qty[\prod_{\hat{z}\text{ abutting }e} \hat{z}]
\]
and $k$ is the number of edge cones abutting the edge $e$. 
\end{enumerate}
\end{defn}
For the rest of the paper, unless otherwise specified, fix $(\CT, \CB) = (q^{-\frac{1}{2}},1)$.

\begin{thm}
There is an $R$-module homomorphism $\Tr_{\mathcal{T}} : \Sk^{\mathfrak{sl}_2}_{A}(Y) \rightarrow \SQGM_{\mathcal{T}}(Y)$ defined as the composition 
\[
\begin{tikzcd}
\Sk^{\mathfrak{sl}_2}_{A}(Y) \arrow[rrrd, swap, "\Tr_\mathcal{T}"] \arrow[r, equal] & \overline{\Sk}^{\mathfrak{sl}_2}_{A}(Y) \arrow[r, "\overline{\sigma}"] & \underset{f\in \mathcal{T}^{(2)}}{\overline{\bigotimes}} \overline{\Sk}^{\mathfrak{sl}_2}_{A}(Sf) \arrow[hookrightarrow, r, "\overline{\otimes}_{f\in \mathcal{T}^{(2)}} \Tr_{Sf}"] &[3em] \underset{f\in \mathcal{T}^{(2)}}{\overline{\bigotimes}} \mathbb{S}f \\
& & & \SQGM_{\mathcal{T}}(Y) \arrow[hookrightarrow, u, ""]
\end{tikzcd}.
\]
\end{thm}

\subsection{Naturality with respect to Pachner moves}
We'll show that the 3d quantum trace map is natural with respect to 2-3 Pachner moves. 
Suppose that $\mathcal{T}_2$ and $\mathcal{T}_3$ are two ideal triangulations of $Y$ related by the 2-3 Pachner move. 
\begin{figure}[htbp]
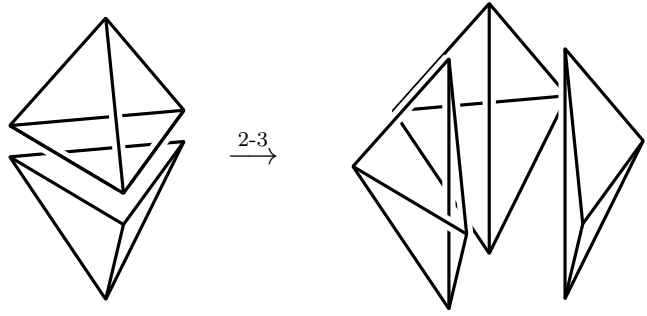

\centering
$
\vcenter{\hbox{
\tdplotsetmaincoords{60}{80}

}}
$
\caption{2-3 Pachner move on the triangular bipyramid}
\label{fig:23Pachner}
\end{figure}

\begin{figure}[htbp]
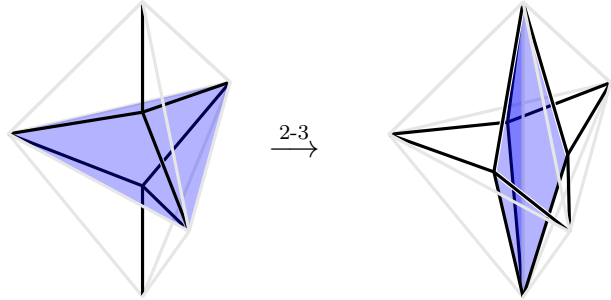

\centering
$
\vcenter{\hbox{
\tdplotsetmaincoords{50}{70}
%
}}
$
\caption{The 2-3 Pachner move turns 7 face suspensions (6 exterior + 1 interior) into 9 face suspensions (6 exterior +  3 interior). The interior faces (1 on the left and 3 on the right) are highlighted in blue.}
\label{fig:23Pachner-face-suspensions}
\end{figure}

Then, we have:
\begin{prop}\label{prop:SQGM23Pachner}
There is an $R$-module homomorphism 
\[
\varphi_{2\rightarrow 3} : \SQGM_{\mathcal{T}_2}(Y) \rightarrow \SQGM_{\mathcal{T}_3}(Y)
\]
defined on the quantum torus generators -- each of which corresponds to an edge cone -- by
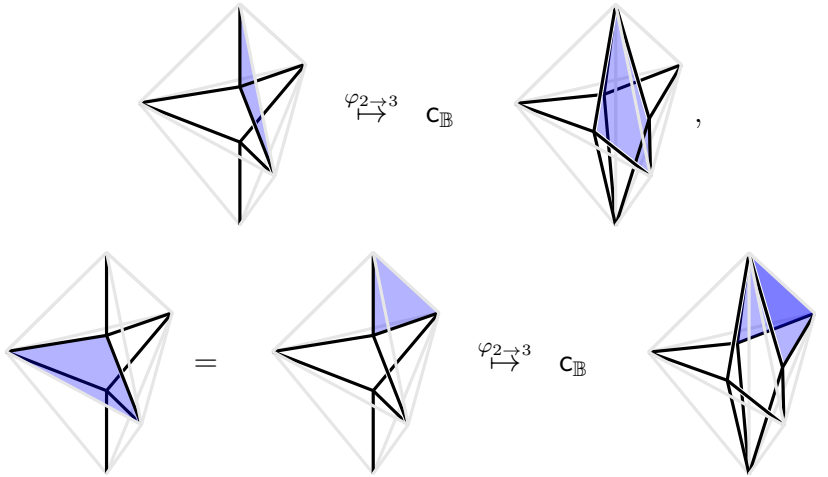
\[
\vcenter{\hbox{
\tdplotsetmaincoords{50}{70}

}}
. 
\]
\end{prop}
\begin{proof}
To check the well-definedness of $\varphi_{2\rightarrow 3}$ defined above for the generators of the quantum torus, we need to check that the map descends to the two-sided quotient -- that is, we need to check that they satisfy the vertex equation \eqref{item:vertex}
\[
[\hat{z}\hat{z}'\hat{z}''] = q^{-1} \CT^{-3} \quad \text{(central)}
\]
and the Lagrangian equation \eqref{item:Lagrangian}
\[
1 = \CB^{-2}\hat{z}^{-2} + \CB^2 \hat{z}''^2 \quad\text{(as left actions)},
\]
where $\hat{z}, \hat{z}', \hat{z}''$ are any triple of edge cones sharing a vertex, which are cyclically ordered in an anticlockwise manner when viewed from a vertex, as 
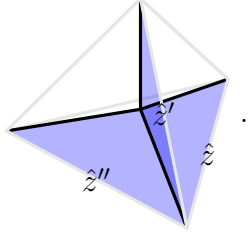
\[
\vcenter{\hbox{
\tdplotsetmaincoords{50}{70}

}}.
\]
Note, there are no gluing equations \eqref{item:gluing} to check, since the bipyramid triangulated into two tetrahedra has no interior edges. 

\textbf{Vertex equation:}
For either the top or the bottom vertex of the bipyramid, we have, as a central relation, 
\begin{align*}
\varphi_{2 \rightarrow 3}\qty(
1 - q \CT^3 
\vcenter{\hbox{
\tdplotsetmaincoords{50}{70}
%
}} \\
&\overset{\eqref{item:gluing}}{=} 
1 - 
q \CT^3 \CB^3 (q^{-1} \CT^{-3})^3 (q \CB^{-3})^{-1}\\
&= 
1 - \qty(q^{-1} \CT^{-2} \CB^2)^3
\overset{\eqref{eq:scaling-parameters}}{=} 0,
\end{align*}
where we have used red to indicate inverses. 
Note, for the third equality, we have used the fact that the gluing relation associated to the interior edge of the 3-tetrahedra triangulation of the bipyramid is actually central, as the product of the 3 edge cones is central in the image of $\varphi_{2 \rightarrow 3}$. 

Since we are identifying the square-root quantized shape parameters associated to the opposite edge cones, the proof is identical for the side vertices of the bipyramid. 

\textbf{Lagrangian equation:}
The proof is analogous to that of \cite[Lem. 4.1]{GY2}. 
Let's label the edge cones as in the figure below: 
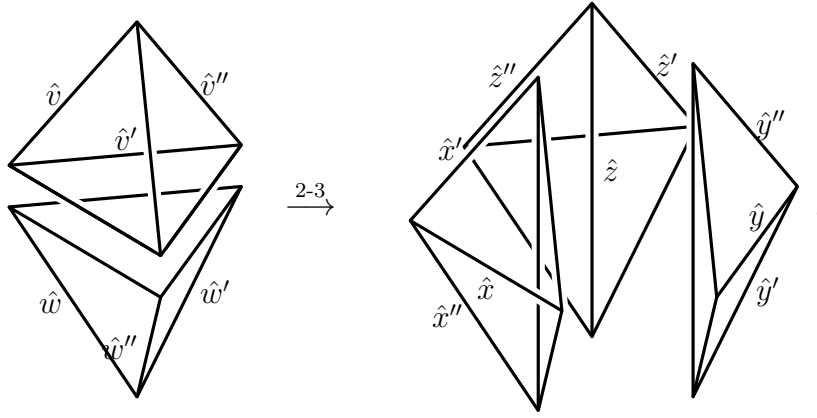
\[
\vcenter{\hbox{
\tdplotsetmaincoords{60}{80}

}}.
\]
We have, as a left relation, 
\begin{align*}
&\quad \varphi_{2 \rightarrow 3} \qty(1 - \CB^{-2} \hat{v}^{-2} - \CB^2 \hat{v}''^2)\\ 
&= 
1 - \CB^{-4} (\hat{x}' \hat{z}'')^{-2} - \CB^4 (\hat{z}' \hat{y}'')^2 \\
&= (1 - \CB^{-2} \hat{z}''^{-2} - \CB^2 \hat{z}'^2) + (1 - \CB^{-2} \hat{x}'^{-2} - \CB^2 \hat{x}^2)\CB^{-2} \hat{z}''^{-2}
 \\
&\quad + (1 - \CB^{-2} \hat{y}^{-2} + \CB^2 \hat{y}''^2) \CB^2 \hat{z}'^2 + \hat{x}^2 \hat{z}''^{-2} + \hat{y}^{-2} \hat{z}'^{2}\\
&\overset{\eqref{item:Lagrangian}}{=} \hat{x}^2 \hat{z}''^{-2} + \hat{y}^{-2} \hat{z}'^{2} 
= \hat{z}''^{-2}\hat{z}^{-2}\hat{y}^{-2}\qty(\hat{x}^2 \hat{y}^2 \hat{z}^2 - q^2 \CB^{-6}) + \qty(q^2 \CB^{-6} \hat{z}''^{-2} \hat{z}^{-2} \hat{z}'^{-2} + 1) \hat{y}^{-2}\hat{z}'^2\\
&\overset{\eqref{item:gluing}}{=} \qty(q^2 \CB^{-6} \hat{z}''^{-2} \hat{z}^{-2} \hat{z}'^{-2} + 1) \hat{y}^{-2}\hat{z}'^2 
= \qty(q^2 \CB^{-6} A^{-2} [\hat{z} \hat{z}' \hat{z}'']^{-2} + 1) \hat{y}^{-2}\hat{z}'^2 \\
&= \qty(-q \CB^{-6} [\hat{z} \hat{z}' \hat{z}'']^{-2} + 1) \hat{y}^{-2}\hat{z}'^2
\overset{\eqref{eq:scaling-parameters}}{=} \qty(-q^{-2} \CT^{-6} [\hat{z} \hat{z}' \hat{z}'']^{-2} + 1) \hat{y}^{-2}\hat{z}'^2 \\
&\overset{\eqref{item:vertex}}{=} 0.
\end{align*}
Again, we have used the fact that the gluing relation for the interior edge of the 3-tetrahedra triangulation, 
\[
\hat{x}\hat{y}\hat{z} - q \CB^{-3} = 0,
\]
which is a priori a relation among right actions, is actually central in the image of $\varphi_{2 \rightarrow 3}$. 
\end{proof}

\begin{thm}
The 3d quantum trace map is compatible with the Pachner move in the sense that we have the following commutative diagram:
\[
\begin{tikzcd}
 & \SQGM_{\mathcal{T}_2}(Y) \arrow[dd, "\varphi_{2 \rightarrow 3}"]\\
\Sk^{\mathfrak{sl}_2}_{A}(Y) \arrow[ur, "\Tr_2"] \arrow[dr, "\Tr_3"] & \\
 & \SQGM_{\mathcal{T}_3}(Y)
\end{tikzcd}.
\]
\end{thm}
\begin{proof}
Since the 3d quantum trace map factors through the splitting map for reduced stated skein modules, it suffices to check the commutativity for the triangular bipyramid $\mathrm{BP}$ on which we are performing the 2-3 Pachner move. 
Topologically, the triangular bipyramid is a 3-ball, whose boundary is combinatorially foliated (see Figure \ref{fig:boundary-marked-triangular-bipyramid}), so we have an explicit generator-and-relation description of $\overline{\Sk}_A^{\mathfrak{sl}_2}(\mathrm{BP})$. 
\begin{figure}[htbp]
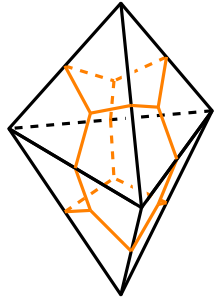

\centering
$
\vcenter{\hbox{
\tdplotsetmaincoords{60}{80}

}}
$
\caption{Triangular bipyramid $\mathrm{BP}$ with the standard boundary marking (all the edges of the boundary marking are oriented toward the center of each face)}
\label{fig:boundary-marked-triangular-bipyramid}
\end{figure}
Concretely, we have
\[
\overline{\Sk}_A^{\mathfrak{sl}_2}(\mathrm{BP}) \cong \frac{\mathbb{T}^{\otimes 6} \otimes (\mathbb{B}^{\mathrm{op}})^{\otimes 9}}{\mathrm{Ann}([\emptyset])},
\]
where we have one factor of $\mathbb{T}$ for each face and one factor of $\mathbb{B}^{\mathrm{op}}$ for each edge of the triangular bipyramid, and 
$\mathrm{Ann}([\emptyset])$ is the left ideal generated by relations coming from each face of the boundary tessellation. 
The triangular bipyramid splits into 7 face suspensions in the case of $\mathcal{T}_2$ and 9 face suspensions in the case of $\mathcal{T}_3$, 
so we need to check the commutativity of the diagram
\[
\begin{tikzcd}
 & \sim\backslash\qty(\widetilde{\mathbb{T}}^{\otimes 2})^{\otimes 7}/\sim \arrow[dd, "\varphi_{2 \rightarrow 3}"]\\
\overline{\Sk}^{\mathfrak{sl}_2}_{A}(\mathrm{BP}) \arrow[ur, "\Tr_2"] \arrow[dr, "\Tr_3"] & \\
 & \sim\backslash\qty(\widetilde{\mathbb{T}}^{\otimes 2})^{\otimes 9}/\sim
\end{tikzcd}, 
\]
where $\sim$ denotes the relations defining the square-root quantum gluing module (Definition \ref{defn:SQGM}). 
Since the maps are bimodule homomorphisms, it suffices to check this for each generator of the triangle and biangle algebras. 
Even though there are $6$ factors of triangle algebras and $9$ factors of biangle algebras, by symmetry, 
there are only 3 different types of generators we need to check: 
\begin{enumerate}
\item triangle algebra of a face,
\item biangle algebra of a non-horizontal edge,
\item biangle algebra of a horizontal edge. 
\end{enumerate}
The commutativity for the triangle algebra generators is obvious, as the 2-3 Pachner move doesn't affect the tangles localized at the center of a face of the bipyramid. 
The commutativity for the non-horizontal biangle algebra generators is also immediate; in fact, we have defined $\varphi_{2 \rightarrow 3}$ in such a way that the diagram commutes. 
Thus, the only non-trivial check is the commutativity for the horizontal biangle algebra generators (as right relations): 
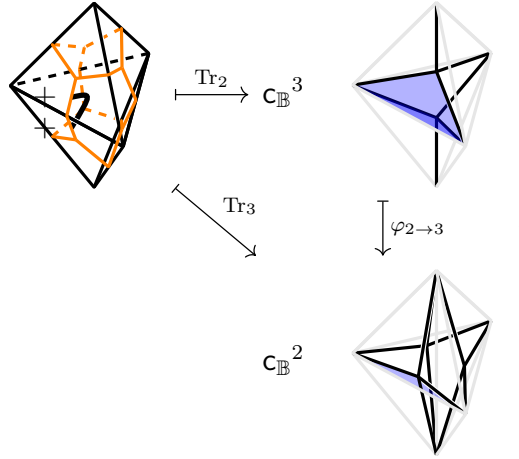
\[
\begin{tikzcd}
\vcenter{\hbox{
\tdplotsetmaincoords{50}{70}

}} 
\arrow[r, maps to, "\Tr_2"] \arrow[dr, maps to, "\Tr_3"] &
\CB^3
\vcenter{\hbox{
\tdplotsetmaincoords{50}{70}
%
}}
\end{tikzcd}. 
\]

This diagram commutes because: 
\begin{align*}
&\varphi_{2 \rightarrow 3} \qty(
\CB^3
\vcenter{\hbox{
\tdplotsetmaincoords{50}{70}
%
}},
\end{align*}
where, as before, we have used red to indicate inverses. 
\end{proof}

\begin{rmk}
When identifying the square-root quantized shape parameters of opposite edge cones, we had to choose between setting them equal or setting them to negatives of each other, since only their squares agreed a priori.
It turns out that for the compatibility with the 2-3 Pachner move, only the first identification works; for the second choice, we get an extra $-$ sign in the computation above. 
\end{rmk}


\section{Quantum UV-IR map}\label{sec:quantumUVIR}
In this section, we review the quantum UV-IR map \cite{NY} for surfaces and discuss its generalization to 3-manifolds.

\subsection{2d quantum UV-IR map}\label{subsec:2dUVIR}
Let $\Sigma$ be a surface equipped with an ideal triangulation $\tau$. 
Then, there is a branched double cover $\widetilde{\Sigma} = \widetilde{\Sigma}_{\tau}$ of $\Sigma$ associated to the ideal triangulation. 
This branched double cover can be obtained by putting a branch point at the center of each ideal triangle and drawing branch cuts as in Figure \ref{fig:triangle-branch-cut}; the corresponding branched double cover is an ideal hexagon. 
We will often trivialize the double cover away from the branch cuts so that we can call the two sheets of $\widetilde{\Sigma}$ by ``sheet 1'' and ``sheet 2''. 
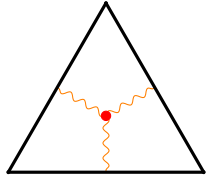
\begin{figure}[htbp]
\centering
$\vcenter{\hbox{
\begin{tikzpicture}
\begin{scope}[scale = 1.5]
    \coordinate (o) at (0, 0);
    \coordinate (a) at (0, 1);
    \coordinate (a1) at (0, -1/2);
    \coordinate (b) at ({sqrt(3)/2}, -1/2);
    \coordinate (b1) at ({-sqrt(3)/4}, 1/4);
    \coordinate (c) at (-{sqrt(3)/2}, -1/2);
    \coordinate (c1) at ({sqrt(3)/4}, 1/4);

    \tikzset{decoration={snake, amplitude=.4mm, segment length=2mm, post length=0mm, pre length=0mm}}
    \draw[orange, decorate] (o) -- (a1);
    \draw[orange, decorate] (o) -- (b1);
    \draw[orange, decorate] (o) -- (c1);
    
    \draw[very thick] (a) -- (b);
    \draw[very thick] (a) -- (c);
    \draw[very thick] (b) -- (c);
    \filldraw[red] (o) circle (0.1em);
    \filldraw (a) circle (0.03em);
    \filldraw (b) circle (0.03em);
    \filldraw (c) circle (0.03em);
\end{scope}
\end{tikzpicture}
}}$
\caption{Branch point (red dot) and branch cuts (orange squiggly lines) for an ideal triangle}
\label{fig:triangle-branch-cut}
\end{figure}

Given such a branched double cover $\widetilde{\Sigma} \rightarrow \Sigma$, 
\cite{NY} constructed an algebra homomorphism, called the \emph{quantum UV-IR map}, from the $\mathfrak{gl}_2$-skein algebra of $\Sigma$ to the $\mathfrak{gl}_1$-skein algebra of the double cover $\widetilde{\Sigma}$. 
Before reviewing this construction, let's briefly recall the definition of $\mathfrak{gl}_2$- and $\mathfrak{gl}_1$-skein modules. 
\begin{defn}
Let $Y$ be an oriented 3-manifold. 
The \emph{$\mathfrak{gl}_2$-skein module} $\Sk_q^{\mathfrak{gl}_2}(Y)$ of $Y$ 
is defined as
\[
\Sk^{\mathfrak{gl}_2}_q(Y) := \frac{R\langle \text{isotopy classes of framed oriented links in }Y \rangle}{\langle \mathfrak{gl}_2\text{-skein relations}\rangle},
\]
where the $\mathfrak{gl}_2$-skein relations are given by
\begin{gather}
\;\vcenter{\hbox{

}},
\end{gather}
with $[2] = q+q^{-1}$, the quantum $2$.\footnote{
Here, we used the half-twist relation \eqref{eq:gl2skeinrel3} -- instead of the usual full-twist relation -- for convenience; this doesn't change the skein module. 
We do the same for $\mathfrak{gl}_1$ skein modules (see \eqref{eq:gl1skeinrel4}). 
} 
\footnote{
In \cite[Sec 3.1]{NY}, the last two skein relations \eqref{eq:gl2extraskeinrel1}-\eqref{eq:gl2extraskeinrel2} (which correspond, for instance, to \cite[Eqn. 2.3]{QueffelecWedrich}) were missing, but here we add them for completeness, since the quantum UV-IR map indeed respects all these relations. 
}
\end{defn}

For our $\mathfrak{gl}_1$-skein modules, we need to twist the usual notion of $\mathfrak{gl}_1$-skein modules by introducing ``sign defects''. 
\begin{defn}\label{defn:gl1skeins-with-sign-defects}
Let $Y$ be an oriented 3-manifold decorated by a ``sign defect'' $B \subset Y$, an embedded 1-manifold. 
The \emph{$\mathfrak{gl}_1$-skein module} $\Sk_q^{\mathfrak{gl}_1}(Y) = \Sk_q^{\mathfrak{gl}_1}(Y,B)$ of $Y$
is defined as
\[
\Sk^{\mathfrak{gl}_1}_q(Y) := \frac{R\langle \text{isotopy classes of framed oriented links in }Y\setminus B\rangle}{\langle \mathfrak{gl}_1\text{-skein relations}\rangle},
\]
where the $\mathfrak{gl}_1$-skein relations are given by
\begin{gather}
q^{-1}\;\vcenter{\hbox{

}}
\;, \label{eq:gl1skeinrel3}
\end{gather}
with the red line in the final skein relation representing a part of the sign defect $B$.\footnote{This last skein relation justifies the use of the name ``sign defect'', as our skein acquires an extra sign as it passes through the sign defect.} 
\end{defn}
Whenever we talk about the $\mathfrak{gl}_1$-skein module of a branched double cover, we will always assume that we are decorating the branched double cover by a sign defect along the branch locus.\footnote{The reason we need to twist the $\mathfrak{gl}_1$-skein module of $\widetilde{Y}$ by the sign defect along the branch locus has to do with the fact that the pull-back of a spin structure on $Y$ to $\widetilde{Y} \setminus B$ doesn't extend over $B$, i.e. it is the non-bounding spin structure over the $S^1$ linking the branch locus $B$; see \cite{FN, ELPS}.} 

The 2d quantum UV-IR map \cite{NY} is an algebra homomorphism 
\[
F : \SkAlg_{q}^{\mathfrak{gl}_2}(\Sigma) \rightarrow \SkAlg_{q}^{\mathfrak{gl}_1}(\widetilde{\Sigma})
\]
from the $\mathfrak{gl}_2$-skein algebra of $\Sigma$ to the $\mathfrak{gl}_1$-skein algebra of the branched double cover $\widetilde{\Sigma}$ (associated to a choice of ideal triangulation $\tau$) decorated by sign defects at the branch points. 
The construction is best described in terms of projection to the leaf space of a 1-dimensional foliation of $\Sigma \times I$ called the \emph{WKB foliation}.\footnote{The name comes from the exact WKB analysis of Schr\"odinger equations on Riemann surfaces, where such a foliation can be obtained from a holomorphic quadratic differential.} 
Topologically, the WKB foliation is a 1-dimensional foliation that looks like Figure \ref{fig:2d-WKB} in each ideal triangle. 
Note that near each edge of the ideal triangle, the leaves are parallel and thus glue well to give a 1-dimensional foliation of $\Sigma$ in the complement of the branch locus. 
\begin{figure}[htbp]
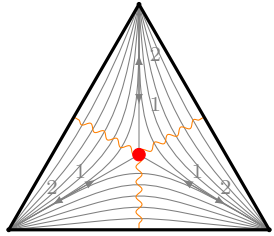

\centering
$\vcenter{\hbox{

}}$
\caption{2d WKB foliation and its orientation}
\label{fig:2d-WKB}
\end{figure}

The WKB foliation of $\Sigma$ lifts to an oriented 1-dimensional foliation on $\widetilde{\Sigma}$: 
we orient the leaves of the foliation of $\widetilde{\Sigma}$ so that the orientation is away from (resp. toward) the ideal vertices in sheet 1 (resp. sheet 2). 
This orientation is also indicated in Figure \ref{fig:2d-WKB}. 

While the generic leaves of the WKB foliation of $\Sigma$ are homeomorphic to $\mathbb{R}$, 
there are 3 \emph{singular leaves} in each ideal triangle which are homeomorphic to $\mathbb{R}_{\geq 0}$; these are the three leaves emanating out of the branch point. 
The union of singular leaves is called the \emph{2d spectral network} \cite{GMNsn}. 

The space of leaves in the WKB foliation, or the \emph{leaf space} for short, of an ideal triangle times $I$ is given by 3 facets glued along a binder, as in Figure \ref{fig:triangle-leaf-space}. 
The leaf space for $\Sigma \times I$ is obtained by gluing the leaf space for $T \times I$ along the edges. 
Note, we can canonically identify the binders of the leaf space to the branch locus and the facets of the leaf space to the branch cuts. 
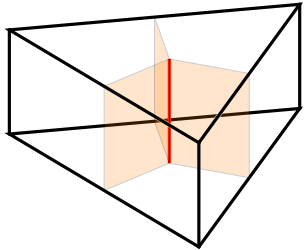
\begin{figure}[htbp]
\centering
$\vcenter{\hbox{
\tdplotsetmaincoords{60}{80}
\begin{tikzpicture}[tdplot_main_coords]
\begin{scope}[scale = 0.8, tdplot_main_coords]
    \coordinate (o) at (0, 0, -1);
    \coordinate (o2) at (0, 0, 1);
    
    \coordinate (b) at ({2*sqrt(2)}, 0, -1);
    \coordinate (c) at ({-sqrt(2)}, {sqrt(6)}, -1);
    \coordinate (d) at ({-sqrt(2)}, {-sqrt(6)}, -1);
    
    \coordinate (b2) at ({2*sqrt(2)}, 0, 1);
    \coordinate (c2) at ({-sqrt(2)}, {sqrt(6)}, 1);
    \coordinate (d2) at ({-sqrt(2)}, {-sqrt(6)}, 1);
    
    \coordinate (bc) at ({sqrt(2)/2}, {sqrt(6)/2}, -1);
    \coordinate (bd) at ({sqrt(2)/2}, {-sqrt(6)/2}, -1);
    \coordinate (cd) at ({-sqrt(2)}, 0, -1);

    \coordinate (bc2) at ({sqrt(2)/2}, {sqrt(6)/2}, 1);
    \coordinate (bd2) at ({sqrt(2)/2}, {-sqrt(6)/2}, 1);
    \coordinate (cd2) at ({-sqrt(2)}, 0, 1);

    \draw[very thick] (b) -- (c) -- (d) -- cycle;
    \draw[red, very thick] (o) -- (o2);
    \draw[white, ultra thick] (b2) -- (c2) -- (d2) -- cycle;
    
    \filldraw [very thin, fill= orange, opacity = 0.2] (o) -- (bc) -- (bc2) -- (o2) -- (o);
    \filldraw [very thin, fill= orange, opacity = 0.2] (o) -- (bd) -- (bd2) -- (o2) -- (o);
    \filldraw [very thin, fill= orange, opacity = 0.2] (o) -- (cd) -- (cd2) -- (o2) -- (o);
    
    \draw[very thick] (b2) -- (c2) -- (d2) -- cycle;
    \draw[very thick] (b) -- (b2);
    \draw[very thick] (c) -- (c2);
    \draw[very thick] (d) -- (d2);

    
    \filldraw (b) circle (0.05em);
    \filldraw (c) circle (0.05em);
    \filldraw (d) circle (0.05em);

    \filldraw (b2) circle (0.05em);
    \filldraw (c2) circle (0.05em);
    \filldraw (d2) circle (0.05em);
\end{scope}
\end{tikzpicture}
}}$
\caption{Leaf space for an ideal triangle times $I$}
\label{fig:triangle-leaf-space}
\end{figure}

\subsubsection{2d quantum UV-IR map}\label{subsec:2d-quantum-UV-IR-map}
Now, we are ready to review the construction of the 2d quantum UV-IR map. 
Let $L$ be a framed oriented link in $\Sigma \times I$. 
By making a small isotopy if necessary, let's assume that $L$ is in general position with respect to the WKB foliation, so that the projection of $L$ on the leaf space gives a nice link diagram. 
The link diagram of $L$ on the leaf space will have finitely many crossings, half-twists, and intersection points with the binders. 
In terms of the WKB foliation, crossings correspond to the leaves that intersect with $L$ twice, half-twists correspond to points of $L$ where the tangent vector is parallel to that of the foliation, and intersection points with the binders correspond to singular leaves (i.e. the leaves of the WKB foliation that meet the binder) that intersect with $L$.

The 2d quantum UV-IR map \cite{NY} is defined as follows:
\begin{align*}
F : \SkAlg_{q}^{\mathfrak{gl}_2}(\Sigma) &\rightarrow \SkAlg_{q}^{\mathfrak{gl}_1}(\widetilde{\Sigma})\\
[L] &\mapsto \sum_{\widetilde{L}} c_{\widetilde{L}} [\widetilde{L}],
\end{align*}
where $\widetilde{L}$ ranges over all possible lifts $\widetilde{L} \subset \widetilde{\Sigma}$ of $L \subset \Sigma$ that can be constructed out of \emph{direct lifts}, \emph{detours} and \emph{exchanges}, and $c_{\widetilde{L}} \in \mathbb{Z}[q^{\pm 1}]$ is the product of \emph{exchange factors}, \emph{turning factors}, and \emph{framing factors}, each of which we describe below. 

The lifts $\widetilde{L}$ are constructed from $L$ by attaching the leaves of the foliation as follows: 
\begin{itemize}
\item (Direct lifts) 
Any part of $L$ can be directly lifted to either sheet $1$ or sheet $2$; 
we denote the lift to sheet $i \in \{1, 2\}$ by labeling that part of $L$ by $i$: 
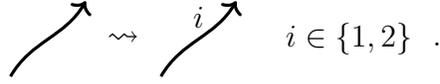
\[
\vcenter{\hbox{

}}.
\]
\item (Detours) 
If $L$ intersects a singular leaf at some point (i.e.\ when $L$ crosses a binding of the leaf space), that part of $L$ can be lifted using a detour in the following way: 
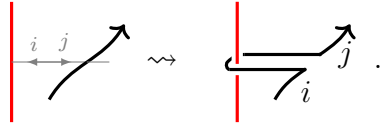
\[
\vcenter{\hbox{
%
}}.
\]
\item (Exchanges) 
If there is a generic leaf that meets $L$ at two points (i.e.\ at a crossing of $L$ on the leaf space), that part of $L$ can be lifted using an exchange in the following way: 
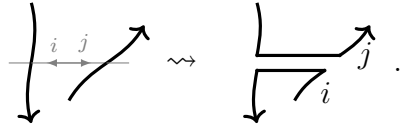
\[
\vcenter{\hbox{
%
}}.
\]
\end{itemize}
The coefficient $c_{\widetilde{L}}$ is defined as the product of the following factors: 
\begin{itemize}
\item (Exchange factor) For each exchange, we get a factor of $\pm (q-q^{-1})$, where the sign is given by the sign of the crossing where the exchange was used. 
\item (Turning factor) We get a factor $q^{t}$, where $t$ is the turning number of $\widetilde{L}$ projected on the leaf space; for part of $\widetilde{L}$ in sheet $i \in \{1, 2\}$, we compute the turning number by viewing the leaf space from the side where the $i$-direction is pointing away from our eyes. 
\item (Framing factor) We get an overall framing factor given by the product of $q^{\frac{1}{2}}$, one for each positive half-twist, and $q^{-\frac{1}{2}}$, one for each negative half-twist. 
\end{itemize}

See \cite[Sec. 5]{NY} for explicit examples. 

\begin{thm}[{\cite[Sec. 7]{NY}}]\label{thm:2dquantumUVIR-welldefinedness}
The 2d quantum UV-IR map $F : \SkAlg_{q}^{\mathfrak{gl}_2}(\Sigma) \rightarrow \SkAlg_{q}^{\mathfrak{gl}_1}(\widetilde{\Sigma})$ defined as above is a well-defined algebra homomorphism. 
\end{thm}
\begin{proof}
As shown in \cite[Sec. 7]{NY}, the quantum UV-IR map is invariant under isotopy and respects the skein relations \eqref{eq:gl2skeinrel1}-\eqref{eq:gl2skeinrel3}. 
Since the skein relations \eqref{eq:gl2extraskeinrel1}-\eqref{eq:gl2extraskeinrel2} were missing in their paper, we add a proof that the quantum UV-IR map indeed respects these two extra skein relations as well. 

We draw all the diagrams on a facet of the leaf space -- which locally looks like $\mathbb{R}^2$ --  with the leaves oriented in such a way that the ``1''(resp., ``2'')-direction is pointing away from (resp. toward) our eyes. 

\begin{itemize}
\item The relation \eqref{eq:gl2extraskeinrel1}: 
Defining
\[
\tikzset{
  >={To[length=5pt]}
}
\vcenter{\hbox{

}},
\]
the relation \eqref{eq:gl2extraskeinrel1} is equivalent to the following: 
\[
\tikzset{
  >={To[length=5pt]}
}
\vcenter{\hbox{
%
}}.
\]
The following small calculation will be useful:
\begin{equation}\label{eq:small-lemma}
\tikzset{
  >={To[length=5pt]}
}
F\qty(
\vcenter{\hbox{
%
}}
.
\end{equation}
From this, we see that 
\begin{align*}
\tikzset{
  >={To[length=5pt]}
}
F\qty(
\vcenter{\hbox{
%
}}
),
\end{align*}
where, in the second and the third line, we have drawn the two sheets separately, and $\sigma \in S_2^4$ ranges over all possible ways to perturb the relative positions of the boundary points of the tangle in sheet $1$ and $2$, on all 4 edges of the square; 
when the tangle is oriented upward, the trivial permutation $1 \in S_2$ corresponds to $2$ on the left of $1$, and the non-trivial permutation $\sigma_1 \in S_2$ corresponds to $1$ on the left of $2$. 
The sum of the lengths of the 4 permutations is denoted $l(\sigma)$. 
\end{itemize}
\end{proof}

\subsubsection{Naturality of the 2d quantum UV-IR map with respect to flips}\label{subsubsec:2d-quantum-UVIR-naturality}
Suppose $\tau$ and $\tau'$ are ideal triangulations of $\Sigma$ that differ by a flip. 
Then, there is a transition map
\[
\psi_{\tau \rightarrow \tau'} : 
\SkAlg_q^{\mathfrak{gl}_1}(\widetilde{\Sigma}_\tau) \rightarrow \SkAlg_q^{\mathfrak{gl}_1}(\widetilde{\Sigma}_{\tau'}),
\]
which is an algebra map defined on the algebra generators by
\begin{align}
\psi : 
\SkAlg^{\mathfrak{gl}_1}_{q} \qty(
\vcenter{\hbox{

}}, \label{eq:psi-map-3-term}
\end{align}
where $i \in \{1, 2\}$ ranges over the two sheets.\footnote{Here, we are using the fact that the $\mathfrak{gl}_1$-splitting map is injective, so it suffices to define the transition map for the branched double cover of the ideal quadrilateral as above. 
The corresponding stated $\mathfrak{gl}_1$-skein algebra is a quantum torus of rank 8; the first 8 corner tangles generate rank 7 part of it, and together with the last generator, they generate the whole quantum torus.} 

It is straightforward to check that the following diagram commutes:
\[
\begin{tikzcd}
 & \SkAlg^{\mathfrak{gl}_1}_{q}(\widetilde{\Sigma}_{\tau}) \arrow[dd, "\psi_{\tau \rightarrow \tau'}"]\\
\overline{\SkAlg}^{\mathfrak{gl}_2}_{q}(\Sigma) \arrow[ur, "F_{\tau}"] \arrow[dr, "F_{\tau'}"] & \\
 & \SkAlg^{\mathfrak{gl}_1}_{q}(\widetilde{\Sigma}_{\tau'})
\end{tikzcd},
\]
and hence the 2d quantum UV-IR map is natural with respect to changes of triangulation. 

\begin{rmk}\label{rmk: conjugation-by-quantum-dilog}
The coordinate change map $\psi_{\tau \rightarrow \tau'}$ can also be characterized as a conjugation by the quantum dilogarithm $\Psi = \frac{1}{(qx;q^2)_\infty}$; see \cite[Sec. 8.3]{NY}. 
To be more precise, the branched double cover of a quadrilateral triangulated into two ideal triangles is an annulus, whose (non-stated) $\mathfrak{gl}_1$-skein algebra is isomorphic to the polynomial ring $R[x^{\pm 1}]$, where $x$ represents the longitude of this annulus. 
The quantum dilogarithm $\Psi = \frac{1}{(qx;q^2)_\infty}$ is an element of a completion of the $\mathfrak{gl}_1$-skein algebra of the annulus, that solves the recurrence relation
\begin{equation}\label{eq:Psi-3-term}
(1  - \hat{y} -q \hat{x})\Psi = 0, 
\end{equation}
where $\hat{y}$ is the meridian so that $\hat{y}\hat{x} = q^2 \hat{x}\hat{y}$. 
While $\Psi$ itself lives in the completion, conjugation by $\Psi$ gives a well-defined map between ordinary (i.e., non-completed) skein algebras and, in fact, gives exactly the map described above: 
it is easy to see that conjugation by $\Psi$ acts as the identity on the first 8 corner tangles, as those corner tangles commute with $\Psi$, while the last relation \eqref{eq:psi-map-3-term} is exactly the 3-term relation \eqref{eq:Psi-3-term} satisfied by $\Psi$. 

The pentagon identity for the transition maps $\psi$ follows from that of the quantum dilogarithm; see, e.g., \cite[Sec. 2.1]{KLNPS} for an illustration and its interpretation in terms of the unlinking moves on symmetric quivers. 
\end{rmk}

\subsection{3d spectral networks}\label{subsec:3dSpecNet}
The crucial ingredient in the construction of the 2d quantum UV-IR map was the 1-dimensional foliation on ideally triangulated surfaces and the associated leaf space. 
Such combinatorial topological structures can be generalized to ideally triangulated 3-manifolds, as studied in \cite{FN}, and are called \emph{3d spectral networks}.\footnote{To be more precise, the term spectral network usually refers to the stratification structure given by the union of singular leaves of the WKB foliation, but here we are using the term in a vague sense to refer to the whole combinatorial topological structure of the WKB foliations and the associated leaf space.} 

Let $Y$ be a 3-manifold equipped with a triangulation $\mathcal{T}$ into ideal tetrahedra. 
Then, there is a branched double cover $\widetilde{Y} = \widetilde{Y}_{\mathcal{T}}$ of $Y$ associated to the ideal triangulation, obtained by putting the branch locus given by the embedded 4-valent graph connecting the barycenter of each ideal tetrahedron with the barycenters of its 4 faces; see Figure \ref{fig:tetrahedron-double-cover}. 
Note that on each face, the branching structure is exactly what we had for surfaces (Figure \ref{fig:triangle-branch-cut}). 
\begin{figure}[htbp]
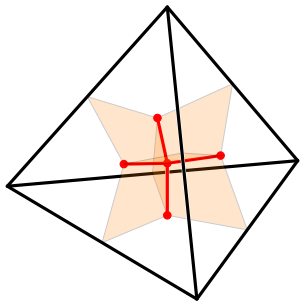

\centering
$\vcenter{\hbox{
\tdplotsetmaincoords{60}{80}

}}$
\caption{Branch locus (red) and branch cuts (orange) for an ideal tetrahedron}
\label{fig:tetrahedron-double-cover}
\end{figure}

A new feature in this 3-dimensional setup is that, unlike the surface case we saw earlier, the branched double cover $\widetilde{Y}$ of $Y$ is no longer a manifold; it is a pseudomanifold. 
Near the barycenters, the branched double cover is a cone over a torus with 4 branch points, as illustrated in Figure \ref{fig:cone-point}. 
\begin{figure}[htbp]
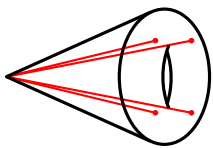

\centering
%
\caption{Cone over a torus with 4 branch points}
\label{fig:cone-point}
\end{figure}

There is a 1-dimensional foliation on $Y$ in the complement of the branch locus, which we will still call the \emph{WKB foliation}, even though its connection to exact WKB analysis is less clear. 
On each ideal tetrahedron, the WKB foliation is given by the cone over the WKB foliation of the boundary surface, which is $S^2$ with 4 punctures, triangulated into 4 ideal triangles; see Figure \ref{fig:3d-WKB}. 
\begin{figure}[htbp]
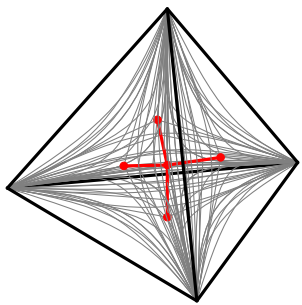

\centering
$\vcenter{\hbox{
\tdplotsetmaincoords{60}{80}
%
}}$
\caption{3d WKB foliation}
\label{fig:3d-WKB}
\end{figure}

The corresponding leaf space looks exactly like the branch cuts shown in Figure \ref{fig:tetrahedron-double-cover} for each tetrahedron; it has 6 facets joined along 4 binders, and has a singular point in the middle where the 4 binders meet. 

The singular leaves of the 3d WKB foliation are shown on the right side of Figure \ref{fig:2d-and-3d-spectral-networks}. 
\begin{figure}[htbp]
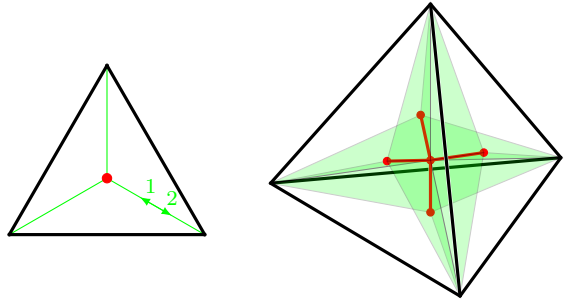

\centering
$\vcenter{\hbox{
%
}}$
\caption{2d and 3d spectral networks}
\label{fig:2d-and-3d-spectral-networks}
\end{figure}
As in the previous subsection, we always orient the WKB foliation so that the ``1'' (resp. ``2'')-direction is always pointing away from (resp. toward) the ideal vertices.  
This convention will determine our dictionary between the sheet labels ($1$ and $2$) and the states ($+$ and $-$) when we extend the quantum UV-IR map to stated skein modules in Section \ref{subsec:stated_quantum_UV-IR}.

\subsection{$\mathfrak{gl}_1$-skein modules with defects}
We saw in the previous subsection that the branched double cover associated to an ideal trangulation of a 3-manifold has a cone point at the barycenter of each tetrahedron. 
The appropriate notion of the $\mathfrak{gl}_1$-skein module for such a pseudomanifold turns out to be the one with an extra 3-term relation for each cone point, that depends on a choice of a generalized angle structure of the ideal triangulation. 
We describe this $\mathfrak{gl}_1$-skein module in this subsection. 

\subsubsection{Generalized angle structures}
\begin{defn}
For an ideally triangulated 3-manifold $Y$, 
a \emph{generalized angle structure} is 
an assignment of a real number\footnote{Can be thought of as formal commutative parameters satisfying certain linear relations.} to each dihedral angle of the tetrahedra
satisfying the following properties: 
\begin{enumerate}
\item The numbers assigned to the opposite dihedral angles of each tetrahedron are the same; see Fig. \ref{fig:dihedral_angles} below. 
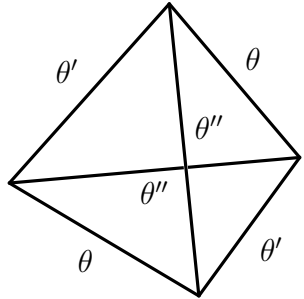
\begin{figure}[htbp]
\centering
$\vcenter{\hbox{
\tdplotsetmaincoords{60}{80}
\begin{tikzpicture}[tdplot_main_coords]
\begin{scope}[scale = 0.8, tdplot_main_coords]
    \coordinate (o) at (0, 0, 0);
    \coordinate (a) at (0, 0, 3);
    \coordinate (a1) at (0, 0, -1);
    \coordinate (b) at ({2*sqrt(2)}, 0, -1);
    \coordinate (b1) at ({-2*sqrt(2)/3}, 0, 1/3);
    \coordinate (c) at ({-sqrt(2)}, {sqrt(6)}, -1);
    \coordinate (c1) at ({sqrt(2)/3}, {-sqrt(6)/3}, 1/3);
    \coordinate (d) at ({-sqrt(2)}, {-sqrt(6)}, -1);    \coordinate (d1) at ({sqrt(2)/3}, {sqrt(6)/3}, 1/3);
    \coordinate (ab) at ({sqrt(2)}, 0, 1);
    \coordinate (ac) at ({-sqrt(2)/2}, {sqrt(6)/2}, 1);
    \coordinate (ad) at ({-sqrt(2)/2}, {-sqrt(6)/2}, 1);
    \coordinate (bc) at ({sqrt(2)/2}, {sqrt(6)/2}, -1);
    \coordinate (bd) at ({sqrt(2)/2}, {-sqrt(6)/2}, -1);
    \coordinate (cd) at ({-sqrt(2)}, 0, -1);
    \draw[very thick] (c) -- (d);
    \draw[white, ultra thick] (a) -- (b);
    \draw[very thick] (a) -- (b);
    \draw[very thick] (a) -- (c);
    \draw[very thick] (a) -- (d);
    \draw[very thick] (b) -- (c);
    \draw[very thick] (b) -- (d);
    \filldraw (a) circle (0.05em);
    \filldraw (b) circle (0.05em);
    \filldraw (c) circle (0.05em);
    \filldraw (d) circle (0.05em);
    \node[anchor = south east] at (ad) {$\theta'$};
    \node[anchor = south west] at (ac) {$\theta$};
    \node[anchor = north east] at (bd) {$\theta$};
    \node[anchor = north west] at (bc) {$\theta'$};
    \node[anchor = south west] at (ab) {$\theta''$};
    \node[anchor = north] at (cd) {$\theta''$};
\end{scope}
\end{tikzpicture}
}}$
\caption{Dihedral angles $\theta, \theta', \theta''$}
\label{fig:dihedral_angles}
\end{figure}
\item The 3 numbers $\theta$, $\theta'$, and $\theta''$ add up to $\pi$, for each tetrahedron. 
\item For each internal edge of the ideal triangulation, the angles add up to $2\pi$. 
\end{enumerate}
\end{defn}
We will denote a generalized angle structure by $\Theta = \{\theta_i\}_{i\in I}$, where $I$ is some indexing set for the set of dihedral angles of the triangulation. 

\begin{rmk}
It is known \cite[Theorem 1]{LuoTillmann} that, if $\mathcal{T}$ is an ideal triangulation with $t$ ideal tetrahedra of a compact 3-manifold $Y$ with $v$ boundary components, then $\mathcal{T}$ admits a generalized angle structure if and only if each boundary component of $Y$ is either a torus or a Klein bottle;
since we are dealing with oriented 3-manifolds, they have to be tori. 
Moreover, the space of generalized angle structures (when non-empty) is an affine space of dimension $t+v$. 
\end{rmk}

For our purposes, what's important is that a choice of a generalized angle structure induces a flat connection on the tangent bundle of the leaf space. 
That is, once equipped with a generalized angle structure, each corner of the leaf space carries a definite angle; see Figure \ref{fig:leafspace_Euclideanstructure}. 
\begin{figure}[htbp]
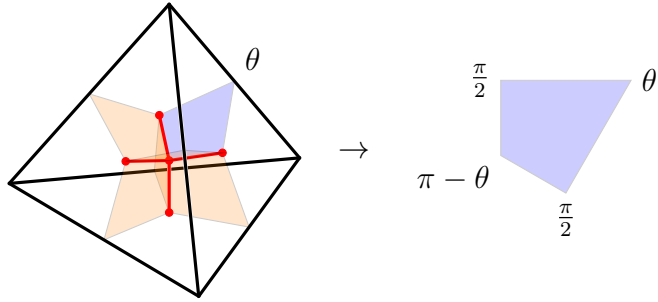

\centering
$\vcenter{\hbox{
\tdplotsetmaincoords{60}{80}

}}$
\caption{Euclidean structure on each facet of the leaf space}
\label{fig:leafspace_Euclideanstructure}
\end{figure}
In particular, because 
\[
(\pi - \theta) + (\pi - \theta') + (\pi - \theta'') = 3\pi - (\theta + \theta' + \theta'') = 2\pi,
\]
for any three facets of the leaf space of a tetrahedron around a fixed vertex, the sum of the three inner angles add up to $2\pi$. 
This, along with the condition that the angles around each edge of the triangulation add up to $2\pi$, means that the \emph{turning number} is a smooth isotopy invariant of a closed curve drawn on the leaf space. 
Unlike the 2d case, however, the turning number on the leaf space of a 3d WKB foliation is some $\mathbb{Z}$-linear combination of $\frac{\theta}{2\pi}$'s and $\frac{1}{2}$, so in general is \emph{not} an integer.

\subsubsection{$\mathfrak{gl}_1$-skein module of branched double covers}
Recall that the branched double cover $\widetilde{Y}$ has branch locus and cone points. 
Accordingly, there are additional skein relations for $\Sk_{q}^{\mathfrak{gl}_1}(\widetilde{Y})$:

\begin{defn}\label{defn:SkeinModuleOfBranchedDoubleCover}
Let $Y$ be an ideally triangulated 3-manifold with an associated branched double cover $\widetilde{Y}$. 
Suppose that the ideal triangulation of $Y$ is equipped with a generalized angle structure $\Theta = \{\theta_i\}_{i\in I}$. 
Let $R^{\Theta} := \mathbb{Z}[q^{\pm 1}]\otimes \mathbb{Z}[\{q^{\pm \frac{\theta_i}{\pi}}\}_{i\in I}]$. 
Then, the \emph{$\mathfrak{gl}_1$-skein module} $\mathrm{Sk}_{q}^{\mathfrak{gl}_1}(\widetilde{Y},\Theta)$ of $\widetilde{Y}$
is the free $R^{\Theta}$-module spanned by the isotopy classes of framed oriented links in $\widetilde{Y}$ away from the branch locus, modulo \eqref{eq:gl1skeinrel1}-\eqref{eq:gl1skeinrel4}, as well as the following extra relation near the cone points: 
\begin{equation}
\vcenter{\hbox{

}}
\;. \label{eq:gl1skeinrel5}
\end{equation}

Note, the relations \eqref{eq:gl1skeinrel1}-\eqref{eq:gl1skeinrel4} are drawn in $\widetilde{Y}$, while the last relation, \eqref{eq:gl1skeinrel5}, is drawn in the projection to $Y$. 
The meaning of these skein relations except for the last one should be clear; \eqref{eq:gl1skeinrel1}-\eqref{eq:gl1skeinrel2} are the usual relations (with a half twist relation \eqref{eq:gl1skeinrel4}), and \eqref{eq:gl1skeinrel3} says that when a strand of a link passes through the branch locus, the skein acquires a minus sign. 
The last skein relation \eqref{eq:gl1skeinrel5} is a local relation near the singular cone point, viewed from a vertex of the ideal tetrahedron. 
It can be drawn 3-dimensionally as in Fig. \ref{fig:3termrelation}. 
The labels $1$ and $2$ denote which sheet (in the complement of the branch cuts) that part of the skein belongs to. 
\begin{figure}[htbp]
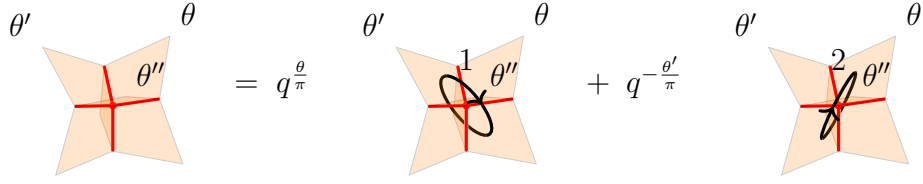

\centering
$\vcenter{\hbox{
\tdplotsetmaincoords{60}{80}

}}
$
\caption{The 3-term relation \eqref{eq:gl1skeinrel5} near the cone point}
\label{fig:3termrelation}
\end{figure}
\end{defn}

\begin{rmk}
When the angles are $\pi, 0, 0$, the 3-term relation \eqref{eq:gl1skeinrel5} is nothing but the recurrence relation for the quantum dilogarithm $\Psi$ in Remark \ref{rmk: conjugation-by-quantum-dilog}. 
This fact is used crucially in \cite{ELPS} to study a non-singular version of the quantum UV-IR map and give a geometric interpretation in terms of holomorphic curve counts. 
\end{rmk}

\begin{rmk}
The appearance of non-integer powers of $q$ in \eqref{eq:gl1skeinrel5} may seem strange, but there is an easy way to see why such factors are necessary for the skein module to be non-zero. 
Firstly, let's take a look at the following simple observation: 
\begin{lem}\label{lem:WeylOrderedProduct-q}
Let $\mathbb{T}$ be a quantum torus defined as
\[
\mathbb{T} := \frac{\mathbb{C}\langle x^{\pm 1}, {x'}^{\pm 1}, {x''}^{\pm 1}\rangle}{(x x' = q^2 x' x,\; x' x'' = q^2 x'' x',\; x'' x = q^2 x x'')}.
\]
Let $M$ be a cyclic left $\mathbb{T}$-module defined by
\[
M = \mathbb{T}/I,
\]
where $I$ is the left ideal generated by the relations
\[
x'' = 1-x^{-1},\; x=1-{x'}^{-1},\; x'=1-{x''}^{-1}.
\]
Then, the Weyl ordered product $[x x' x''] := q^{-1} x x' x'' = q x' x x''$ acts on the cyclic vector of $M$ by $-q$. 
\end{lem}
\begin{proof}
Let $[\emptyset]$ denote the cyclic vector of $M$. Then, 
\begin{align*}
[x x' x''][\emptyset] &= q x' x x''[\emptyset]\\
&= q x' x (1-x^{-1})[\emptyset]\\
&= q x'(x - 1)[\emptyset]\\
&= q x'(-x'^{-1})[\emptyset]\\
&= -q[\emptyset].
\end{align*}
\end{proof}
On the other hand, in the $\mathfrak{gl}_1$ skein algebra of the torus with 4 branch points, the Weyl ordered product of the three natural cycles of $T^2$ is $-1$: 
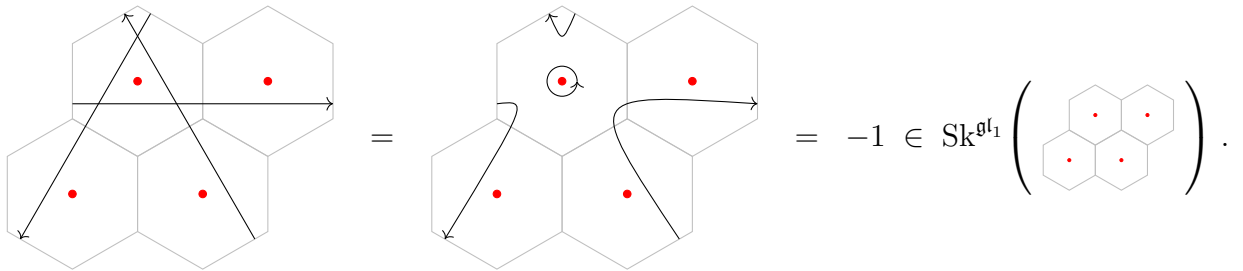
\[
\vcenter{\hbox{

}}
)
\;.
\]
It follows that, if we call these three natural cycles 
\[
x, x', x'' \in \mathrm{Sk}^{\mathfrak{gl}_1}\qty(
\vcenter{\hbox{
%
}}
),
\]
then the three-term recursion relations associated to the singular point cannot be the ones in Lemma \ref{lem:WeylOrderedProduct-q}. 
That is, there must be some factor of $q$ in some of the relations. 
We also see immediately that some non-integer power of $q$ is necessary to have the $\mathbb{Z}/3$ symmetry of the recursion relations. 
\end{rmk}

\subsection{3d quantum UV-IR map}\label{subsec:3dUVIR}
We are ready to state the 3d generalization of the quantum UV-IR map. 
The construction is essentially the same as the 2d quantum UV-IR map described in Section \ref{subsec:2d-quantum-UV-IR-map}:
For each framed oriented link $L \subset Y$ representing an element $[L] \in \Sk_q^{\mathfrak{gl}_2}(Y)$, which we assume to be in general position with respect to the 3d WKB foliation, 
we take the linear combination $\sum_{\widetilde{L}} c_{\widetilde{L}}[\widetilde{L}] \in \Sk_q^{\mathfrak{gl}_1}(\widetilde{Y})$ of all possible lifts $\widetilde{L}$ that can be constructed out of direct lifts, detours and exchanges (using the 3d WKB foliation), 
where the coefficients $c_{\widetilde{L}}$ are given by the product of exchange factors, turning factors, and framing factors, exactly as in Section \ref{subsec:2d-quantum-UV-IR-map}. 
The new phenomenon in 3d is that the turning numbers $t$ in the turning factors are no longer integers and depend on the choice of generalized angle structure. 
The claim is that the resulting 3d quantum UV-IR map is well-defined: 
\begin{thm}\label{thm:3dquantumUVIR}
Let $Y$ be a 3-manifold equipped with an ideal triangulation and a generalized angle structure. 
Then, there is a well-defined $R$-module homomorphism
\[
F : \Sk_q^{\mathfrak{gl}_2}(Y) \rightarrow \Sk^{\mathfrak{gl}_1}_q(\widetilde{Y}).
\]
\end{thm}
\begin{proof}
We need to show that $F : L \mapsto \sum_{\widetilde{L}} c_{\widetilde{L}}[\widetilde{L}]$ is invariant under the isotopy of $L$ and that this map respects the $\mathfrak{gl}_2$-skein relations. 
Any isotopy of $L$ is a finite composition of elementary isotopies, each corresponding to a violation of a general position requirement of $L$ with respect to the 3d WKB foliation. 
The elementary isotopies consist of
\begin{enumerate}
\item 3 types of framed Reidemeister moves, 
\item 4 types of moves near the bindings, 
\item and an extra move near the singular point. 
\end{enumerate}
The 3 framed Reidemeister moves and the 4 moves near the bindings are already present in the 2d case and were analyzed in detail in \cite[Sec. 7]{NY}, where it is shown that $F$ is indeed invariant under those isotopies. 
The last move, which is new in 3d, is an elementary isotopy of a link diagram on the leaf space when it crosses the singular point (Figure \ref{fig:isotopy-crossing-singular-point}). 
\begin{figure}[htbp]
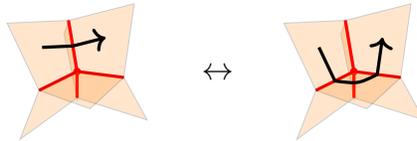

\centering
$
\vcenter{\hbox{
\tdplotsetmaincoords{30}{80}

}}
$
\caption{An isotopy of a link diagram on the leaf space crossing the singular point}
\label{fig:isotopy-crossing-singular-point}
\end{figure}
The only non-trivial case to check is when there are detours, in which case the images of the left-hand side and the right-hand side of Figure \ref{fig:isotopy-crossing-singular-point} under $F$ are given by
\[
\vcenter{\hbox{
%
}}, 
\]
and the invariance is ensured exactly by the 3-term relation \eqref{eq:gl1skeinrel5} near the cone point. 

Finally, that $F$ respects the $\mathfrak{gl}_2$-skein relations was already shown in \cite[Sec. 7]{NY} (see Theorem \ref{thm:2dquantumUVIR-welldefinedness}). 
\end{proof}

\begin{rmk}
The 3d quantum UV-IR map gives an intuitive explanation for why the 2d quantum UV-IR map gets conjugated by the quantum dilogarithm $\Psi$ under a flip:
A flip in an ideal triangulation can be thought of as a bordism given by attaching a taut ideal tetrahedron (i.e., the one with angles $\pi, 0, 0$), and for such an ideal tetrahedron, the corresponding $\mathfrak{gl}_1$-skein module of the branched double cover knows the 3-term relation \eqref{eq:gl1skeinrel5} -- which is equivalent to an insertion of the quantum dilogarithm, after resolving the singular cone point. 
\end{rmk}

\begin{rmk}
Theorem \ref{thm:3dquantumUVIR} admits a natural generalization to a map between HOMFLYPT skein modules
\[
F : \Sk_{a^2,z}(Y) \rightarrow \Sk_{a,z}(\widetilde{Y}),
\]
which specializes to the above map after setting $a=q$ and $z = q-q^{-1}$. 
In \cite{ELPS}, such a map is interpreted in terms of skein-valued counts of holomorphic curves (in the spirit of \cite{ES}) and is vastly generalized to branched covers coming from Lagrangians in $T^*Y$. 
\end{rmk}

\subsection{Naturality with respect to Pachner moves}\label{subsec:quantumUVIR-Pachner}
Suppose that $\mathcal{T}_2$ and $\mathcal{T}_3$ are two ideal triangulations of $Y$ related by a 2-3 Pachner move, and let $\widetilde{Y}_2$ and $\widetilde{Y}_3$ be the corresponding branched double covers. 
The leaf space of the WKB foliation corresponding to the ideal triangulation $\mathcal{T}_3$ looks identical to that of $\mathcal{T}_2$ outside of the bipyramid where we are performing the 2-3 Pachner move. 
Inside the bipyramid, the leaf space undergoes the transformation depicted in Figure \ref{fig:leaf-space-23Pachner}; the two singular points collide and split into three singular points, which bound a newly created triangular facet. 
\begin{figure}[htbp]
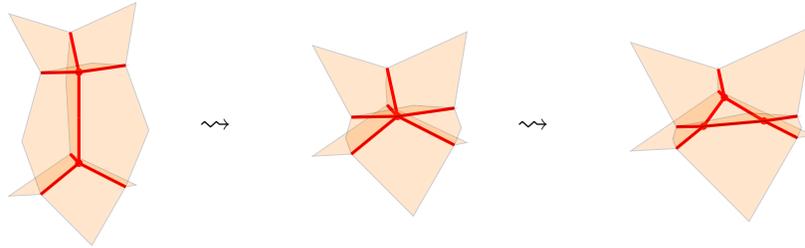

\centering
$
\vcenter{\hbox{
\tdplotsetmaincoords{60}{80}

}}
$
\caption{Change of leaf space under the 2-3 Pachner move}
\label{fig:leaf-space-23Pachner}
\end{figure}

Let $\Theta_2$ be a generalized angle structure on $\mathcal{T}_2$, and let $\Theta_3$ be a generalized angle structure on $\mathcal{T}_3$ compatible with $\Theta_2$; there is always a 1-dimensional space of such $\Theta_3$'s. 
Explicitly, when viewed from the top of Figure \ref{fig:leaf-space-23Pachner}, the angles of the facets of the leaf space are given as in Figure \ref{fig:23Pachner-angles}. 
\begin{figure}[htbp]
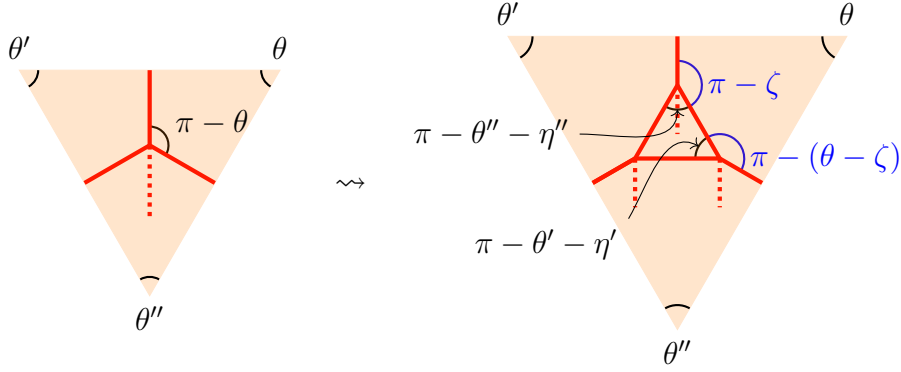

\centering
$
\vcenter{\hbox{
%
}}
$
\caption{Compatible generalized angle structures under the 2-3 Pachner move. 
Here, $\eta, \eta', \eta''$ are the dihedral angles of the bottom tetrahedron which are directly below $\theta, \theta', \theta''$, respectively, and $\zeta$ is a free parameter. 
The seams of the leaf space are perpendicular to the boundary. }
\label{fig:23Pachner-angles}
\end{figure}

In this setup, there is a natural map 
\[
\phi_{2 \rightarrow 3} : \Sk_q^{\mathfrak{gl}_1}(\widetilde{Y}_2, \Theta_2) \rightarrow \Sk_q^{\mathfrak{gl}_1}(\widetilde{Y}_3, \Theta_3)
\]
between the two $\mathfrak{gl}_1$-skein modules that can be constructed as follows. 
Let $L$ be any framed oriented link in $\widetilde{Y}_2$. 
We can visualize it by its projection $\overline{L}$ to $Y$; 
when put in a general position with respect to the leaf space (thought of as an embedded foam $F_2$ in $Y$), the projection $\overline{L} \subset Y$ carries a sheet label $i \in \{1, 2\}$ in each component of $\overline{L} \setminus F_2$, and the sheet label flips every time $\overline{L}$ crosses the leaf space $F_2 \subset Y$. 
Since $\overline{L} \subset Y$ is disjoint from the branch locus, we can perform the 1-parameter deformation $F_2 \rightsquigarrow F_3$ of the leaf space as in Figure \ref{fig:leaf-space-23Pachner}, while keeping it in general position with respect to $\overline{L}$. 
As a result, we obtain a link $\overline{L} \subset Y$ in general position with respect to the new foam $F_3 \subset Y$, equipped with a sheet label $i \in \{1,2\}$ for each component of $\overline{L} \setminus F_3$, which is nothing but a link in $\widetilde{Y}_3$. 
Let's call this resulting link $L' \subset \widetilde{Y}_3$. 
\begin{prop}\label{prop:gl1skein_23Pachner}
The map
\begin{align*}
\phi_{2 \rightarrow 3} : \Sk_q^{\mathfrak{gl}_1}(\widetilde{Y}_2) &\rightarrow \Sk_q^{\mathfrak{gl}_1}(\widetilde{Y}_3) \\
[L] &\mapsto [L']
\end{align*}
is well-defined. 
\end{prop}
\begin{proof}
We need to check that the map $L \mapsto [L']$ is invariant under isotopy of $L$ and respects the $\mathfrak{gl}_1$-skein relations on $L$. 
Invariance under isotopy (away from the branch locus) is trivial, as any isotopy of $L \subset \widetilde{Y}_2$ corresponds to an isotopy of $\overline{L}$ (where we allow crossing changes between strands with different sheet labels), which induces an isotopy of $L' \subset \widetilde{Y}_3$. 
This map also straightforwardly respects most of the $\mathfrak{gl}_1$-skein relations. 
The only non-trivial skein relations we need to check are:
\begin{enumerate}
    \item The sign relation \eqref{eq:gl1skeinrel3} on the branch locus connecting the two singular cone points of $\widetilde{Y}_2$. 
    \item The 3-term relations \eqref{eq:gl1skeinrel5} for the two singular cone points of $\widetilde{Y}_2$. 
\end{enumerate}
For the sign relations, let's say $L$ is a meridian of the branch locus connecting the two singular cone points of $\widetilde{Y}_2$. 
We need to show that the image of $L$ evaluates to $-1$ in $\Sk_q^{\mathfrak{gl}_1}(\widetilde{Y}_3)$. 
It is enough to observe that the image of $L$ is (modulo the usual $\mathfrak{gl}_1$-skein relations) a union of three meridians around some branch loci in $\widetilde{Y}_3$, and thus evaluates to $(-1)^3 = -1$; see Figure \ref{fig:sign-relation}. 
\begin{figure}[htbp]
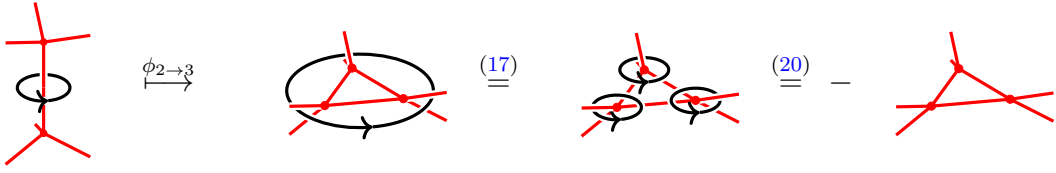

\centering
$
\vcenter{\hbox{
\tdplotsetmaincoords{60}{80}

}}
$
\caption{Checking the sign relation. Here, each strand is labeled by both sheets, so to get the corresponding links in $\widetilde{Y}_2$ and $\widetilde{Y}_3$, take the preimage under the projection.}
\label{fig:sign-relation}
\end{figure}

Now, it suffices to show that $\phi_{2 \rightarrow 3}$ respects the following relative version of the 3-term relation 
\[
\vcenter{\hbox{
%
}},
\]
which is equivalent to the closed version of the 3-term relations \eqref{eq:gl1skeinrel5}. 
Using the angles as in Figure \ref{fig:23Pachner-angles}, 
this can be directly verified as follows: 
\begin{align*}
&\phi_{2 \rightarrow 3}
\qty(
\vcenter{\hbox{
%
}}
),
\end{align*}
where, in the third to last equality, we changed the framing of the last tangle by $+1$, hence absorbing the factor of $q$. 
\end{proof}

\begin{rmk}
Proposition \ref{prop:gl1skein_23Pachner} is the singular analog of the pentagon relation for the quantum dilogarithm. 
In fact, when the angle structure is taut, we can resolve the singular cone points in a certain way and insert the quantum dilogarithm -- which satisfies the same 3-term recurrence relation -- at each of those resolved cone points. 
Then, Proposition \ref{prop:gl1skein_23Pachner} becomes exactly the well-known pentagon relation for the quantum dilogarithm. 
\end{rmk}

\begin{thm}
The quantum UV-IR map is natural with respect to the transition maps $\phi_{2 \rightarrow 3}$. 
That is, the following diagram commutes:
\[
\begin{tikzcd}
 & \Sk_q^{\mathfrak{gl}_1}(\widetilde{Y}_2) \arrow[dd, "\phi_{2 \rightarrow 3}"]\\
\Sk^{\mathfrak{gl}_2}_{q}(Y) \arrow[ur, "F_2"] \arrow[dr, "F_3"] & \\
 & \Sk_q^{\mathfrak{gl}_1}(\widetilde{Y}_3)
\end{tikzcd}.
\]
\end{thm}
\begin{proof}
The deformation of the leaf space (Figure \ref{fig:leaf-space-23Pachner}) preserves the Euclidean structures on the facets of the leaf space, so this is immediate from the construction of the quantum UV-IR map. 
\end{proof}

\subsection{Stated quantum UV-IR map}\label{subsec:stated_quantum_UV-IR}
So far, we have reviewed the construction of the quantum trace map (Section \ref{sec:quantumtrace}) and the quantum UV-IR map (Section \ref{subsec:2dUVIR}-\ref{subsec:quantumUVIR-Pachner}). 
While they are both maps that ``abelianize'' the skein modules in a sense, their constructions were quite different -- 
the quantum trace map is defined mostly algebraically, while the quantum UV-IR map is defined topologically, crucially using the 1-dimensional foliation induced by a choice of ideal triangulation. 

In this subsection, we bring the quantum UV-IR map closer to the quantum trace map by constructing a stated version of the quantum UV-IR map
\[
F : \Sk_q^{\mathfrak{gl}_2}(Y) \rightarrow \Sk_q^{\mathfrak{gl}_1}(\widetilde{Y}).
\]
Eventually, this will allow us to use the cut-and-glue approach to check the compatibility between the two maps locally. 

The $\mathfrak{gl}_2$-skeins and $\mathfrak{gl}_1$-skeins with defects admit natural extensions to stated tangles; 
the precise definition of the stated $\mathfrak{gl}_2$-skein modules and the stated $\mathfrak{gl}_1$-skein modules with defects that we use are given in Appendix \ref{sec:stated-gl2-skeins} and \ref{sec:stated-gl1-skeins}, respectively, where we also show that these stated skein modules behave nicely under splitting of 3-manifolds, analogously to stated $\mathfrak{sl}_2$-modules. 
Since the construction of the quantum UV-IR map was local, in order to extend it to a map between stated skein modules, we simply need to carefully determine the dictionary between the boundary states of stated tangles and the boundary condition for lifts (i.e., which of the two sheets it needs to be lifted to). 

Before stating the rules for determining the boundary condition, let's briefly recall (from Section \ref{subsec:2dUVIR}-\ref{subsec:3dSpecNet}) our conventions on the sheet labels $1$ and $2$. 
Given an ideal triangulation and the associated WKB foliation, we orient the foliation so that, in the complement of the branch cut, 
\begin{itemize}
\item the ``sheet 2''-direction is always pointing toward the ideal vertices, and
\item the ``sheet 1''-direction is always pointing away from the ideal vertices. 
\end{itemize}

With this orientation convention in mind, here's the rule determining the boundary condition for the lifts of stated tangles:
\begin{defn}[Dictionary between boundary states and sheets]\label{defn:dictionary-boundary-state-sheet}
Depending on the boundary state and orientation of the tangle, we impose the following boundary conditions on the lifts of stated tangles:
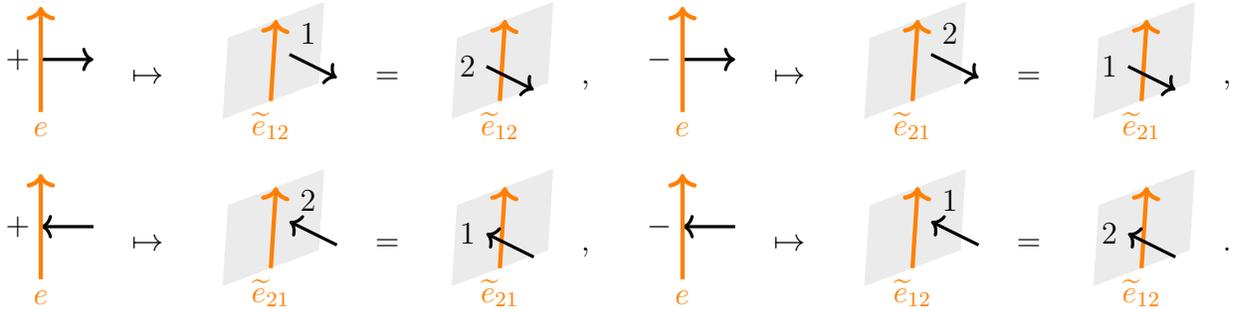
\[
\vcenter{\hbox{

}}
\;\;.
\]
Here, $\widetilde{e}_{ij}$ ($ij \in \{12, 21\}$) denotes the lift of the boundary marking $e$ to $\widetilde{Y}$, for which the sheet labels are $i$ and $j$ on the left and the right of $\widetilde{e}_{ij}$, respectively, if $\widetilde{e}_{ij}$ is pointing upward and viewed from outside of $\widetilde{Y}$. 
\end{defn}

\begin{thm}\label{thm:stated-quantum-UV-IR}
With the boundary conditions given as in Definition \ref{defn:dictionary-boundary-state-sheet}, we get a well-defined quantum UV-IR map
\[
F : \Sk_q^{\mathfrak{gl}_2}(Y) \rightarrow \Sk_q^{\mathfrak{gl}_1}(\widetilde{Y}), 
\]
between the stated skein modules, which factors through the reduced $\mathfrak{gl}_2$-skein module
\[
F : \Sk_q^{\mathfrak{gl}_2}(Y) \rightarrow \overline{\Sk}_q^{\mathfrak{gl}_2}(Y) \rightarrow\Sk_q^{\mathfrak{gl}_1}(\widetilde{Y}).
\]
\end{thm}
\begin{proof}
The proof of invariance under isotopy and that it respects the interior (i.e., non-stated) $\mathfrak{gl}_2$-skein relations \eqref{eq:gl2skeinrel1}-\eqref{eq:gl2extraskeinrel2} is the same as in the proof of Theorem \ref{thm:3dquantumUVIR}. 
Thus, we need to check that the quantum UV-IR map respects the boundary (i.e., stated) $\mathfrak{gl}_2$-skein relations \eqref{eq:gl2boundaryskeinrel1}-\eqref{eq:gl2boundaryskeinrel4}. 
\begin{itemize}
\item The first boundary skein relation \eqref{eq:gl2boundaryskeinrel1}:
\begin{align*}
F\qty(
\qty[
\vcenter{\hbox{

}}
]^{\mathfrak{gl}_2}
)
\end{align*}
\item The second boundary skein relation \eqref{eq:gl2boundaryskeinrel2}: 
\begin{align*}
F\qty(
\qty[
\vcenter{\hbox{
%
}}
]^{\mathfrak{gl}_2}
).
\end{align*}
\item The third boundary skein relation \eqref{eq:gl2boundaryskeinrel3}:
\begin{align*}
F\qty(
\qty[
%
]^{\mathfrak{gl}_2}
).
\end{align*}
\item The fourth boundary skein relation \eqref{eq:gl2boundaryskeinrel4}: 
\begin{align*}
F\qty(
\qty[
%
]^{\mathfrak{gl}_2}
).
\end{align*}
\end{itemize}
For other skein relations obtained by simultaneous orientation reversal of the tangles, the proof is essentially identical. 

Finally, it is easy to see that the bad arcs map to $0$, and hence $F$ descends to a map from the reduced $\mathfrak{gl}_2$-skein module. 
\end{proof}

Recall (e.g. from Proposition \ref{prop:bimodule-structure}) that the stated skein modules carry natural bimodule structures;
specifically, 
$\overline{\Sk}_q^{\mathfrak{gl}_2}(Y, \Gamma)$ is a $\overline{\SkAlg}_q^{\mathfrak{gl}_2}(V(\Gamma)^+)$-$\overline{\SkAlg}_q^{\mathfrak{gl}_2}(V(\Gamma)^-)$-bimodule, and 
$\Sk_q^{\mathfrak{gl}_1}(\widetilde{Y}, \widetilde{\Gamma})$ is a $\SkAlg_q^{\mathfrak{gl}_1}(V(\widetilde{\Gamma})^+)$-$\SkAlg_q^{\mathfrak{gl}_1}(V(\widetilde{\Gamma})^-)$-bimodule. 
Since the quantum UV-IR map $F : \overline{\Sk}_q^{\mathfrak{gl}_2}(Y, \Gamma) \rightarrow \Sk_q^{\mathfrak{gl}_1}(\widetilde{Y}, \widetilde{\Gamma})$ is defined locally, it is in fact a bimodule map when $\Sk_q^{\mathfrak{gl}_1}(\widetilde{Y}, \widetilde{\Gamma})$ is regarded as a $\overline{\SkAlg}_q^{\mathfrak{gl}_2}(V(\Gamma)^+)$-$\overline{\SkAlg}_q^{\mathfrak{gl}_2}(V(\Gamma)^-)$-bimodule.
This bimodule structure is induced from the algebra homomorphisms 
\[
F: \overline{\SkAlg}_q^{\mathfrak{gl}_2}(V(\Gamma)^{\pm}) \rightarrow \SkAlg_q^{\mathfrak{gl}_1}(V(\widetilde{\Gamma})^{\pm})
\]
determined by the quantum UV-IR map computed near the vertices of the boundary markings. 

\begin{eg}\label{eg:angled-homomorphism}
The algebra homomorphism 
\[
F: \overline{\SkAlg}_q^{\mathfrak{gl}_2}(V(\Gamma)^{\pm}) \rightarrow \SkAlg_q^{\mathfrak{gl}_1}(V(\widetilde{\Gamma})^{\pm})
\]
depends in general on the angles at the vertices. 
Here, we describe this map explicitly in the case of vertices of degree 2 and 3. 
Around those vertices, the leaf space is locally modeled by the angled biangular and triangular prism depicted in Figure \ref{fig:angled-biangular-prism} and \ref{fig:angled-prism}. 
\begin{figure}[htbp]
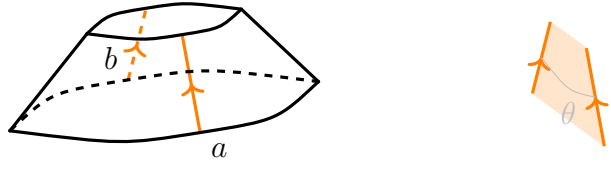

\centering
$
\vcenter{\hbox{
\tdplotsetmaincoords{70}{110}

}}
$
\caption{An angled biangular prism $(D_2 \times I)_{\theta}$ (left) and the associated leaf space (right)}
\label{fig:angled-biangular-prism}
\end{figure}

\begin{figure}[htbp]
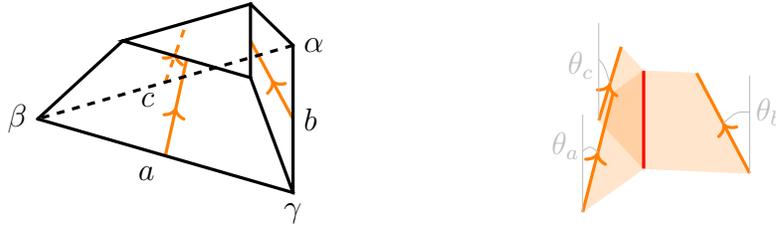

\centering
$
\vcenter{\hbox{
\tdplotsetmaincoords{60}{60}
%
}}
$
\caption{An angled triangular prism $(D_3 \times I)_{\theta_a, \theta_b, \theta_c}$ (left) and the associated leaf space (right)}
\label{fig:angled-prism}
\end{figure}
Direct calculation shows that the quantum UV-IR map on the angled biangular prism $(D_2 \times I)_{\theta}$ is the algebra homomorphism given by\footnote{
Here, we used a notational convention where $\overrightarrow{x_\mu y_\nu}$ ($\mu, \nu \in \{\pm\}$) denotes the elementary stated $\mathfrak{gl}_2$ tangle connecting the boundary markings $x$ and $y$ -- with boundary states $\mu$ and $\nu$, respectively -- in the simplest possible way, oriented from $x$ to $y$. 
We use similar notations for elementary $\mathfrak{gl}_1$ tangles, with $x_{ij}$ ($ij \in \{12, 21\}$) denoting the lift of the boundary marking $x$ to the double cover, where the sheet labels are $i$ and $j$ on the left and right of $x_{ij}$, when viewed from outside of the 3-manifold and if $x_{ij}$ is oriented upward;
we have already used a similar notational convention in Definition \ref{defn:dictionary-boundary-state-sheet}. 
} \footnote{These elementary tangles generate the reduced $\mathfrak{gl}_2$-skein algebra of the biangle; see Lemma \ref{lem:reduced-gl2-biangle}.}
\[
F : \overline{\SkAlg}_q^{\mathfrak{gl}_2}(D_2) \rightarrow \SkAlg_q^{\mathfrak{gl}_1}(\widetilde{D_2}) = \SkAlg_q^{\mathfrak{gl}_1}(D_2)^{\otimes 2}
\]
\begin{gather*}
\overrightarrow{a_+b_+} \mapsto q^{-\frac{\theta}{2\pi}} 1 \otimes \overrightarrow{a_{12}b_{21}}, \quad 
\overleftarrow{a_-b_-} \mapsto q^{\frac{\theta}{2\pi}}1 \otimes \overleftarrow{a_{12}b_{21}}, \\
\overrightarrow{a_-b_-} \mapsto q^{\frac{\theta}{2\pi}} \overrightarrow{a_{21}b_{12}} \otimes 1, \quad 
\overleftarrow{a_+b_+} \mapsto q^{-\frac{\theta}{2\pi}} \overleftarrow{a_{21}b_{12}} \otimes 1. 
\end{gather*}
Likewise, the quantum UV-IR map on the angled triangular prism $(D_3 \times I)_{\theta_a, \theta_b, \theta_c}$ is the algebra homomorphism given by\footnote{These elementary tangles (and their inverses) generate the reduced $\mathfrak{gl}_2$-skein algebra of the triangle, as shown in Lemma \ref{lem:red_gl2_triangle}.}
\[
F : \overline{\SkAlg}_q^{\mathfrak{gl}_2}(D_3) \rightarrow \SkAlg_q^{\mathfrak{gl}_1}(\widetilde{D_3})
\]
\begin{gather*}
\overrightarrow{a_+b_+} \mapsto q^{-\frac{\theta_a + \theta_b}{2\pi}} \overrightarrow{a_{12}b_{21}}, \quad
\overleftarrow{a_+b_+} \mapsto q^{-\frac{\theta_a + \theta_b}{2\pi}} \overleftarrow{a_{21}b_{12}}, \\
\overrightarrow{b_+c_+} \mapsto q^{-\frac{\theta_b + \theta_c}{2\pi}} \overrightarrow{b_{12}c_{21}}, \quad
\overleftarrow{b_+c_+} \mapsto q^{-\frac{\theta_b + \theta_c}{2\pi}} \overleftarrow{b_{21}c_{12}}, \\
\overrightarrow{c_+a_+} \mapsto q^{-\frac{\theta_c + \theta_a}{2\pi}} \overrightarrow{c_{12}a_{21}}, \quad
\overleftarrow{c_+a_+} \mapsto q^{-\frac{\theta_c + \theta_a}{2\pi}} \overleftarrow{c_{21}a_{12}}. 
\end{gather*}

\end{eg}

\begin{prop}\label{prop:quantum-UV-IR-bimodule-homomorphism}
Considering $\Sk_q^{\mathfrak{gl}_1}(\widetilde{Y}, \widetilde{\Gamma})$ as a $\overline{\SkAlg}_q^{\mathfrak{gl}_2}(V(\Gamma)^+)$-$\overline{\SkAlg}_q^{\mathfrak{gl}_2}(V(\Gamma)^-)$-bimodule, with the bimodule structure induced from the algebra homomorphism 
\[
F: \overline{\SkAlg}_q^{\mathfrak{gl}_2}(V(\Gamma)^{\pm}) \rightarrow \SkAlg_q^{\mathfrak{gl}_1}(V(\widetilde{\Gamma})^{\pm}),
\]
the quantum UV-IR map
\[
F : \overline{\Sk}_q^{\mathfrak{gl}_2}(Y, \Gamma) \rightarrow \Sk_q^{\mathfrak{gl}_1}(\widetilde{Y}, \widetilde{\Gamma})
\]
is a bimodule homomorphism. 
\end{prop}
\begin{proof}
This is immediate from the locality of the quantum UV-IR map. 
\end{proof}

It also follows immediately from the fact that the quantum UV-IR map is constructed locally that it is compatible with the splitting maps on both sides: 
\begin{prop} \label{prop:UVIRMapCompatibleWithSplitting}
The quantum UV-IR map is compatible with the splitting maps on both sides. 
That is, if $Y$ is obtained by gluing $Y_i$'s with compatible branched covers and generalized angle structures, we have a commuting square
\[
\begin{tikzcd}
\overline{\Sk}^{\mathfrak{gl}_2}_{q}(Y) \arrow[d, "\sigma^{\mathfrak{gl}_2}"] \arrow[r, "F_{Y}"] & \Sk^{\mathfrak{gl}_1}_{q}(\widetilde{Y}) \arrow[d, "\sigma^{\mathfrak{gl}_1}"]\\
\overline{\bigotimes}_{i}\overline{\Sk}^{\mathfrak{gl}_2}_{q}(Y_i) \arrow[r, "\overline{\otimes}_i F_{Y_i}"] & \overline{\bigotimes}_{i}\Sk^{\mathfrak{gl}_1}_{q}(\widetilde{Y}_i)
\end{tikzcd}. 
\]
\end{prop}
Here, $\overline{\bigotimes}_{i}\Sk^{\mathfrak{gl}_1}_{q}(\widetilde{Y}_i)$ denotes the relative tensor product of the stated $\mathfrak{gl}_1$-skein modules $\Sk^{\mathfrak{gl}_1}_q(\widetilde{Y}_i)$, i.e., the quotient of the ordinary tensor product $\bigotimes_{i}\Sk^{\mathfrak{gl}_1}_{q}(\widetilde{Y}_i)$ by the gluing relations, as well as extra 3-term relations for each cone point of $\widetilde{Y}$; see Appendix \ref{sec:stated-gl1-skeins}.

\section{Compatibility for surfaces}\label{sec:surface-compatibility}
Now that we have carefully reviewed both the quantum trace map and the quantum UV-IR map (both in 2d and 3d), we are ready to compare them; 
we will do this for surfaces in this section, and for 3-manifolds in Section \ref{sec: 3-manifold compatibility}. 

Let $\Sigma$ be a surface without boundary (but with punctures), and let $\widetilde{\Sigma} = \widetilde{\Sigma}_{\tau}$ be its branched double cover associated to some ideal triangulation $\tau$ of $\Sigma$. 
In this section, we will construct algebra homomorphisms
\[
\pi : \overline{\SkAlg}^{\mathfrak{gl}_2}_{q}(\Sigma) \rightarrow \overline{\SkAlg}^{\mathfrak{sl}_2}_A(\Sigma) \otimes \SkAlg^{\mathfrak{gl}_1}_{-A}(\Sigma)
\]
and
\[
\evmap : \SkAlg^{\mathfrak{gl}_1}_q(\widetilde{\Sigma}_\tau) \rightarrow \SQTS_{\tau}(\Sigma) \otimes \SkAlg^{\mathfrak{gl}_1}_{-A}(\Sigma)
\]
which fit into the following commutative square of algebra maps: 
\begin{equation}\label{eqn: surface compatibility}
\begin{tikzcd}
\overline{\SkAlg}^{\mathfrak{gl}_2}_{q}(\Sigma) \arrow[d, "\pi"] \arrow[r, "F_{\tau}"] & \SkAlg^{\mathfrak{gl}_1}_{q}(\widetilde{\Sigma}_{\tau}) \arrow[d, "\evmap"]\\
\overline{\SkAlg}^{\mathfrak{sl}_2}_{A}(\Sigma) \otimes \SkAlg^{\mathfrak{gl}_1}_{-A}(\Sigma) \arrow[r, "\Tr_{\tau} \otimes \mathrm{id}"] & \SQTS_{\tau}(\Sigma) \otimes \SkAlg^{\mathfrak{gl}_1}_{-A}(\Sigma)
\end{tikzcd}.
\end{equation}
We will do so by first constructing a similar commutative square for ideal triangles and then show that they glue consistently. 
For surfaces with boundary (such as ideal triangles), the top left and right corner of such commutative diagrams will carry a product twisted by a sign that we describe below.

\subsection{From $\mathfrak{gl}_2$ to $\mathfrak{sl}_2$}\label{subsec:gl2tosl2}
For each stated oriented tangle $\vec{L}$ in $Y$, we will define a quantity $b(\vec{L}) \in \frac{1}{2}\mathbb{Z}/2\mathbb{Z}$ that depends only on the boundary behavior of $\vec{L}$ (hence $b$ for boundary). 
This will play a crucial role in the construction of the left vertical map $\pi$ in the commutative square. 
We will also use it to twist the product structures on $\SkAlg^{\mathfrak{gl}_2}_{q}(\Sigma)$ and on $\SkAlg^{\mathfrak{gl}_1}_{q}(\widetilde{\Sigma})$. 

\begin{defn}
Let $\vec{L}$ be a stated oriented tangle in a boundary marked 3-manifold $(Y,\Gamma)$. 
Define $b(\vec{L}) \in \frac{1}{2}\mathbb{Z}/2\mathbb{Z}$ to be
\[
b(\vec{L}) := \sum_{e \in E(\Gamma)}\sum_{\{p_1, p_2\} \subset e} b(\{p_1, p_2\}),
\]
where the second sum is over all (unordered) pairs $\{p_1, p_2\}$ of boundary points of $\vec{L}$ on the edge $e$ of the boundary marking $\Gamma$, 
and $b(\{p_1, p_2\}) \in \{\pm \frac{1}{2}, 0\}$ is determined by the following rules: 
\[
\vcenter{\hbox{

}}
\;\;\rightsquigarrow\;\; 
\frac{1}{2}
\;\;
,
\]
and all other contributions are zero.\footnote{Note that the only non-zero contributions are from pairs of boundary points which lift to different boundary markings under the quantum UV-IR map (see the dictionary \ref{defn:dictionary-boundary-state-sheet}. } 
\end{defn}

Now, we are ready to construct the natural ``$\mathfrak{gl}_2$-$\mathfrak{sl}_2$ map'' from the $\mathfrak{gl}_2$-skeins to the tensor product of $\mathfrak{sl}_2$ and $\mathfrak{gl}_1$-skeins. 
\begin{prop}\label{prop:gl2tosl2}
The map
\begin{align*}
\pi : \Sk^{\mathfrak{gl}_2}_{q}(Y) &\rightarrow \Sk^{\mathfrak{sl}_2}_A(Y) \otimes \Sk^{\mathfrak{gl}_1}_{-A}(Y) \\
[\vec{L}]^{\mathfrak{gl}_2} &\mapsto (-1)^{b(\vec{L})} [L]^{\mathfrak{sl}_2} \otimes [\vec{L}]^{\mathfrak{gl}_1}
\end{align*}
is a well-defined $R$-module homomorphism. 
\end{prop}
\begin{proof}
We need to check that this map respects all the stated $\mathfrak{gl}_2$-skein relations, i.e., the relations \eqref{eq:gl2skeinrel1}-\eqref{eq:gl2extraskeinrel2} 
and the boundary relations \eqref{eq:gl2boundaryskeinrel1}-\eqref{eq:gl2boundaryskeinrel4}. 

\begin{itemize}
\item The first skein relation \eqref{eq:gl2skeinrel1}: 
\begin{align*}
&\pi\qty(
\qty[
\vcenter{\hbox{

}}
]^{\mathfrak{gl}_2}
).
\end{align*}

\item The second skein relation \eqref{eq:gl2skeinrel2}:
\begin{align*}
\pi\qty(
\qty[
\vcenter{\hbox{
%
}}
]^{\mathfrak{gl}_1}
= (q+q^{-1}) [\emptyset]^{\mathfrak{sl}_2} \otimes [\emptyset]^{\mathfrak{gl}_1}
= \pi\qty((q+q^{-1})[\emptyset]^{\mathfrak{gl}_2}).
\end{align*}

\item The third skein relation \eqref{eq:gl2skeinrel3}: 
\[
\pi\qty(
\qty[
\vcenter{\hbox{
%
}}
]^{\mathfrak{gl}_2}
)
\]

\item The fourth skein relation \eqref{eq:gl2extraskeinrel1}: 
Using
\begin{equation}\label{eq:small-lemma-gl2-sl2}
\pi\qty(
\qty[
q
\vcenter{\hbox{
%
}}
]^{\mathfrak{gl}_1}
\\
&= 
0.
\end{align*}

\item The fifth skein relation \eqref{eq:gl2extraskeinrel2}: 
\begin{align*}
&\pi\qty(
\qty[
\vcenter{\hbox{
%
}}
]^{\mathfrak{gl}_2}
).
\end{align*}

\item The first boundary skein relation \eqref{eq:gl2boundaryskeinrel1}:
\begin{align*}
\pi\qty(
\qty[
\vcenter{\hbox{
%
}}
]^{\mathfrak{gl}_2}
).
\end{align*}

\item The second boundary skein relation \eqref{eq:gl2boundaryskeinrel2}:
\begin{align*}
\pi\qty(
\qty[
\vcenter{\hbox{
%
}}
]^{\mathfrak{gl}_2}
).
\end{align*}

\item The third boundary skein relations \eqref{eq:gl2boundaryskeinrel3}: 
\begin{align*}
\pi\qty(
\qty[
%
]^{\mathfrak{gl}_2}
).
\end{align*}

\item The fourth boundary skein relations \eqref{eq:gl2boundaryskeinrel4}: 
\begin{align*}
\pi\qty(
\qty[
%
]^{\mathfrak{gl}_2}
).
\end{align*}
\end{itemize}
For the other skein relations obtained by simultaneous orientation reversal of the tangles, the proof is identical. 
\end{proof}

\subsection{Sign-twisted products}
Because of the extra factor $b(\vec{L})$, the $\mathfrak{gl}_2$-$\mathfrak{sl}_2$ map
\begin{align*}
\pi : \Sk^{\mathfrak{gl}_2}_{q}(\Sigma \times I) &\rightarrow \Sk^{\mathfrak{sl}_2}_A(\Sigma \times I) \otimes \Sk^{\mathfrak{gl}_1}_{-A}(\Sigma \times I) \\
[\vec{L}]^{\mathfrak{gl}_2} &\mapsto (-1)^{b(\vec{L})} [L]^{\mathfrak{sl}_2} \otimes [\vec{L}]^{\mathfrak{gl}_1}
\end{align*}
is \emph{not} an algebra homomorphism. 
We would like to fix this by modifying the product on $\SkAlg^{\mathfrak{gl}_2}_{q}(\Sigma)$ accordingly. 

The quantity $b(\vec{L})$ itself unfortunately does not behave nicely with respect to the $\mathfrak{gl}_2$ skein relations (i.e., it does \emph{not} give a grading on the $\mathfrak{gl}_2$ skein module). 
However, the relative version of $b$ that we define below behaves well. 

\begin{defn}
Given two stated oriented tangles $\vec{L}_1$ and $\vec{L}_2$ in $\Sigma \times I$, define
\[
b(\vec{L}_1, \vec{L}_2) := b(\vec{L}_1 \cdot \vec{L}_2) - b(\vec{L}_1) - b(\vec{L}_2) \in \frac{1}{2}\mathbb{Z}/2\mathbb{Z}, 
\]
where $\vec{L}_1 \cdot \vec{L}_2$ denotes the stated oriented tangle in $\Sigma \times I$ obtained by stacking $\vec{L}_1$ above $\vec{L}_2$ as usual. 
\end{defn}
It is useful to note that $b(\vec{L}_1, \vec{L}_2) = -b(\vec{L}_2, \vec{L}_1)$. 

\begin{prop}\label{prop:b_grading}
The quantity $b(\vec{L}_1, \vec{L}_2)$ is invariant under stated $\mathfrak{gl}_2$-skein relations applied to $\vec{L}_1$ or to $\vec{L}_2$. 
Therefore, for any fixed $\vec{L}$, 
\[
b(\vec{L}, \cdot) = -b(\cdot, \vec{L})
\]
gives a grading on $\Sk^{\mathfrak{gl}_2}_q(\Sigma \times I)$. 
\end{prop}
\begin{proof}
The only thing we need to check is the invariance under the stated skein relations \eqref{eq:gl2skeinrel3} and \eqref{eq:gl2boundaryskeinrel1} which creates (or annihilates) a pair of boundary points, which is immediate from our definition of $b(\vec{L})$. 
\end{proof}

\begin{defn}
Define the \emph{sign-twisted $\mathfrak{gl}_2$-skein algebra} $\SkAlg^{\mathfrak{gl}_2, \st}_q(\Sigma)$ to be 
the usual $\mathfrak{gl}_2$-skein module $\Sk^{\mathfrak{gl}_2}_q(\Sigma \times I)$, 
but with the product structure twisted by a sign determined by $(-1)^{b(\cdot, \cdot)}$. 
That is, 
\[
[\vec{L}_1] \overset{\st}{\cdot} [\vec{L}_2] := (-1)^{-b(\vec{L}_1, \vec{L}_2)}[\vec{L}_1 \cdot \vec{L}_2].
\]
This gives a well-defined associative product, as $b(\cdot, \cdot)$ satisfies the following obvious cocycle condition:
\[
b(\vec{L}_1, \vec{L}_2 \cdot \vec{L}_3) + b(\vec{L}_2, \vec{L}_3) = b(\vec{L}_1, \vec{L}_2) + b(\vec{L}_1, \vec{L}_3) + b(\vec{L}_2, \vec{L}_3) = b(\vec{L}_1, \vec{L}_2) + b(\vec{L}_1 \cdot \vec{L}_2, \vec{L}_3). 
\]

Likewise, define the \emph{sign-twisted $\mathfrak{gl}_2$-skein module} $\Sk^{\mathfrak{gl}_2, \st}_q(Y,\Gamma)$ to be the usual $\mathfrak{gl}_2$-skein module $\Sk^{\mathfrak{gl}_2}_q(Y,\Gamma)$, 
but with the sign-twisted bimodule structure, i.e., considered as a $\otimes_{v\in V(\Gamma)^+}\SkAlg^{\mathfrak{gl}_2,\st}_q(D_{\deg v})$-$\otimes_{w\in V(\Gamma)^-}\SkAlg^{\mathfrak{gl}_2,\st}_q(D_{\deg w})$-bimodule. 
\end{defn}

\begin{rmk}
Since our sign-twisting only changes the product by a sign, 
bad arcs generate the same ideal as before, and it follows that the sign-twisting descends well to reduced skein modules. 
\end{rmk}

\begin{prop}\label{prop:gl2-sl2-map}
The map
\begin{align*}
\pi : \SkAlg^{\mathfrak{gl}_2, \st}_{q}(\Sigma) &\rightarrow \SkAlg^{\mathfrak{sl}_2}_A(\Sigma) \otimes \SkAlg^{\mathfrak{gl}_1}_{-A}(\Sigma) \\
[\vec{L}]^{\mathfrak{gl}_2} &\mapsto (-1)^{b(\vec{L})} [L]^{\mathfrak{sl}_2} \otimes [\vec{L}]^{\mathfrak{gl}_1}
\end{align*}
is an algebra homomorphism. 
\end{prop}
\begin{proof}
This is immediate from the definition of the sign-twisted product: 
\begin{align*}
\pi\qty(
[\vec{L}_1]^{\mathfrak{gl}_2} \overset{\st}{\cdot} [\vec{L}_2]^{\mathfrak{gl}_2} 
)
&= 
\pi\qty(
(-1)^{-b(\vec{L}_1, \vec{L}_2)}
[\vec{L}_1 \cdot \vec{L}_2]^{\mathfrak{gl}_2} 
) \\
&= 
(-1)^{-b(\vec{L}_1, \vec{L}_2) + b(\vec{L}_1 \cdot \vec{L}_2)}
[L_1 \cdot L_2]^{\mathfrak{sl}_2} \otimes [\vec{L}_1 \cdot \vec{L}_2]^{\mathfrak{gl}_1}\\
&= 
\qty(
(-1)^{b(\vec{L}_1)} [L_1]^{\mathfrak{sl}_2} \otimes [\vec{L}_1]^{\mathfrak{gl}_1}
) \cdot 
\qty(
(-1)^{b(\vec{L}_2)} [L_2]^{\mathfrak{sl}_2} \otimes [\vec{L}_2]^{\mathfrak{gl}_1}
)\\
&= \pi([\vec{L}_1]) \cdot \pi([\vec{L}_2]).
\end{align*}
\end{proof}
From the above, it follows that both $\Sk^{\mathfrak{gl}_2, \st}_{q}(Y)$ and $\Sk^{\mathfrak{sl}_2}_A(Y) \otimes \Sk^{\mathfrak{gl}_1}_{-A}(Y)$ can be considered as $\otimes_{v\in V(\Gamma)^+}\SkAlg^{\mathfrak{gl}_2,\st}_q(D_{\deg v})$-$\otimes_{w\in V(\Gamma)^-}\SkAlg^{\mathfrak{gl}_2,\st}_q(D_{\deg w})$-bimodules, where the bimodule structure on the latter is defined using $\pi$, and $\Gamma$ is the boundary marking for $Y$. 
The following proposition is clear: 
\begin{prop}
The map
\begin{align*}
\pi : \Sk^{\mathfrak{gl}_2, \st}_{q}(Y) &\rightarrow \Sk^{\mathfrak{sl}_2}_A(Y) \otimes \Sk^{\mathfrak{gl}_1}_{-A}(Y) \\
[\vec{L}]^{\mathfrak{gl}_2} &\mapsto (-1)^{b(\vec{L})} [L]^{\mathfrak{sl}_2} \otimes [\vec{L}]^{\mathfrak{gl}_1}
\end{align*}
is a bimodule homomorphism. 
\end{prop}

We can define the sign-twisted $\mathfrak{gl}_1$-skein algebra of $\widetilde{\Sigma}$ in exactly the same way, using the dictionary translating between the states and the lift described in Section \ref{subsec:stated_quantum_UV-IR}. 
That is, for any tangle $\vec{L}$ in $\widetilde{Y}$ whose boundary points are in general position so that no two boundary points project down to the same point in $Y$, 
we define $b(\vec{L}) \in \frac{1}{2}\mathbb{Z}/2\mathbb{Z}$ to be
\[
b(\vec{L}) := \sum_{e \in E(\Gamma)}\sum_{\substack{p_1 \in \widetilde{e}_{12} \\ p_2 \in \widetilde{e}_{21}}} b(p_1, p_2),
\]
where the second sum is over all pairs $(p_1, p_2)$ of boundary points of $\vec{L}$ which lie in different lifts $\widetilde{e}_{12}$ and $\widetilde{e}_{21}$ of the same edge $e$ of the boundary marking $\Gamma$ of $Y$, 
and $b(p_1, p_2) \in \{\pm \frac{1}{2}\}$ is determined by the following rules: 
\[
\vcenter{\hbox{

}}
\;\;\rightsquigarrow\;\; 
\frac{1}{2}
\;\;
.
\]

Then, given two oriented tangles $\vec{L}_1$ and $\vec{L}_2$ in $\widetilde{\Sigma} \times I$, define
\[
b(\vec{L}_1, \vec{L}_2) := b(\vec{L}_1 \cdot \vec{L}_2) - b(\vec{L}_1) - b(\vec{L}_2) \in \frac{1}{2}\mathbb{Z}/2\mathbb{Z}, 
\]
where $\vec{L}_1 \cdot \vec{L}_2$ denotes the oriented tangle in $\widetilde{\Sigma} \times I$ obtained by stacking $\vec{L}_1$ above $\vec{L}_2$ as usual. 
Analogous to Proposition \ref{prop:b_grading}, we have the following: 
\begin{prop}
The quantity $b(\vec{L}_1, \vec{L}_2)$ is invariant under $\mathfrak{gl}_1$-skein relations or isotopies that exchange the height of some pair of boundary points $p_1 \in \widetilde{e}_{12}$ and $p_2 \in \widetilde{e}_{21}$, applied to $\vec{L}_1$ or to $\vec{L}_2$. 
Therefore, for any fixed $\vec{L}$ in $\widetilde{\Sigma} \times I$, 
\[
b(\vec{L}, \cdot) = -b(\cdot, \vec{L})
\]
gives a grading on $\Sk^{\mathfrak{gl}_1}_q(\widetilde{\Sigma} \times I)$. 
\end{prop}

\begin{defn}
The \emph{sign-twisted} $\mathfrak{gl}_1$-skein algebra $\SkAlg^{\mathfrak{gl}_1,\st}_q(\widetilde{\Sigma})$ of $\widetilde{\Sigma}$ is defined to be the usual $\mathfrak{gl}_1$-skein module $\Sk^{\mathfrak{gl}_1}_q(\widetilde{\Sigma} \times I)$, but with the product structure twisted by a sign determined by $(-1)^{b(\cdot, \cdot)}$. 
That is, 
\[
[\vec{L}_1] \overset{\st}{\cdot} [\vec{L}_2] := (-1)^{-b(\vec{L}_1, \vec{L}_2)}[\vec{L}_1 \cdot \vec{L}_2].
\]
\end{defn}
Since we are using the same sign-twisting for both $\mathfrak{gl}_2$-skeins and $\mathfrak{gl}_1$-skeins, the following is clear. 
\begin{prop}
The quantum UV-IR map is still a bimodule homomorphism after sign-twisting. 
That is, 
\[
F : \SkAlg^{\mathfrak{gl}_2,\st}_q(\Sigma) \rightarrow \SkAlg^{\mathfrak{gl}_1,\st}_q(\widetilde{\Sigma})
\]
is still an algebra homomorphism, and
\[
F : \Sk^{\mathfrak{gl}_2,\st}_q(Y) \rightarrow \Sk^{\mathfrak{gl}_1,\st}_q(\widetilde{Y})
\]
is a bimodule homomorphism. 
\end{prop}

The $\mathfrak{gl}_2$-$\mathfrak{sl}_2$ map $\pi$ also behaves nicely with respect to reduced stated skein modules, as well as to $3$-manifolds split into face suspensions. 

\begin{prop} \label{prop:splitpi}
The map 
\[
\underset{f\in \mathcal{T}^{(2)}}{\overline{\bigotimes}}\pi_{Sf} : 
\underset{f\in \mathcal{T}^{(2)}}{\overline{\bigotimes}}\overline{\Sk}^{\mathfrak{gl}_2,\st}_q(Sf)
\rightarrow 
\underset{f\in \mathcal{T}^{(2)}}{\overline{\bigotimes}} \overline{\Sk}^{\mathfrak{sl}_2}_{A}(Sf)\otimes 
\underset{f\in \mathcal{T}^{(2)}}{\overline{\bigotimes}}\Sk^{\mathfrak{gl}_1}_{-A}(Sf)
\]
is well-defined.
\end{prop}
\begin{proof}
We need to show that the composition 
\[
\bigotimes_{f\in \mathcal{T}^{(2)}} \overline{\Sk}^{\mathfrak{gl}_2,\st}_q(Sf)
\overset{\bigotimes_{f\in \mathcal{T}^{(2)}} \pi_{Sf}}{\longrightarrow} 
\bigotimes_{f\in \mathcal{T}^{(2)}} \overline{\Sk}^{\mathfrak{sl}_2}_{A}(Sf)\otimes 
\bigotimes_{f\in \mathcal{T}^{(2)}} \Sk^{\mathfrak{gl}_1}_{-A}(Sf)
\rightarrow 
\underset{f\in \mathcal{T}^{(2)}}{\overline{\bigotimes}} \overline{\Sk}^{\mathfrak{sl}_2}_{A}(Sf)\otimes 
\underset{f\in \mathcal{T}^{(2)}}{\overline{\bigotimes}}\Sk^{\mathfrak{gl}_1}_{-A}(Sf)
\]
factors through $\overline{\bigotimes}_{f\in \mathcal{T}^{(2)}} \overline{\Sk}^{\mathfrak{gl}_2,\st}_q(Sf)$. 
This can be seen immediately by comparing the gluing relations in $\overline{\bigotimes}_{f\in \mathcal{T}^{(2)}}\overline{\Sk}^{\mathfrak{gl}_2,\st}_q(Sf)$ to those in $\overline{\bigotimes}_{f\in \mathcal{T}^{(2)}} \overline{\Sk}^{\mathfrak{sl}_2}_{A}(Sf)\otimes 
\overline{\bigotimes}_{f\in \mathcal{T}^{(2)}} \Sk^{\mathfrak{gl}_1}_{-A}(Sf)$. 
\end{proof}
Under splitting, the factors $(-1)^{b(\vec{L})}$ come in pairs and cancel themselves out. 
Hence, the following is immediate: 
\begin{prop} \label{prop:gl2sl2CompatibleWithSplitting}
The $\mathfrak{gl}_2$-$\mathfrak{sl}_2$ map is compatible with the splitting maps on both sides. 
That is, if $Y$ is obtained by gluing $Y_i$'s, we have a commuting square
\[
\begin{tikzcd}
\overline{\Sk}^{\mathfrak{gl}_2,\st}_{q}(Y) \arrow[d, "\sigma^{\mathfrak{gl}_2}"] \arrow[r, "\pi_{Y}"] &  \overline{\Sk}^{\mathfrak{sl}_2}_{A}(Y) \otimes \Sk^{\mathfrak{gl}_1}_{-A}(Y) \arrow[d, "\sigma^{\mathfrak{sl}_2} \otimes \sigma^{\mathfrak{gl}_1}"]\\
\overline{\bigotimes}_{i}\overline{\Sk}^{\mathfrak{gl}_2, \st}_{q}(Y_i) \arrow[r, "\overline{\otimes}_i \pi_{Y_i}"] & \overline{\bigotimes}_{i}\overline{\Sk}^{\mathfrak{sl}_2}_{A}(Y_i) \otimes \overline{\bigotimes}_{i}\Sk^{\mathfrak{gl}_1}_{-A}(Y_i)
\end{tikzcd}. 
\]
\end{prop}

\subsection{Commutative square for a triangle}\label{subsec:triangle-commutative-square}

Before defining the evaluation map for $\widetilde{\Sigma}$, we will do so first for the triangle. 
Let's first recall some basic facts about skein algebras of the triangle and its branched double cover. 

\begin{lem}
The $\mathfrak{gl}_1$-skein algebra
$\SkAlg^{\mathfrak{gl}_1}_{-A}(\tri)$ of the triangle $\tri$ 
-- with vertices labeled by $\alpha, \beta, \gamma$ in a counterclockwise manner and the opposite edges by $a, b, c$ -- 
is the quantum torus generated by
\[
\alpha := 
\vcenter{\hbox{

}}
\;,
\]
and their inverses, 
with relations
\[
\alpha\beta = (-A)\beta \alpha,\quad \beta\gamma = (-A)\gamma \beta,\quad \gamma \alpha = (-A) \alpha \gamma,
\]
\[
[\alpha\beta\gamma] = 1.
\]
\end{lem}

\begin{lem}\label{lem:sign_twisted_hexagon}
The sign-twisted $\mathfrak{gl}_1$-skein algebra $\SkAlg^{\mathfrak{gl}_1, \st}_q(\widetilde{\tri})$ of the branched double cover $\widetilde{\tri}$ of the triangle is the quantum torus generated by 
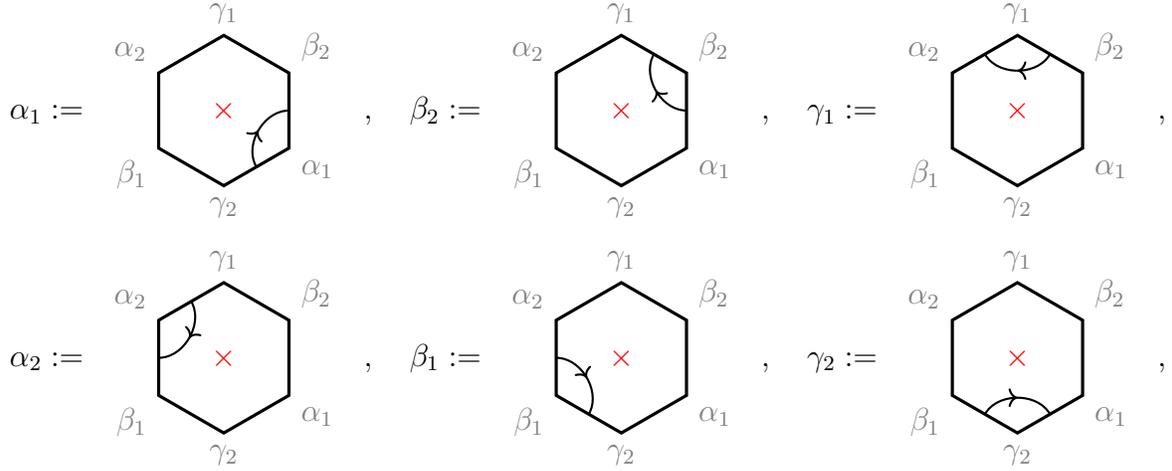
\begin{gather*}
\alpha_1 := 
\vcenter{\hbox{
%
}}\;,
\end{gather*}
and their inverses, 
with commutation relations
\begin{gather*}
\alpha_1 \beta_2 = q \beta_2 \alpha_1,\quad 
\beta_2 \gamma_1 = q \gamma_1 \beta_2,\quad
\gamma_1 \alpha_2 = q \alpha_2 \gamma_1, \\
\alpha_2 \beta_1 = q \beta_1 \alpha_2,\quad 
\beta_1 \gamma_2 = q \gamma_2 \beta_1,\quad
\gamma_2 \alpha_1 = q \alpha_1 \gamma_2, \\
\alpha_1 \gamma_1 = - \gamma_1 \alpha_1,\quad
\beta_2 \alpha_2 = - \alpha_2 \beta_2,\quad
\gamma_1 \beta_1 = - \beta_1 \gamma_1, \\
\alpha_2 \gamma_2 = - \gamma_2 \alpha_2,\quad
\beta_1 \alpha_1 = - \alpha_1 \beta_1,\quad
\gamma_2 \beta_2 = - \beta_2 \gamma_2.
\end{gather*}
All other pairs of generators commute, 
modulo one relation:
\begin{equation}\label{eq:gl1hexagonrel}
[\alpha_1 \beta_2 \gamma_1 \alpha_2 \beta_1 \gamma_2] = -1.
\end{equation}
\end{lem}

\begin{rmk}
The $-1$ in the last relation comes from the evaluation of a small unknot around the branch point. 
The sign-twisting does not change the sign of the Weyl-ordered product in this case. 
\end{rmk}

\begin{lem}\label{lem:red_gl2_triangle}
The sign-twisted reduced $\mathfrak{gl}_2$-skein algebra $\overline{\SkAlg}^{\mathfrak{gl}_2,\st}_q(\tri)$ of the triangle $\tri$ is generated by
\[
\overrightarrow{\alpha}_{++},
\overleftarrow{\alpha}_{++},
\overrightarrow{\beta}_{++}, 
\overleftarrow{\beta}_{++},
\overrightarrow{\gamma}_{++},
\overleftarrow{\gamma}_{++},
\]
and their inverses, 
where
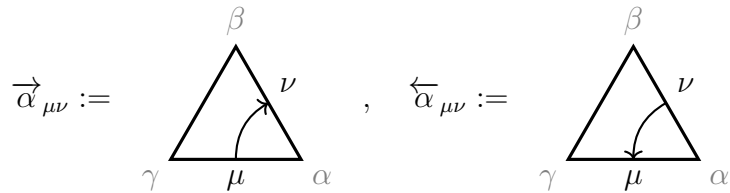
\[
\overrightarrow{\alpha}_{\mu\nu} := 
\vcenter{\hbox{
\begin{tikzpicture}[scale=1.0]
\coordinate (alpha) at ({sqrt(3)/2}, -1/2);
\coordinate (beta) at (0, 1);
\coordinate (gamma) at ({-sqrt(3)/2}, -1/2);
\coordinate (a) at ({-sqrt(3)/4}, 1/4);
\coordinate (b) at (0, -1/2);
\coordinate (c) at ({sqrt(3)/4}, 1/4);
\draw[very thick] (alpha) -- (beta) -- (gamma) -- cycle;
\node[anchor = north west, gray] at (alpha){$\alpha$}; 
\node[anchor = south, gray] at (beta){$\beta$};
\node[anchor = north east, gray] at (gamma){$\gamma$}; 
\draw[->, thick] (b) to[out=90, in=-150] (c);
\node[anchor = north] at (b){$\mu$};
\node[anchor = south west] at (c){$\nu$};
\end{tikzpicture}
}}
\;,\quad
\overleftarrow{\alpha}_{\mu\nu} := 
\vcenter{\hbox{
\begin{tikzpicture}[scale=1.0]
\coordinate (alpha) at ({sqrt(3)/2}, -1/2);
\coordinate (beta) at (0, 1);
\coordinate (gamma) at ({-sqrt(3)/2}, -1/2);
\coordinate (a) at ({-sqrt(3)/4}, 1/4);
\coordinate (b) at (0, -1/2);
\coordinate (c) at ({sqrt(3)/4}, 1/4);
\draw[very thick] (alpha) -- (beta) -- (gamma) -- cycle;
\node[anchor = north west, gray] at (alpha){$\alpha$}; 
\node[anchor = south, gray] at (beta){$\beta$};
\node[anchor = north east, gray] at (gamma){$\gamma$}; 
\draw[<-, thick] (b) to[out=90, in=-150] (c);
\node[anchor = north] at (b){$\mu$};
\node[anchor = south west] at (c){$\nu$};
\end{tikzpicture}
}}
\]
and likewise for $\beta$ and $\gamma$.

\end{lem}
\begin{proof}
It is straightforward to verify that 
\begin{equation} \label{eqn:gl2_triangle_rels}
\begin{gathered}
\overrightarrow{\alpha}_{--} = \overleftarrow{\alpha}_{++}^{-1}, \quad
\overleftarrow{\alpha}_{--} = \overrightarrow{\alpha}_{++}^{-1}, \\
\overrightarrow{\alpha}_{-+} = [\overleftarrow{\beta}_{++}\overleftarrow{\gamma}_{--}], \quad
\overleftarrow{\alpha}_{-+} = [\overrightarrow{\beta}_{++}\overrightarrow{\gamma}_{--}],
\end{gathered}
\end{equation}
in addition to the relations obtained from those above by cyclic permutations of $\alpha, \beta,$ and $\gamma$.

Let $\vec{L}$ be any $\mathfrak{gl}_2$-stated tangle in $\tri \times I$ representing an element $[\vec{L}]\in \overline{\SkAlg}^{\mathfrak{gl}_2,\st}_q(\tri)$. 
Consider its projection to the leaf space of $\tri \times I$. 
On each facet of the leaf space, we can ``push" the boundary marking closer to the binder while absorbing the crossings, cups, and caps, using the $\mathfrak{gl}_2$-stated skein relations. 
An example of this process, carried out on one facet of the leaf space, is shown below:
\begin{gather*}
\vcenter{\hbox{

}}.
\end{gather*}
Once the boundary marking is pushed sufficiently close to the binder, we obtain an expression of $[\vec{L}]$ as a linear combination of link diagrams that -- after a minor isotopy -- are products of stated tangles at constant heights. 
Such diagrams can be written as a word in $\overrightarrow{\alpha}_{\mu\nu}, \overrightarrow{\beta}_{\mu\nu}, \overrightarrow{\gamma}_{\mu\nu}, \overleftarrow{\alpha}_{\mu\nu}, \overleftarrow{\beta}_{\mu\nu}$, and $\overleftarrow{\gamma}_{\mu\nu}$, which can be written in terms of the $++$ generators using the relations in (\ref{eqn:gl2_triangle_rels}).

\end{proof}

\begin{lem}\label{lem:qUVIR_triangle_surjectivity}
The quantum UV-IR map on a triangle
\begin{align*}
F_{\tri} : \overline{\SkAlg}^{\mathfrak{gl}_2,\st}_q(\tri) &\rightarrow \SkAlg^{\mathfrak{gl}_1, \st}_q(\widetilde{\tri})
\end{align*}
is surjective. 
\end{lem}
\begin{proof}
Since it is an algebra homomorphism, it is enough to observe that all the generators of $\SkAlg^{\mathfrak{gl}_1, \st}_q(\widetilde{\tri})$ are in the image of $F_{\tri}$: 
\begin{align*}
F_{\tri} : \overline{\SkAlg}^{\mathfrak{gl}_2,\st}_q(\tri) &\rightarrow \SkAlg^{\mathfrak{gl}_1, \st}_q(\widetilde{\tri})\\ 
\overrightarrow{\alpha}_{--} &\mapsto \alpha_1, \\
\overrightarrow{\alpha}_{++} &\mapsto \alpha_2, \\ 
\overrightarrow{\beta}_{--} &\mapsto \beta_1, \\
\overrightarrow{\beta}_{++} &\mapsto \beta_2,\\ 
\overrightarrow{\gamma}_{--} &\mapsto \gamma_1, \\
\overrightarrow{\gamma}_{++} &\mapsto \gamma_2. 
\end{align*}
\end{proof}

\begin{cor}
If there exists a linear map
\[
\evmap_{\tri} : \SkAlg^{\mathfrak{gl}_1, \st}_q(\widetilde{\tri}) \rightarrow \widetilde{\mathbb{T}} \otimes \SkAlg^{\mathfrak{gl}_1}_{-A}(\tri)
\]
making the following square 
\begin{equation}\label{cd:triangle_compatibility}
\begin{tikzcd}
\overline{\SkAlg}^{\mathfrak{gl}_2, \st}_{q}(\tri) \arrow[d, "\pi_{\tri}"] \arrow[r, "F_{\tri}"] & \SkAlg^{\mathfrak{gl}_1, \st}_{q}(\widetilde{\tri}) \arrow[d, dashed, "\evmap_{\tri}"] \\
\overline{\SkAlg}^{\mathfrak{sl}_2}_{A}(\tri) \otimes \SkAlg^{\mathfrak{gl}_1}_{-A}(\tri) \arrow[r, "\Tr_{\tri} \otimes \mathrm{id}"] & \widetilde{\mathbb{T}} \otimes \SkAlg^{\mathfrak{gl}_1}_{-A}(\tri)
\end{tikzcd}
\end{equation}
commutative, then it is unique, and it must be an algebra homomorphism.
\end{cor}
\begin{proof}
This follows immediately from the surjectivity of $F$ (Lemma \ref{lem:qUVIR_triangle_surjectivity}) and the fact that the other $3$ arrows in the diagram are algebra homomorphisms. 
\end{proof}

Below, we show that the evaluation map $\evmap$ indeed exists:
\begin{thm} \label{thm:evMapTriangle}
There is an algebra homomorphism
\[
\evmap_{\tri} : \SkAlg^{\mathfrak{gl}_1, \st}_q(\widetilde{\tri}) \rightarrow \widetilde{\mathbb{T}} \otimes \SkAlg^{\mathfrak{gl}_1}_{-A}(\tri)
\]
defined on the generators by
\begin{align*}
\alpha_1 &\mapsto [bc]^{-1} \otimes \alpha, \\
\beta_1 &\mapsto [ca]^{-1} \otimes \beta, \\
\gamma_1 &\mapsto [ab]^{-1} \otimes \gamma, \\
\alpha_2 &\mapsto [bc] \otimes \alpha, \\
\beta_2 &\mapsto [ca] \otimes \beta, \\
\gamma_2 &\mapsto [ab] \otimes \gamma,
\end{align*}
which makes the square \eqref{cd:triangle_compatibility} commutative. 
\end{thm}
\begin{proof}
Choosing the preimages of the generators as in Lemma \ref{lem:qUVIR_triangle_surjectivity}, we get
\[
\alpha_1
\;\overset{F^{-1}}{\mapsto}\;
\overrightarrow{\alpha}_{--}
\;\overset{\pi}{\mapsto}\;
\alpha_{--} \otimes \alpha
\;\overset{\Tr \otimes \mathrm{id}}{\mapsto}\;
[bc]^{-1} \otimes \alpha,
\]
and
\[
\alpha_2
\;\overset{F^{-1}}{\mapsto}\;
\overrightarrow{\alpha}_{++}
\;\overset{\pi}{\mapsto}\;
\alpha_{++} \otimes \alpha
\;\overset{\Tr \otimes \mathrm{id}}{\mapsto}\;
[bc] \otimes \alpha. 
\]
Images of the remaining generators can be computed in the same way. 
Now, it suffices to check that this map respects the relation \eqref{eq:gl1hexagonrel}:
\[
q^{-2}\alpha_1 \beta_2 \gamma_1 \alpha_2 \beta_1 \gamma_2 =: [\alpha_1 \beta_2 \gamma_1 \alpha_2 \beta_1 \gamma_2] = -1.
\]
Indeed, 
\begin{align*}
&q^{-2} \evmap(\alpha_1) \evmap(\beta_2) \evmap(\gamma_1) \evmap(\alpha_2) \evmap(\beta_1) \evmap(\gamma_2)\\
&= q^{-2}
\qty([bc]^{-1} \otimes \alpha)
\qty([ca] \otimes \beta)
\qty([ab]^{-1} \otimes \gamma)
\qty([bc] \otimes \alpha)
\qty([ca]^{-1} \otimes \beta)
\qty([ab] \otimes \gamma) \\
&= q^{-2}[bc]^{-1}[ca][ab]^{-1}[bc][ca]^{-1}[ab] \otimes \alpha\beta\gamma\alpha\beta\gamma\\
&= q^{-2}
\qty(A^{\frac{1}{2}} b^{-1} c^{-1})
\qty(A^{\frac{1}{2}} c a)
\qty(A^{\frac{1}{2}} a^{-1} b^{-1})
\qty(A^{\frac{1}{2}} b c)
\qty(A^{\frac{1}{2}} c^{-1} a^{-1})
\qty(A^{\frac{1}{2}} a b)
\otimes 
(-A)^{\frac{1}{2}}(-A)^{\frac{1}{2}}\\
&=-1.
\end{align*}
\end{proof}

\subsection{Gluing the commutative squares}
To finish the construction of the desired commutative square \eqref{eqn: surface compatibility}, we simply need to ``glue'' together the commutative squares for triangles \eqref{cd:triangle_compatibility} built in Theorem \ref{thm:evMapTriangle}. 
That is, given an ideally triangulated surface $\Sigma$ without boundary, 
we have the following commutative diagram (commutativity follows from commutativity of the 3 squares): 
\begin{equation}\label{eqn:surface-commutative-diagram-gluing}
\adjustbox{scale=0.7625}{
\begin{tikzcd}[row sep=1.25em, column sep = .5em]
\SkAlg^{\mathfrak{gl}_2}_{q}(\Sigma) \arrow[rr, "F_{\tau}"] \arrow[dr, "\sigma^{\mathfrak{gl}_2}"] \arrow[dd, "\pi"] \arrow[drrr,phantom,""] &&
\SkAlg^{\mathfrak{gl}_1}_{q}(\widetilde{\Sigma}) \arrow[dr, "\sigma^{\mathfrak{gl}_1}"] \\
& \underset{\tri \in \tau^{(2)}}{\bigotimes} \overline{\SkAlg}^{\mathfrak{gl}_2, \st}_{q}(\tri) \arrow[rr, "\bigotimes_{\tri \in \tau^{(2)}} F_{\tri}"] \arrow[dd, "\bigotimes_{\tri \in \tau^{(2)}} \pi_{\tri}"] \arrow[ddrr,phantom,""] && \underset{\tri \in \tau^{(2)}}{\bigotimes} \SkAlg^{\mathfrak{gl}_1, \st}_{q}(\widetilde{\tri}) \arrow[dd, "\bigotimes_{\tri \in \tau^{(2)}} \evmap_{\tri}"] \\
\SkAlg^{\mathfrak{sl}_2}_{A}(\Sigma) \otimes \SkAlg^{\mathfrak{gl}_1}_{-A}(\Sigma) \arrow[dr, "\sigma^{\mathfrak{sl}_2} \otimes \sigma^{\mathfrak{gl}_1}"] \arrow[drrr, bend right=50, swap, "\Tr_{\tau} \otimes \sigma^{\mathfrak{gl}_1}"] \arrow[ur,phantom,""] && \\
& \underset{\tri \in \tau^{(2)}}{\bigotimes} \qty( \overline{\SkAlg}^{\mathfrak{sl}_2}_{A}(\tri) \otimes \SkAlg^{\mathfrak{gl}_1}_{-A}(\tri) ) \arrow[rr, "\bigotimes_{\tri \in \tau^{(2)}}\Tr_{\tri}\otimes \mathrm{id}"] && \underset{\tri \in \tau^{(2)}}{\bigotimes} \qty(\widetilde{\mathbb{T}} \otimes \SkAlg^{\mathfrak{gl}_1}_{-A}(\tri))
\end{tikzcd},
}
\end{equation}
where $\sigma^{\mathfrak{sl}_2}$, $\sigma^{\mathfrak{gl}_2}$, $\sigma^{\mathfrak{gl}_1}$ are the splitting homomorphisms for $\mathfrak{sl}_2$, $\mathfrak{gl}_2$, and $\mathfrak{gl}_1$-skein algebras, respectively. 

Moreover, note that the image of $\Tr_\tau \otimes \sigma^{\mathfrak{gl}_1}$ (i.e., the bottom arrow) is contained in the subalgebra
\[
\SQTS_\tau(\Sigma) \otimes \qty(\bigotimes_{\tri \in \tau^{(2)}} \SkAlg^{\mathfrak{gl}_1}_{-A}(\tri) )_0 \subset \bigotimes_{\tri \in \tau^{(2)}}  \widetilde{\mathbb{T}} \otimes \bigotimes_{\tri \in \tau^{(2)}}\SkAlg^{\mathfrak{gl}_1}_{-A}(\tri),
\]
where $\qty(\bigotimes_{\tri \in \tau^{(2)}} \SkAlg^{\mathfrak{gl}_1}_{-A}(\tri) )_0$ denotes the $0$-graded part of $\bigotimes_{\tri \in \tau^{(2)}} \SkAlg^{\mathfrak{gl}_1}_{-A}(\tri)$, with respect to the $\mathbb{Z}^{\tau^{(1)}}$-grading given by the number of end points on each edge $e\in \tau^{(1)}$, counted with sign by orientation. 
The same is true for the image of $\qty(\otimes_{\tri \in \tau^{(2)}}\evmap_{\tri}) \circ \sigma^{\mathfrak{gl}_1}$. 

A useful fact is that: 
\begin{lem} \label{lem:gl1_iso_surfaces}
The splitting map
\[
\SkAlg_{-A}^{\mathfrak{gl}_1}(\Sigma) \overset{\sigma^{\mathfrak{gl}_1}}{\rightarrow} \qty(\bigotimes_{\tri\in \tau^{(2)}} \SkAlg_{-A}^{\mathfrak{gl}_1}(\tri))_0
\]
is an isomorphism of $R$-algebras. 
\end{lem}
\begin{proof}
Each copy of $\SkAlg_{-A}^{\mathfrak{gl}_1}(\tri)$ has a basis labeled by $\{ [L_{\vec{n}_{\tri}}] \;\vert\; \vec{n}_{\tri} \in \mathbb{Z}^3, \; n_{\tri, 1} + n_{\tri, 2} + n_{\tri, 3} = 0\}$, 
where $L_{\vec{n}_{\tri}}$ denotes the distinguished $\mathfrak{gl}_1$-web in $\tri \times I$ with boundary condition $\vec{n}_{\tri}$; see Figure \ref{fig:triangle-gl1-web-basis}. 
\begin{figure}[htbp]
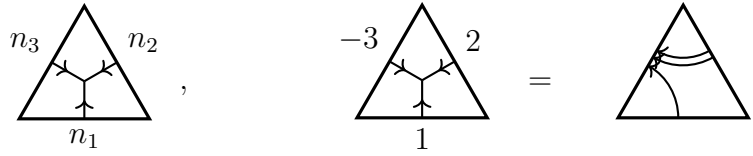

\centering
\[
\vcenter{\hbox{

}}
\]
\caption{Left: the $\mathfrak{gl}_1$-web $L_{n_1, n_2, n_3}$; Right: an example, $L_{1,2,-3}$, is shown, which should be understood as the Weyl-ordered product of the shown tangles.}
\label{fig:triangle-gl1-web-basis}
\end{figure}
Likewise, $\qty(\bigotimes_{\tri\in \tau^{(2)}} \SkAlg_{-A}^{\mathfrak{gl}_1}(\tri))_0$ has a basis given by the form $\otimes_{\tri \in \tau^{(2)}} [L_{\vec{n}_{\tri}}]$ with matching boundary conditions along each $e \in \tau^{(1)}$ (i.e., the sum of the two $n$'s must vanish on each edge $e$). 
Since the $\mathfrak{gl}_1$-webs with matching boundary conditions can be glued, the splitting map is surjective. 
Also, such gluing gives the inverse of the splitting map, showing that the splitting map is injective. 
\end{proof}

As a result, we can replace the bottom right corner of the commutative square by $\SQTS_{\tau}(\Sigma) \otimes \SkAlg^{\mathfrak{gl}_1}_{-A}(\Sigma)$ to obtain:
\begin{thm}[Compatibility theorem for surfaces]\label{thm:2d-compatibility}
The 2d quantum trace map $\Tr_\tau$ is compatible with the 2d quantum UV-IR map $F_\tau$ in the sense that they fit into the commutative square
\[
\begin{tikzcd}
\SkAlg^{\mathfrak{gl}_2}_{q}(\Sigma) \arrow[d, "\pi"] \arrow[r, "F_{\tau}"] & \SkAlg^{\mathfrak{gl}_1}_{q}(\widetilde{\Sigma}_{\tau}) \arrow[d, "\evmap"]\\
\SkAlg^{\mathfrak{sl}_2}_{A}(\Sigma) \otimes \SkAlg^{\mathfrak{gl}_1}_{-A}(\Sigma) \arrow[r, "\Tr_{\tau} \otimes \mathrm{id}"] & \SQTS_{\tau}(\Sigma) \otimes \SkAlg^{\mathfrak{gl}_1}_{-A}(\Sigma)
\end{tikzcd},
\]
where $\pi$ is the $\mathfrak{gl}_2$-$\mathfrak{sl}_2$ map (Proposition \ref{prop:gl2-sl2-map}), and $\evmap$ is the composition $\qty(\otimes_{\tri \in \tau^{(2)}}\evmap_{\tri}) \circ \sigma^{\mathfrak{gl}_1}$. 
\end{thm}

\subsection{Proof of Neitzke-Yan conjecture} 
A version of the above compatibility was conjectured earlier by Neitzke and Yan. 
To state their conjecture precisely, 
let $\Gamma$ be the lattice $H_1(\widetilde{\Sigma}; \mathbb{Z})$ with the standard intersection pairing $\Gamma \times \Gamma \rightarrow \mathbb{Z}$, and consider 3 sublattices $2\Gamma$, $\Gamma^{\mathrm{odd}}$, $\Gamma^{\mathrm{even}}$, where
\[
2\Gamma := \{2\gamma \;\vert\; \gamma \in \Gamma\}, \quad
\Gamma^{\mathrm{odd}} := \{\gamma \in \Gamma \;\vert\; \sigma(\gamma) = -\gamma\}, 
\quad\text{and}\quad
\Gamma^{\mathrm{even}} := \{\gamma \in \Gamma \;\vert\; \sigma(\gamma) = \gamma\},
\]
with $\sigma$ being the $\mathbb{Z}/2$-deck transformation action on $\Gamma$. 
Each of those sublattices has an induced intersection pairing, 
so we can consider the corresponding quantum tori $Q_{2\Gamma}$, $Q_{\Gamma^{\mathrm{odd}}}$, and $Q_{\Gamma^{\mathrm{even}}}$ (as in Definition \ref{defn:quantum-torus}) based on those lattices. 
As explained in \cite[Sec. 3.5]{NY}, there is an algebra isomorphism 
\begin{align}\label{eq:gl1-quantum-torus-iso}
\iota : \SkAlg^{\mathfrak{gl}_1, *}_{q}(\widetilde{\Sigma}_{\tau}) &\overset{\sim}{\rightarrow} Q_{2\Gamma} \\
[\vec{L}]^{\mathfrak{gl}_1} &\mapsto (-1)^{n(\vec{L})}q^{w(\vec{L})}x_{2\vec{L}}, \nonumber
\end{align}
where $n(\vec{L})$ denotes the number of non-local crossings (i.e., crossings in the projection to $\Sigma$ which do not come from a crossing on $\widetilde{\Sigma}$),
and $*$ denotes the twist of the product structure by a sign given by the mod 2 intersection number between the two links in the projection to $\Sigma$.\footnote{
Here, the factor $(-1)^{n(\vec{L})}$ ensures that the map is invariant under an isotopy of $\vec{L}$ across a branch point, 
and the twist in the product structure is due to the fact that the monomials in $Q_{2\Gamma}$ are $A^2 = -q$-commuting, instead of $q$-commuting. 
} 
By twisting the product structures on $\SkAlg^{\mathfrak{gl}_2}_{q}(\Sigma)$ and $\SkAlg^{\mathfrak{gl}_1}_{-A}(\Sigma)$ in the same way (i.e., by a sign given by the mod 2 intersection number), the quantum UV-IR map $F$ and the $\mathfrak{gl}_2$-$\mathfrak{sl}_2$ map remain algebra homomorphisms. 

\begin{conj}[{\cite[Sec. 9.2]{NY}}]\label{conj:AndyAndFei}
There is a commutative diagram
\begin{equation}\label{eq:NeitzkeYan-commutative-diagram}
\begin{tikzcd}
\SkAlg^{\mathfrak{gl}_2, *}_{q}(\Sigma) \arrow[d, "\pi"] \arrow[r, "F_{\tau}"] & \SkAlg^{\mathfrak{gl}_1, *}_{q}(\widetilde{\Sigma}_{\tau}) \cong Q_{2\Gamma} \arrow[d, "\rho"]\\
\SkAlg^{\mathfrak{sl}_2}_{A}(\Sigma) \otimes \SkAlg^{\mathfrak{gl}_1, *}_{-A}(\Sigma) \arrow[r, "F^{\mathrm{odd}} \otimes F^{\mathrm{even}}"] & Q_{\Gamma^{\mathrm{odd}}} \otimes Q_{\Gamma^{\mathrm{even}}}
\end{tikzcd},
\end{equation}
where 
\begin{itemize}
\item $\rho$ is an algebra homomorphism given by
\begin{align}\label{eq:rho-map}
\rho : Q_{2\Gamma} &\rightarrow Q_{\Gamma^{\mathrm{odd}}} \otimes Q_{\Gamma^{\mathrm{even}}} \\
x_{2\gamma} &\mapsto x_{\gamma - \sigma(\gamma)} \otimes x_{\gamma + \sigma(\gamma)}, \nonumber
\end{align}
\item $F^{\mathrm{odd}}$ is an algebra homomorphisms given by\footnote{Here, and also in $F^{\mathrm{even}}$ below, the factor written in \cite{NY} is $q^{-\frac{1}{2}w(\vec{L})}$, which is corrected to $(-A)^{-w(\vec{L})}$ here.} 
\begin{align}\label{eq:Fodd}
F^{\mathrm{odd}} : \SkAlg^{\mathfrak{sl}_2}_{A}(\Sigma) &\rightarrow Q_{\Gamma^{\mathrm{odd}}} \\
[L]^{\mathfrak{sl}_2} &\mapsto (-A)^{-w(\vec{L})} \rho^{\mathrm{odd}}(\iota \circ F([\vec{L}]^{\mathfrak{gl}_2})), \nonumber
\end{align}
where $\vec{L}$ is $L$ with an arbitrary choice of orientation, 
$w(\vec{L})$ is the writhe, 
and $\rho^{\mathrm{odd}} : Q_{2\Gamma} \rightarrow Q_{\Gamma^{\mathrm{odd}}}$ is the $R$-linear map $x_{2\gamma} \mapsto x_{\gamma - \sigma(\gamma)}$, and
\item $F^{\mathrm{even}}$ is an algebra homomorphism given by
\begin{align}\label{eq:Feven}
F^{\mathrm{even}} : \SkAlg^{\mathfrak{gl}_1, *}_{-A}(\Sigma) &\rightarrow Q_{\Gamma^{\mathrm{even}}} \\
[\vec{L}]^{\mathfrak{gl}_1} &\mapsto (-A)^{w(\vec{L})} x_{p^{-1}(\vec{L})}, \nonumber
\end{align}
with $p : \widetilde{\Sigma} \rightarrow \Sigma$ being the projection. 
\end{itemize}
Morever, $F^{\mathrm{odd}}$ coincides with the Bonahon-Wong quantum trace map. 
\end{conj}

\begin{thm}\label{thm:NeitzkeYanConjIsTrue}
The Neitzke-Yan conjecture \ref{conj:AndyAndFei} is true. 
\end{thm}
\begin{proof}
Thanks to Theorem \ref{thm:2d-compatibility}, it suffices to construct an algebra embedding
\[
\SQTS_\tau(\Sigma) \overset{i^{\mathrm{odd}}}{\hookrightarrow} Q_{\Gamma^{\mathrm{odd}}}
\]
such that 
\begin{equation}\label{eq:NYconjproof1}
(i^{\mathrm{odd}} \otimes i^{\mathrm{even}}) \circ \evmap = \rho \circ \iota,
\end{equation}
where $i^{\mathrm{even}} = F^{\mathrm{even}}$,
and
\begin{equation}\label{eq:NYconjproof2}
i^{\mathrm{odd}} \circ \Tr_{\tau} = F^{\mathrm{odd}}.     
\end{equation}

For each square-root quantized shear parameter $\hat{x}_e$, set
\begin{equation}\label{eq:iodd}
i^{\mathrm{odd}}(\hat{x}_e) = x_{\gamma_e},
\end{equation}
where $\gamma_e \in \Gamma^{\mathrm{odd}}$ denotes the 1-cycle in $\widetilde{\Sigma}_\tau$ depicted below:
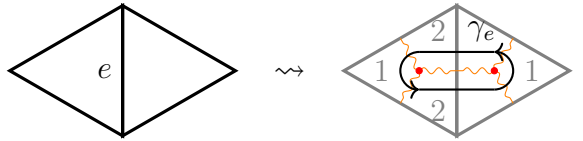
\[
\vcenter{\hbox{

}}.
\]
From the commutation relations, we see that this extends to an algebra homomorphism $i^{\mathrm{odd}} : \SQTS_\tau(\Sigma) \rightarrow Q_{\Gamma^{\mathrm{odd}}}$. 
It is easy to see that each monomial in $\SQTS_{\tau}(\Sigma)$ gets sent to a different lattice point in $\Gamma^{\mathrm{odd}}$ (e.g., by looking at the intersection numbers with ideal arcs in $\widetilde{\Sigma}$), so $i^{\mathrm{odd}}$ is indeed an embedding. 

It remains to show equations \eqref{eq:NYconjproof1} and \eqref{eq:NYconjproof2}. 
Just like how we constructed the compatibility map, we follow the strategy of cutting and gluing. 
To that end, we first show the following lemma which extends the isomorphism between $\mathfrak{gl}_1$-skeins and quantum tori \eqref{eq:gl1-quantum-torus-iso} to surfaces with boundary: 
\begin{lem}
Let $\Sigma$ be a punctured, bordered surface with ideal triangulation $\tau$, and let $\widetilde{\Sigma}_\tau$ be the corresponding branched double cover. 
Set $\Gamma := H_1(\widetilde{\Sigma}_\tau, \partial \widetilde{\Sigma}_\tau)$, equipped with the intersection pairing $\Gamma \times \Gamma \rightarrow \frac{1}{2}\mathbb{Z}$. 
Then, there is an algebra isomorphism
\begin{align*}
\iota : \SkAlg^{\mathfrak{gl}_1, \st, *}_q(\widetilde{\Sigma}_\tau) &\overset{\sim}{\rightarrow} Q_{2\Gamma} \\
[\vec{L}]^{\mathfrak{gl}_1} &\mapsto (-1)^{b(\vec{L})} (-1)^{w(p(\vec{L}))}A^{2 w(\vec{L})} x_{2\vec{L}}\\
&\quad = (-1)^{b(\vec{L})} (-1)^{n(\vec{L})} q^{w(\vec{L})} x_{2\vec{L}},
\end{align*}
where $n(\vec{L}), w(\vec{L}) \in \frac{1}{2}\mathbb{Z}$ are the non-local writhe and the writhe of $\vec{L}$, respectively, 
$w(p(\vec{L})) = n(\vec{L}) + w(\vec{L}) \in \frac{1}{2}\mathbb{Z}$ is the writhe of the projection $p(\vec{L})$ of $\vec{L}$ to $\Sigma$, 
and the product structure on $\SkAlg^{\mathfrak{gl}_1, \st}_q(\widetilde{\Sigma}_\tau)$ is -- on top of the sign-twisting $(-1)^{-b(\vec{L}_1, \vec{L}_2)}$ -- further twisted by $(-1)^{-w(p(\vec{L}_1), p(\vec{L}_2))}$, where $w(p(\vec{L}_1), p(\vec{L}_2)) := w(p(\vec{L}_1 \cdot \vec{L}_2)) - w(p(\vec{L}_1)) - w(p(\vec{L}_2)) \in \frac{1}{2}\mathbb{Z}$.\footnote{
Unlike in \eqref{eq:gl1-quantum-torus-iso}, here, the intersection pairing, as well as various (local and non-local) writhes, $n(\vec{L}), w(\vec{L}), w(p(\vec{L}))$, are half-integer-valued because of boundary contributions. 
} 
\end{lem}
\begin{proof}
Firstly, this map is well-defined because it respects the ordinary $\mathfrak{gl}_1$-skein relations -- thanks to the factor $q^{w(\vec{L})}$ -- and the relation for sign defects -- thanks to the factor $(-1)^{n(\vec{L})}$. 
This is an algebra map because, for any two tangles $\vec{L}_1$ and $\vec{L}_2$ flat on $\widetilde{\Sigma}_{\tau}$ so that $w(\vec{L}_1) = w(\vec{L}_2) = 0$, 
\[
A^{2w(\vec{L}_1 \cdot \vec{L}_2)} x_{2(\vec{L}_1 \cdot\vec{L}_2)}
= A^{\frac{\langle 2\vec{L}_1, 2\vec{L}_2\rangle}{2}} x_{2\vec{L}_1 + 2\vec{L}_2}
= x_{2\vec{L}_1} \cdot x_{2\vec{L}_2}.
\]
Finally, both the domain and codomain of this map are graded by $H_1(\widetilde{\Sigma}_\tau, \partial \widetilde{\Sigma}_\tau)$, with each graded piece being isomorphic to the base ring $R$. 
It is easy to see that $\iota$ respects this grading and is an isomorphism in each graded piece. 
Therefore, $\iota$ is an algebra isomorphism. 
\end{proof}

Now, back to the case where $\Sigma$ is an ideally triangulated surface without boundary. 
For each ideal triangle $\tri$, we have: 
\[
\begin{tikzcd}
\overline{\SkAlg}^{\mathfrak{gl}_2, \st, *}_{q}(\tri) \arrow[d, "\pi"] \arrow[r, "F_{\tri}"] & \SkAlg^{\mathfrak{gl}_1, \st, *}_{q}(\widetilde{\tri}) \cong Q_{2\Gamma_{\widetilde{\tri}}} \arrow[d, "\evmap"] \arrow[dashed, ddr, bend left=30, "\rho"]\\
\overline{\SkAlg}^{\mathfrak{sl}_2}_{A}(\tri) \otimes \SkAlg^{\mathfrak{gl}_1, *}_{-A}(\tri) \arrow[r, "\Tr_{\tri} \otimes \mathrm{Id}"] \arrow[dashed, drr, bend right=20, swap, "F^{\mathrm{odd}} \otimes F^{\mathrm{even}}"] & \widetilde{\mathbb{T}} \otimes \SkAlg^{\mathfrak{gl}_1, *}_{-A}(\tri) \arrow[dr, "i^{\mathrm{odd}} \otimes i^{\mathrm{even}}"] \\
&& Q_{\Gamma_{\widetilde{\tri}}^{\mathrm{odd}}} \otimes Q_{\Gamma_{\widetilde{\tri}}^{\mathrm{even}}}
\end{tikzcd},
\]
where $\Gamma_{\widetilde{\tri}} := H_1(\widetilde{\tri}, \partial \widetilde{\tri})$ is a lattice of rank 5, 
and $i^{\mathrm{odd}}$ and $i^{\mathrm{even}}$ are algebra maps (in fact, isomorphisms) between rank 3 and 2 quantum tori, respectively, defined by
\begin{align*}
i^{\mathrm{odd}} : \widetilde{\mathbb{T}} &\rightarrow Q_{\Gamma_{\widetilde{\tri}}^{\mathrm{odd}}}\\
\vcenter{\hbox{

}}
\end{align*}
for each generator $e \in \{a, b, c\}$ of $\widetilde{\mathbb{T}}$, 
and
\begin{align*}
i^{\mathrm{even}} : \SkAlg^{\mathfrak{gl}_1, *}_{-A}(\tri) &\rightarrow Q_{\Gamma_{\widetilde{\tri}}^{\mathrm{even}}}\\
\vcenter{\hbox{
%
}},
\end{align*}
and similarly for the remaining generators $\beta, \gamma$ of $\SkAlg^{\mathfrak{gl}_1, *}_{-A}(\tri)$.\footnote{Note, these generators $\alpha, \beta, \gamma$, which were originally $-A$-commuting, become $A$-commuting after the twist in the product.} 
A straightforward computation on each generator of $\SkAlg^{\mathfrak{gl}_1, \st, *}_{q}(\widetilde{\tri})$ shows that the composition $(i^{\mathrm{odd}} \otimes i^{\mathrm{even}}) \circ \evmap$ is indeed given by $\rho \circ \iota$ where $\rho$ is as in \eqref{eq:rho-map} but for $\Gamma = \Gamma_{\widetilde{\tri}}$. 

Now, what remains is a simple matter of gluing these commutative squares back together to get the desired commutative diagram \eqref{eq:NeitzkeYan-commutative-diagram}. 
The splitting of $\Sigma$ into ideal triangles induces the corresponding splitting map on the quantum tori
\[
Q_{2\Gamma} \rightarrow \bigotimes_{\tri \in \tau^{(2)}} Q_{2\Gamma_{\widetilde{\tri}}}
\]
and analogous maps for $Q_{\Gamma^{\mathrm{odd}}}$ and $Q_{\Gamma^{\mathrm{even}}}$, 
which are isomorphisms onto the degree-0 subalgebra (i.e., the subalgebra generated by elements of the form $\otimes_{\tri \in \tau^{(2)}} x_{\gamma_{\tri}}$ such that $\partial \gamma_{\tri}\vert_{e} + \partial \gamma_{\tri'}\vert_{e} = 0$ for each edge $e$, if $\tri$ and $\tri'$ are the two triangles sharing the edge $e$).  
From the construction, it is clear that, after gluing, $i^{\mathrm{odd}}$ becomes the map described earlier in \eqref{eq:iodd}. 
That \eqref{eq:NYconjproof1} holds follows directly, since we have already checked it locally. 
That $i^{\mathrm{even}} = F^{\mathrm{even}}$ is indeed given by the formula \eqref{eq:Feven} is straightforward: 
for links flat on $\Sigma$, this is evident from our local definition of $i^{\mathrm{even}}$, and the prefactor $(-A)^{w(\vec{L})}$ is uniquely determined in order for the map to be well-defined. 
Finally, that \eqref{eq:NYconjproof2} holds, or equivalently, that $i^{\mathrm{odd}} \circ \Tr$ is indeed given by the formula \eqref{eq:Fodd} is an immediate corollary of the commutative diagram: 
for any flat link $L$ on $\Sigma$, we have
\[
i^{\mathrm{odd}} \circ \Tr([L]^{\mathfrak{sl}_2}) = i^{\mathrm{odd}} \circ \evmap \circ F([\vec{L}]^{\mathfrak{gl}_2}) = \rho^{\mathrm{odd}} \circ \iota \circ F([\vec{L}]^{\mathfrak{gl}_2}).
\]
The prefactor $(-A)^{-w(\vec{L})}$ is uniquely determined for the map to be well-defined. 
\end{proof}

\subsection{Naturality with respect to flips}
Here, we show that the commutative squares constructed in Theorem \ref{thm:2d-compatibility} are natural with respect to change of triangulation of the surface. 
\begin{thm}
Under a change of triangulation $\tau \rightarrow \tau'$, we have the following commutative diagram
\[
\adjustbox{scale=.9}{
\begin{tikzcd}
\overline{\SkAlg}^{\mathfrak{gl}_2}_{q}(\Sigma) \arrow[dd, "\pi"] \arrow[rr, "F_{\tau'}"] \arrow[dr, "F_{\tau}"] &  & \SkAlg^{\mathfrak{gl}_1}_{q}(\widetilde{\Sigma}_{\tau'}) \arrow[dd, "\evmap_{\tau'}"]\\
 & \SkAlg^{\mathfrak{gl}_1}_{q}(\widetilde{\Sigma}_{\tau}) \arrow[ur, "\psi_{\tau \rightarrow \tau'}"] & \\
\overline{\SkAlg}^{\mathfrak{sl}_2}_{A}(\Sigma) \otimes \SkAlg^{\mathfrak{gl}_1}_{-A}(\Sigma) \arrow[rr, "\Tr_{\tau'} \otimes \mathrm{id}" {xshift=-3em}] \arrow[dr, "\Tr_{\tau} \otimes \mathrm{id}"] &  & \widehat{\SQTS}_{\tau'}(\Sigma) \otimes \SkAlg^{\mathfrak{gl}_1}_{-A}(\Sigma) \\
 & \widehat{\SQTS}_{\tau}(\Sigma) \otimes \SkAlg^{\mathfrak{gl}_1}_{-A}(\Sigma) \arrow[leftarrow,uu,crossing over,"\evmap_\tau" swap,pos=.7] \arrow[ur, swap, "\theta_{\tau \rightarrow \tau'} \otimes \mathrm{id}"] & 
\end{tikzcd},
}
\]
where $\theta_{\tau \rightarrow \tau'}$ are the transition maps of the square-root quantum Teichm\"uller space discussed in Section \ref{subsubsec:2d-quantum-trace-naturality}, and $\psi_{\tau \rightarrow \tau'}$ are the transition maps for the $\mathfrak{gl}_1$-skein algebras of the branched double covers discussed in Section \ref{subsubsec:2d-quantum-UVIR-naturality}. 
\end{thm}
\begin{proof}
In the triangular-prism-shaped diagram above, it suffices to check that the front right face commutes, as we already know that all the other faces commute. 
Thanks to the local nature of our construction, it suffices to check this just for the ideal quadrilateral where the flip happens. 
Since the stated quantum UV-IR map $F_\tau$ is surjective for the ideal quadrilateral, the commutativity of the front right face follows from commutativity of the other faces. 
\end{proof}

\begin{rmk}
That the evaluation map $\evmap$ is natural with respect to coordinate transformation maps $\psi_{\tau \rightarrow \tau'}$ and $\theta_{\tau \rightarrow \tau'}$, i.e. that the following diagram commutes 
\[
\begin{tikzcd}
\SkAlg^{\mathfrak{gl}_1, \st}_{q}
\qty(
\vcenter{\hbox{

}}
)
\end{tikzcd},
\]
can also be checked explicitly. 
Since every arrow is an algebra map, it suffices to check the commutativity for the generators of the top left corner of this diagram, which is a rank 8 quantum torus.\footnote{Rank 8 because that's the rank of the relative first homology group of the branched double cover of the ideally triangulated quadrilateral, which is an annulus with 4 punctures on each of the two $S^1$ boundary components.} 
There are 8 corner tangles (4 corners $\times$ 2 sheets) spanning the rank 7 part of the quantum torus. 
For each of them, we get a simple commutative diagram like
\[
\begin{tikzcd}
\vcenter{\hbox{

}}
\end{tikzcd},
\]
where we are labeling the edges of the ideal triangulations as in Figure \ref{fig:flip}. 
For the remaining 1 generator of the quantum torus, we have
\[
\begin{tikzcd}
\vcenter{\hbox{
%
}}
\end{tikzcd}.
\]
\end{rmk}

\section{Compatibility for 3-manifolds}\label{sec: 3-manifold compatibility}
In this section, we extend the analysis of the previous section to 3-manifolds. 
That is, given a (non-compact) 3-manifold $Y$ without boundary, with an ideal triangulation $\mathcal{T}$, we will construct a commutative square
\begin{equation}\label{eqn: 3-manifold compatibility}
\begin{tikzcd}
\overline{\Sk}^{\mathfrak{gl}_2}_{q}(Y) \arrow[d, "\pi"] \arrow[r, "F_{\mathcal{T}}"] & \Sk^{\mathfrak{gl}_1}_{q}(\widetilde{Y}_{\mathcal{T}}) \arrow[d, "\evmap"]\\
\overline{\Sk}^{\mathfrak{sl}_2}_{A}(Y) \otimes \Sk^{\mathfrak{gl}_1}_{-A}(Y) \arrow[r, "\Tr_{\mathcal{T}} \otimes \mathrm{id}"] & \SQGM_{\mathcal{T}}(Y) \otimes \Sk^{\mathfrak{gl}_1}_{-A}(Y)
\end{tikzcd}.
\end{equation}
Compared to the surface cases, the main difference is that there are extra gluing relations (in the ``relative tensor products'') coming from the splitting maps, and we have to check that commutative squares for little pieces (i.e., face suspensions) glue well (i.e., respects the gluing relations).

\subsection{Commutative square for a face suspension}

In order to define the evaluation map of a face suspension, we first need to establish the structure of various skein algebras of the biangle. 
The following lemmas are easy to show: 
\begin{lem}\label{lem:gl1-skein-bigon}
The $\mathfrak{gl}_1$-skein algebra $\SkAlg^{\mathfrak{gl}_1}_q(D_2)$ is isomorphic to the ring of Laurent polynomials in $1$-variable,
\[
\SkAlg^{\mathfrak{gl}_1}_q(D_2)\cong R[x^{\pm1}],
\] 
with
\[
\overrightarrow{ab}
:=
\vcenter{\hbox{

}}
\mapsto
x^{-1}
.
\]
Furthermore, the sign-twisted $\mathfrak{gl}_1$-skein algebra of the double cover $\widetilde{D_2}$ is
\[
\SkAlg^{\mathfrak{gl}_1,\st}_q(\widetilde{D_2})\cong \SkAlg^{\mathfrak{gl}_1}_q(D_2)^{\otimes 2} \cong R[x^{\pm 1}, y^{\pm 1}],
\]
with
\[
\overrightarrow{\widetilde{a}_{12}\widetilde{b}_{21}}
:=
\vcenter{\hbox{
%
}}
\mapsto
y
.
\]
That is, in this case, the sign-twisted $\mathfrak{gl}_1$-skein algebra is isomorphic to the untwisted one. 
\end{lem}

\begin{lem}\label{lem:reduced-gl2-biangle}
The sign-twisted reduced $\mathfrak{gl}_2$-skein algebra $\overline{\SkAlg}^{\mathfrak{gl}_2,\st}_q(D_2)$ of the biangle $D_2$ is generated by 
\[
\overrightarrow{a_+ b_+}, \;\; \overleftarrow{a_+ b_+}, 
\]
and their inverses, where 
\[
\overrightarrow{a_\mu b_\mu}
:=
\vcenter{\hbox{
%
}}
.
\]
These generators satisfy
\[
\overleftarrow{a_+ b_+} = \overrightarrow{a_- b_-}^{-1},\quad 
\overleftarrow{a_- b_-} = \overrightarrow{a_+ b_+}^{-1},
\]
and 
\[
\overleftarrow{a_+ b_+} \overrightarrow{a_+ b_+} =
\overrightarrow{a_+ b_+} \overleftarrow{a_+ b_+}. 
\]
\end{lem}
Here, we have used a notational convention for stated tangles similar to the one used in Example \ref{eg:angled-homomorphism}. 
For the rest of the paper, we adopt similar notation for the generators of the various skein algebras we consider. 

To describe the structure of $\mathfrak{gl}_1$ and $\mathfrak{gl}_2$ stated skein modules of face suspensions, first recall the following definition:  
\begin{defn}[{\cite[Def. 3.17]{PP}}] \label{defn:combinatorialFoliation}
A \textit{combinatorial foliation} of a bordered, punctured surface $\Sigma$ is a decomposition of $\Sigma$ into pieces, each of which is topologically an elementary quadrilateral depicted below:
\[
\vcenter{\hbox{
\begin{tikzpicture}
\draw[thick] (0,-1)--(1,0)--(0,1)--(-1,0)--cycle;
\fill[orange] (0,-1) circle (1pt);
\fill[orange] (0,1) circle (1pt);
\begin{scope}[decoration={
markings,
mark=at position 0.5 with {\arrow{>}}}]
\draw[thick,orange,postaction={decorate}] (0,-1)--(0,1); 
\end{scope}
\draw[fill=white,draw=black,thick] (1,0) circle (3pt);
\draw[fill=white,draw=black,thick] (-1,0) circle (3pt);
\end{tikzpicture}
}}.
\]
That is, each piece is a quadrilateral with a diagonal marking and two vertices removed. 
\end{defn}
The boundary of a face suspension, as well as its double cover, admit a combinatorial foliation. 
The corresponding boundary marking determines a tessellation of the boundary into polygonal faces, one for each boundary puncture. 

The following theorems follow from an argument almost identical to that in \cite[Sec. 4.1]{PP}. 
\begin{thm} \label{thm:gl2skeinmodule}
Let $B$ be a $3$-ball whose boundary is combinatorially foliated, and let $\Gamma \subset \partial B$ be the associated boundary marking. 
Then, the stated $\mathfrak{gl}_2$ skein module of $B$ has the following presentation:
\[
\Sk_q^{\mathfrak{gl}_2,\st}(B,\Gamma) \cong \frac{\SkAlg_q^{\mathfrak{gl}_2,\st}(V(\Gamma)^+) \otimes \SkAlg_q^{\mathfrak{gl}_2,\st}(V(\Gamma)^-)^{\mathrm{op}}}{\mathrm{Ann}([\emptyset])},
\]
where $\mathrm{Ann}([\emptyset])$ is the left ideal generated by the following relations (as well as simultaneous orientation reversal of all of the tangles) for each face of the tessellation of $\partial B$:\footnote{In both this and the following theorem, the relation is drawn in a hexagon for illustrative purposes, but the face associated to a puncture can be any $2n$-gon.} 
\begin{gather*}
\vcenter{\hbox{

}}
\end{gather*}
\end{thm}

A similar statement holds for $\mathfrak{gl}_1$ skein modules of $3$-balls:
\begin{thm} \label{thm:gl1skeinmodule}
In the setup of Theorem \ref{thm:gl2skeinmodule}, the stated $\mathfrak{gl}_1$ skein module of $B$ has the following presentation:
\[
\Sk_q^{\mathfrak{gl}_1}(B,\Gamma) \cong \frac{\SkAlg_q^{\mathfrak{gl}_1}(V(\Gamma)^+) \otimes \SkAlg_q^{\mathfrak{gl}_1}(V(\Gamma)^-)^{\mathrm{op}}}{\mathrm{Ann}([\emptyset])},
\]
where $\mathrm{Ann}([\emptyset])$ is the left ideal generated by the following relations (as well as simultaneous reversal of all of the tangles) for each face of the tessellation of $\partial B$: 
\begin{gather*}
\vcenter{\hbox{
\begin{tikzpicture}[scale = 1.5]
\foreach \i in {0,...,5} {
\coordinate (P\i) at ({cos(30 + \i*60)}, {sin(30 + \i*60)});
}
\foreach \i in {0,...,5} {
\pgfmathsetmacro\j{mod(\i+1,6)};
\coordinate (P1\i) at ($(P\i)!.25!(P\j)$);
\coordinate (P2\i) at ($(P\i)!.75!(P\j)$);
}
\begin{scope}[very thick,decoration={
markings,
mark=at position 0.5 with {\arrow{>}}}]
\draw[thick,->] (P11) to[out=-60,in=-120] (P20);
\draw[orange,postaction={decorate}] (P0)--(P1);
\draw[orange,postaction={decorate}] (P2)--(P1);
\draw[orange,postaction={decorate}] (P2)--(P3);
\draw[orange,postaction={decorate}] (P4)--(P3);
\draw[orange,postaction={decorate}] (P4)--(P5);
\draw[orange,postaction={decorate}] (P0)--(P5);
\end{scope}
\foreach \i in{0,...,5} {
\fill[orange] (P\i) circle (1pt);
\draw[] (0,0) circle (2pt);
}
\end{tikzpicture}
}}
\;\;=\;\;
\vcenter{\hbox{
\begin{tikzpicture}[scale = 1.5]
\foreach \i in {0,...,5} {
\coordinate (P\i) at ({cos(30 + \i*60)}, {sin(30 + \i*60)});
}
\foreach \i in {0,...,5} {
\pgfmathsetmacro\j{mod(\i+1,6)};
\coordinate (P1\i) at ($(P\i)!.25!(P\j)$);
\coordinate (P2\i) at ($(P\i)!.75!(P\j)$);
}
\begin{scope}[very thick,decoration={
markings,
mark=at position 0.5 with {\arrow{>}}}]
\draw[thick,->] (P21) to[out=-60,in=0] (P12);
\draw[thick,->] (P22) to[out=0,in=60] (P13);
\draw[thick,->] (P23) to[out=60,in=120] (P14);
\draw[thick,->] (P24) to[out=120,in=180] (P15);
\draw[thick,->] (P25) to[out=180,in=240] (P10);
\draw[orange,postaction={decorate}] (P0)--(P1);
\draw[orange,postaction={decorate}] (P2)--(P1);
\draw[orange,postaction={decorate}] (P2)--(P3);
\draw[orange,postaction={decorate}] (P4)--(P3);
\draw[orange,postaction={decorate}] (P4)--(P5);
\draw[orange,postaction={decorate}] (P0)--(P5);
\end{scope}
\foreach \i in{0,...,5} {
\fill[orange] (P\i) circle (1pt);
\draw[] (0,0) circle (2pt);
}
\end{tikzpicture}
}}
\end{gather*}
\end{thm}

From Theorems \ref{thm:gl2skeinmodule} and \ref{thm:gl1skeinmodule}, we immediately obtain the following corollaries on the structure of stated skein modules of face suspensions. 
For simplicity of notation, in the rest of the paper, we will often suppress tensor product symbols; for example, we express the element $\overrightarrow{L} \otimes \overrightarrow{K} \in \SkAlg^{\mathfrak{gl}_1}_{-A}(D_3)^{\otimes 2}$ as $\overrightarrow{L}\overrightarrow{K}$.

\begin{cor}
As a $\overline{\SkAlg}^{\mathfrak{gl}_2,\st}_q(D_3)^{\otimes 2}$--$\overline{\SkAlg}^{\mathfrak{gl}_2,\st}_q(D_2)^{\otimes 3}$-bimodule, 
$\overline{\Sk}^{\mathfrak{gl}_2,\st}_q(Sf)$ is a cyclic bimodule generated by the empty skein $[\emptyset]$. 
More explicitly, 
\[
\overline{\Sk}^{\mathfrak{gl}_2,\st}_q(Sf) \cong \frac{\overline{\SkAlg}^{\mathfrak{gl}_2,\st}_q(D_3)^{\otimes 2} \otimes \qty(\overline{\SkAlg}^{\mathfrak{gl}_2,\st}_q(D_2)^{\mathrm{op}})^{\otimes 3}}{\mathrm{Ann}([\emptyset])},
\]
and $\mathrm{Ann}([\emptyset])$ is the left ideal generated by relations, one for each face of the boundary tessellation: 
\begin{equation} 
\overrightarrow{x_{S+}y_{S+}}\overrightarrow{y_{T+}x_{T+}}
\;\;
=
\;\;
\vcenter{\hbox{
\begin{tikzpicture}[decoration={
    markings,
    mark=at position 0.5 with {\arrow{>}}}]
\draw[thick,->] (-2/5,3/5) to[out = -45, in = -135] (2/5,3/5);
\draw[thick,<-] (-2/5,-3/5) to[out = 45, in = 135] (2/5,-3/5);
\draw[very thick, orange, postaction = {decorate}] (-1,0) -- (0,1);
\draw[very thick, orange, postaction = {decorate}] (-1,0) -- (0,-1);
\draw[very thick, orange, postaction = {decorate}] (1,0) -- (0,1);
\draw[very thick, orange, postaction = {decorate}] (1,0) -- (0,-1);
\filldraw[orange] (-1,0) circle (1pt);
\filldraw[orange] (1,0) circle (1pt);
\filldraw[orange] (0,1) circle (1pt);
\filldraw[orange] (0,-1) circle (1pt);
\node[left] at (-1/2,1/2) {$x_S$};
\node[right] at (1/2,1/2) {$y_S$};
\node[left] at (-1/2,-1/2) {$x_T$};
\node[right] at (1/2,-1/2) {$y_T$};
\node[above] at (-1/2,1/2) {$+$};
\node[above] at (1/2,1/2) {$+$};
\node[below] at (-1/2,-1/2) {$+$};
\node[below ] at (1/2,-1/2) {$+$};
\end{tikzpicture}
}}
\;\;
=
\;\;
\vcenter{\hbox{
\begin{tikzpicture}[decoration={
    markings,
    mark=at position 0.5 with {\arrow{>}}}]
\draw[thick,->] (-7/10,3/10) to[out = -45, in = 45] (-7/10,-3/10);
\draw[thick,<-] (7/10,3/10) to[out = -135, in = 135] (7/10,-3/10);
\draw[very thick, orange, postaction = {decorate}] (-1,0) -- (0,1);
\draw[very thick, orange, postaction = {decorate}] (-1,0) -- (0,-1);
\draw[very thick, orange, postaction = {decorate}] (1,0) -- (0,1);
\draw[very thick, orange, postaction = {decorate}] (1,0) -- (0,-1);
\filldraw[orange] (-1,0) circle (1pt);
\filldraw[orange] (1,0) circle (1pt);
\filldraw[orange] (0,1) circle (1pt);
\filldraw[orange] (0,-1) circle (1pt);
\node[left] at (-1/2,1/2) {$x_S$};
\node[right] at (1/2,1/2) {$y_S$};
\node[left] at (-1/2,-1/2) {$x_T$};
\node[right] at (1/2,-1/2) {$y_T$};
\node[above] at (-1/2,1/2) {$+$};
\node[above] at (1/2,1/2) {$+$};
\node[below] at (-1/2,-1/2) {$+$};
\node[below ] at (1/2,-1/2) {$+$};
\end{tikzpicture}
}}
\;\;
=
\;\;
\overrightarrow{x_{S+}x_{T+}}\overrightarrow{y_{T+}y_{S+}}.
\end{equation}
\end{cor}

\begin{cor}
As a
$\SkAlg^{\mathfrak{gl}_1,\st}_q(\widetilde{D_3})^{\otimes 2}$--$\SkAlg^{\mathfrak{gl}_1,\st}_q(\widetilde{D_2})^{\otimes 3}$-bimodule,  $\SkAlg^{\mathfrak{gl}_1,\st}_q(\widetilde{Sf})$ is a cyclic bimodule generated by the empty skein $[\emptyset]$. 
More explicitly, 
\[
\Sk^{\mathfrak{gl}_1,\st}_q(\widetilde{Sf}) \cong \frac{\SkAlg^{\mathfrak{gl}_1,\st}_q(\widetilde{D_3})^{\otimes 2} \otimes \qty(\SkAlg^{\mathfrak{gl}_1,\st}_q(\widetilde{D_2})^{\mathrm{op}})^{\otimes 3}}{\mathrm{Ann}([\emptyset])},
\]
where $\mathrm{Ann}([\emptyset])$ is the left ideal generated by 
the relations, one for each face of the boundary tessellation:

\begin{equation}\label{eqn:fsDoubleCoverBimoduleRelations} 
\overrightarrow{x_Sy_S}\overrightarrow{y_Tx_T}
\;\;
=
\;\;
\vcenter{\hbox{
\begin{tikzpicture}[decoration={
    markings,
    mark=at position 0.5 with {\arrow{>}}}]
\draw[thick,->] (-2/5,3/5) to[out = -45, in = -135] (2/5,3/5);
\draw[thick,<-] (-2/5,-3/5) to[out = 45, in = 135] (2/5,-3/5);
\draw[very thick, orange, postaction = {decorate}] (-1,0) -- (0,1);
\draw[very thick, orange, postaction = {decorate}] (-1,0) -- (0,-1);
\draw[very thick, orange, postaction = {decorate}] (1,0) -- (0,1);
\draw[very thick, orange, postaction = {decorate}] (1,0) -- (0,-1);
\filldraw[orange] (-1,0) circle (1pt);
\filldraw[orange] (1,0) circle (1pt);
\filldraw[orange] (0,1) circle (1pt);
\filldraw[orange] (0,-1) circle (1pt);
\node[above left] at (-1/2,1/2) {$x_S$};
\node[above right] at (1/2,1/2) {$y_S$};
\node[below left] at (-1/2,-1/2) {$x_T$};
\node[below right] at (1/2,-1/2) {$y_T$};
\end{tikzpicture}
}}
\;\;
=
\;\;
\vcenter{\hbox{
\begin{tikzpicture}[decoration={
    markings,
    mark=at position 0.5 with {\arrow{>}}}]
\draw[thick,->] (-7/10,3/10) to[out = -45, in = 45] (-7/10,-3/10);
\draw[thick,<-] (7/10,3/10) to[out = -135, in = 135] (7/10,-3/10);
\draw[very thick, orange, postaction = {decorate}] (-1,0) -- (0,1);
\draw[very thick, orange, postaction = {decorate}] (-1,0) -- (0,-1);
\draw[very thick, orange, postaction = {decorate}] (1,0) -- (0,1);
\draw[very thick, orange, postaction = {decorate}] (1,0) -- (0,-1);
\filldraw[orange] (-1,0) circle (1pt);
\filldraw[orange] (1,0) circle (1pt);
\filldraw[orange] (0,1) circle (1pt);
\filldraw[orange] (0,-1) circle (1pt);
\node[above left] at (-1/2,1/2) {$x_S$};
\node[above right] at (1/2,1/2) {$y_S$};
\node[below left] at (-1/2,-1/2) {$x_T$};
\node[below right] at (1/2,-1/2) {$y_T$};
\end{tikzpicture}
}}
\;\;
=
\;\;
\overrightarrow{x_S x_T}\overrightarrow{y_T y_S}.
\end{equation}
\end{cor}

\begin{cor}
As a $\SkAlg^{\mathfrak{gl}_1}_{-A}(D_3)^{\otimes 2}$-$\SkAlg^{\mathfrak{gl}_1}_{-A}(D_2)^{\otimes 3}$-bimodule, 
$\Sk^{\mathfrak{gl}_1}_{-A}(Sf)$ is a cyclic bimodule generated by the empty skein $[\emptyset]$. 
More explicitly, 
\[
\Sk^{\mathfrak{gl}_1}_{-A}(Sf) \cong \frac{\SkAlg^{\mathfrak{gl}_1}_{-A}(D_3)^{\otimes 2} \otimes \qty(\SkAlg^{\mathfrak{gl}_1}_{-A}(D_2)^{\mathrm{op}})^{\otimes 3}}{\mathrm{Ann}([\emptyset])},
\]
where $\mathrm{Ann}([\emptyset])$ is the left ideal generated by the relations, one for each face of the boundary tesselation:
\begin{equation}\label{eqn:fsBimoduleRelations} 
\overrightarrow{x_S y_S}\overrightarrow{y_T x_T}
\;\;
=
\;\;
\vcenter{\hbox{
\begin{tikzpicture}[decoration={
    markings,
    mark=at position 0.5 with {\arrow{>}}}]
\draw[thick,->] (-2/5,3/5) to[out = -45, in = -135] (2/5,3/5);
\draw[thick,<-] (-2/5,-3/5) to[out = 45, in = 135] (2/5,-3/5);
\draw[very thick, orange, postaction = {decorate}] (-1,0) -- (0,1);
\draw[very thick, orange, postaction = {decorate}] (-1,0) -- (0,-1);
\draw[very thick, orange, postaction = {decorate}] (1,0) -- (0,1);
\draw[very thick, orange, postaction = {decorate}] (1,0) -- (0,-1);
\filldraw[orange] (-1,0) circle (1pt);
\filldraw[orange] (1,0) circle (1pt);
\filldraw[orange] (0,1) circle (1pt);
\filldraw[orange] (0,-1) circle (1pt);
\node[above left] at (-1/2,1/2) {$x_S$};
\node[above right] at (1/2,1/2) {$y_S$};
\node[below left] at (-1/2,-1/2) {$x_T$};
\node[below right] at (1/2,-1/2) {$y_T$};
\end{tikzpicture}
}}
\;\;
=
\;\;
\vcenter{\hbox{
\begin{tikzpicture}[decoration={
    markings,
    mark=at position 0.5 with {\arrow{>}}}]
\draw[thick,->] (-7/10,3/10) to[out = -45, in = 45] (-7/10,-3/10);
\draw[thick,<-] (7/10,3/10) to[out = -135, in = 135] (7/10,-3/10);
\draw[very thick, orange, postaction = {decorate}] (-1,0) -- (0,1);
\draw[very thick, orange, postaction = {decorate}] (-1,0) -- (0,-1);
\draw[very thick, orange, postaction = {decorate}] (1,0) -- (0,1);
\draw[very thick, orange, postaction = {decorate}] (1,0) -- (0,-1);
\filldraw[orange] (-1,0) circle (1pt);
\filldraw[orange] (1,0) circle (1pt);
\filldraw[orange] (0,1) circle (1pt);
\filldraw[orange] (0,-1) circle (1pt);
\node[above left] at (-1/2,1/2) {$x_S$};
\node[above right] at (1/2,1/2) {$y_S$};
\node[below left] at (-1/2,-1/2) {$x_T$};
\node[below right] at (1/2,-1/2) {$y_T$};
\end{tikzpicture}
}}
\;\;
=
\;\;
\overrightarrow{x_S x_T}\overrightarrow{y_T y_S}.
\end{equation}
\end{cor}

Recall from Section \ref{subsubec:3d-quantum-trace} that the generators of the face suspension module are each naturally associated to an edge cone, and thus to an edge of some tetrahedron in the triangulation of $Y$. 
For the face suspension module generator $x$, set $\theta_x$ to be the angle associated to this edge. 

Furthermore, label edge cones of the double cover of a face suspension as in Figure \ref{fig:liftsOfFaceSuspensionModuleVariables}.
$S$ (resp. $T$) is the top (resp. bottom) tetrahedron. 
Vertices of the double cover of the face suspension are decorated with sheet labels. 
While we have used the notations like $\widetilde{x}_{12}$ and $\widetilde{x}_{21}$ in Section \ref{subsec:stated_quantum_UV-IR} to denote the two lifts of the boundary marking $x$, to further simplify the notation, here we use asterisks (*) to denote the two lifts; we simply write $x$ and $x^*$ for $\widetilde{x}_{12}$ and $\widetilde{x}_{21}$, respectively. 
\begin{figure}[htbp]
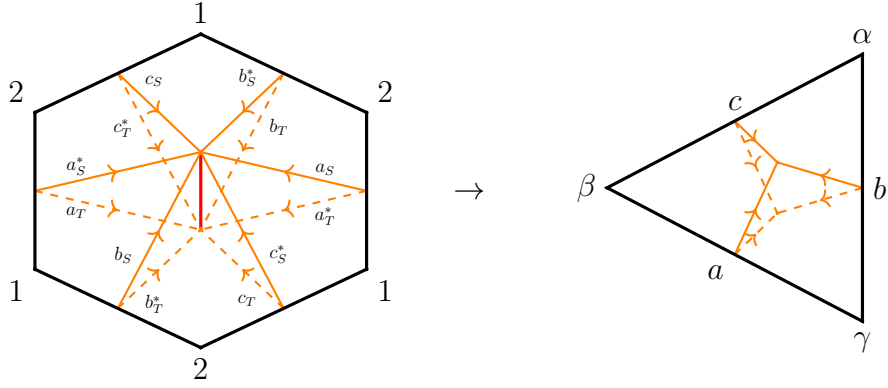

\[
\vcenter{\hbox{
\tdplotsetmaincoords{35}{-90}

}}
\]
\caption{Labeling the edge cones of a face suspension and its double cover with face suspension module variables and their lifts.}
\label{fig:liftsOfFaceSuspensionModuleVariables}
\end{figure}
\begin{lem}\label{lem:qUVIR_Sf_surjectivity}
The quantum UV-IR map on a face suspension,
\[
F_{Sf} :\overline{\Sk}^{\mathfrak{gl}_2,\st}_q(Sf)\rightarrow \Sk^{\mathfrak{gl}_1,\st}_q(\widetilde{Sf}),
\]
is surjective.  
\end{lem}
\begin{proof}
Recall from Proposition \ref{prop:quantum-UV-IR-bimodule-homomorphism} that
$F$ is a bimodule homomorphism mapping the empty skein to the empty skein. 
Thus, it is enough to check that the associated algebra maps 
\[
F: \overline{\SkAlg}^{\mathfrak{gl}_2,\st}_q(D_3)^{\otimes 2} \rightarrow \SkAlg^{\mathfrak{gl}_1,\st}_q(\widetilde{D_3})^{\otimes 2}
\]
and 
\[
F: \overline{\SkAlg}^{\mathfrak{gl}_2,\st}_q(D_2)^{\otimes 3} \rightarrow \SkAlg^{\mathfrak{gl}_1,\st}_q(\widetilde{D_2})^{\otimes 3}
\]
are surjective. 
\begin{figure}[htbp]
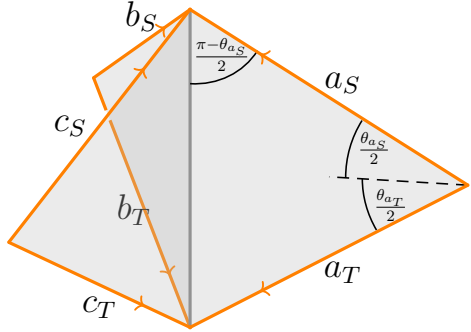

    \centering
    \includestandalone[scale=1.5]{figures/faceSuspensionLeafSpaceWithAngles}
    \caption{
    The leaf space of a face suspension with some angles labeled. 
    $\theta_{a_S}$ and $\theta_{a_T}$ are the generalized angles assigned to the edges of the tetrahedra associated to the face suspension module variables $a_S$ and $a_T$, respectively.
    }
    \label{fig:faceSuspensionAngles}
\end{figure}
Recall from Example \ref{eg:angled-homomorphism} that these maps are given by (the relevant leaf space is shown in Figure \ref{fig:faceSuspensionAngles}) 
\begin{align*}
F : \overline{\SkAlg}^{\mathfrak{gl}_2,\st}_q(D_3)^{\otimes 2} &\rightarrow \SkAlg^{\mathfrak{gl}_1, \st}_q(\widetilde{D_3})^{\otimes 2} \\
\overrightarrow{x_+ y_+} &\mapsto q^{-\frac{1}{2\pi}(\pi - \frac{\theta_x}{2} - \frac{\theta_y}{2})} \overrightarrow{xy^*}\\
\overrightarrow{x_- y_-} &\mapsto q^{\frac{1}{2\pi}(\pi - \frac{\theta_x}{2} - \frac{\theta_y}{2})} \overrightarrow{x^*y}, 
\end{align*}
for $xy \in \{ a_S b_S, b_S c_S, c_S a_S, b_T a_T, c_T b_T, a_T c_T \}$, 
and
\begin{align*}
F : \overline{\SkAlg}^{\mathfrak{gl}_2,\st}_q(D_2)^{\otimes 3} &\rightarrow \SkAlg^{\mathfrak{gl}_1, \st}_q(\widetilde{D_2})^{\otimes 3}\\ 
\overrightarrow{x_+y_+} &\mapsto q^{\frac{1}{2\pi}(\frac{\theta_{x}}{2} + \frac{\theta_{y}}{2})} \overrightarrow{xy^*},\\
\overrightarrow{x_-y_-} &\mapsto q^{-\frac{1}{2\pi}(\frac{\theta_{x}}{2} + \frac{\theta_{y}}{2})} \overrightarrow{x^*y},\\
\end{align*}
for $xy \in \{ a_S a_T, b_S b_T, c_S c_T \}$. 
The surjectivity then follows immediately from Lemmas \ref{lem:sign_twisted_hexagon} and \ref{lem:gl1-skein-bigon}, as those images generate the $\mathfrak{gl}_1$-skein algebras. 
\end{proof}

\begin{cor}
If there exists a linear map
\[
\evmap_{Sf} : \Sk^{\mathfrak{gl}_1, \st}_q(\widetilde{Sf}) \rightarrow \mathbb{S}f \otimes \Sk^{\mathfrak{gl}_1}_{-A}(Sf)
\]
making the following square 
\begin{equation}\label{cd:Sf_compatibility}
\begin{tikzcd}
\overline{\Sk}^{\mathfrak{gl}_2, \st}_{q}(Sf) \arrow[d, "\pi_{Sf}"] \arrow[r, "F_{Sf}"] & \Sk^{\mathfrak{gl}_1, \st}_{q}(\widetilde{Sf}) \arrow[d, dashed, "\evmap_{Sf}"] \\
\overline{\Sk}^{\mathfrak{sl}_2}_{A}(Sf) \otimes \Sk^{\mathfrak{gl}_1}_{-A}(Sf) \arrow[r, "\Tr_{Sf} \otimes \mathrm{id}"] & \mathbb{S}f \otimes \Sk^{\mathfrak{gl}_1}_{-A}(Sf)
\end{tikzcd}
\end{equation}
commutative, then it is unique, and it must be a bimodule homomorphism.
\end{cor}
\begin{proof}
This follows immediately from the surjectivity of $F_{Sf}$ (Lemma \ref{lem:qUVIR_Sf_surjectivity}) and the fact that the other $3$ arrows in the diagram are bimodule homomorphisms. 
\end{proof}

Below, we show that the evaluation map $\evmap_{Sf}$ indeed exists. 
\begin{thm}\label{thm: Sf commutative square}
There is a bimodule homomorphism
\[
\evmap_{Sf} : \Sk^{\mathfrak{gl}_1, \st}_q(\widetilde{Sf}) \rightarrow \mathbb{S}f \otimes \Sk^{\mathfrak{gl}_1}_{-A}(Sf)
\]
mapping the empty skein $[\emptyset]$ to $1 \otimes [\emptyset]$, 
defined on the algebra generators by 
\begin{align*}
\evmap : \SkAlg^{\mathfrak{gl}_1}_q(\widetilde{D_3})^{\otimes 2} &\rightarrow \widetilde{\mathbb{T}}^{\otimes 2} \otimes \SkAlg^{\mathfrak{gl}_1}_{-A}(D_3)^{\otimes 2}\\
\overrightarrow{xy^*} &\mapsto (q^{\frac{1}{2}}\CT) q^{-\frac{1}{2\pi}(\frac{\theta_x}{2} + \frac{\theta_y}{2})}[xy] \otimes \overrightarrow{xy}\\
\overrightarrow{x^*y} &\mapsto (q^{\frac{1}{2}}\CT)^{-1} q^{\frac{1}{2\pi}(\frac{\theta_x}{2} + \frac{\theta_y}{2})} [xy]^{-1} \otimes \overrightarrow{xy}, 
\end{align*}
for $xy \in \{ a_S b_S, b_S c_S, c_S a_S, b_T a_T, c_T b_T, a_T c_T \},$
and
\begin{align*}
\evmap : \SkAlg^{\mathfrak{gl}_1}_q(\widetilde{D_2})^{\otimes 3} &\rightarrow \widetilde{\mathbb{T}}^{\otimes 2} \otimes \SkAlg^{\mathfrak{gl}_1}_{-A}(D_2)^{\otimes 3}\\
\overrightarrow{xy^*} &\mapsto \CB q^{-\frac{1}{2\pi}(\frac{\theta_x}{2} + \frac{\theta_y}{2})} [xy]\otimes \overrightarrow{xy},\\
\overrightarrow{x^*y} &\mapsto \CB^{-1} q^{\frac{1}{2\pi}(\frac{\theta_x}{2} + \frac{\theta_y}{2})} [xy]^{-1}\otimes \overrightarrow{xy},
\end{align*}
for $xy \in \{a_Sa_T,b_Sb_T,c_Sc_T\}$,
which makes the square \eqref{cd:Sf_compatibility} commutative. 
\end{thm}

\begin{proof}
Choosing the preimages of the generators as in Lemma \ref{lem:qUVIR_Sf_surjectivity}, we find, for the generators of $\SkAlg^{\mathfrak{gl}_1}_q(\widetilde{D_3})^{\otimes 2}$, 
\begin{align*}
\overrightarrow{x y^*}
\;&\overset{F^{-1}}{\mapsto}\;
q^{\frac{1}{2\pi}(\pi-\frac{\theta_{x}}{2} - \frac{\theta_y}{2})}\overrightarrow{x_+ y_+} 
\;\;\overset{\pi}{\mapsto}\;\;
q^{\frac{1}{2\pi}(\pi-\frac{\theta_{x}}{2} - \frac{\theta_y}{2})} \overline{x_+ y_+} \otimes \overrightarrow{xy} \\
\;&\overset{\Tr \otimes \mathrm{id}}{\mapsto}\;
q^{\frac{1}{2\pi}(\pi-\frac{\theta_{x}}{2} - \frac{\theta_y}{2})} \CT [xy] \otimes \overrightarrow{xy}
\; = (q^{\frac{1}{2}} \CT) q^{-\frac{1}{2\pi}(\frac{\theta_x}{2} + \frac{\theta_y}{2})} [xy] \otimes \overrightarrow{xy}, 
\end{align*}
and
\begin{align*}
\overrightarrow{x^* y}
\;&\overset{F^{-1}}{\mapsto}\;
q^{-\frac{1}{2\pi}(\pi-\frac{\theta_{x}}{2} - \frac{\theta_y}{2})}\overrightarrow{x_- y_-} 
\;\;\overset{\pi}{\mapsto}\;\;
q^{-\frac{1}{2\pi}(\pi-\frac{\theta_{x}}{2} - \frac{\theta_y}{2})} \overline{x_- y_-} \otimes \overrightarrow{xy} \\
\;&\overset{\Tr \otimes \mathrm{id}}{\mapsto}\;
q^{-\frac{1}{2\pi}(\pi-\frac{\theta_{x}}{2} - \frac{\theta_y}{2})} q^{\frac{1}{2}} \CT^{-1} [xy]^{-1} \otimes \overrightarrow{xy}
\; = (q^{\frac{1}{2}} \CT)^{-1} q^{\frac{1}{2\pi}(\frac{\theta_x}{2} + \frac{\theta_y}{2})} [xy]^{-1} \otimes \overrightarrow{xy}. 
\end{align*}
Likewise, for the generators of $\SkAlg^{\mathfrak{gl}_1}_q(\widetilde{D_2})^{\otimes 3}$,  
\begin{align*}
\overrightarrow{x y^*}
\;&\overset{F^{-1}}{\mapsto}\;
q^{-\frac{1}{2\pi}(\frac{\theta_x}{2} + \frac{\theta_y}{2})} \overrightarrow{x_+ y_+}
\;\overset{\pi}{\mapsto}\;
q^{-\frac{1}{2\pi}(\frac{\theta_x}{2} + \frac{\theta_y}{2})} \overline{x_+ y_+}
\otimes \overrightarrow{xy}
\;\overset{\Tr \otimes \mathrm{id}}{\mapsto}\;
\CB q^{-\frac{1}{2\pi}(\frac{\theta_x}{2} + \frac{\theta_y}{2})} [xy] \otimes \overrightarrow{xy},
\end{align*}
and
\begin{align*}
\overrightarrow{x^* y}
\;&\overset{F^{-1}}{\mapsto}\;
q^{\frac{1}{2\pi}(\frac{\theta_x}{2} + \frac{\theta_y}{2})} \overrightarrow{x_- y_-}
\;\overset{\pi}{\mapsto}\;
q^{\frac{1}{2\pi}(\frac{\theta_x}{2} + \frac{\theta_y}{2})} \overline{x_- y_-}
\otimes \overrightarrow{xy}
\;\overset{\Tr \otimes \mathrm{id}}{\mapsto}\;
\CB^{-1} q^{\frac{1}{2\pi}(\frac{\theta_x}{2} + \frac{\theta_y}{2})} [xy]^{-1} \otimes \overrightarrow{xy}.
\end{align*}

The fact $\evmap$ respects relation \eqref{eq:gl1hexagonrel} in both copies of $\SkAlg^{\mathfrak{gl}_1}_q(\widetilde{D_3})$ follows from an almost identical computation to the one performed in Theorem \ref{thm:evMapTriangle}. 

Now, it suffices to check that the relations \eqref{eqn:fsDoubleCoverBimoduleRelations} are preserved by $\evmap$, and by symmetry, we just need to check just one of the relations. 
Observe 
\begin{align*}
\overrightarrow{x^*_S y_S} \overrightarrow{y^*_T x_T} [\emptyset] &\mapsto 
\qty((q^{\frac{1}{2}} \CT)^{-1} q^{\frac{1}{2\pi}(\frac{\theta_{x_S}}{2} + \frac{\theta_{y_S}}{2})} [x_S y_S]^{-1} \otimes \overrightarrow{x_S y_S} ) 
\qty((q^{\frac{1}{2}} \CT)^{-1} q^{\frac{1}{2\pi}(\frac{\theta_{x_T}}{2} + \frac{\theta_{y_T}}{2})} [x_T y_T]^{-1} \otimes \overrightarrow{x_T y_T} ) \\ 
&= (q^{-1} \CT^{-2}) q^{\frac{1}{2\pi}\qty(\frac{\theta_{x_S}}{2} + \frac{\theta_{x_T}}{2} + \frac{\theta_{y_S}}{2} + \frac{\theta_{y_T}}{2})} [x_Sy_S]^{-1}[x_Ty_T]^{-1} \otimes \overrightarrow{x_Sy_S}\overrightarrow{x_Ty_T}, 
\end{align*}
and
\begin{align*}
[\emptyset] \overrightarrow{x^*_S x_T} \overrightarrow{y^*_T y_S} &\mapsto 
\qty(\CB^{-1} q^{\frac{1}{2\pi}(\frac{\theta_{x_S}}{2} + \frac{\theta_{x_T}}{2})} [x_S x_T]^{-1} \otimes \overrightarrow{x_S x_T} ) 
\qty(\CB^{-1} q^{\frac{1}{2\pi}(\frac{\theta_{y_S}}{2} + \frac{\theta_{y_T}}{2})} [y_S y_T]^{-1} \otimes \overrightarrow{y_T y_S} ) \\
&= \CB^{-2} q^{\frac{1}{2\pi}\qty(\frac{\theta_{x_S}}{2} + \frac{\theta_{x_T}}{2} + \frac{\theta_{y_S}}{2} + \frac{\theta_{y_T}}{2})} x^{-1}_S x^{-1}_T y^{-1}_S y^{-1}_T\otimes \overrightarrow{x_S x_T}\overrightarrow{y_T y_S}\\
&= \CB^{-2} q^{\frac{1}{2\pi}\qty(\frac{\theta_{x_S}}{2} + \frac{\theta_{x_T}}{2} + \frac{\theta_{y_S}}{2} + \frac{\theta_{y_T}}{2})} [x_S y_S]^{-1} [x_T y_T]^{-1} \otimes \overrightarrow{x_S x_T} \overrightarrow{y_T y_S}.
\end{align*}
Using \eqref{eq:scaling-parameters} and the relations in \eqref{eqn:fsBimoduleRelations}, we see that these are equal in $\mathbb{S}f \otimes \Sk^{\mathfrak{gl}_1}_{-A}(Sf)$. 
\end{proof}

\subsection{Gluing the commutative squares}

In this subsection, we complete the construction of the commutative square in \eqref{eqn: 3-manifold compatibility}. 
To do so, consider the following diagram analogous to \eqref{eqn:surface-commutative-diagram-gluing}: 
\begin{equation} \label{eqn:3dgluing}
\adjustbox{scale=.8375}{
\begin{tikzcd}[row sep=1.25em, column sep = .5em] 
\overline{\Sk}^{\mathfrak{gl}_2}_{q}(Y) \arrow[rr, "F_{\mathcal{T}}"] \arrow[dr, "\sigma^{\mathfrak{gl}_2}"] \arrow[dddd, "\pi_{Y}"] \arrow[drrr,phantom,""] &&
\Sk^{\mathfrak{gl}_1}_{q}(\widetilde{Y}) \arrow[dr, "\sigma^{\mathfrak{gl}_1}"] \\
& \underset{f \in \mathcal{T}^{(2)}}{\overline{\bigotimes}} \overline{\Sk}^{\mathfrak{gl}_2, \st}_{q}(Sf) \arrow[rr, "\overline{\bigotimes}_{f \in \mathcal{T}^{(2)}} F_{Sf}"] \arrow[dd, "\overline{\bigotimes}_{f \in \mathcal{T}^{(2)}} \pi_{Sf}"] 
&&
\underset{f \in \mathcal{T}^{(2)}}{\overline{\bigotimes}} \Sk^{\mathfrak{gl}_1, \st}_{q}(\widetilde{Sf}) \arrow[dddd, "\overline{\bigotimes}_{f \in \mathcal{T}^{(2)}} \evmap_{Sf}"] \\
 \\
& \underset{f \in \mathcal{T}^{(2)}}{\overline{\bigotimes}} \qty(\overline{\Sk}^{\mathfrak{sl}_2}_{A}(Sf) \otimes \Sk^{\mathfrak{gl}_1}_{-A}(Sf)) \arrow[dd] 
\\
\overline{\Sk}^{\mathfrak{sl}_2}_{A}(Y) \otimes \Sk^{\mathfrak{gl}_1}_{-A}(Y) \arrow[dr, "\sigma^{\mathfrak{sl}_2} \otimes \sigma^{\mathfrak{gl}_1}"] \arrow[drrr, bend right=50, swap, "\Tr_{\mathcal{T}} \otimes \sigma^{\mathfrak{gl}_1}"] \arrow[ur,phantom,""] && \\
& \underset{f \in \mathcal{T}^{(2)}}{\overline{\bigotimes}} \overline{\Sk}^{\mathfrak{sl}_2}_{A}(Sf) \otimes 
\underset{f \in \mathcal{T}^{(2)}}{\overline{\bigotimes}} \Sk^{\mathfrak{gl}_1}_{-A}(Sf) \arrow[rr, "\overline{\bigotimes}_{f \in \mathcal{T}^{(2)}}\Tr_{Sf}\otimes \mathrm{id}"] && 
\underset{f \in \mathcal{T}^{(2)}}{\overline{\bigotimes}} \mathbb{S}f \otimes \underset{f \in \mathcal{T}^{(2)}}{\overline{\bigotimes}} \Sk^{\mathfrak{gl}_1}_{-A}(Sf)
\end{tikzcd}
}.
\end{equation}
Here, $\sigma^{\mathfrak{sl}_2}$, $\sigma^{\mathfrak{gl}_2}$, $\sigma^{\mathfrak{gl}_1}$ are the splitting homomorphisms for $\mathfrak{sl}_2$, $\mathfrak{gl}_2$, and $\mathfrak{gl}_1$-skein modules, respectively; see Theorem \ref{thm:3d-splitting-map-for-Sf}, Corollary \ref{cor:gl2-splitting-map}, Corollary \ref{cor:gl1splittingmap} and \ref{cor:gl1splittingwithconepoints}. 
The tensor products in the front square are the relative tensor products. 
In the top-right corner of the front square, the relative tensor product $\overline{\bigotimes}_{f \in \mathcal{T}^{(2)}} \Sk^{\mathfrak{gl}_1, \st}_{q}(\widetilde{Sf})$ must remember the extra 3-term relations near the cone points in $\widetilde{Y}$; see Corollary \ref{cor:gl1splittingwithconepoints} in Appendix \ref{sec:stated-gl1-skeins}. 

We need to first check that all of the maps in the front square of \eqref{eqn:3dgluing} are well-defined; i.e., that the naive tensor product of maps descend to the corresponding quotients. 
For the top horizontal arrow and the left vertical arrow, this follows from Propositions \ref{prop:UVIRMapCompatibleWithSplitting} and \ref{prop:splitpi}, respectively. 
That the bottom horizontal arrow is well-defined is by design; the relative tensor product $\overline{\bigotimes}_{f\in\mathcal{T}^{(2)}}\mathbb{S}f$ is defined in such a way that $\overline{\bigotimes}_{f\in\mathcal{T}^{(2)}}\mathrm{Tr}_{Sf}$ is well-defined; see \cite[Sec. 5.2]{PP}. 

As for the right vertical arrow, $\overline{\bigotimes}_{f\in\mathcal{T}^{(2)}} \evmap_{Sf}$, this is addressed in the following proposition.

\begin{prop}
The map 
\[
\underset{f \in \mathcal{T}^{(2)}}{\overline{\bigotimes}} \evmap_{Sf} 
\;:\; 
\underset{f \in \mathcal{T}^{(2)}}{\overline{\bigotimes}} \Sk^{\mathfrak{gl}_1, \st}_{q}(\widetilde{Sf}) 
\;\rightarrow\;
\underset{f \in \mathcal{T}^{(2)}}{\overline{\bigotimes}} \mathbb{S}f \otimes \underset{f \in \mathcal{T}^{(2)}}{\overline{\bigotimes}} \Sk^{\mathfrak{gl}_1}_{-A}(Sf)
\]
is well-defined. 
\end{prop}
\begin{proof}
We need to show that the composition
\[
\bigotimes_{f\in \mathcal{T}^{(2)}}\Sk_q^{\mathfrak{gl}_1, \st}(\widetilde{Sf}) \overset{\otimes_{f \in \mathcal{T}^{(2)}} \evmap_{Sf}}{\longrightarrow}
\bigotimes_{f \in \mathcal{T}^{(2)}} \mathbb{S}f \otimes \bigotimes_{f \in \mathcal{T}^{(2)}} \Sk^{\mathfrak{gl}_1}_{-A}(Sf)
\rightarrow
\underset{f \in \mathcal{T}^{(2)}}{\overline{\bigotimes}} \mathbb{S}f \otimes \underset{f \in \mathcal{T}^{(2)}}{\overline{\bigotimes}} \Sk^{\mathfrak{gl}_1}_{-A}(Sf)
\]
factors through $\overline{\bigotimes}_{f \in \mathcal{T}^{(2)}} \Sk^{\mathfrak{gl}_1, \st}_{q}(\widetilde{Sf})$. 

First, let's check that the relations for the pre-relative tensor product -- \eqref{eqn:gl1-vertex-cone} and \eqref{eqn:gl1-gluing} in Definition \ref{defn:gl1GluingRelationsDoubleCover} -- are respected. 
\begin{itemize}
    \item To check the relation \eqref{eqn:gl1-vertex-cone}, suppose a vertex cone is surrounded by the face suspensions $Sf_1$, $Sf_2$, and $Sf_3$. 
    Then, as a left relation,  
    \begin{align*}
    &\evmap \qty(
    \vcenter{\hbox{

    }}
    )\\
    &= \evmap \qty( \overrightarrow{z_1 z_1'^*} \overrightarrow{z_2' z_2''^*} \overrightarrow{z_3'' z_3^*} - 1 ) \\
    &= \qty( \qty(q^{\frac{1}{2}}\CT) q^{-\frac{1}{2\pi}(\frac{\theta}{2} + \frac{\theta'}{2})}[z_1 z_1'] ) 
    \qty( \qty(q^{\frac{1}{2}}\CT) q^{-\frac{1}{2\pi}(\frac{\theta'}{2} + \frac{\theta''}{2})}[z_2' z_2''] ) 
    \qty( \qty(q^{\frac{1}{2}}\CT) q^{-\frac{1}{2\pi}(\frac{\theta''}{2} + \frac{\theta}{2})}[z_3'' z_3] ) \\
    &\qquad \otimes \overrightarrow{z_1 z_1'} \overrightarrow{z_2' z_2''} \overrightarrow{z_3'' z_3} - 1 \otimes [\emptyset] \\
    &\overset{\eqref{eq:gl1GluingRel1}}{=} q\CT^3[\hat{z}\hat{z}'\hat{z}''] \otimes [\emptyset] - 1 \otimes [\emptyset] \\
    &\overset{\eqref{item:vertex}}{=} 0,
    \end{align*}
    where $\hat{z}, \hat{z}'$, and $\hat{z}''$ are the square-root quantized shape parameters (as in Section \ref{subsubec:3d-quantum-trace}) associated to the three edge cones around the vertex cone. 
    \item To check the relation \eqref{eqn:gl1-gluing}, consider face suspensions surrounding an edge. 
    Then, as a right relation, 
    \begin{align*}
    &\evmap \qty(
    \vcenter{\hbox{

    }}   
    ) \\
    &= \evmap \qty( 
    \overrightarrow{x_{1,e_1}x_{1,e_2}^*} \overrightarrow{x_{2,e_2}x_{2,e_3}^*}\cdots \overrightarrow{x_{k,e_k}x_{k,e_1}^*} - 1 
    ) \\
    &= \qty( \CB q^{-\frac{1}{2\pi}(\frac{\theta_{e_1}}{2} + \frac{\theta_{e_2}}{2})} [x_{1,e_1} x_{1,e_2}] ) 
    \qty( \CB q^{-\frac{1}{2\pi}(\frac{\theta_{e_2}}{2} + \frac{\theta_{e_3}}{2})} [x_{2,e_2}x_{2,e_3}] ) 
    \cdots 
    \qty( \CB q^{-\frac{1}{2\pi}(\frac{\theta_{e_k}}{2} + \frac{\theta_{e_1}}{2})} [x_{k,e_k}x_{k,e_1}] )\\
    &\quad\quad \otimes \overrightarrow{x_{1,e_1}x_{1,e_2}^*} \overrightarrow{x_{2,e_2}x_{2,e_3}^*}\cdots \overrightarrow{x_{k,e_k}x_{k,e_1}^*} 
    - 1 \otimes [\emptyset] \\ 
    &\overset{\eqref{eq:gl1GluingRel2}}{=} q^{-1} \CB^k [\hat{x}_{e_1}\hat{x}_{e_2}\cdots \hat{x}_{e_k}] \otimes [\emptyset] - 1 \otimes [\emptyset] \\
    &\overset{\eqref{item:gluing}}{=} 0,
    \end{align*}
    where $\hat{x}_{e_i}$ is the square-root quantized shape parameter associated to the edge cone $e_i$. 
\end{itemize}
For the remaining relations, which can be obtained from the above diagrams through orientation reversal and/or an involution of the double cover, the computations proceed identically to the ones given above. 

Lastly, let's check that the relations for the relative tensor product -- \eqref{eq:3termRelProjected} in Definition \ref{defn:gl1-relative-tensor-product} are respected. 
\begin{align*}
&\evmap \qty(
q^{\frac{\theta}{\pi}} \!\!\!
\vcenter{\hbox{

}}
) 
\\
&= 
q^{\frac{\theta}{\pi}} \evmap 
\qty( \overrightarrow{z_2z_2^{\prime* }}\overrightarrow{z_2^{\prime*}z_2^{\prime\prime}} \overrightarrow{z_2^{\prime\prime}z_2^*} \overrightarrow{z_1z_1^{\prime*}}\overrightarrow{z_1^{\prime*} z_1^{\prime\prime}} \overrightarrow{z_1^{\prime\prime}z_1^*} 
) 
+q^{-\frac{\theta'}{\pi}} \evmap 
\qty(
\overrightarrow{z_3^{\prime*}z_3^{\prime\prime }}\overrightarrow{z_3^{\prime\prime}z_3^{*}} \overrightarrow{z_3^{*}z_3^\prime} \overrightarrow{z_1^{\prime*}z_1^{\prime\prime}}\overrightarrow{z_1^{\prime\prime} z_1^{*}} \overrightarrow{z_1^{*}z_1^{\prime}}
)
- \evmap\left(1\right) \\
&= 
q^{\frac{\theta}{\pi}} \qty( \qty(q^{\frac{1}{2}}\CT ) q^{-\frac{1}{2\pi}(\frac{\theta}{2} + \frac{\theta'}{2})} [z_2 z'_2] )
\qty( \qty(q^{\frac{1}{2}}\CT )^{-1} q^{\frac{1}{2\pi}(\frac{\theta'}{2} + \frac{\theta''}{2})} [z'_2 z''_2]^{-1} )
\qty( \qty(q^{\frac{1}{2}}\CT ) q^{-\frac{1}{2\pi}(\frac{\theta''}{2} + \frac{\theta}{2})}[z''_2 z_2] ) \\
&\qquad
\qty( \qty(q^{\frac{1}{2}}\CT ) q^{-\frac{1}{2\pi}(\frac{\theta}{2} + \frac{\theta'}{2})} [z_1 z'_1] )
\qty( \qty(q^{\frac{1}{2}}\CT )^{-1} q^{\frac{1}{2\pi}(\frac{\theta'}{2} + \frac{\theta''}{2})} [z'_1 z''_1]^{-1} )
\qty( \qty(q^{\frac{1}{2}}\CT ) q^{-\frac{1}{2\pi}(\frac{\theta''}{2} + \frac{\theta}{2})} [z''_1 z_1] ) \\
&\qquad
\otimes 
\overrightarrow{z_2 z_2'}
\overrightarrow{z_2' z_2''} 
\overrightarrow{z_2'' z_2} 
\overrightarrow{z_1 z_1'}
\overrightarrow{z_1' z_1''} 
\overrightarrow{z_1'' z_1} \\
&\;\;+ 
q^{-\frac{\theta'}{\pi}} \qty( \qty(q^{\frac{1}{2}}\CT )^{-1} q^{\frac{1}{2\pi}(\frac{\theta'}{2} + \frac{\theta''}{2})} [z'_3 z''_3]^{-1} )
\qty( \qty(q^{\frac{1}{2}}\CT ) q^{-\frac{1}{2\pi}(\frac{\theta''}{2} + \frac{\theta}{2})} [z''_3 z_3] )
\qty( \qty(q^{\frac{1}{2}}\CT )^{-1} q^{\frac{1}{2\pi}(\frac{\theta}{2} + \frac{\theta'}{2})} [z_3 z'_3]^{-1} )
\\
&\qquad 
\qty( \qty(q^{\frac{1}{2}}\CT)^{-1} q^{\frac{1}{2\pi}(\frac{\theta'}{2} + \frac{\theta''}{2})} [z'_1 z''_1]^{-1} ) 
\qty( \qty(q^{\frac{1}{2}}\CT) q^{-\frac{1}{2\pi}(\frac{\theta''}{2} + \frac{\theta}{2})} [z''_1 z_1] )
\qty( \qty(q^{\frac{1}{2}}\CT)^{-1} q^{\frac{1}{2\pi}(\frac{\theta}{2} + \frac{\theta'}{2})} [z_1 z'_1]^{-1} ) \\
&\qquad
\otimes 
\overrightarrow{z_3' z_3''}
\overrightarrow{z_3'' z_3} 
\overrightarrow{z_3 z_3'} 
\overrightarrow{z_1' z_1''}
\overrightarrow{z_1'' z_1} 
\overrightarrow{z_1 z_1'}\\
&\;\;- 1 \otimes [\emptyset]\\
&\overset{\eqref{eq:statedgl1skeinrel3}}{=} q\CT^2\hat{z}^{2} \otimes [\emptyset] + q^{-1}\CT^{-2}\hat{z}^{\prime-2} \otimes [\emptyset]-1\otimes [\emptyset]\\
&\overset{\eqref{item:Lagrangian}}{=} 0,
\end{align*}
where $\hat{z} = z_2 \otimes z_3$ and $\hat{z}'=z'_1\otimes z'_3$, and we used that $\CB^2=q\CT^2.$ 
\end{proof}

Now that we have checked that all the arrows in the the diagram \eqref{eqn:3dgluing} are well-defined, the fact that it is a commutative diagram is almost immediate; 
the top (Proposition \ref{prop:UVIRMapCompatibleWithSplitting}) and left (Proposition \ref{prop:gl2sl2CompatibleWithSplitting}) squares commute simply because all of the maps are locally defined, and the front square commutes since it is the relative tensor product of the commutative squares constructed in Theorem \ref{thm: Sf commutative square}. 

From the construction, note that the images of both $\Tr_{\mathcal{T}} \otimes \sigma^{\mathfrak{gl}_1}$ and of $\overline{\bigotimes}_{f \in \mathcal{T}^{(2)}} \evmap_{Sf} \circ \sigma^{\mathfrak{gl}_1}$ are contained in the submodule
\[
\SQGM_{\mathcal{T}}(Y) \otimes \qty(
\underset{f \in \mathcal{T}^{(2)}}{\overline{\bigotimes}} \Sk_{-A}^{\mathfrak{gl}_1}(Sf)
)_0 
\subset 
\underset{f \in \mathcal{T}^{(2)}}{\overline{\bigotimes}} \mathbb{S}f \otimes \underset{f \in \mathcal{T}^{(2)}}{\overline{\bigotimes}} \Sk_{-A}^{\mathfrak{gl}_1}(Sf). 
\]
Moreover, by Lemma \ref{lem:gl1_iso}, the $\mathfrak{gl}_1$-splitting map 
\[
\sigma^{\mathfrak{gl}_1} : \Sk^{\mathfrak{gl}_1}_{-A}(Y) \rightarrow \qty(
\underset{f \in \mathcal{T}^{(2)}}{\overline{\bigotimes}} \Sk_{-A}^{\mathfrak{gl}_1}(Sf)
)_0
\]
is an isomorphism. 
Therefore, by replacing the bottom right corner of the commutative diagram by $\SQGM_{\mathcal{T}}(Y) \otimes \Sk^{\mathfrak{gl}_1}_{-A}(Y)$, we obtain the first part of Theorem \ref{thm:main-thm-compatibility}. 

\subsection{Naturality with respect to Pachner moves}
Here, we show the second part of Theorem \ref{thm:main-thm-compatibility}, that the commutative squares constructed above behave naturally with respect to 2-3 Pachner moves: 
\begin{thm}
Under a 2-3 Pachner move $\mathcal{T}_2 \rightarrow \mathcal{T}_3$, we have the following commutative diagram
\[
\adjustbox{scale=.9}{
\begin{tikzcd}
\overline{\Sk}^{\mathfrak{gl}_2}_{q}(Y) \arrow[dd, "\pi"] \arrow[rr, "F_{\mathcal{T}_3}"] \arrow[dr, "F_{\mathcal{T}_2}"] &  & \Sk^{\mathfrak{gl}_1}_{q}(\widetilde{Y}_{\mathcal{T}_3}) \arrow[dd, "\evmap_{\mathcal{T}_3}"]\\
 & \Sk^{\mathfrak{gl}_1}_{q}(\widetilde{Y}_{\mathcal{T}_2}) \arrow[ur, "\phi_{2 \rightarrow 3}"] & \\
\overline{\Sk}^{\mathfrak{sl}_2}_{A}(Y) \otimes \Sk^{\mathfrak{gl}_1}_{-A}(Y) \arrow[rr, "\Tr_{\mathcal{T}_3} \otimes \mathrm{id}" {xshift=-3em}] \arrow[dr, "\Tr_{\mathcal{T}_2} \otimes \mathrm{id}"] &  & \SQGM_{\mathcal{T}_3}(Y) \otimes \Sk^{\mathfrak{gl}_1}_{-A}(Y) \\
 & \SQGM_{\mathcal{T}_2}(Y) \otimes \Sk^{\mathfrak{gl}_1}_{-A}(Y) \arrow[leftarrow,uu,crossing over,"\evmap_{\mathcal{T}_2}" swap,pos=.7] \arrow[ur, swap, "\varphi_{2 \rightarrow 3} \otimes \mathrm{id}"] & 
\end{tikzcd},
}
\]
where $\varphi_{2 \rightarrow 3}$ and $\phi_{2 \rightarrow 3}$ are the coordinate change maps given in Propositions \ref{prop:SQGM23Pachner} and \ref{prop:gl1skein_23Pachner}, respectively. 
\end{thm}
\begin{proof}
It suffices to check that the front right face commutes, since we already know that all the other faces commute. 
Thanks to the local nature of our construction, it is enough to check this for the triangular bipyramid $\mathrm{BP}$ where the Pachner move is applied, and since we know that all the other faces commute, it suffices to show that the stated quantum UV-IR map $\overline{\Sk}^{\mathfrak{gl}_2, \st}_q(\mathrm{BP}) \overset{F_{\mathcal{T}_2}}{\rightarrow} \Sk^{\mathfrak{gl}_1, \st}_q(\widetilde{\mathrm{BP}}_{\mathcal{T}_2})$ is surjective for the bipyramid triangulated into 2 ideal tetrahedra. 

Indeed, by the same argument as in \cite[Lem. 4.1]{PP}, $\Sk^{\mathfrak{gl}_1, \st}_q(\widetilde{\mathrm{BP}}_{\mathcal{T}_2})$ is generated by the empty skein $[\emptyset]$ as a $\SkAlg^{\mathfrak{gl}_1, \st}_q(\widetilde{D_3})^{\otimes 6}$-$\SkAlg^{\mathfrak{gl}_1, \st}_q(\widetilde{D_2})^{\otimes 9}$ bimodule, and the desired surjectivity of $F_{\mathcal{T}_2}$ follows, essentially by the same computation as in the proof of Lemma \ref{lem:qUVIR_Sf_surjectivity}. 
\end{proof}

\subsection{Recovering the quantum trace map from the quantum UV-IR map}\label{sec: decouple gl(1)}
In this short subsection, we show Theorem \ref{thm:quantum-trace-from-quantum-UV-IR};
that, under a mild assumption on the homology of $Y$, it is possible to recover the genuine quantum trace map from the quantum UV-IR map. 

From Theorem \ref{thm:main-thm-compatibility}, using the quantum UV-IR map, we can compute
\[
\Tr_{\mathcal{T}}([L]) \otimes [\vec{L}] \in \SQGM_{\mathcal{T}}(Y)\otimes\Sk_{-A}^{\mathfrak{gl}_1}(Y)
\]
for any unoriented link $L \subset Y$, where $\vec{L}$ is $L$ endowed with an arbitrary orientation. 
This is almost as good as having the quantum trace map
\[
\Sk_{A}^{\mathfrak{sl}_2}(Y) \overset{\Tr_{\mathcal{T}} }{\rightarrow} \SQGM_{\mathcal{T}}(Y)
\]
itself, since, as long as $[\vec{L}] \in \Sk_{-A}^{\mathfrak{gl}_1}(Y)$ is not torsion, we can recover $\Tr_{\mathcal{T}}([L])$ from $\Tr_{\mathcal{T}}([L]) \otimes [\vec{L}]$. 
This is because, if $[\vec{L}] \in \Sk_{-A}^{\mathfrak{gl}_1}(Y)$ is not torsion and $\alpha \in H_1(Y;\mathbb{Z})$ is the homology class of $\vec{L}$, the $\alpha$-graded part of the $\mathfrak{gl}_1$-skein module is $\Sk^{\mathfrak{gl}_1}_{-A}(Y)_\alpha = R[\vec{L}] \cong R$, and the map
\begin{align*}
\SQGM_{\mathcal{T}} &\rightarrow \SQGM_{\mathcal{T}} \otimes \Sk^{\mathfrak{gl}_1}_{-A}(Y)_\alpha \\
x &\mapsto x \otimes [\vec{L}]
\end{align*}
is an isomorphism. 

The structure of $\mathfrak{gl}_1$-skein modules was completely characterized in \cite[Thm 2.3]{Przytycki0} (see also \cite[Thm IX.3.7]{Przytycki1}) where it is shown that, for $\alpha \in H_1(Y;\mathbb{Z})$, 
the $\alpha$-graded part $\Sk^{\mathfrak{gl}_1}_{-A}(Y)_\alpha$ of $\Sk_{-A}^{\mathfrak{gl}_1}(Y)$ is torsion-free (and isomorphic to the base ring $R$) iff the intersection pairing $(\alpha, \beta) \in \mathbb{Z}$ is $0$ for all $\beta \in H_2(Y; \mathbb{Z})$. 
In particular, we obtain Theorem \ref{thm:quantum-trace-from-quantum-UV-IR} as a corollary.

\section{Examples}\label{sec:examples}

In this section, we demonstrate the compatibility theorem -- Theorem \ref{thm:main-thm-compatibility} (and also Theorem \ref{thm:quantum-trace-from-quantum-UV-IR}) -- for a skein in the figure-8 knot complement $Y = S^3 \setminus \mathbf{4_1}$.  
Consider the triangulation $\mathcal{T}$ of $Y$ shown in Figure \ref{fig:figure8_triangulation}. 
\begin{figure}[htbp]
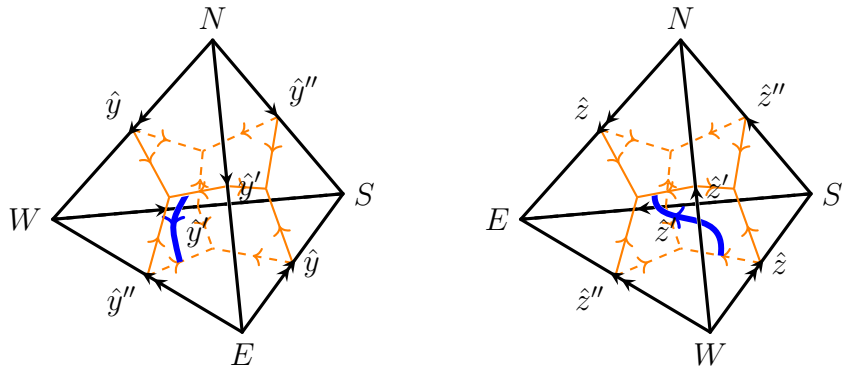

    \centering
    \includestandalone[scale=1.0]{figures/fig8_tet1}
    \hspace{.5cm}
    \includestandalone[scale=1.0]{figures/fig8_tet2}
    \caption{A triangulation of the figure-8 knot complement, as well as the tangle $\vec{K}_b$ contained inside of it (shown in blue). 
    The gluing of the tetrahedra is controlled by the edge markings.
    The faces of the tetrahedra are labeled $N, S, E,$ and $W$, and edges are labeled with their corresponding square-root quantized shape parameters.}
    \label{fig:figure8_triangulation}
\end{figure}
Let us compute the quantum trace of the knot $K_b \subset Y$ from \cite[Sec. 6]{PP}, using the quantum UV-IR map; an oriented version of this knot, $\vec{K}_b$, is shown in blue in Figure \ref{fig:figure8_triangulation}. 

While the image of $\vec{K}_b$ under the quantum UV-IR map can be computed without any reference to splitting, we will do our computation locally, by applying the splitting map, to demonstrate the compatibility in each face suspension and illustrate how our proof works in practice. 
Our knot $\vec{K}_b$, after splitting into face suspensions, is depicted in Figure \ref{fig:kb_in_face_suspensions}; only the face suspensions containing part of $\vec{K}_b$ are shown. 
\begin{figure}[htbp]
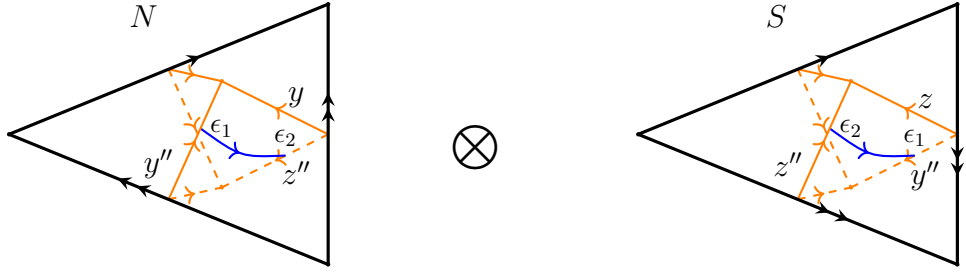

    \centering
    \begin{minipage}{.45\textwidth}  
        \centering
        \includestandalone[scale=1.0]{figures/kb_face_susp1}
    \end{minipage}
    \begin{minipage}{.04\textwidth}
        \centering
        \begin{equation*}
            \bigotimes
        \end{equation*}
    \end{minipage}
    \begin{minipage}{.45\textwidth}
        \centering
        \includestandalone[scale=1.0]{figures/kb_face_susp2}
    \end{minipage}
    \caption{The image of $\vec{K}_b$ under the splitting map. The face suspension variables associated to edge cones that will be used in later computations are labeled with the corresponding shape parameter.} 
    \label{fig:kb_in_face_suspensions}
\end{figure}
Before iterating through the compatible states, observe that the components of $\vec{K}_b$ in the face suspensions above can be written explicitly as elements of $\overline{\SkAlg}^{\mathfrak{gl}_2,\st}_q(D_3)^{\otimes 2}$ and $\overline{\SkAlg}^{\mathfrak{gl}_2,\st}_q(D_2)^{\otimes 3}$ acting on the empty skein as follows: 
\begin{gather*}
    \vcenter{\hbox{
    \includestandalone[scale=.8]{figures/kb_face_susp1}
    }}
    \;\;=\;\;
    q^{-\frac{\epsilon_2}{2}}
    \vcenter{\hbox{
    \includestandalone[scale=.8]{figures/kb_face_susp1_reduced}
    }}
    \;\;=\;\;
    q^{-\frac{\epsilon_2}{2}}\overrightarrow{y''_{\epsilon_1}y_{-\epsilon_2}}\,[\emptyset]\,\overrightarrow{y_{\epsilon_2}z''_{\epsilon_2}},
\end{gather*}
\begin{gather*}
    \vcenter{\hbox{
    \includestandalone[scale=.8]{figures/kb_face_susp2}
    }}
    \;\;=\;\;
    q^{-\frac{\epsilon_1}{2}}
    \vcenter{\hbox{
    \includestandalone[scale=.8]{figures/kb_face_susp2_reduced}
    }}
    \;\;=\;\;
    q^{-\frac{\epsilon_1}{2}}\overrightarrow{z''_{\epsilon_2}z_{-\epsilon_1}}\,[\emptyset]\,\overrightarrow{z_{\epsilon_1}y''_{\epsilon_1}}.
\end{gather*}

There are three nonzero compatible states. 
In the computations that follow, we will denote by $x_f$ the generator of $\mathbb{S}f$ corresponding to the shape parameter $x$.
\begin{itemize}
    \item $\epsilon_1 = +$ and $\epsilon_2 = -$: We compute on each face suspension separately to find
    \begin{equation*}
    \begin{aligned}
    &\evmap \circ F \Bigl(
        q^{\frac{1}{2}}\overrightarrow{y''_{S_+} y_{S_+}}\,[\emptyset]\,\overrightarrow{y_{S_-} z''_{S_-}} 
    \Bigr) \\
    &= q^{\tfrac{1}{2}}\evmap\Bigl(
        q^{-\tfrac{1}{2\pi}\Bigl(\pi-\tfrac{\theta_{y''}}{2}-\tfrac{\theta_{y}}{2}\Bigr)}
        \overrightarrow{y''_S y_S^*}\,[\emptyset]\,
        q^{-\tfrac{1}{2\pi}\Bigl(\tfrac{\theta_y}{2}+\tfrac{\theta_{z''}}{2}\Bigr)}
        \overrightarrow{y_S^* z''_S}\Bigr) \\
    &= q^{\tfrac{1}{2\pi}\qty( \tfrac{\theta_{y''}}{2} - \tfrac{\theta_{z''}}{2})}
       \Bigl((q^{\tfrac{1}{2}}\CT)q^{-\tfrac{1}{2\pi}\Bigl(\tfrac{\theta_{y''}}{2}+\tfrac{\theta_y}{2}\Bigr)}
         [y''_S y_S] \otimes \overrightarrow{y''_S y_S} \Bigr) 
       \Bigl(\CB^{-1} q^{\tfrac{1}{2\pi}\Bigl(\tfrac{\theta_{y}}{2}+\tfrac{\theta_{z''}}{2}\Bigr)}
         [y_S z''_S]^{-1} \otimes \overrightarrow{y_S z''_S} \Bigr) \\
    &= q^{\frac{1}{2}}\CT\CB^{-1}[y''_S y_S][y_S z''_S]^{-1}
       \otimes \overrightarrow{y''_S y_S}\,\overrightarrow{y_S z''_S}
    \end{aligned}
    \end{equation*}
    and
    \begin{equation*}
    \evmap \circ F \Bigl(
        q^{-\frac{1}{2}}\overrightarrow{z''_{N_-}z_{N_-}}\,[\emptyset]\,\overrightarrow{z_{N_+}y''_{N_+}}
    \Bigr)
    = q^{-\frac{1}{2}}\CT^{-1}\CB[z''_Nz_N]^{-1}[z_Ny''_N] \otimes \overrightarrow{z''_Nz_N}\,\overrightarrow{z_Ny''_N}.
    \end{equation*}
    Combining these,
    \begin{align*}
    &\bigotimes_{i=1}^2 \Bigl(\evmap \circ F \Bigr) \Bigl(
    q^{\frac{1}{2}}\overrightarrow{y''_{S_+}y_{S_+}}\,[\emptyset]\,\overrightarrow{y_{S_-}z''_{S_-}} \otimes
    q^{-\frac{1}{2}}\overrightarrow{z''_{N_-}z_{N_-}}\,[\emptyset]\,\overrightarrow{z_{N_+}y''_{N_+}} \Bigr) \\
    &= A\hat{y}''\hat{z}''^{-1}\otimes \Bigl(\overrightarrow{y''_Sy_S}\,\overrightarrow{y_Sz''_S} \otimes \overrightarrow{z''_Nz_N}\,\overrightarrow{z_Ny''_N}\Bigr) 
    \quad \in \SQGM_{\mathcal{T}}(Y) \otimes \bigotimes_{i=1}^2 \Sk^{\mathfrak{gl}_1}_{-A}(Sf).
    \end{align*}
    On the other hand, we find 
    \begin{equation*}
    \begin{aligned}
    (\Tr_{Sf} \otimes \mathrm{Id}) \circ \pi_{Sf} \Bigl(
        q^{\frac{1}{2}}\overrightarrow{y''_{S_+}y_{S_+}}\,[\emptyset]\,\overrightarrow{y_{S_-}z''_{S_-}}
    \Bigr) 
    &= q^{\frac{1}{2}}\left(\CT[y''_Sy_S]\otimes \overrightarrow{y''_Sy_S}\right)
        \left(\CB^{-1}[y_Sz''_S]^{-1}\otimes \overrightarrow{y_Sz''_S}\right)\\
    &= q^{\frac{1}{2}}\CT\CB^{-1}[y''_Sy_S][y_Sz''_S]^{-1}\otimes 
    \overrightarrow{y''_Sy_S}\,\overrightarrow{y_Sz''_S},
    \end{aligned}
    \end{equation*}
    and 
    \begin{equation*}
    (\Tr_{Sf} \otimes \mathrm{Id}) \circ \pi_{Sf} \Bigl(
    q^{-\frac{1}{2}}\overrightarrow{z''_{N_-}z_{N_-}}\,[\emptyset]\,\overrightarrow{z_{N_+}y''_{N_+}}
    \Bigr)
    = q^{-\frac{1}{2}}\CT^{-1}\CB[z''_Nz_N]^{-1}[z_Ny''_N] \otimes \overrightarrow{z''_Nz_N}\,\overrightarrow{z_Ny''_N},
    \end{equation*}
    so that
    \begin{align*}
    &\bigotimes_{i=1}^2 \Bigl((\Tr_{Sf} \otimes \mathrm{Id})\circ \pi_{Sf} \Bigr) \Bigl(
    q^{\frac{1}{2}}\overrightarrow{y''_{S_+}y_{S_+}}\,[\emptyset]\,\overrightarrow{y_{S_-}z''_{S_-}} \otimes
    q^{-\frac{1}{2}}\overrightarrow{z''_{N_-}z_{N_-}}\,[\emptyset]\,\overrightarrow{z_{N_+}y''_{N_+}} \Bigr) \\
    &= A\hat{y}''\hat{z}''^{-1}\otimes \Bigl(\overrightarrow{y''_Sy_S}\,\overrightarrow{y_Sz''_S} \otimes \overrightarrow{z''_Nz_N}\,\overrightarrow{z_Ny''_N}\Bigr) 
    \quad \in \SQGM_{\mathcal{T}}(Y) \otimes \bigotimes_{i=1}^2 \Sk^{\mathfrak{gl}_1}_{-A}(Sf),
    \end{align*} 
    yielding the same result as above, as expected. 
    \item $\epsilon_1 = -$ and $\epsilon_2 = +$: The computation proceeds almost identically to above; the result is 
    \begin{align*}
    &\bigotimes_{i=1}^2 \Bigl(\evmap \circ F \Bigr)\Bigl(
    q^{-\frac{1}{2}}\overrightarrow{y''_{S_-}y_{S_-}}\,[\emptyset]\,\overrightarrow{y_{S_+}z''_{S_+}} \otimes
    q^{\frac{1}{2}}\overrightarrow{z''_{N_+}z_{N_+}}\,[\emptyset]\,\overrightarrow{z_{N_-}y''_{N_-}} \Bigr) \\
    &= A\hat{y}''^{-1}\hat{z}''\otimes \Bigl(\overrightarrow{y''_Sy_S}\,\overrightarrow{y_Sz''_S} \otimes \overrightarrow{z''_Nz_N}\,\overrightarrow{z_Ny''_N}\Bigr) \\
    &= \bigotimes_{i=1}^2 \Bigl((\Tr_{Sf} \otimes \mathrm{Id})\circ \pi_{Sf} \Bigr) \Bigl(
    q^{-\frac{1}{2}}\overrightarrow{y''_{S_-}y_{S_-}}\,[\emptyset]\,\overrightarrow{y_{S_+}z''_{S_+}} \otimes
    q^{\frac{1}{2}}\overrightarrow{z''_{N_+}z_{N_+}}\,[\emptyset]\,\overrightarrow{z_{N_-}y''_{N_-}} \Bigr). 
    \end{align*}
    \item $\epsilon_1 = \epsilon_2 = -$: 
    First, observe that in $\Sk_q^{\mathfrak{gl}_1,\st}(\widetilde{Sf})$, 
    \begin{gather*}
        \overrightarrow{y''^*_S y_S^*}
        \;\;=\;\;
        \vcenter{\hbox{
        \includestandalone[scale=0.6]{figures/lift_with_detour}
        }}
        \;\;=\;\;
        q^{-\frac{1}{2}}
        \vcenter{\hbox{
        \includestandalone[scale=0.6]{figures/lift_with_detour_reduced}
        }}
        \;\;=\;\;
        q^{-\frac{1}{2}} \overrightarrow{y'_S y_S^*}\,\overrightarrow{y''^*_S y'_S}.
    \end{gather*}
    Then,
    \begin{align*}
        &\evmap\circ F \Bigl(q^{\tfrac{1}{2}}\overrightarrow{y''_{S_-}y_{S_+}}\,[\emptyset]\,\overrightarrow{y_{S_-}z''_{S_-}} \Bigr) \\
        &=
        q^{\tfrac{1}{2}}\evmap\Bigl(q^{\tfrac{1}{2\pi}\Bigl(\tfrac{\theta_{y}}{2}-\tfrac{\theta_{y''}}{2}\Bigr)} \overrightarrow{y''^*_S y^*_S}\,[\emptyset] \, q^{-\tfrac{1}{2\pi}\Bigl(\tfrac{\theta_y}{2}+\tfrac{\theta_{z''}}{2}\Bigr)}
        \overrightarrow{y_S^*z''_S}\Bigr) \\
        &= q^{\tfrac{1}{2}}\evmap \Bigl( 
        q^{\tfrac{1}{2\pi}\Bigl(\tfrac{\theta_{y}}{2}-\tfrac{\theta_{y''}}{2}\Bigr)} q^{-\tfrac{1}{2}} \overrightarrow{y'_S y_S^*} \overrightarrow{y''^*_S y'_S}\, [\emptyset] \, q^{-\tfrac{1}{2\pi}\Bigl(\tfrac{\theta_y}{2}+\tfrac{\theta_{z''}}{2}\Bigr)}
        \overrightarrow{y_S^*z''_S}
        \Bigr)\\
        &= q^{-\frac{1}{2\pi}\qty(\frac{\theta_{y''}}{2}+\frac{\theta_{z''}}{2})}
        \Bigl((q^{\tfrac{1}{2}}\CT) q^{-\tfrac{1}{2\pi}\Bigl(\tfrac{\theta_{y'}}{2}+\tfrac{\theta_{y}}{2}\Bigr)}[y'_S y_S] \otimes \overrightarrow{y'_S y_S} \Bigr) 
        \Bigl( (q^{\tfrac{1}{2}}\CT)^{-1}q^{\tfrac{1}{2\pi}\Bigl(\tfrac{\theta_{y''}}{2}+\tfrac{\theta_{y'}}{2}\Bigr)}[y''_S y'_S]^{-1}\otimes \overrightarrow{y''_S y'_S}\Bigr) \\
        &\quad \Bigl(\CB^{-1} q^{\tfrac{1}{2\pi}\Bigl(\tfrac{\theta_{y}}{2}+\tfrac{\theta_{z''}}{2}\Bigr)}
         [y_S z''_S]^{-1} \otimes \overrightarrow{y_S z''_S}\Bigr)\\
        &= \CB^{-1}[y'_S y_S][y''_S y'_S]^{-1}[y_S z''_S]^{-1}\otimes \overrightarrow{y'_S y_S}\,\overrightarrow{y''_S y'_S}\,\overrightarrow{y_S z''_S}\\
        &=\CB^{-1} y''^{-1}_S z''^{-1}_S \otimes \overrightarrow{y'_S y_S}\,\overrightarrow{y''_S y'_S}\,\overrightarrow{y_S z''_S},
    \end{align*}
    and 
    \begin{equation*}
        \evmap\circ F \Bigl(q^{\tfrac{1}{2}}\overrightarrow{z''_{N_-}z_{N_+}}\,[\emptyset]\,\overrightarrow{z_{N_-}y''_{N_-}} \Bigr) 
        = \CB^{-1}y''^{-1}_N z''^{-1}_N \otimes \overrightarrow{z'_N z_N}\,\overrightarrow{z''_N z'_N}\,\overrightarrow{z_N y''_N}.
    \end{equation*}
    Tensoring these together, we get
    \begin{align*}
    &\bigotimes_{i=1}^2 \Bigl(\evmap \circ F \Bigr) \Bigl(
    q^{\frac{1}{2}}\overrightarrow{y''_{S_-}y_{S_+}}\,[\emptyset]\,\overrightarrow{y_{S_-}z''_{S_-}} \otimes
    q^{\frac{1}{2}}\overrightarrow{z''_{N_-}z_{N_+}}\,[\emptyset]\,\overrightarrow{z_{N_-}y''_{N_-}} \Bigr) \\
    &= \CB^{-2}\hat{y}''^{-1}\hat{z}''^{-1}\otimes \Bigl(\overrightarrow{y'_S y_S}\,\overrightarrow{y''_S y'_S}\,\overrightarrow{y_S z''_S} \otimes \overrightarrow{z'_N z_N}\,\overrightarrow{z''_N z'_N}\,\overrightarrow{z_N y''_N} \Bigr) \\
    &= \CB^{-2}(-A)\hat{y}''^{-1}\hat{z}''^{-1}\otimes \Bigl(\overrightarrow{y''_S y_S}\,\overrightarrow{y_S z''_S} \otimes \overrightarrow{z''_N z_N}\,\overrightarrow{z_N y''_N} \Bigr)
    \quad \in \SQGM_{\mathcal{T}}(Y) \otimes \bigotimes_{i=1}^2 \Sk^{\mathfrak{gl}_1}_{-A}(Sf),
    \end{align*}
    where, in the last line, we have used the following relations in $\Sk^{\mathfrak{gl}_1}_{-A}(Sf)$: 
    \begin{align*}
    \overrightarrow{y'_S y_S}\,\overrightarrow{y''_S y'_S}&=(-A)^{\frac{1}{2}}\overrightarrow{y''_S y_S}\\
    \overrightarrow{z_N' z_N}\,\overrightarrow{z_N'' z'_N}&=(-A)^{\frac{1}{2}}\overrightarrow{z''_N z_N}. 
    \end{align*}
    On the other hand, easy computations also show that, 
    \begin{align*}
        (\Tr_{Sf}\otimes \mathrm{Id}) \circ \pi_{Sf}\Bigl(q^{\tfrac{1}{2}}\overrightarrow{y''_{S_-}y_{S_+}}\,[\emptyset]\,\overrightarrow{y_{S_-}z''_{S_-}} \Bigr) &=
        q^{\tfrac{1}{2}}\Bigl([y''^{-1}_Sy_S]\otimes \overrightarrow{y''_Sy_S}\Bigr)\Bigl(\CB^{-1}[y_{S}z''_{S}]^{-1}\otimes \overrightarrow{y_{S}z''_S}\Bigr)\\
        &= \CB^{-1}q^{\tfrac{1}{2}}A^{-\tfrac{1}{2}}y''^{-1}_Sz''^{-1}_S\otimes \overrightarrow{y''_Sy_S}\,\overrightarrow{y_{S}z''_S},
    \end{align*}
    and 
    \begin{align*}
        (\Tr_{Sf}\otimes \mathrm{Id}) \circ \pi_{Sf}\Bigl(q^{\tfrac{1}{2}}\overrightarrow{z''_{N_-}z_{N_+}}\,[\emptyset]\,\overrightarrow{z_{N_-}y''_{N_-}}  \Bigr) &=
        q^{\tfrac{1}{2}}\Bigl([z''^{-1}_Nz_N]\otimes \overrightarrow{z''_Nz_N}\Bigr)\Bigl(\CB^{-1}[z_{N}y''_{N}]^{-1}\otimes \overrightarrow{z_{N}y''_N}\Bigr)\\
        &= \CB^{-1}q^{\tfrac{1}{2}}A^{-\tfrac{1}{2}}y''^{-1}_Nz''^{-1}_N\otimes \overrightarrow{z''_Ny_N}\,\overrightarrow{z_{N}y''_N},
    \end{align*}
    which after combining yields
    \begin{align*}
    &\bigotimes_{i=1}^2 \Bigl((\Tr_{Sf} \otimes \mathrm{Id})\circ \pi_{Sf} \Bigr) \Bigl(
    q^{\frac{1}{2}}\overrightarrow{y''_{S_-}y_{S_+}}\,[\emptyset]\,\overrightarrow{y_{S_-}z''_{S_-}} \otimes
    q^{\frac{1}{2}}\overrightarrow{z''_{N_-}z_{N_+}}\,[\emptyset]\,\overrightarrow{z_{N_-}y''_{N_-}} \Bigr) \\
    &= \CB^{-2} (-A) \hat{y}''^{-1} \hat{z}''^{-1}\otimes \Bigl(\overrightarrow{y''_Sy_S}\,\overrightarrow{y_Sz''_S} \otimes \overrightarrow{z''_Nz_N}\,\overrightarrow{z_Ny''_N}\Bigr) 
    \quad \in \SQGM\otimes \bigotimes_{i=1}^2 \Sk^{\mathfrak{gl}_1}_{-A}(Sf),
    \end{align*}
    as expected. 
\end{itemize}

Lastly, since 
\[
\overrightarrow{y''_S y_S}\,\overrightarrow{y_S z''_S} \otimes \overrightarrow{z''_N z_N}\,\overrightarrow{z_N y''_N} = [\vec{K}_b] \in \Sk^{\mathfrak{gl}_1}_{-A}(Y)
\]
under the isomorphism in Lemma \ref{lem:gl1_iso}, 
we have
\begin{align*}
p_{\vec{K}_b}\circ \evmap \circ F_{\mathcal{T}}(\vec{K}_b)&= p_{\vec{K}_b}\Bigl(A\hat{y}''\hat{z}''^{-1}\otimes \Bigl(\overrightarrow{y''_Sy_S}\,\overrightarrow{y_Sz''_S} \otimes \overrightarrow{z''_Nz_N}\,\overrightarrow{z_Ny''_N}\Bigr) \Bigr)\\
&+p_{\vec{K}_b}\Bigl( A\hat{y}''^{-1}\hat{z}''\otimes \Bigl(\overrightarrow{y''_Sy_S}\,\overrightarrow{y_Sz''_S} \otimes \overrightarrow{z''_Nz_N}\,\overrightarrow{z_Ny''_N}\Bigr)\Bigr) \\
&+p_{\vec{K}_b}\Bigl( \CB^{-2}(-A)\hat{y}''^{-1}\hat{z}''^{-1}\otimes \Bigl(\overrightarrow{y''_Sy_S}\,\overrightarrow{y_Sz''_S} \otimes \overrightarrow{z''_Nz_N}\,\overrightarrow{z_Ny''_N}\Bigr)\Bigr)\\
&=A\hat{y}''\hat{z}''^{-1}+ A\hat{y}''^{-1}\hat{z}''+(-A)\CB^{-2}\hat{y}''^{-1}\hat{z}''^{-1}\\
&=\Tr_{\mathcal{T}}(K_b),
\end{align*}
confirming Theorem \ref{thm:quantum-trace-from-quantum-UV-IR} in this case.\footnote{This result for $\Tr(K_b)$ agrees with that from \cite[Sec. 6.2.2]{PP} after adjusting for the difference in convention for shape parameters and recalling that $\CB=(-1)^{-\frac{1}{2}}$ in that work.}

\appendix
\section{Stated $\mathfrak{gl}_2$-skein modules}\label{sec:stated-gl2-skeins}
We need to first define precisely what we mean by the stated (i.e. relative) version of the $\mathfrak{gl}_2$-skein module. 

\begin{defn}
Let $(Y,\Gamma)$ be a boundary marked 3-manifold, and $R := \mathbb{Z}[q^{\pm \frac{1}{2}}]$. 
The \emph{(stated) $\mathfrak{gl}_2$-skein module} of $(Y,\Gamma)$, $\Sk_q^{\mathfrak{gl}_2}(Y,\Gamma)$ is defined as the $R$-span of isotopy classes of all framed, oriented, stated tangles in $(Y,\Gamma)$, modulo usual relations \eqref{eq:gl2skeinrel1}-\eqref{eq:gl2extraskeinrel2}, as well as the following boundary skein relations: 
\begin{gather}
\vcenter{\hbox{

\label{eq:gl2boundaryskeinrel4}
\;\;,
\end{gather}
plus all the skein relations obtained by simultaneous orientation reversal of the tangles in the skein relations above. 
\end{defn}

\begin{rmk}
Note, the boundary height exchange relations (i.e., the ones involving matrices) correspond exactly to the second part of Lemma 3.26 in \cite{PP}, upon projection to $\mathfrak{sl}_2$-skeins; see Proposition \ref{prop:gl2tosl2}. 
These are also the $R$-matrices for the Jones polynomial:\footnote{It is worth noting that these $R$-matrices are exactly the matrices we get by applying the quantum UV-IR map to a positive and a negative crossing.}
\[
\vcenter{\hbox{

\;\;,
\]
etc., 
with cups and caps
\[
\vcenter{\hbox{
%
\;\;,
\]
etc. 
In particular, as we gradually collide a link with a boundary marking, we obtain the usual state sum formula for its Jones polynomial. 
\end{rmk}

For convenience, again write 
\[
\SkAlg^{\mathfrak{gl}_2}_q(V(\Gamma)^\pm):=\bigotimes_{v \in V(\Gamma)^\pm}\SkAlg^{\mathfrak{gl}_2}_q(D_{\mathrm{deg}\,v}),
\]
where $V(\Gamma^+)$ and $V(\Gamma^-)$ denote the sets of sink and source vertices of $\Gamma$, respectively. 
By an identical argument to the $\mathfrak{sl}_2$ case, we have: 
\begin{prop}
The stated skein module $\Sk^{\mathfrak{gl}_2}_q(Y,\Gamma)$ has a natural $\SkAlg^{\mathfrak{gl}_2}_q(V(\Gamma)^+)$-$\SkAlg^{\mathfrak{gl}_2}_q(V(\Gamma)^-)$-bimodule structure.
\end{prop}

The stated skein relations above are necessary to obtain a well-defined splitting map for stated $\mathfrak{gl}_2$-skein modules. 

\begin{thm} \label{thm:gl2splittingmap}
Let $(Y_1,\Gamma_1)$ and $(Y_2,\Gamma_2)$ be boundary marked $3$-manifolds. 
Suppose that $\Sigma_1 \subset \partial Y_1$ and $\Sigma_2 \subset Y_2$, along with their markings, are homeomorphic combinatorial foliated surfaces (see Definition \ref{defn:combinatorialFoliation}) of opposite orientations. 
Write $\Sigma$ for the common image of $\Sigma_1$ and $\Sigma_2$ after gluing, and set $Y=Y_1 \cup_\Sigma Y_2$ and $\Gamma = (\Gamma_1 \setminus \mathrm{int}\,\Sigma_1)\cup(\Gamma_2 \setminus \mathrm{int}\, \Gamma_2).$
Let $\vec{L}$ be a stated, oriented tangle representing $[\vec{L}]\in\Sk^{\mathfrak{gl}_2}(Y,\Gamma)$, isotoped so that $\vec{L}\cap \Sigma \subset \Gamma,$ guaranteeing that $\vec{L}_1:=\vec{L}\cap Y_1$ and $\vec{L}_2:=\vec{L}\cap Y_2$ are tangles in $(Y_1,\Gamma_1)$ and $(Y_2,\Gamma_2)$.
Then, there is an R-module homomorphism,
\begin{align*}
    \sigma^{\mathfrak{gl}_2}: \Sk^{\mathfrak{gl}_2}_q(Y,\Gamma) &\rightarrow \Sk^{\mathfrak{gl}_2}_q(Y_1,\Gamma_1) \overline{\otimes} \Sk^{\mathfrak{gl}_2}_q(Y_2,\Gamma_2)\\
    [\vec{L}] &\mapsto \left[ \sum_{\vec{\epsilon}\in \{\pm\}^{\Gamma \cap \vec{L}}}[\vec{L}^{\vec{\epsilon}}_1]\otimes [\vec{L}^{\vec{\epsilon}}_2]\right],
\end{align*}
where the sum is over all functions $\vec{\epsilon}:\vec{L} \cap \Sigma \rightarrow \{\pm\}$ assigning states to endpoints of $\vec{L}$, $\vec{L}^{\vec{\epsilon}}_1$ and $\vec{L}^{\vec{\epsilon}}_2$ denote the stated tangles obtained from $\vec{L}_1$ and $\vec{L}_2$ by assigning states according to $\vec{\epsilon},$ and the relative tensor product $\Sk^{\mathfrak{gl}_2}_q(Y_1,\Gamma_1) \overline{\otimes} \Sk^{\mathfrak{gl}_2}_q(Y_2,\Gamma_2)$ denotes the quotient of the usual tensor product $\Sk^{\mathfrak{gl}_2}_q(Y_1,\Gamma_1) \otimes \Sk^{\mathfrak{gl}_2}_q(Y_2,\Gamma_2)$ by the following relations:
\begin{enumerate}
\item For each internal edge e of of $\Sigma$ adjacent to a sink, we have the following relations among left actions:
\begin{align*}
&
\text{the left action of }\;
q^{\frac{\mu + \nu}{4}}
\vcenter{\hbox{

}}
\text{ on }\Sk^{\mathfrak{gl}_2}_q(Y_1, \Gamma_1)\\
=\;\;&\text{the left action of }\;
q^{-\frac{\mu + \nu}{4}}
\vcenter{\hbox{
%
}}
\text{ on }\Sk^{\mathfrak{gl}_2}_q(Y_2, \Gamma_2),
\end{align*}
along with the relations obtained by simultaneous orientation reversal of the tangles in both figures. 
\item Likewise, for each internal edge e of $\Sigma$ adjacent to a source, there are the following relations among right actions:
\begin{align*}
&
\text{the right action of }\;
q^{-\frac{\mu + \nu}{4}}
\vcenter{\hbox{
%
}}
\text{ on }\Sk^{\mathfrak{gl}_2}_q(Y_1, \Gamma_1) \\
=\;\;&\text{the right action of }\;
q^{\frac{\mu + \nu}{4}}
\vcenter{\hbox{
%
}}
\text{ on }\Sk^{\mathfrak{gl}_2}_q(Y_2, \Gamma_2), 
\end{align*}
along with the relations obtained by simultaneous orientation reversal of the tangles in both figures. 
\end{enumerate}
\end{thm}
\begin{proof}
As in \cite[Sec. 3]{PP}, we need to study how the element 
\begin{equation*}
\tilde{\sigma}^{\mathfrak{gl}_2}(\vec{L}):=\sum_{\vec{\epsilon}\in \{\pm\}^{\Gamma \cap \vec{L}}}[\vec{L}^{\vec{\epsilon}}_1]\otimes [\vec{L}^{\vec{\epsilon}}_2]
\end{equation*}
behaves under isotopy of $\vec{L}$. 
The proof is completely analogous to that of \cite[Thm 3.21]{PP}. 
The only difference lies in the invariance under the height exchange isotopy near a component of $\Gamma$:
\begin{equation*}
\vcenter{\hbox{

}}
\end{equation*}
Using the new $\mathfrak{gl}_2$-stated skein relations and starting from the image of the RHS, we compute,
\begin{align*}
\sum_{\mu,\nu \in \{\pm\}}
\vcenter{\hbox{
%
}}
\;=\; 
\tilde{\sigma}^{\mathfrak{gl}_2}(\mathrm{LHS}).
\end{align*}
The computation is similar for other choices of orientations of the tangles. 
\end{proof}

Reduced stated $\mathfrak{gl}_2$-skein modules are defined analogously to their $\mathfrak{sl}_2$ counterparts. 
\textit{Bad arcs} in the $\mathfrak{gl}_2$-skein module are shown below:
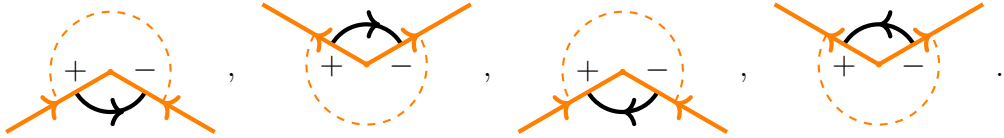
\begin{equation*}
\vcenter{\hbox{
%
}}.
\end{equation*}

\begin{defn}
The \textit{reduced stated skein module} $\overline{\Sk}^{\mathfrak{gl}_2}_q(Y,\Gamma)$ is the quotient 
\[
\overline{\Sk}^{\mathfrak{gl}_2}_q(Y, \Gamma) := 
I^{\mathrm{bad}, +} \backslash \Sk^{\mathfrak{gl}_2}_q(Y, \Gamma)/I^{\mathrm{bad}, -}
 = \frac{\Sk^{\mathfrak{gl}_2}_q(Y, \Gamma)}{I^{\mathrm{bad}, +}\Sk^{\mathfrak{gl}_2}_q(Y,\Gamma) + \Sk^{\mathfrak{gl}_2}_q(Y,\Gamma)I^{\mathrm{bad}, -}},
\]
where $I^{\mathrm{bad}, +}$ denotes the right ideal of $\SkAlg_q^{\mathfrak{gl}_2}(V(\Gamma)^+)$ generated by the bad arcs near the sinks, and $I^{\mathrm{bad}, -}$ denotes the left ideal of $\SkAlg_q^{\mathfrak{gl}_2}(V(\Gamma)^-)$ generated by the bad arcs near the sources. 
\end{defn}

As before, set 
\[
\overline{\SkAlg}_q^{\mathfrak{gl}_2}(V(\Gamma)^{\pm}) := \bigotimes_{v\in V(\Gamma)^{\pm}} \overline{\SkAlg}_q^{\mathfrak{gl}_2}(D_{\deg v}),
\]
and note that the $\SkAlg^{\mathfrak{gl}_2}_q(V(\Gamma)^+)$-$\SkAlg^{\mathfrak{gl}_2}_q(V(\Gamma)^-)$-bimodule structure on $\Sk^{\mathfrak{gl}_2}_q(Y,\Gamma)$ clearly descends to a $\overline{\SkAlg}^{\mathfrak{gl}_2}_q(V(\Gamma)^+)$-$\overline{\SkAlg}^{\mathfrak{gl}_2}_q(V(\Gamma)^-)$-bimodule structure on $\overline{\Sk}^{\mathfrak{gl}_2}_q(Y,\Gamma)$.

The following corollary -- the 2d splitting map for $\mathfrak{gl}_2$-skeins -- follows from Theorem \ref{thm:gl2splittingmap}; 
this is because each ideal arc times an interval is homeomorphic to an elementary quadrilateral and hence admits a natural combinatorial foliation as follows:
\begin{gather*}
e\times I \,=\,
\vcenter{\hbox{
\begin{tikzpicture}[decoration={
    markings,
    mark=at position 0.5 with {\arrow{>}}}]
\draw[dotted, very thick] (0,0)--(0,2);
\draw[dotted, very thick] (2,0)--(2,2);
\draw[black, very thick] (0,0)--(2,0);
\draw[black, very thick] (0,2)--(2,2);
\draw[orange, very thick, postaction={decorate}] (1,0)
--(1,2);
\node[above, scale=.75] at (1, 2) {$e\times \{1\}$};
\node[below, scale=.75] at (1, 0) {$e\times \{0\}$};
\filldraw[orange] (1,0) circle (1pt);
\filldraw[orange] (1,2) circle (1pt);
\end{tikzpicture}
}}
\,\cong\,
\vcenter{\hbox{
\begin{tikzpicture}[decoration={
    markings,
    mark=at position 0.5 with {\arrow{>}}}]
\draw[black,very thick] (1,0)--(0,1);
\draw[black,very thick] (0,1)--(1,2);
\draw[black,very thick] (1,0)--(2,1);
\draw[black,very thick] (2,1)--(1,2);
\draw[orange,very thick,postaction={decorate}] (1,0)
--(1,2);
\draw[fill = white] (2,1) circle (2pt);
\draw[fill = white] (0,1) circle (2pt);
\filldraw[orange] (1,0) circle (1pt);
\filldraw[orange] (1,2) circle (1pt);
\end{tikzpicture}
}}.
\end{gather*}
Since the boundary markings involved in the decomposition of an ideally triangulated surface into the ideal triangles have no vertices, there are no extra relations in the tensor product. 
\begin{cor}
Let $\Sigma=\bigcup_{\tri\in \tau^{(2)}}\tri$ be a decomposition of an ideally triangulated surface $\Sigma$ into ideal triangles. 
Then, there is a well-defined splitting map 
\begin{align*}
\overline{\sigma}^{\mathfrak{gl}_2}:
\overline{\SkAlg}^{\mathfrak{gl}_2}_q(\Sigma) &\rightarrow \bigotimes_{\tri\in \tau^{(2)}} \overline{\SkAlg}^{\mathfrak{gl}_2}_q(\tri), \\
[\vec{L}] &\mapsto \sum_{\vec{\epsilon} \in \{\pm\}^{(E\times I)\cap L}} \otimes_{\tri \in \tau^{(2)}}[\vec{L}^{\vec{\epsilon}}_{\tri}],
\end{align*}
where $E\in \Sigma$ denotes the union of all of the edges of the triangulation, and $\vec{L}^{\vec{\epsilon}}_{\tri}$ denotes the part of $\vec{L}$ in $\tri \times I$ after splitting, with states assigned to the newly created boundary points according to $\vec{\epsilon}.$
\end{cor}

To obtain the $\mathfrak{gl}_2$ splitting map for ideally triangulated $3$-manifolds, we first need the following definition:
\begin{defn}
The \emph{relative tensor product} $\overline{\bigotimes}_{f\in \mathcal{T}^{(2)}}\overline{\Sk}^{\mathfrak{gl}_2}_q(Sf)$ of the bimodules $\overline{\Sk}^{\mathfrak{gl}_2}_q(Sf)$, $f\in \mathcal{T}^{(2)}$ is defined to be the quotient of the ordinary tensor product (as $R$-modules) $\bigotimes_{f\in \mathcal{T}^{(2)}}\overline{\Sk}^{\mathfrak{gl}_2}_q(Sf)$ by the following relations: 
\begin{itemize}
\item For each vertex cone $Cv$, we have the following relations among left actions on $\bigotimes_{f\in \mathcal{T}^{(2)}} \overline{\Sk}^{\mathfrak{gl}_2}_q(Sf)$:
\begin{gather} 
\vcenter{\hbox{

}}
\;\;=\;\; 
0
\;,
\end{gather}
where each sector in the above diagrams represents one of the three face suspensions surrounding $Cv$ (viewed from the vertex $v$), and the markings shown are on the bare edge cones abutting $Cv$. 
\item For each internal edge $e \in \mathcal{T}^{(1)}$ 
we have the following relations among right actions of on $\bigotimes_{f\in \mathcal{T}^{(2)}}\overline{\Sk}^{\mathfrak{gl}_2}_q(Sf)$: 
\begin{equation} 
\vcenter{\hbox{
%
}}
\;,
\quad \epsilon \in \{\pm\},
\end{equation}
where each sector in the above diagrams represents one of the face suspensions surrounding $e$ (as many as the number of tetrahedra abutting $e$), and the markings shown are on the bare edge cones abutting $e$. 
\end{itemize} 
\end{defn}

The splitting map of Theorem \ref{thm:gl2splittingmap} then implies the following important corollary:
\begin{cor}\label{cor:gl2-splitting-map}
Let $Y= \bigcup_{f\in \mathcal{T}^{(2)}} Sf$ be a decomposition of an ideally triangulated $3$-manifold (without boundary except for cusps at infinity) into face suspensions. Then, there is a well-defined splitting map 
\begin{align*}              
\overline{\sigma}^{\mathfrak{gl}_2}:\overline{\Sk}^{\mathfrak{gl}_2}_q(Y) &\rightarrow 
\underset{f\in \mathcal{T}^{(2)}}{\overline{\bigotimes}} \overline{\Sk}^{\mathfrak{gl}_2}_q(Sf),\\
[\vec{L}] &\mapsto  \qty[\sum_{\vec{\epsilon} \;\in\; \{\pm\}^{\cup_f Sf\cap  \vec{L}}}\otimes_{f\in \mathcal{T}^{(2)}} [\vec{L}_{f}^{\vec{\epsilon}}]], 
\end{align*}
where $\vec{L}_f^{\vec{\epsilon}}$ denotes the part of $\vec{L}$ in $Sf$ after splitting, with boundary states determined by $\vec{\epsilon}$.
\end{cor}

\section{Stated $\mathfrak{gl}_1$-skein modules with defects} \label{sec:stated-gl1-skeins}

In this section of the appendix, we define the stated $\mathfrak{gl}_1$-skein module, both with and without sign defects.\footnote{
Even though we call it the ``stated'' $\mathfrak{gl}_1$-skein module in analogy with the stated $\mathfrak{sl}_2$- or $\mathfrak{gl}_2$-skein modules, there is a unique ``state'' that we can assign to the boundary of a $\mathfrak{gl}_1$-skein, so there's no extra data at the boundary points.
} 
Furthermore, we construct and prove the well-definedness of the splitting maps for the $\mathfrak{gl}_1$-skein modules above. 
The key results are Corollaries \ref{cor:gl1splittingmap} and \ref{cor:gl1splittingwithconepoints}, which allow us to split the $\mathfrak{gl}_1$-skein modules of both an ideally triangulated $3$-manifold as well as its double cover. 

\begin{defn}
Let $(Y,\Gamma)$ be a boundary marked $3$-manifold, and let $R:=\mathbb{Z}[q^{\pm1}]$.
The \emph{(stated) $\mathfrak{gl}_1$-skein module} $\Sk^{\mathfrak{gl}_1}_q(Y,\Gamma)$ of $Y$ is defined as the $R$-span of isotopy classes of all framed, oriented tangles in $(Y,\Gamma)$ modulo the usual $\mathfrak{gl}_1$-skein relations (\ref{eq:gl1skeinrel1}), (\ref{eq:gl1skeinrel2}), and (\ref{eq:gl1skeinrel4}) in addition to the following: 
\begin{gather}
\vcenter{\hbox{
\begin{tikzpicture}[scale=0.7]
\draw[->, orange, ultra thick] (0, -1) -- (0, 1);
\draw[very thick] (0, -0.5) to[out=0, in=-90] (0.5, 0) to[out=90, in=0] (0, 0.5);
\draw[->, very thick] (0, -0.5) to[out=0, in=-90] (0.5, 0);
\end{tikzpicture}
}}
\;\;=\;\;
\vcenter{\hbox{
\begin{tikzpicture}[scale=0.7]
\draw[->, orange, ultra thick] (0, -1) -- (0, 1);
\end{tikzpicture}
}}
\;\;,
\label{eq:statedgl1skeinrel3}
\\
\vcenter{\hbox{
\begin{tikzpicture}[scale=0.7]
\draw[->, orange, ultra thick] (0, -1) -- (0, 1);
\draw[very thick] (1, 0.5) to[out=180, in=90] (0.5, 0) to[out=-90, in=180] (1, -0.5);
\draw[->, very thick] (1, 0.5) to[out=180, in=90] (0.5, 0);
\end{tikzpicture}
}}
\;\;=\;\;
\vcenter{\hbox{
\begin{tikzpicture}[scale=0.7]
\draw[->, orange, ultra thick] (0, -1) -- (0, 1);
\draw[very thick] (1, 0.5) -- (0, 0.5);
\draw[->, very thick] (1, 0.5) -- (0.5, 0.5);
\draw[very thick] (1, -0.5) -- (0, -0.5);
\draw[->, very thick] (0, -0.5) -- (0.5, -0.5);
\end{tikzpicture}
}}
\;\;,
\label{eq:statedgl1skeinrel4}
\end{gather}
as well as all of the skein relations obtained by simultaneous orientation reversal of the tangles in the aforementioned relations. 

Let $\widetilde{Y}$ be the associated double cover with boundary marking $\widetilde{\Gamma}$.
The stated $\mathfrak{gl}_1$-skein module with defects $\Sk^{\mathfrak{gl}_1}_q(\widetilde{Y},\widetilde{\Gamma})$ is the $R$-module spanned by isotopy classes of framed oriented links away from the branch locus modulo the relations above in addition to (\ref{eq:gl1skeinrel3}).
\end{defn}

These skein modules of course admit natural bimodule structures.
\begin{prop}
The stated $\mathfrak{gl}_1$-skein modules $\Sk^{\mathfrak{gl}_1}_q(Sf,\Gamma)$ and $\Sk^{\mathfrak{gl}_1}_q(\widetilde{Sf},\widetilde{\Gamma})$) admit natural $\SkAlg^{\mathfrak{gl}_1}_q(V(\Gamma)^+)$-$\SkAlg^{\mathfrak{gl}_1}_q(V(\Gamma)^-)$ and  $\SkAlg^{\mathfrak{gl}_1}_q(V(\widetilde{\Gamma})^+)$-$\SkAlg^{\mathfrak{gl}_1}_q(V(\widetilde{\Gamma})^-)$-bimodule structures, respectively. 
\end{prop}

In much the same way as Theorem \ref{thm:gl2splittingmap}, 
we have:
\begin{thm}
Let $(Y,\Gamma)$ be either a boundary marked $3$-manifold or its double cover $(\widetilde{Y},\widetilde{\Gamma})$.
We write $\Sk^{\mathfrak{gl}_1}_q(Y,\Gamma)$ for the corresponding skein module in each case.

Then, there is an $R$-module homomorphism, 
\begin{align*}
    \sigma^{\mathfrak{gl}_1}: \Sk^{\mathfrak{gl}_1}_q(Y,\Gamma) &\rightarrow \Sk^{\mathfrak{gl}_1}_q(Y_1,\Gamma_1) \overline{\otimes} \Sk^{\mathfrak{gl}_1}_q(Y_2,\Gamma_2)\\
    [\vec{L}] &\mapsto \left[[\vec{L}_1] \otimes [\vec{L}_2]\right],
\end{align*}
where the relative tensor product $\Sk^{\mathfrak{gl}_1}_q(Y_1,\Gamma_1) \overline{\otimes} \Sk^{\mathfrak{gl}_1}_q(Y_2,\Gamma_2)$ denotes the quotient of the usual tensor product $\Sk^{\mathfrak{gl}_1}_q(Y_1,\Gamma_1) \otimes \Sk^{\mathfrak{gl}_1}_q(Y_2,\Gamma_2)$ by the following relations:
\begin{enumerate}
\item For each internal edge e of of $\Sigma$ adjacent to a sink, we have the following relations among left actions:
\begin{align*}
&
\text{the left action of }\;
\vcenter{\hbox{

}}
\text{ on }\Sk^{\mathfrak{gl}_1}_q(Y_1, \Gamma_1)\\
=\;\;&\text{the left action of }\;
\vcenter{\hbox{
%
}}
\text{ on }\Sk^{\mathfrak{gl}_1}_q(Y_2, \Gamma_2),
\end{align*}
along with the relations obtained by simultaneous orientation reversal of the tangles in both figures. 
\item Likewise, for each internal edge e of $\Sigma$ adjacent to a source, there are the following relations among right actions:
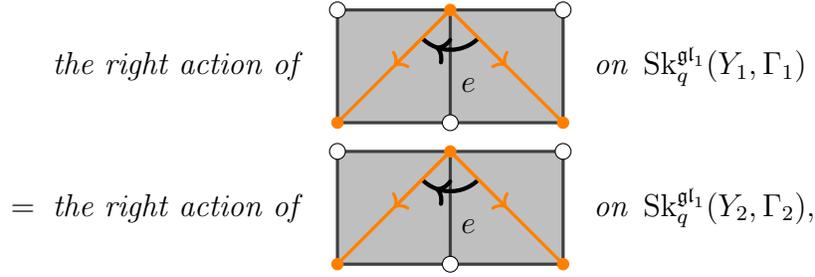
\begin{align*}
&
\text{the right action of }\;
\vcenter{\hbox{
%
}}
\text{ on }\Sk^{\mathfrak{gl}_1}_q(Y_1, \Gamma_1) \\
=\;\;&\text{the right action of }\;
\vcenter{\hbox{
%
}}
\text{ on }\Sk^{\mathfrak{gl}_1}_q(Y_2, \Gamma_2), 
\end{align*}
along with the relations obtained by simultaneous orientation reversal of the tangles in both figures. 
\end{enumerate}
\end{thm}

For the purpose of comparing the 3d quantum trace map with the 3d quantum UV-IR map, we are most interested in the special case of splitting an ideally triangulated 3-manifold $Y$ into face suspensions and the corresponding splitting for the branched double cover (i.e., splitting $\widetilde{Y}$ into $\widetilde{Sf}$'s). 

\subsubsection*{Splitting $Y$ into face suspensions}
Here, we use the $\mathfrak{gl}_1$-skein modules with parameter $-A$, since that is the relevant skein module used in our main construction. 
\begin{defn}
The \textit{relative tensor product} $\overline{\bigotimes}_{f\in \mathcal{T}^{(2)}}\Sk^{\mathfrak{gl}_1}_{-A}(Sf)$ of the bimodules $\Sk^{\mathfrak{gl}_1}_{-A}(Sf)$ is the quotient of the naive tensor product $\bigotimes_{f\in \mathcal{T}^{(2)}}\Sk^{\mathfrak{gl}_1}_{-A}(Sf)$ by the following relations:
\begin{itemize}
\item For each vertex cone $Cv$, we have the following relations among left actions on $\bigotimes_{f\in \mathcal{T}^{(2)}} \Sk^{\mathfrak{gl}_1}_{-A}(Sf)$:
\begin{equation} \label{eq:gl1GluingRel1}
\vcenter{\hbox{

}}
\;,
\end{equation}
where each sector in the above diagrams represents one of the three face suspensions surrounding $Cv$ (viewed from the vertex $v$), in addition to the relation obtained by simultaneously reversing the orientations of all of the tangles in the figures above. 
\item For each internal edge $e \in \mathcal{T}^{(1)}$ 
we have the following relations among right actions on $\bigotimes_{f\in \mathcal{T}^{(2)}}\Sk^{\mathfrak{gl}_1}_{-A}(Sf)$: 
\begin{equation} \label{eq:gl1GluingRel2}
\vcenter{\hbox{
%
}}
\;,
\end{equation}
where each sector in the above diagrams represents one of the face suspensions surrounding $e$ (as many as the number of tetrahedra abutting $e$), in addition to the relation obtained by simultaneously reversing the orientations of all of the tangles in the above figures. 
\end{itemize} 
\end{defn}

\begin{cor} \label{cor:gl1splittingmap}
Let $Y = \bigcup_{f\in \mathcal{T}^{(2)}} Sf$ be a decomposition of an ideally triangulated 3-manifold into face suspensions. 
Then, there is a well-defined splitting map
\begin{align*}              
\sigma^{\mathfrak{gl}_1}:\Sk^{\mathfrak{gl}_1}_{-A}(Y) &\rightarrow 
\underset{f\in \mathcal{T}^{(2)}}{\overline{\bigotimes}}\Sk^{\mathfrak{gl}_1}_{-A}(Sf)\\
[\vec{L}] &\mapsto  \qty[\otimes_{f\in \mathcal{T}^{(2)}} [\vec{L}_{f}]], 
\end{align*}
where $\vec{L}_f$ denotes the part of $\vec{L}$ in $Sf$ after splitting.
\end{cor}

The relative tensor product $\overline{\bigotimes}_{f\in \mathcal{T}^{(2)}} \Sk_{-A}^{\mathfrak{gl}_1}(Sf)$ is graded by $\oplus_{Ce}\mathbb{Z}$; i.e., for each edge cone $Ce$, it has a $\mathbb{Z}$-grading given by the signed count of the end points of a $\mathfrak{gl}_1$-tangle on that edge cone. 
The image of the splitting map lies in the $0$-graded part, with respect to this grading. 
We denote by $\qty(\overline{\bigotimes}_{f\in \mathcal{T}^{(2)}} \Sk_{-A}^{\mathfrak{gl}_1}(Sf))_0$ this 0-graded piece. 

\begin{lem}\label{lem:gl1_iso}
The splitting map
\[
\Sk_{-A}^{\mathfrak{gl}_1}(Y) \overset{\sigma^{\mathfrak{gl}_1}}{\rightarrow} 
\qty(
\underset{f\in \mathcal{T}^{(2)}}{\overline{\bigotimes}} \Sk_{-A}^{\mathfrak{gl}_1}(Sf)
)_0
\]
is an isomorphism of $R$-modules. 
\end{lem}
\begin{proof}
The proof is similar to that of Lemma \ref{lem:gl1_iso_surfaces}, except that in this case we have some relations to check. 
Firstly, it is easy to see that this map is surjective; $\qty(\overline{\bigotimes}_{f\in \mathcal{T}^{(2)}} \Sk_{-A}^{\mathfrak{gl}_1}(Sf))$ has a spanning set given by the form $\{\otimes_{f\in \mathcal{T}^{(2)}} [L_{\vec{n}_f}]\;\vert\; \vec{n}_f \in \mathbb{Z}^6,\, n_{f,1} + \cdots + n_{f,6} = 0\}$, where $L_{\vec{n}_f} \in \Sk_{-A}^{\mathfrak{gl}_1}(Sf)$ denotes the distinguished $\mathfrak{gl}_1$-skein ($\mathfrak{gl}_1$-web) in $Sf$ with boundary condition $\vec{n}_f$; see Figure \ref{fig:gl1-web-basis-for-Sf}. 
\begin{figure}[htbp]
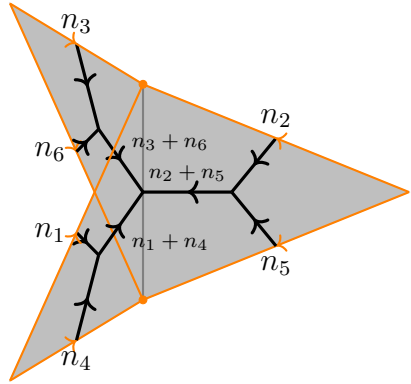

\centering
\[
\vcenter{\hbox{
\tdplotsetmaincoords{35}{60}

}}
\]
\caption{The $\mathfrak{gl}_1$-web $L_{\vec{n}}$ in $Sf$; each strand labeled by $n \in \mathbb{Z}$ denotes the $n$-colored strand (equivalent to $n$ parallel strands), and they are all flat on the leaf space.}
\label{fig:gl1-web-basis-for-Sf}
\end{figure}
Likewise, $\qty(\overline{\bigotimes}_{f\in \mathcal{T}^{(2)}} \Sk_{-A}^{\mathfrak{gl}_1}(Sf))_0$ has a spanning set given by the form $\otimes_{f\in \mathcal{T}^{(2)}} [L_{\vec{n}_f}]$ but with matching boundary conditions along boundary markings that are glued, which is the image of the corresponding glued $\mathfrak{gl}_1$-skein in $\Sk_{-A}^{\mathfrak{gl}_1}(Y)$. 

To show the injectivity of the splitting map, we can construct the inverse of the splitting map in the following way. 
There is a map
\[
\qty(\bigotimes_{f\in \mathcal{T}^{(2)}} \Sk_{-A}^{\mathfrak{gl}_1}(Sf))_0 \overset{g}{\rightarrow} \Sk_{-A}^{\mathfrak{gl}_1}(Y)
\]
which can be defined on the basis elements $\otimes_{f\in \mathcal{T}^{(2)}} [L_{\vec{n}_f}]$ (with matching boundary conditions) by just gluing the skeins. 
Furthermore, it is easy to see that this map respects the left and right relations, and thus descends to the quotient to give a map
\[
\qty(
\underset{f\in \mathcal{T}^{(2)}}{\overline{\bigotimes}} \Sk_{-A}^{\mathfrak{gl}_1}(Sf)
)_0 \overset{\overline{g}}{\rightarrow} \Sk_{-A}^{\mathfrak{gl}_1}(Y).
\]
By construction, $\overline{g} \circ \sigma^{\mathfrak{gl}_1}$ is an identity on $\Sk_{-A}^{\mathfrak{gl}_1}(Y)$. 
Therefore, the splitting map
\[
\Sk_{-A}^{\mathfrak{gl}_1}(Y) \overset{\sigma^{\mathfrak{gl}_1}}{\rightarrow} \qty(
\underset{f\in \mathcal{T}^{(2)}}{\overline{\bigotimes}} \Sk_{-A}^{\mathfrak{gl}_1}(Sf)
)_0
\]
is an isomorphism. 
\end{proof}

\subsubsection*{Splitting $\widetilde{Y}$ into branched double covers of face suspensions}
Obtaining a splitting map for $\Sk^{\mathfrak{gl}_1}_q(\widetilde{Y},\Theta)$ requires a bit more work, due to the presence of the 3-term relations \eqref{eq:gl1skeinrel5} around the cone points. 

\begin{defn} \label{defn:gl1GluingRelationsDoubleCover}
The \emph{pre-relative tensor product} $\overline{\bigotimes}^{\circ}_{f\in \mathcal{T}^{(2)}}\Sk^{\mathfrak{gl}_1}_q(\widetilde{Sf})$ of the bimodules $\Sk^{\mathfrak{gl}_1}_q(\widetilde{Sf})$ is the quotient of $\bigotimes_{f\in \mathcal{T}^{(2)}}\Sk^{\mathfrak{gl}_1}_q(\widetilde{Sf})$ by the following relations (as well as the ones obtained by simultaneous orientation reversal of all the tangles): 
\begin{itemize}
\item For each vertex cone $Cv$, we have the following relations among left actions on $\bigotimes_{f\in \mathcal{T}^{(2)}} \Sk^{\mathfrak{gl}_1}_q(\widetilde{Sf})$:
\begin{equation} \label{eqn:gl1-vertex-cone}
\vcenter{\hbox{

}}
\;,
\quad
i \in \{1, 2\},
\end{equation}
where each sector in the above diagrams represents the projection of one of the three face suspensions surrounding $Cv$ (viewed from the vertex $v$), and $i$ denotes the sheet labels of the tangles. 
\item For each internal edge $e \in \mathcal{T}^{(1)}$ 
we have the following relations among right actions on $\bigotimes_{f\in \mathcal{T}^{(2)}}\Sk^{\mathfrak{gl}_1}_q(\widetilde{Sf})$: 
\begin{equation}\label{eqn:gl1-gluing}
\vcenter{\hbox{
%
}}
\;,
\quad
i \in \{1,2\},
\end{equation}
where each sector in the above diagrams represents the projection one of the face suspensions surrounding $e$ and $i$ labels the sheet of the tangles. 
\end{itemize}
\end{defn}

\begin{cor} \label{cor:gl1splittingwithoutconepoints}
Suppose $Y$ is an ideally triangulated $3$-manifold with an associated branched double cover $\widetilde{Y}$. 
Let $\widetilde{Y}^\circ$ denote $\widetilde{Y}$ with a small neighborhood of the cone points removed, and let $\Sk_q^{\mathfrak{gl}_1}(\widetilde{Y}^\circ)$ denote the $\mathfrak{gl}_1$-skein module of $\widetilde{Y}^\circ$, i.e., with sign defects along the branch locus but \emph{without} the 3-term relations at cone points. 
Then, for the decomposition $\widetilde{Y}^\circ = \bigcup_{f\in \mathcal{T}^{(2)}}\widetilde{Sf}$ of $\widetilde{Y}^\circ$ into branched double covers of face suspensions, there is a well-defined splitting map 
\begin{align*}
\sigma^{\mathfrak{gl}_1}:\Sk^{\mathfrak{gl}_1}_q(\widetilde{Y}^\circ) &\rightarrow \, 
\underset{f\in \mathcal{T}^{(2)}}{\overline{\bigotimes}^\circ}\Sk^{\mathfrak{gl}_1}_q(\widetilde{Sf})\\
[\vec{L}] &\mapsto \qty[\otimes_{f\in \mathcal{T}^{(2)}} [\vec{L}_{f}]]. 
\end{align*}
\end{cor}

For each neighborhood of a cone point that is removed from $\widetilde{Y}$, $\widetilde{Y}^\circ$ has a boundary component which is a torus $T^2$ with 4 branch points. 
It follows that $\Sk^{\mathfrak{gl}_1}_q(\widetilde{Y}^\circ)$ is a module over $\bigotimes_{T \in \mathcal{T}^{(3)}} \SkAlg^{\mathfrak{gl}_1}_q(T^2)$, where $T^2$ here means the torus with 4 branch points; see Figure \ref{fig:boundaryProjection} for an illustration of this boundary torus. 
\begin{figure}[htbp]
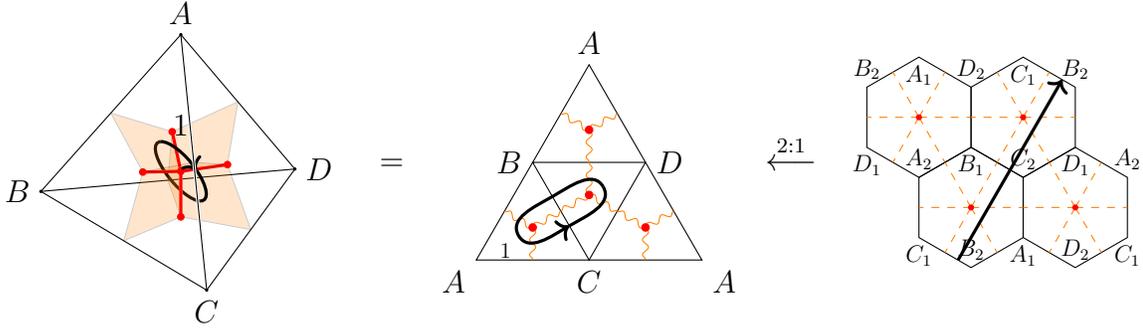

\[
\vcenter{\hbox{
\tdplotsetmaincoords{60}{80}

}}
\]
\caption{The boundary of a tetrahedron has been unraveled; its double cover is a torus. 
Vertices $A, B, C,$ and $D$ are lifted to the double cover as $A_1, B_1, C_1,$ and $D_1$ on sheet $1$ and as $A_2, B_2, C_2,$ and $D_2$ on sheet 2.
We also show a cycle, its image under projection to the boundary, and its appearance in the double cover. 
}
\label{fig:boundaryProjection}
\end{figure}
While the naive tensor product $\bigotimes_{f\in \mathcal{T}^{(2)}}\Sk^{\mathfrak{gl}_1}_q(\widetilde{Sf})$ is no longer a module over $\bigotimes_{T \in \mathcal{T}^{(3)}} \SkAlg^{\mathfrak{gl}_1}_q(T^2)$, the pre-relative tensor product $\overline{\bigotimes}^\circ_{f\in \mathcal{T}^{(2)}}\Sk^{\mathfrak{gl}_1}_q(\widetilde{Sf})$ recovers this module structure. 
Indeed, while an isotopy of a link diagram on $T^2$ does not preserve the corresponding element in $\SkAlg^{\mathfrak{gl}_1}_q(\widetilde{D_3})^{\otimes 4}$ upon splitting of $T^2$ into 4 hexagons, the relations \eqref{eqn:gl1-vertex-cone} around each vertex cone ensures that we have a well-defined splitting
\[
\SkAlg^{\mathfrak{gl}_1}_q(T^2) \rightarrow \underset{f \in \mathcal{T}^{(2)}}{\overline{\bigotimes}} \SkAlg^{\mathfrak{gl}_1}_q(\widetilde{D_3}),
\]
where the relative tensor product on the right-hand side is over the 4 face suspensions around the same barycenter. 
It follows that $\overline{\bigotimes}^\circ_{f\in \mathcal{T}^{(2)}}\Sk^{\mathfrak{gl}_1}_q(\widetilde{Sf})$ is a $\bigotimes_{T \in \mathcal{T}^{(3)}} \SkAlg^{\mathfrak{gl}_1}_q(T^2)$-module. 
It is also immediate that the splitting map $\sigma^{\mathfrak{gl}_1}:\Sk^{\mathfrak{gl}_1}_q(\widetilde{Y}^\circ) \rightarrow \overline{\bigotimes}^\circ_{f\in \mathcal{T}^{(2)}}\Sk^{\mathfrak{gl}_1}_q(\widetilde{Sf})$
is a $\bigotimes_{T \in \mathcal{T}^{(3)}} \SkAlg^{\mathfrak{gl}_1}_q(T^2)$-module homomorphism. 

To this end, in order to get a splitting map out of $\Sk^{\mathfrak{gl}_1}_q(\widetilde{Y})$, we simply need to take a further quotient of the pre-relative tensor product $\overline{\bigotimes}^\circ_{f\in \mathcal{T}^{(2)}}\Sk^{\mathfrak{gl}_1}_q(\widetilde{Sf})$ by imposing the 3-term relations \eqref{eq:gl1skeinrel5} using this $\bigotimes_{T \in \mathcal{T}^{(3)}} \SkAlg^{\mathfrak{gl}_1}_q(T^2)$-module structure. 

\begin{defn}\label{defn:gl1-relative-tensor-product}
The \emph{relative tensor product} $\overline{\bigotimes}_{f\in \mathcal{T}^{(2)}} \Sk^{\mathfrak{gl}_1}_q(\widetilde{Sf})$ of the bimodules $\Sk^{\mathfrak{gl}_1}_q(\widetilde{Sf})$ is the quotient of the pre-relative tensor product $\overline{\bigotimes}^\circ_{f\in \mathcal{T}^{(2)}}\Sk^{\mathfrak{gl}_1}_q(\widetilde{Sf})$, thought of as a left $\bigotimes_{T \in \mathcal{T}^{(3)}} \SkAlg^{\mathfrak{gl}_1}_q(T^2)$-module, by the 3-term relations \eqref{eq:gl1skeinrel5} among left actions: 
\begin{equation} \label{eq:3termRelProjected}
\vcenter{\hbox{

}}.
\end{equation}
\end{defn}
The following corollary is immediate: 
\begin{cor} \label{cor:gl1splittingwithconepoints}
Suppose $Y$ is an ideally triangulated $3$-manifold with an associated branched double cover $\widetilde{Y}$.
Further suppose that $Y$ is equipped with a generalized angle structure $\Theta$.
Let $\widetilde{Y} = \bigcup_{f\in \mathcal{T}^{(2)}}\widetilde{Sf}$ be the decomposition of $\widetilde{Y}$ into branched double covers of face suspensions. 
Then, there is a well-defined splitting map 
\begin{align*}              
\sigma^{\mathfrak{gl}_1}:\Sk^{\mathfrak{gl}_1}_q(\widetilde{Y},\Theta) &\rightarrow \, 
\underset{f\in \mathcal{T}^{(2)}}{\overline{\bigotimes}} \Sk^{\mathfrak{gl}_1}_q(\widetilde{Sf})\\
[\vec{L}] &\mapsto \qty[\otimes_{f\in \mathcal{T}^{(2)}} [\vec{L}_{f}]]. 
\end{align*} 
\end{cor}

\bibliography{ref}
\bibliographystyle{alpha}
\end{document}

%% file: figures/module_structureLHS.tex
\tdplotsetmaincoords{60}{80}
\begin{tikzpicture}[tdplot_main_coords]
\begin{scope}[scale = 0.5, tdplot_main_coords]
    \coordinate (o) at (0, 0, 0);
    \coordinate (a) at (3, 0, 0);
    \coordinate (b) at ({-1/2*3}, {sqrt(3)/2*3}, 0);
    \coordinate (c) at ({-1/2*3}, {-sqrt(3)/2*3}, 0);
    
    \coordinate (a1) at (1*3, 0, -4);
    \coordinate (b1) at ({-1/2*3}, {sqrt(3)/2*3}, -4);
    \coordinate (c1) at ({-1/2*3}, {-sqrt(3)/2*3}, -4);

    \coordinate (a2) at (1*3*2, 0, -6);
    \coordinate (a3/2) at (1*3, 0, -6);

    \coordinate (b2) at ({-1/2*3*2}, {sqrt(3)/2*3*2}, -6);
    \coordinate (b3/2) at ({-1/2*3}, {sqrt(3)/2*3}, -6);

    \coordinate (c2) at ({-1/2*3*2}, {-sqrt(3)/2*3*2}, -6);
    \coordinate (c3/2) at ({-1/2*3}, {-sqrt(3)/2*3}, -6);
    
    \coordinate (ab) at ({1/4*3}, {sqrt(3)/4*3}, 0);
    \coordinate (bc) at ({-1/2*3}, 0, 0);
    \coordinate (ca) at ({1/4*3}, {-sqrt(3)/4*3}, 0);

    \coordinate (3ab) at ({1/4*3*3}, {sqrt(3)/4*3*3}, 0);
    \coordinate (3bc) at ({-1/2*3*3}, 0, 0);
    \coordinate (3ca) at ({1/4*3*3}, {-sqrt(3)/4*3*3}, 0);

    \coordinate (ab1) at ({1/4*3}, {sqrt(3)/4*3}, -4);
    \coordinate (bc1) at ({-1/2*3}, 0, -4);
    \coordinate (ca1) at ({1/4*3}, {-sqrt(3)/4*3}, -4);

    \coordinate (ab2) at ({1/4*3*3}, {sqrt(3)/4*3*3}, -6);
    \coordinate (ab3/2) at ({1/4*3}, {sqrt(3)/4*3}, -6);

    \coordinate (bc2) at ({-1/2*3*3}, 0, -6);
    \coordinate (bc3/2) at ({-1/2*3}, 0, -6);

    \coordinate (ca2) at ({1/4*3*3}, {-sqrt(3)/4*3*3}, -6);
    \coordinate (ca3/2) at ({1/4*3}, {-sqrt(3)/4*3}, -6);

    \coordinate (3a) at ({3*3}, 0, 0);
    \coordinate (3b) at ({-1/2*3*3}, {sqrt(3)/2*3*3}, 0);
    \coordinate (3c) at ({-1/2*3*3}, {-sqrt(3)/2*3*3}, 0);

    \fill[lightgray] (3a) -- (3b) -- (3c) -- cycle;

    \begin{scope}[thick,decoration={
    markings,
    mark=at position 0.5 with {\arrow{>}}}
    ]
    \draw[postaction={decorate}, orange, very thick] (3ab) -- (o);
    \draw[postaction={decorate}, orange, very thick] (3bc) -- (o);
    \draw[postaction={decorate}, orange, very thick] (3ca) -- (o);
    \end{scope}

\end{scope}
\end{tikzpicture}

%% file: figures/module_structureRHS.tex
\tdplotsetmaincoords{60}{80}
\begin{tikzpicture}[tdplot_main_coords]
\begin{scope}[scale = 0.5, tdplot_main_coords]
    \coordinate (o) at (0, 0, 0);
    \coordinate (a) at (3, 0, 0);
    \coordinate (b) at ({-1/2*3}, {sqrt(3)/2*3}, 0);
    \coordinate (c) at ({-1/2*3}, {-sqrt(3)/2*3}, 0);
    
    \coordinate (a1) at (1*3, 0, -4);
    \coordinate (b1) at ({-1/2*3}, {sqrt(3)/2*3}, -4);
    \coordinate (c1) at ({-1/2*3}, {-sqrt(3)/2*3}, -4);

    \coordinate (a2) at (1*3*2, 0, -6);
    \coordinate (a3/2) at (1*3, 0, -6);

    \coordinate (b2) at ({-1/2*3*2}, {sqrt(3)/2*3*2}, -6);
    \coordinate (b3/2) at ({-1/2*3}, {sqrt(3)/2*3}, -6);

    \coordinate (c2) at ({-1/2*3*2}, {-sqrt(3)/2*3*2}, -6);
    \coordinate (c3/2) at ({-1/2*3}, {-sqrt(3)/2*3}, -6);
    
    \coordinate (ab) at ({1/4*3}, {sqrt(3)/4*3}, 0);
    \coordinate (bc) at ({-1/2*3}, 0, 0);
    \coordinate (ca) at ({1/4*3}, {-sqrt(3)/4*3}, 0);

    \coordinate (ab1) at ({1/4*3}, {sqrt(3)/4*3}, -4);
    \coordinate (bc1) at ({-1/2*3}, 0, -4);
    \coordinate (ca1) at ({1/4*3}, {-sqrt(3)/4*3}, -4);

    \coordinate (ab2) at ({1/4*3*3}, {sqrt(3)/4*3*3}, -6);
    \coordinate (ab3/2) at ({1/4*3}, {sqrt(3)/4*3}, -6);

    \coordinate (bc2) at ({-1/2*3*3}, 0, -6);
    \coordinate (bc3/2) at ({-1/2*3}, 0, -6);

    \coordinate (ca2) at ({1/4*3*3}, {-sqrt(3)/4*3*3}, -6);
    \coordinate (ca3/2) at ({1/4*3}, {-sqrt(3)/4*3}, -6);

    \coordinate (3a) at ({3*3}, 0, -6);
    \coordinate (3b) at ({-1/2*3*3}, {sqrt(3)/2*3*3}, -6);
    \coordinate (3c) at ({-1/2*3*3}, {-sqrt(3)/2*3*3}, -6);
    \fill[lightgray, opacity=0.4] (a) -- (b) -- (c) -- cycle;
    \fill[lightgray, opacity=0.9] (a) -- (a1) .. controls (a3/2) and (a3/2) .. (a2) -- (3a) -- (3b) -- (b2) .. controls (b3/2) and (b3/2) .. (b1) -- (b) -- cycle;
    \fill[lightgray, opacity=0.8] (b) -- (b1) -- (c1) -- (c) -- cycle;
    \fill[lightgray, opacity=0.8] (b1) .. controls (b3/2) and (b3/2) .. (b2) -- (3b) -- (3c) -- (c2) .. controls (c3/2) and (c3/2) .. (c1) -- cycle;
    \fill[lightgray, opacity=0.9] (c) -- (c1) .. controls (c3/2) and (c3/2) .. (c2) -- (3c) -- (3a) -- (a2) .. controls (a3/2) and (a3/2) .. (a1) -- (a) -- cycle;

    \draw[orange, very thick] (ab) -- (ab1);
    \draw[orange, very thick, dashed] (bc) -- (bc1);
    \draw[orange, very thick] (ca) -- (ca1);

    \draw[orange, very thick, <-] (ab1) .. controls (ab3/2) and (ab3/2) .. (ab2);
    \draw[orange, very thick, dashed, <-] (bc1) .. controls (bc3/2) and (bc3/2) .. (bc2);
    \draw[orange, very thick, <-] (ca1) .. controls (ca3/2) and (ca3/2) .. (ca2);
    
    \draw[very thick] (a) -- (b) -- (c) -- cycle;

    \draw[very thick] (a) -- (a1);
    \draw[very thick] (b) -- (b1);
    \draw[very thick] (c) -- (c1);

    \draw[very thick] (a1) .. controls (a3/2) and (a3/2) .. (a2);
    \draw[very thick] (b1) .. controls (b3/2) and (b3/2) .. (b2);
    \draw[very thick] (c1) .. controls (c3/2) and (c3/2) .. (c2);

\end{scope}
\end{tikzpicture}

%% file: figures/triangular_pillow.tex
\tdplotsetmaincoords{25}{60}
\begin{tikzpicture}[tdplot_main_coords]
\begin{scope}[scale = 0.8, tdplot_main_coords]
    \coordinate (o) at (0, 0, 1);
    \coordinate (p) at (0, 0, -1);
    \coordinate (a) at (0, 0, 4);
    \coordinate (a1) at (0, 0, 0);
    
    \coordinate (b) at ({2*sqrt(2)}, 0, 0);
    \coordinate (b1) at ({-2*sqrt(2)/3}, 0, 4/3);
    
    \coordinate (c) at ({-sqrt(2)}, {sqrt(6)}, 0);
    \coordinate (c1) at ({sqrt(2)/3}, {-sqrt(6)/3}, 4/3);
    
    \coordinate (d) at ({-sqrt(2)}, {-sqrt(6)}, 0);
    \coordinate (d1) at ({sqrt(2)/3}, {sqrt(6)/3}, 4/3);
    
    \coordinate (ab) at ({sqrt(2)}, 0, 2);
    \coordinate (ac) at ({-sqrt(2)/2}, {sqrt(6)/2}, 2);
    \coordinate (ad) at ({-sqrt(2)/2}, {-sqrt(6)/2}, 2);
    \coordinate (bc) at ({sqrt(2)/2}, {sqrt(6)/2}, 0);
    \coordinate (bd) at ({sqrt(2)/2}, {-sqrt(6)/2}, 0);
    \coordinate (cd) at ({-sqrt(2)}, 0, 0);

    \draw[white, ultra thick] (o) -- (a1);
    
    \begin{scope}[thick,decoration={
    markings,
    mark=at position 0.5 with {\arrow{>}}}
    ] 
    \draw[postaction={decorate}, orange] (bd) -- (o);
    \draw[postaction={decorate}, orange] (bc) -- (o);
    \draw[postaction={decorate}, orange] (cd) -- (o);
    \draw[postaction={decorate}, dashed, orange] (bd) -- (p);
    \draw[postaction={decorate}, dashed, orange] (bc) -- (p);
    \draw[postaction={decorate}, dashed, orange] (cd) -- (p);
    \end{scope}

    \draw[very thick] (b) -- (c);
    \draw[very thick] (b) -- (d);
    \draw[very thick] (c) -- (d);
    
    \filldraw[orange] (o) circle (0.05em);
    \filldraw[orange] (p) circle (0.05em);
    \filldraw (b) circle (0.05em);
    \filldraw (c) circle (0.05em);
    \filldraw (d) circle (0.05em);

    \node[anchor=north east] at (bd) {$a$};
    \node[anchor=west] at (bc) {$b$};
    \node[anchor=south] at (cd) {$c$};
    \node[anchor=south] at (c) {$\alpha$};
    \node[anchor=east] at (d) {$\beta$};
    \node[anchor=north] at (b) {$\gamma$};
    
\end{scope}
\end{tikzpicture}

%% file: figures/faceSuspensionLeafSpaceWithAngles.tex
\tdplotsetmaincoords{70}{130}
\begin{tikzpicture}[tdplot_main_coords,
decoration={markings,mark=at position 0.75 with {\arrow[scale=.6]{>}}}]
    \fill[color=lightgray, opacity=0.3] (0,0,0) -- ({2.5*cos(240)},{2.5*sin(240)},1.5) -- (0,0,3) -- cycle;
    \draw[thick,orange] ({2.5*cos(240)},{2.5*sin(240)},1.5) -- ($({2.5*cos(240)},{2.5*sin(240)},1.5)!.13!(0,0,0)$);
    \draw[thick,orange,postaction=decorate] ($({2.5*cos(240)},{2.5*sin(240)},1.5)!.17!(0,0,0)$) -- coordinate[pos=.3] (b2) (0,0,0);   
    \draw[thick,orange,postaction=decorate] ({2.5*cos(240)},{2.5*sin(240)},1.5) -- coordinate[midway] (b1) (0,0,3);
    \node[above,scale=.75] at (b1) {$b_S$};
    \node[below,scale=.75] at (b2) {$b_T$};
    \draw[thick, gray] (0,0,0)--(0,0,3);
    \fill[color=lightgray, opacity=0.3] (0,0,0) -- (2.5,0,1.5) -- (0,0,3)--cycle;
    \fill[color=lightgray, opacity=0.3] (0,0,0) -- ({2.5*cos(120)},{2.5*sin(120)},1.5) -- (0,0,3) -- cycle;
    \draw[thick,orange,postaction=decorate] ({2.5*cos(0)},{2.5*sin(0)},1.5) -- coordinate[midway] (c2) (0,0,0);
    \draw[thick,orange,postaction=decorate] ({2.5*cos(0)},{2.5*sin(0)},1.5) -- coordinate[midway] (c1) (0,0,3);
    \node[left,scale=.75] at (c1) {$c_S$};
    \node[below,scale=.75] at (c2) {$c_T$};
       
    \tdplotsetthetaplanecoords{120}
    \begin{scope}[tdplot_rotated_coords] 
        \draw[] (1.5,1.55,0) arc (-90:-120:.95) coordinate[midway] (angleLabel);
        \draw[] (2.3,0,0) arc (180:120:.7) node[pos=.4,above,scale=.5]{$\frac{\pi-\theta_{a_S}}{2}$};
    \end{scope}
    \draw[densely dashed] ({2.5*cos(120)},{2.5*sin(120)},1.5)--({1.25*cos(120)},{1.25*sin(120)},1.5);
    \tdplotsetthetaplanecoords{120}
    \begin{scope}[tdplot_rotated_coords] 
        \draw[] (1.5,1.4,0) arc (-90:-60:1.1) node[midway,right,scale=.5]{$\frac{\theta_{a_S}}{2}$};
    \end{scope}
    \draw[thick,orange,postaction=decorate] ({2.5*cos(120)},{2.5*sin(120)},1.5) -- (0,0,0) node[black,pos=.45,below,scale=.75]{$a_T$};
    \draw[thick,orange,postaction=decorate] ({2.5*cos(120)},{2.5*sin(120)},1.5) -- (0,0,3)
    node[black,pos=.45,above,scale=.75]{$a_S$};
    \node[scale=.5,right] at (angleLabel) {$\frac{\theta_{a_T}}{2}$};
    \filldraw[orange] (0,0,3) circle (.3pt);
    \filldraw[orange] (0,0,0) circle (.3pt);
    \filldraw[orange] ({2.5*cos(240)},{2.5*sin(240)},1.5) circle (.2pt);
    \filldraw[orange] ({2.5*cos(120)},{2.5*sin(120)},1.5) circle (.2pt);
    \filldraw[orange] ({2.5*cos(0)},{2.5*sin(0)},1.5) circle (.2pt);
\end{tikzpicture}

%% file: figures/fig8_tet1.tex
\tdplotsetmaincoords{60}{80}
\begin{tikzpicture}[tdplot_main_coords]
\begin{scope}[scale = 0.8, tdplot_main_coords]
    \coordinate (o) at (0, 0, 0);
    \coordinate (a) at (0, 0, 3);
    \coordinate (a1) at (0, 0, -1);
    
    \coordinate (b) at ({2*sqrt(2)}, 0, -1);
    \coordinate (b1) at ({-2*sqrt(2)/3}, 0, 1/3);
    
    \coordinate (c) at ({-sqrt(2)}, {sqrt(6)}, -1);
    \coordinate (c1) at ({sqrt(2)/3}, {-sqrt(6)/3}, 1/3);
    
    \coordinate (d) at ({-sqrt(2)}, {-sqrt(6)}, -1);
    \coordinate (d1) at ({sqrt(2)/3}, {sqrt(6)/3}, 1/3);
    
    \coordinate (ab) at ({sqrt(2)}, 0, 1);
    \coordinate (ac) at ({-sqrt(2)/2}, {sqrt(6)/2}, 1);
    \coordinate (ad) at ({-sqrt(2)/2}, {-sqrt(6)/2}, 1);
    \coordinate (bc) at ({sqrt(2)/2}, {sqrt(6)/2}, -1);
    \coordinate (bd) at ({sqrt(2)/2}, {-sqrt(6)/2}, -1);
    \coordinate (cd) at ({-sqrt(2)}, 0, -1);
    
    \draw[very thick, dashed] (c) -- (d);

    \begin{scope}[thick,decoration={
    markings,
    mark=at position .4 with \arrow{stealth}}
    ] 
    \draw[postaction={decorate},very thick] (d) -- (c);
    \end{scope}
    \begin{scope}[thick,decoration={
    markings,
    mark=at position 0.5 with {\arrow{>}}}
    ] 
    \draw[postaction={decorate}, dashed, orange] (cd) -- (a1);
    \draw[postaction={decorate}, dashed, orange] (bd) -- (a1);
    \draw[postaction={decorate}, dashed, orange] (bc) -- (a1);
    \draw[postaction={decorate}, dashed, orange] (cd) -- (b1);
    \draw[postaction={decorate}, dashed, orange] (ad) -- (b1);
    \draw[postaction={decorate}, dashed, orange] (ac) -- (b1);
    \end{scope}
    \coordinate (km1) at ($(cd)!0.5!(b1)$); 
    \coordinate (km2) at ($(cd)!0.5!(a1)$); 
    \coordinate (km1w) at ($(cd)!0.57!(b1)$); 
    \coordinate (km2w) at ($(cd)!0.57!(a1)$);
    \coordinate (kb1) at ($(bd)!0.5!(a1)$); 
    \coordinate (kb2) at ($(ab)!0.7!(c1)$); 
    \coordinate (kb1w) at ($(bd)!0.43!(a1)$); 
    \coordinate (kb2w) at ($(ab)!0.77!(c1)$);
    \begin{scope}[decoration={
    markings,
    mark=at position .5 with \arrow[very thick]{>}}]
    \draw[thick,white] (kb2) .. controls (.75,-.85,0) .. coordinate[pos=.9] (end) (kb1);
    \draw[line width=2.5,white] (kb2) .. controls (.75,-.85,0) ..(end); 
    \draw[line width=2.25,blue,postaction = {decorate}] (kb2) .. controls (.75,-.85,0) .. (kb1); 
    \end{scope}
    
    \begin{scope}[thick,decoration={
    markings,
    mark=at position 0.5 with {\arrow{>}}}
    ] 
    \draw[white,ultra thick] (ab) -- (c1);
    \draw[postaction={decorate}, orange] (bd) -- (c1);
    \draw[postaction={decorate}, orange] (ad) -- (c1);
    \draw[postaction={decorate}, orange] (ab) -- (c1);
    \draw[postaction={decorate}, orange] (ac) -- (d1);
    \draw[postaction={decorate}, orange] (bc) -- (d1);
    \draw[postaction={decorate}, orange] (ab) -- (d1);
    \end{scope}

    \draw[white, ultra thick] (a) -- (b);
   
    \begin{scope}[thick,decoration={
    markings,
    mark=between positions 0.47 and .53 step 0.06 with \arrow{stealth}}
    ] 
    \draw[postaction={decorate},very thick] (b) -- (c);
    \draw[postaction={decorate},very thick] (b) -- (d);
    \draw[postaction={decorate},very thick] (a) -- (d);
    \end{scope}

    \begin{scope}[thick,decoration={
    markings,
    mark=at position .5 with \arrow{stealth}}
    ] 
    \draw[postaction={decorate},very thick] (a) -- (b);
    \draw[postaction={decorate},very thick] (a) -- (c);
    \end{scope}
    
    \filldraw (a) circle (0.05em);
    \filldraw (b) circle (0.05em);
    \filldraw (c) circle (0.05em);
    \filldraw (d) circle (0.05em);

    \node[above] at (a) {$N$};
    \node[below] at (b) {$E$};
    \node[right] at (c) {$S$};
    \node[left] at (d) {$W$};
    \node[above right] at (ac) {$\hat{y}''$};
    \node[right] at (ab) {$\hat{y}'$};
    \node[above left] at (ad) {$\hat{y}$};
    \node[right] at (bc) {$\hat{y}$};
    \node[below left] at (bd) {$\hat{y}''$};
    \node[below] at (cd) {$\hat{y}'$};
\end{scope}
\end{tikzpicture}

%% file: figures/fig8_tet2.tex
\tdplotsetmaincoords{60}{80}
\begin{tikzpicture}[tdplot_main_coords]
\begin{scope}[scale = 0.8, tdplot_main_coords]
    \coordinate (o) at (0, 0, 0);
    \coordinate (a) at (0, 0, 3);
    \coordinate (a1) at (0, 0, -1);
    
    \coordinate (b) at ({2*sqrt(2)}, 0, -1);
    \coordinate (b1) at ({-2*sqrt(2)/3}, 0, 1/3);
    
    \coordinate (c) at ({-sqrt(2)}, {sqrt(6)}, -1);
    \coordinate (c1) at ({sqrt(2)/3}, {-sqrt(6)/3}, 1/3);
    
    \coordinate (d) at ({-sqrt(2)}, {-sqrt(6)}, -1);
    \coordinate (d1) at ({sqrt(2)/3}, {sqrt(6)/3}, 1/3);
    
    \coordinate (ab) at ({sqrt(2)}, 0, 1);
    \coordinate (ac) at ({-sqrt(2)/2}, {sqrt(6)/2}, 1);
    \coordinate (ad) at ({-sqrt(2)/2}, {-sqrt(6)/2}, 1);
    \coordinate (bc) at ({sqrt(2)/2}, {sqrt(6)/2}, -1);
    \coordinate (bd) at ({sqrt(2)/2}, {-sqrt(6)/2}, -1);
    \coordinate (cd) at ({-sqrt(2)}, 0, -1);
    
    \draw[very thick, dashed] (c) -- (d);

    \begin{scope}[thick,decoration={
    markings,
    mark=at position .6 with \arrow{stealth}}
    ] 
    \draw[postaction={decorate},very thick] (c) -- (d);
    \end{scope}
    \begin{scope}[thick,decoration={
    markings,
    mark=at position 0.5 with {\arrow{>}}}
    ] 
    \draw[postaction={decorate}, dashed, orange] (cd) -- (a1);
    \draw[postaction={decorate}, dashed, orange] (bd) -- (a1);
    \draw[postaction={decorate}, dashed, orange] (bc) -- (a1);
    \draw[postaction={decorate}, dashed, orange] (cd) -- (b1);
    \draw[postaction={decorate}, dashed, orange] (ad) -- (b1);
    \draw[postaction={decorate}, dashed, orange] (ac) -- (b1);
    \end{scope}
    \coordinate (km1) at ($(ac)!0.5!(d1)$); 
    \coordinate (km2) at ($(cd)!0.5!(a1)$); 
    \coordinate (km1w) at ($(ac)!0.57!(d1)$); 
    \coordinate (km2w) at ($(cd)!0.57!(a1)$); 
    \coordinate (kb1) at ($(bc)!0.5!(a1)$); 
    \coordinate (kb2) at ($(ab)!0.7!(c1)$); 
    \coordinate (kb1w) at ($(bc)!0.57!(a1)$); 
    \coordinate (kb2w) at ($(ab)!0.77!(c1)$); 

    \begin{scope}[decoration={
    markings,
    mark=at position .4 with \arrow[very thick]{<}}]
    \draw[thick,white] (kb2) .. controls (.1,-.5,-.6) and (.3,.75,0) .. coordinate[pos=.9] (end) (kb1); 
    \draw[line width=2.5,white] (kb2) .. controls (.1,-.5,-.6) and (.3,.75,0) .. (end); 
    \draw[line width=2.25,blue,postaction = {decorate}] (kb2) .. controls (.1,-.5,-.6) and (.3,.75,0) .. (kb1); 
    \end{scope}
    
    \begin{scope}[thick,decoration={
    markings,
    mark=at position 0.5 with {\arrow{>}}}
    ] 
    \draw[white,very thick] (ab) -- (c1);
    \draw[postaction={decorate}, orange] (bd) -- (c1);
    \draw[postaction={decorate}, orange] (ad) -- (c1);
    \draw[postaction={decorate}, orange] (ab) -- (c1);
    \draw[postaction={decorate}, orange] (ac) -- (d1);
    \draw[postaction={decorate}, orange] (bc) -- (d1);
    \draw[postaction={decorate}, orange] (ab) -- (d1);
    \end{scope}

    \draw[white, ultra thick] (a) -- (b);
   
    \begin{scope}[thick,decoration={
    markings,
    mark=between positions 0.47 and .53 step 0.06 with \arrow{stealth}}
    ] 
    \draw[postaction={decorate},very thick] (b) -- (c);
    \draw[postaction={decorate},very thick] (b) -- (d);
    \draw[postaction={decorate},very thick] (a) -- (d);
    \end{scope}

    \begin{scope}[thick,decoration={
    markings,
    mark=at position .5 with \arrow{stealth}}
    ] 
    \draw[postaction={decorate},very thick] (b) -- (a);
    \draw[postaction={decorate},very thick] (c) -- (a);
    \end{scope}
    
    \filldraw (a) circle (0.05em);
    \filldraw (b) circle (0.05em);
    \filldraw (c) circle (0.05em);
    \filldraw (d) circle (0.05em);

    \node[above] at (a) {$N$};
    \node[below] at (b) {$W$};
    \node[right] at (c) {$S$};
    \node[left] at (d) {$E$};
    \node[above right] at (ac) {$\hat{z}''$};
    \node[right] at (ab) {$\hat{z}'$};
    \node[above left] at (ad) {$\hat{z}$};
    \node[right] at (bc) {$\hat{z}$};
    \node[below left] at (bd) {$\hat{z}''$};
    \node[below] at (cd) {$\hat{z}'$};
\end{scope}
\end{tikzpicture}

%% file: figures/kb_face_susp1.tex
\tdplotsetmaincoords{45}{60}

\begin{tikzpicture}[tdplot_main_coords]

\begin{scope}[tdplot_main_coords]
\coordinate (N) at (0, 0, 1);
\coordinate (S) at (0, 0, -1);  
\coordinate (b) at ({2*sqrt(2)*cos(0)}, {2*sqrt(2)*sin(0)}, 0);
\coordinate (c) at ({2*sqrt(2)*cos(120)}, {2*sqrt(2)*sin(120)}, 0);
\coordinate (d) at ({2*sqrt(2)*cos(240)}, {2*sqrt(2)*sin(240)}, 0);   
\coordinate (bc) at ($1/2*(b) + 1/2*(c)$);
\coordinate (bd) at ($1/2*(b) + 1/2*(d)$);
\coordinate (cd) at ($1/2*(c) + 1/2*(d)$);
\coordinate (kb1) at ($(bd)!.6!(N)$);
\coordinate (kb2) at ($(bc)!.6!(N)$);
\coordinate (kb3) at ($(bc)!.4!(N)$);
\coordinate (kb4) at ($(bc)!.4!(S)$);   
\begin{scope}[thick,decoration={
markings,
mark=at position 0.5 with {\arrow{>}}}
] 
\draw[postaction={decorate}, dashed, orange] (bc) -- (S);
\draw[postaction={decorate},blue] (kb1) .. controls ({1/2*cos(0)},{1/2*sin(0)},0) .. (kb4);
\draw[postaction={decorate}, orange] (bd) -- (N);
\draw[postaction={decorate}, orange] (bc) -- (N);
\draw[postaction={decorate}, orange] (cd) -- (N);
\draw[postaction={decorate}, dashed, orange] (bd) -- (S);
\draw[postaction={decorate}, dashed, orange] (cd) -- (S);
\end{scope}
\begin{scope}[decoration={
markings,
mark=between positions 0.6 and .66 step 0.06 with \arrow{stealth}},very thick]
\draw[postaction={decorate}] (b)--(c);
\draw[postaction={decorate}] (b)--(d);
\end{scope}
\begin{scope}[decoration={
markings,
mark=at position 0.6 with \arrow{stealth}},very thick]
\draw[postaction={decorate}] (d)--(c);
\end{scope}    
\filldraw[orange] (N) circle (0.05em);
\filldraw[orange] (S) circle (0.05em);
\filldraw (b) circle (0.05em);
\filldraw (c) circle (0.05em);
\filldraw (d) circle (0.05em);
\node[right,scale=.9] at (kb1) {$\epsilon_1$};
\node[above,scale=.9] at (kb4) {$\epsilon_2$};   
\node[left] at ($(bd)!.3!(N)$) {$y''$};
\node[above] at ($(bc)!.3!(N)$) {$y$};
\node[below] at ($(bc)!.3!(S)$) {$z''$};
\node[left] at ($(cd)+(0,0,1)$) {$N$};
\node[] at (3,0,0) {};
\end{scope}
\end{tikzpicture}

%% file: figures/kb_face_susp2.tex
\tdplotsetmaincoords{45}{60}

\begin{tikzpicture}[tdplot_main_coords]

\begin{scope}[tdplot_main_coords]
\coordinate (N) at (0, 0, 1);
\coordinate (S) at (0, 0, -1);  
\coordinate (b) at ({2*sqrt(2)*cos(0)}, {2*sqrt(2)*sin(0)}, 0);
\coordinate (c) at ({2*sqrt(2)*cos(120)}, {2*sqrt(2)*sin(120)}, 0);
\coordinate (d) at ({2*sqrt(2)*cos(240)}, {2*sqrt(2)*sin(240)}, 0);   
\coordinate (bc) at ($1/2*(b) + 1/2*(c)$);
\coordinate (bd) at ($1/2*(b) + 1/2*(d)$);
\coordinate (cd) at ($1/2*(c) + 1/2*(d)$);
\coordinate (kb1) at ($(bd)!.6!(N)$);
\coordinate (kb2) at ($(bc)!.6!(N)$);
\coordinate (kb3) at ($(bc)!.4!(N)$);
\coordinate (kb4) at ($(bc)!.4!(S)$);   
\begin{scope}[thick,decoration={
markings,
mark=at position 0.5 with {\arrow{>}}}
] 
\draw[postaction={decorate}, dashed, orange] (bc) -- (S);
\draw[postaction={decorate},blue] (kb1) .. controls ({1/2*cos(0)},{1/2*sin(0)},0) .. (kb4);
\draw[postaction={decorate}, orange] (bd) -- (N);
\draw[postaction={decorate}, orange] (bc) -- (N);
\draw[postaction={decorate}, orange] (cd) -- (N);
\draw[postaction={decorate}, dashed, orange] (bd) -- (S);
\draw[postaction={decorate}, dashed, orange] (cd) -- (S);
\end{scope}
\begin{scope}[decoration={
markings,
mark=between positions 0.6 and .66 step 0.06 with \arrow{stealth}},very thick]
\draw[postaction={decorate}] (c)--(b);
\draw[postaction={decorate}] (d)--(b);
\end{scope}
\begin{scope}[decoration={
markings,
mark=at position 0.6 with \arrow{stealth}},very thick]
\draw[postaction={decorate}] (d)--(c);
\end{scope}    
\filldraw[orange] (N) circle (0.05em);
\filldraw[orange] (S) circle (0.05em);
\filldraw (b) circle (0.05em);
\filldraw (c) circle (0.05em);
\filldraw (d) circle (0.05em);
\node[right,scale=.9] at (kb1) {$\epsilon_2$};
\node[above,scale=.9] at (kb4) {$\epsilon_1$}; 
\node[left] at ($(bd)!.3!(N)$) {$z''$};
\node[above] at ($(bc)!.3!(N)$) {$z$};
\node[below] at ($(bc)!.3!(S)$) {$y''$};
\node[] at (3,0,0) {};
\node[left] at ($(cd)+(0,0,1)$) {$S$};
\end{scope}
\end{tikzpicture}

%% file: figures/kb_face_susp1_reduced.tex
\tdplotsetmaincoords{45}{60}

\begin{tikzpicture}[tdplot_main_coords]

\begin{scope}[tdplot_main_coords]
\coordinate (N) at (0, 0, 1);
\coordinate (S) at (0, 0, -1);  
\coordinate (b) at ({2*sqrt(2)*cos(0)}, {2*sqrt(2)*sin(0)}, 0);
\coordinate (c) at ({2*sqrt(2)*cos(120)}, {2*sqrt(2)*sin(120)}, 0);
\coordinate (d) at ({2*sqrt(2)*cos(240)}, {2*sqrt(2)*sin(240)}, 0);   
\coordinate (bc) at ($1/2*(b) + 1/2*(c)$);
\coordinate (bd) at ($1/2*(b) + 1/2*(d)$);
\coordinate (cd) at ($1/2*(c) + 1/2*(d)$);
\coordinate (kb1) at ($(bd)!.6!(N)$);
\coordinate (kb2) at ($(bc)!.6!(N)$);
\coordinate (kb3) at ($(bc)!.4!(N)$);
\coordinate (kb4) at ($(bc)!.4!(S)$);   
\begin{scope}[thick,decoration={
markings,
mark=at position 0.5 with {\arrow{>}}}
] 
\draw[postaction={decorate},blue] (kb1) .. controls (0,0,.5) .. (kb2);
\draw[postaction={decorate}, dashed, orange] (bc) -- (S);
\draw[postaction={decorate},blue] (kb3) .. controls ({3/4*cos(60)},{3/4*sin(60)},0) .. (kb4);
\draw[postaction={decorate}, orange] (bd) -- (N);
\draw[postaction={decorate}, orange] (bc) -- (N);
\draw[postaction={decorate}, orange] (cd) -- (N);
\draw[postaction={decorate}, dashed, orange] (bd) -- (S);
\draw[postaction={decorate}, dashed, orange] (cd) -- (S);
\end{scope}
\begin{scope}[decoration={
markings,
mark=between positions 0.6 and .66 step 0.06 with \arrow{stealth}},very thick]
\draw[postaction={decorate}] (b)--(c);
\draw[postaction={decorate}] (b)--(d);
\end{scope}
\begin{scope}[decoration={
markings,
mark=at position 0.6 with \arrow{stealth}},very thick]
\draw[postaction={decorate}] (d)--(c);
\end{scope}    
\filldraw[orange] (N) circle (0.05em);
\filldraw[orange] (S) circle (0.05em);
\filldraw (b) circle (0.05em);
\filldraw (c) circle (0.05em);
\filldraw (d) circle (0.05em);
\node[right,scale=.75] at (kb1) {$\epsilon_1$};
\node[above,scale=.75] at (kb2) {$-\epsilon_2$};
\node[right,scale=.75] at (kb3) {$\epsilon_2$};
\node[below right,scale=.75] at (kb4) {$\epsilon_2$};    
\end{scope}
\end{tikzpicture}

%% file: figures/kb_face_susp2_reduced.tex
\tdplotsetmaincoords{45}{60}

\begin{tikzpicture}[tdplot_main_coords]

\begin{scope}[tdplot_main_coords]
\coordinate (N) at (0, 0, 1);
\coordinate (S) at (0, 0, -1);  
\coordinate (b) at ({2*sqrt(2)*cos(0)}, {2*sqrt(2)*sin(0)}, 0);
\coordinate (c) at ({2*sqrt(2)*cos(120)}, {2*sqrt(2)*sin(120)}, 0);
\coordinate (d) at ({2*sqrt(2)*cos(240)}, {2*sqrt(2)*sin(240)}, 0);   
\coordinate (bc) at ($1/2*(b) + 1/2*(c)$);
\coordinate (bd) at ($1/2*(b) + 1/2*(d)$);
\coordinate (cd) at ($1/2*(c) + 1/2*(d)$);
\coordinate (kb1) at ($(bd)!.6!(N)$);
\coordinate (kb2) at ($(bc)!.6!(N)$);
\coordinate (kb3) at ($(bc)!.4!(N)$);
\coordinate (kb4) at ($(bc)!.4!(S)$);   
\begin{scope}[thick,decoration={
markings,
mark=at position 0.5 with {\arrow{>}}}
] 
\draw[postaction={decorate},blue] (kb1) .. controls (0,0,.5) .. (kb2);
\draw[postaction={decorate}, dashed, orange] (bc) -- (S);
\draw[postaction={decorate},blue] (kb3) .. controls ({3/4*cos(60)},{3/4*sin(60)},0) .. (kb4);
\draw[postaction={decorate}, orange] (bd) -- (N);
\draw[postaction={decorate}, orange] (bc) -- (N);
\draw[postaction={decorate}, orange] (cd) -- (N);
\draw[postaction={decorate}, dashed, orange] (bd) -- (S);
\draw[postaction={decorate}, dashed, orange] (cd) -- (S);
\end{scope}
\begin{scope}[decoration={
markings,
mark=between positions 0.6 and .66 step 0.06 with \arrow{stealth}},very thick]
\draw[postaction={decorate}] (b)--(c);
\draw[postaction={decorate}] (b)--(d);
\end{scope}
\begin{scope}[decoration={
markings,
mark=at position 0.6 with \arrow{stealth}},very thick]
\draw[postaction={decorate}] (d)--(c);
\end{scope}    
\filldraw[orange] (N) circle (0.05em);
\filldraw[orange] (S) circle (0.05em);
\filldraw (b) circle (0.05em);
\filldraw (c) circle (0.05em);
\filldraw (d) circle (0.05em);
\node[right,scale=.75] at (kb1) {$\epsilon_2$};
\node[above,scale=.75] at (kb2) {$-\epsilon_1$};
\node[right,scale=.75] at (kb3) {$\epsilon_1$};
\node[below right,scale=.75] at (kb4) {$\epsilon_1$};    
\end{scope}
\end{tikzpicture}

%% file: figures/lift_with_detour.tex
\tdplotsetmaincoords{45}{-90}
\begin{tikzpicture}[tdplot_main_coords]
\pgfmathsetmacro{\nodeScale}{.75}
\begin{scope}[tdplot_main_coords]
\coordinate (N) at (0, 0, 1);
\coordinate (S) at (0, 0, -1);   
\coordinate (a1) at ({2*sqrt(2)*cos(0)}, {2*sqrt(2)*sin(0)}, 0);
\coordinate (b2) at ({2*sqrt(2)*cos(60)}, {2*sqrt(2)*sin(60)}, 0);
\coordinate (c1) at ({2*sqrt(2)*cos(120)}, {2*sqrt(2)*sin(120)}, 0);
\coordinate (a2) at ({2*sqrt(2)*cos(180)}, {2*sqrt(2)*sin(180)}, 0);
\coordinate (b1) at ({2*sqrt(2)*cos(240)}, {2*sqrt(2)*sin(240)}, 0);
\coordinate (c2) at ({2*sqrt(2)*cos(300)}, {2*sqrt(2)*sin(300)}, 0);    
\coordinate (a1b2) at ($1/2*(a1) + 1/2*(b2)$);
\coordinate (b2c1) at ($1/2*(b2) + 1/2*(c1)$);
\coordinate (c1a2) at ($1/2*(c1) + 1/2*(a2)$);
\coordinate (a2b1) at ($1/2*(a2) + 1/2*(b1)$);
\coordinate (b1c2) at ($1/2*(b1) + 1/2*(c2)$);
\coordinate (c2a1) at ($1/2*(c2) + 1/2*(a1)$);
\coordinate (kb1) at ($(b1c2)!.6!(N)$);
\coordinate (kb2) at ($(c1a2)!.6!(N)$);    
\draw[very thick, red] (N) -- (S);    
\begin{scope}[thick,decoration={
markings,
mark=at position 0.5 with {\arrow[scale=.75]{>}}}
] 
\draw[postaction={decorate}, dashed, orange] (c2a1) -- (S) coordinate[pos=.25] (bt);
\draw[postaction={decorate}, dashed, orange] (a1b2) -- (S) coordinate[pos=.25] (ct*);
\draw[postaction={decorate}, dashed, orange] (b2c1) -- (S)
coordinate[pos=.25] (at);
\draw[postaction={decorate}, dashed, orange] (c1a2) -- (S) coordinate[pos=.25] (bt*);
\draw[postaction={decorate}, dashed, orange] (a2b1) -- (S) coordinate[pos=.25] (ct);
\draw[postaction={decorate}, dashed, orange] (b1c2) -- (S) coordinate[pos=.25] (at*);
\end{scope}
\begin{scope}[thick,decoration={
markings,
mark=at position 0.6 with {\arrow[scale=.75]{>}}}]
\draw[white,very thick] (kb1) .. controls ({1/2*cos(210)},{1/2*sin(210)},.4) ..(kb2);
\draw[blue, ultra thick,postaction={decorate}] (kb1) .. controls ({1/2*cos(210)},{1/2*sin(210)},.4) ..(kb2);
\end{scope}
\begin{scope}[thick,decoration={
markings,
mark=at position 0.5 with {\arrow[scale=.75]{>}}}
] 
\draw[postaction={decorate}, white,very thick] ($(a2b1)!.25!(N)$) -- ($(a2b1)!.75!(N)$) coordinate[pos=.25] (cs*);
\draw[postaction={decorate}, orange] (a2b1) -- (N) coordinate[pos=.25] (cs*);
\draw[postaction={decorate}, orange] (b1c2) -- (N)coordinate[pos=.25] (as);
\draw[postaction={decorate}, orange] (a1b2) -- (N) coordinate[pos=.25] (cs);
\draw[postaction={decorate}, orange] (b2c1) -- (N) coordinate[pos=.25] (as*);
\draw[postaction={decorate}, orange] (c1a2) -- (N) coordinate[pos=.25] (bs);
\draw[postaction={decorate}, orange] (c2a1) -- (N) coordinate[pos=.25] (bs*);
\end{scope}
\draw[very thick] (a1) -- (b2) -- (c1) -- (a2) -- (b1) -- (c2) -- cycle;
\node[scale=\nodeScale, above left] at (bs) {\Huge $y^*$};
\node[scale=\nodeScale, above] at (as) {\Huge $y''^*$};
\node[scale=\nodeScale, above right] at (cs*) {\Huge $y'$};
\filldraw[orange] (N) circle (0.05em);
\filldraw[orange] (S) circle (0.05em);
\filldraw (a1) circle (0.05em);
\filldraw (b2) circle (0.05em);
\filldraw (c1) circle (0.05em);
\filldraw (a2) circle (0.05em);
\filldraw (b1) circle (0.05em);
\filldraw (c2) circle (0.05em);      
\end{scope}
\end{tikzpicture}

%% file: figures/lift_with_detour_reduced.tex
\tdplotsetmaincoords{45}{-90}
\begin{tikzpicture}[tdplot_main_coords]
\pgfmathsetmacro{\nodeScale}{.75}
\begin{scope}[tdplot_main_coords]
\coordinate (N) at (0, 0, 1);
\coordinate (S) at (0, 0, -1);   
\coordinate (a1) at ({2*sqrt(2)*cos(0)}, {2*sqrt(2)*sin(0)}, 0);
\coordinate (b2) at ({2*sqrt(2)*cos(60)}, {2*sqrt(2)*sin(60)}, 0);
\coordinate (c1) at ({2*sqrt(2)*cos(120)}, {2*sqrt(2)*sin(120)}, 0);
\coordinate (a2) at ({2*sqrt(2)*cos(180)}, {2*sqrt(2)*sin(180)}, 0);
\coordinate (b1) at ({2*sqrt(2)*cos(240)}, {2*sqrt(2)*sin(240)}, 0);
\coordinate (c2) at ({2*sqrt(2)*cos(300)}, {2*sqrt(2)*sin(300)}, 0);    
\coordinate (a1b2) at ($1/2*(a1) + 1/2*(b2)$);
\coordinate (b2c1) at ($1/2*(b2) + 1/2*(c1)$);
\coordinate (c1a2) at ($1/2*(c1) + 1/2*(a2)$);
\coordinate (a2b1) at ($1/2*(a2) + 1/2*(b1)$);
\coordinate (b1c2) at ($1/2*(b1) + 1/2*(c2)$);
\coordinate (c2a1) at ($1/2*(c2) + 1/2*(a1)$);
\coordinate (kb1) at ($(b1c2)!.6!(N)$);
\coordinate (kb2) at ($(a2b1)!.4!(N)$);  
\coordinate (kb3) at ($(a2b1)!.6!(N)$);    
\coordinate (kb4) at ($(c1a2)!.6!(N)$);    
\draw[very thick, red] (N) -- (S);    
\begin{scope}[thick,decoration={
markings,
mark=at position 0.5 with {\arrow[scale=.75]{>}}}
] 
\draw[postaction={decorate}, dashed, orange] (c2a1) -- (S) coordinate[pos=.25] (bt);
\draw[postaction={decorate}, dashed, orange] (a1b2) -- (S) coordinate[pos=.25] (ct*);
\draw[postaction={decorate}, dashed, orange] (b2c1) -- (S)
coordinate[pos=.25] (at);
\draw[postaction={decorate}, dashed, orange] (c1a2) -- (S) coordinate[pos=.25] (bt*);
\draw[postaction={decorate}, dashed, orange] (a2b1) -- (S) coordinate[pos=.25] (ct);
\draw[postaction={decorate}, dashed, orange] (b1c2) -- (S) coordinate[pos=.25] (at*);
\end{scope}
\begin{scope}[thick,decoration={
markings,
mark=at position 0.4 with {\arrow[scale=.75]{>}}}]
\draw[white,very thick] (kb1) .. controls ({3/4*cos(240)},{3/4*sin(240)},.4) ..(kb2);
\draw[blue, ultra thick, postaction={decorate}] (kb1) .. controls ({3/4*cos(240)},{3/4*sin(240)},.4) ..(kb2);
\draw[white,very thick] (kb3) .. controls ({3/4*cos(180)},{3/4*sin(180)},.4) ..(kb4);
\draw[blue, ultra thick, postaction={decorate}] (kb3) .. controls ({3/4*cos(180)},{3/4*sin(180)},.4) ..(kb4);
\end{scope}
\begin{scope}[thick,decoration={
markings,
mark=at position 0.5 with {\arrow[scale=.75]{>}}}
] 
\draw[postaction={decorate}, white,very thick] ($(a2b1)!.25!(N)$) -- ($(a2b1)!.75!(N)$) coordinate[pos=.25] (cs*);
\draw[postaction={decorate}, orange] (a2b1) -- (N) coordinate[pos=.25] (cs*);
\draw[postaction={decorate}, orange] (b1c2) -- (N)coordinate[pos=.25] (as);
\draw[postaction={decorate}, orange] (a1b2) -- (N) coordinate[pos=.25] (cs);
\draw[postaction={decorate}, orange] (b2c1) -- (N) coordinate[pos=.25] (as*);
\draw[postaction={decorate}, orange] (c1a2) -- (N) coordinate[pos=.25] (bs);
\draw[postaction={decorate}, orange] (c2a1) -- (N) coordinate[pos=.25] (bs*);
\end{scope}
\draw[very thick] (a1) -- (b2) -- (c1) -- (a2) -- (b1) -- (c2) -- cycle;
\node[scale=\nodeScale, above left] at (bs) {\Huge $y^*$};
\node[scale=\nodeScale, above] at (as) {\Huge $y''^*$};
\node[scale=\nodeScale, above right] at (cs*) {\Huge $y'$};
\filldraw[orange] (N) circle (0.05em);
\filldraw[orange] (S) circle (0.05em);
\filldraw (a1) circle (0.05em);
\filldraw (b2) circle (0.05em);
\filldraw (c1) circle (0.05em);
\filldraw (a2) circle (0.05em);
\filldraw (b1) circle (0.05em);
\filldraw (c2) circle (0.05em);      
\end{scope}
\end{tikzpicture}